\documentclass{amsart}
\usepackage{amsxtra, microtype}
\usepackage{stmaryrd }
\usepackage{mathrsfs}
\usepackage{amssymb,amsmath,amsthm,stmaryrd,latexsym,wasysym}
\usepackage[all]{xy}
\usepackage{hyperref}
\usepackage{tikz}
\usetikzlibrary{decorations.pathmorphing}
\usetikzlibrary{arrows}
\usepackage{cancel}
\usepackage{appendix}
\theoremstyle{plain}
\usepackage{xcolor}
\usepackage{xspace}
\newtheorem{theorem}{Theorem}[section]
\newtheorem{lemma}[theorem]{Lemma}
\newtheorem{proposition}[theorem]{Proposition}
\newtheorem{corollary}[theorem]{Corollary}
\newtheorem{definition}[theorem]{Definition} \theoremstyle{definition}
\newtheorem{example}[theorem]{Example}
\newtheorem{remark}[theorem]{Remark}
\newcommand{\lie}[1]{\mathfrak{#1}}

\newcommand{\duer}{\mathbin{\raisebox{3pt}{\varhexstar}\kern-3.70pt{\rule{0.15pt}{4pt}}}\,}

\newcommand{\E}{{\mathbb{E}}}

\newcommand{\R}{\mathbb{R}} 
 
\newcommand{\inv}{^{-1}}
\newcommand{\N}{\mathbb{N}} 

\newcommand{\mx}{\mathfrak{X}} 
\newcommand{\dr}{\mathbf{d}}
\newcommand{\ldr}[1]{{{\pounds}}_{#1}}
\newcommand{\ip}[1]{{\mathbf{i}}_{#1}}
\newcommand{\an}[1]{\arrowvert_{#1}} 
\newcommand{\lb}{\llbracket} 
\newcommand{\rb}{\rrbracket}
\newcommand{\Beta}{\boldsymbol{\beta}}

\DeclareMathOperator{\dom}{Dom}

\DeclareMathOperator{\pr}{pr}
\DeclareMathOperator{\Hom}{Hom}

\DeclareMathOperator{\Id}{Id}
\DeclareMathOperator{\Jac}{Jac}
\DeclareMathOperator{\der}{Der}
\newcommand{\nsp}[2]  {%
\langle\mspace{-6.8mu}%
\langle\mspace{-6.8mu}%
\langle\mspace{-6.8mu}%
\langle\mspace{-6.8mu}%
\langle\mspace{-6.8mu}%
\langle\mspace{-6.8mu}%
\langle{#1,\,}{#2}%
\rangle%
\mspace{-6.8mu}\rangle%
\mspace{-6.8mu}\rangle%
\mspace{-6.8mu}\rangle%
\mspace{-6.8mu}\rangle%
\mspace{-6.8mu}\rangle%
\mspace{-6.8mu}\rangle}
\allowdisplaybreaks

\begin{document}
%%%%%%%%%%%%%%%%%%%%%%%%%%%%%%%%%%%%%%%%%%%%%%%%%%%%%%%%%%%%%%%%%%%%%%%%%%%
%%%%%%%%%%%%%%%%%%%%%%    Title    %%%%%%%%%%%%%%%%%%%%%%%%%%%%%%%%%%%%%%%%
\title{N-manifolds of degree 2 and  metric double vector bundles.}

%%% author one information

\author{M. Jotz Lean} \address{School of Mathematics and Statistics,
  The University of Sheffield.}  \email{M.Jotz-Lean@sheffield.ac.uk}
\thanks{Research partially supported by a \emph{fellowship for
    prospective researchers (PBELP2\_137534) of the Swiss NSF} for a
  postdoctoral stay at UC Berkeley.}

\begin{abstract}
  This paper shows the equivalence of the categories of $N$-manifolds
  of degree $2$ with the category of double vector bundles endowed
  with a linear metric.

  Split Poisson $N$-manifolds of degree $2$ are shown to be equivalent
  to \emph{self-dual representations up to homotopy}. As a
  consequence, the equivalence above induces an equivalence between so
  called \emph{metric VB-algebroids} and Poisson $N$-manifolds of
  degree $2$.

  Then a new, simple description of split Lie $2$-algebroids is given,
  as well as their ``duals'', the \emph{Dorfman $2$-representations}.
  We show that Dorfman $2$-representations are equivalent in a simple
  manner to \emph{Lagrangian splittings} of VB-Courant
  algebroids. This yields the equivalence of the categories of Lie
  $2$-algebroids and of VB-Courant algebroids.  We give several
  natural classes, some of them new, of examples of split Lie
  $2$-algebroids and of the corresponding VB-Courant algebroids.

  We then show that a split Poisson Lie $2$-algebroid is equivalent to
  the ``matched pair'' of a Dorfman $2$-representation with a
  self-dual representation up to homotopy.  We deduce a new proof of
  the equivalence of categories of LA-Courant algebroids and Poisson
  Lie $2$-algebroids.  We show that the core of an LA-Courant
  algebroid inherits naturally the structure of a degenerate Courant
  algebroid, as a kind of bicrossproduct of the Poisson bracket and
  the Dorfman connection defined by a Lagrangian splitting.  This
  yields a new formula to retrieve in a direct manner the Courant
  algebroid found by Roytenberg to correspond to a symplectic Lie
  $2$-algebroid.

  Finally we study VB- and LA-Dirac structures in VB- and LA-Courant
  algebroids. As an application, we extend Li-Bland's results on
  pseudo-Dirac structures and we construct a Manin pair associated to
  an LA-Dirac structure.
\end{abstract}
\maketitle

\newpage

\tableofcontents

\section{Introduction}

Lie and Courant algebroids are differential geometric objects that are
useful in encoding infinitesimal symmetry. Lie algebroids generalise
tangent bundles and Lie algebras. They were originally defined by Jean
Pradines \cite{Pradines67} as the infinitesimal counterpart of Lie
groupoids, just as Lie algebras describe Lie groups
infinitesimally. The integration of Lie algebroids to Lie groupoids
has only been fully understood around ten years ago, in seminal papers
of Crainic and Fernandes \cite{CrFe05} and Cattaneo and Felder
\cite{CaFe01}.

The standard Courant algebroid over a manifold was discovered in the
late eighties by (Ted) Courant in his study of Dirac structures
\cite{Courant90a}. He had earlier discovered Dirac structures with his
adviser Weinstein \cite{CoWe88}; Dirac structures arose from Dirac's
theory of constraints and encompass Poisson structures, closed
two-forms, and regular foliations \cite{Courant90a}.  A few years
later, Liu, Weinstein and Xu defined general Courant algebroids and
showed that the ``bicrossproducts'' of Lie bialgebroids are examples of
Courant algebroids \cite{LiWeXu97}. The standard Courant algebroids
over smooth manifolds turned out to be crucial in Hitchin's
definition of generalised complex structures \cite{Hitchin03}, which
were then further developed by his students Cavalcanti and Gualtieri
\cite{Gualtieri03}. Severa described in \cite{Severa05} how the
standard Courant algebroid describes symmetries of variational
problems.

Despite a great and growing interest in Courant algebroids, still
little is known about them. The two main goals in their study are now
their integration, and understanding their representation theory.  The
main tool for investigations in that direction is the well-known, but
rather complicated, equivalence between Courant algebroids and
symplectic Lie $2$-algebroids. One of the features of this paper is a
thorough explanation of this equivalence, in a very simple and
innovative manner.

\subsubsection*{Lie $1$- and Lie $2$-algebroids.}
For clarity let us recall a few definitions.  %A Lie
% algebroid is a vector bundle $A\to M$ together with a vector bundle
% morphism $\rho\colon A\to TM$ (the anchor) over the identity and a Lie
% algebra bracket
% $[\cdot\,,\cdot]\colon\Gamma(A)\times\Gamma(A)\to\Gamma(A)$ such that
% $[a_1, f a_2]=f[a_1, a_2]+\ldr{\rho(a_1)}f\cdot a_2$ for all $a_1,
% a_2\in\Gamma(A)$ and $f\in C^\infty(M)$.  This and the Jacobi identity
% imply then $\rho[a_1,a_2]=[\rho(a_1),\rho(a_2)]$ for all
% $a_1,a_2\in\Gamma(A)$.  The reader will immediately notice that a Lie
% algebroid over a point is a Lie algebra, and that the tangent bundle
% of a manifold, together with the identity as anchor and the Lie
% bracket of vector fields, is a Lie algebroid.
An \textbf{$\N$-graded manifold} $\mathcal M$ of degree $n$ and
dimension $(p; r_1,\ldots, r_n)$ is a smooth $p$-dimensional manifold
$M$ endowed with a sheaf $C^\infty(\mathcal M)$ of $\N$-graded
commutative associative unital $\R$-algebras, which can locally be written as
\[C^\infty_U(M)[\xi_1^{1},\ldots,\xi_1^{r_1},\xi_2^1,\ldots,\xi_2^{r_2},\ldots,\xi_n^1,\ldots,\xi_n^{r_n}]\]
with $\xi_i^j$ of degree $i$. % for $i\in\{1,\ldots,n\}$ and
% $j\in\{1,\ldots,r_i\}$
For instance, an $\N$-graded manifold of degree $1$ is just a locally
finitely generated sheaf of $C^\infty(M)$-modules, hence canonicaly
isomorphic to the set of sections of a vector bundle:
$C^\infty(\mathcal M)=\Gamma(\wedge^{\bullet} E^*)$ for a vector
bundle $E\to M$.  On a local chart $U$ of $M$ trivialising $E$, we
have coordinates $x_1,\ldots, x_n$, i.e.~functions of degree $0$ on
$\mathcal M$. We also have local basis sections $e_1,\ldots,e_r$ of
$E$ and the dual local basis sections $\xi_1,\ldots,\xi_r$ of $E^*$;
which are seen as functions of degree $1$ on $\mathcal
M$.  % The graded algebra
% $C^\infty(U)[\xi_1,\ldots,\xi_r]$ generated by $\xi_1,\ldots,\xi_r$
% is exactly $\Gamma_U(\wedge^{\bullet} E^*)$.

A vector field of degree $j$ on an $\N$-graded manifold is a graded
derivation that increases the degree by $j$. % Graded vector fields of
% odd degree do not necessarily commute with themselves.
An $NQ$-manifold of degree $1$ is an $\N$-graded manifold
$\mathcal M$ of degree $1$, with a vector field $\mathcal Q$ of degree
$1$ that commutes with itself; $[\mathcal Q,\mathcal Q]=0$. The vector
field $\mathcal Q$ is then called a \textbf{homological vector field}.  Take
again an $N$-manifold of degree $1$, i.e.~$C^\infty(\mathcal
M)=\Gamma(\wedge^{\bullet} E^*)$ for a vector bundle $E$ over $M$ and take a
trivialising chart $U\subseteq M$ for $E$.  The vector fields
$\partial_{x_i}$ have degree $0$ and the vector fields
$\partial_{\xi_i}$ have degree $-1$. % Furthermore,
% $\partial_{\xi_i}$ can be identified with $e_i\in\Gamma(E)$, since it
% sends $\xi_j$ to the Kronecker symbol $\delta_{i}^j$.
Assume that $E$ has a Lie algebroid structure with anchor $\rho\colon
E\to TM$ and bracket $[\cdot\,,\cdot]$, and define $\mathcal Q$ on
$\mathcal M$ by
\begin{equation}\label{Q_1_in_coor}
\mathcal Q\an{U}=\sum_{i=1}^n\sum_{j=1}^r \rho(e_j)(x_i)\xi_j\partial_{x_i}
-\sum_{i<j}^r\sum_{k=1}^r\langle [e_i,e_j], \xi_k\rangle \xi_i\xi_j\partial_{\xi_k}.
\end{equation}
A quick degree count shows that $\mathcal Q$ has degree $1$.  Let us
check that $\mathcal Q\circ\mathcal Q=0$, i.e.~that $\mathcal Q$ is a
homological vector field. First note that $f\in C^\infty(U)$ is sent
by $\mathcal Q$ to $\mathcal Q(f)% =\sum_{i=1}^n\sum_{j=1}^r
% \rho(e_j)(x_i)\xi_j\frac{\partial f}{\partial_{x_i}}
= \sum_{j=1}^r
\rho(e_j)(f)\xi_j=\rho^*\dr f$.  Then \begin{equation*}
\begin{split}
  \mathcal Q^2(f)&= \mathcal Q\left(\sum_{j=1}^r
    \rho(e_j)(f)\xi_j\right)\\
&= \sum_{j=1}^r\rho^*\dr(\rho(e_j)(f))\xi_j-\sum_{j=1}^r
  \rho(e_j)(f)\sum_{s<t}\langle [e_s,e_t], \xi_j\rangle \xi_s\xi_t\\
  &=\sum_{t=1}^r\sum_{s=1}^r\rho(e_s)\rho(e_t)(f)\xi_s\xi_t-\sum_{s<t}\rho[e_s,e_t](f) \xi_s\xi_t% \\
  % &=
  % \sum_{s<t}(\rho(e_s)\rho(e_t)(f)-\rho(e_t)\rho(e_s)(f)-\rho[e_s,e_t](f))
  % \xi_s\xi_t=0
\end{split}\end{equation*}
which vanishes since $\rho[e_s,e_t]=[\rho(e_s),\rho(e_t)]$. Then we
find that $\mathcal Q(\xi_k)$ equals $-\sum_{i<j}^r\langle [e_i,e_j],
\xi_k\rangle \xi_i\xi_j$. As an element of $\Gamma(\wedge^2E^*)$, this
is $\dr_E\xi_k$, where $\dr_E$ is the operator defined on
$\Gamma(\wedge^\bullet E^*)$ by the Lie algebroid structure on $E$.  A
similar computation as the one above shows that the Jacobi identity
implies $\mathcal Q^2(\xi_k)=0$ for all $k$. Thus we have found that a
Lie algebroid structure on $E$ defines a homological vector field
$\mathcal Q=\dr_E$ on the corresponding N-manifold of degree $1$.

Conversely, any homological vector field $\mathcal Q$ on a degree $1$
graded manifold can be written as
$\sum_{ij}f^{ij}\epsilon_i\partial_{x_j}+\sum_{ijk}g^{ijk}\epsilon_i\epsilon_j\partial_{\epsilon_k}$
with smooth functions $f^{ij}$, $g^{ijk}\in C^\infty(U)$.  Setting
$f^{ij}=\rho(e_j)(x_i)$ and $g^{ijk}=-\langle [e_i,e_j], \xi_k\rangle
$ and extending using the Leibniz identities \emph{defines} then
locally a Lie algebroid structure on $E\an{U}$, which can further be
checked to be global since $\mathcal Q$ does not depend on the local
coordinates.  Hence, $NQ$-manifolds of degree $1$ are equivalent to
Lie algebroids.  This result, due to Arkady Vaintrob
\cite{Vaintrob97}, is the reason why $NQ$-manifolds of degree $1$ are
called Lie $1$-algebroids and $NQ$-manifolds of degree $n\geq 1$ are
called Lie $n$-algebroids.
Let us describe yet another way to get the Lie algebroid structure
from the homological vector field $\mathcal Q$.  A study of Equation
\eqref{Q_1_in_coor} shows that $[\mathcal Q, e](f)=\rho(e)(f)$ and
$[[\mathcal Q, e], e']=[e,e']$, if $e\in\Gamma(E)$ is identified with
the graded vector field $e$ of degree $-1$ that sends
$\xi\in\Gamma(E^*)$ to $\langle e, \xi\rangle $ and $f\in C^\infty(M)$
to $0$. The Lie algebroid structure is hence \emph{derived} from the
homological vector field.

\medskip

We now turn to the case of degree $2$ N-manifolds, and in
particular the one of Lie $2$-algebroids. The equivalence between
\emph{symplectic} degree $2$ N-manifolds and vector bundles
with a metric, and the one between \emph{symplectic} Lie
$2$-algebroids and Courant algebroids are due to Dmitry Roytenberg
\cite{Roytenberg02}.  Li-Bland established in his thesis
correspondences between Lie $2$-algebroids and
VB-Courant algebroids and between \emph{Poisson} Lie
$2$-algebroids and LA-Courant algebroids \cite{Li-Bland12}.

These correspondence are full of insight, but are not explained in the
literature as concretely as the one of NQ-manifolds of degree $1$ with
Lie algebroids. This paper geometrises N-manifolds of degree $2$: We
find an equivalence between the category of N-manifolds of degree $2$
and double vector bundles endowed with a linear metric. Then we deduce
correspondences between geometric structures on both sides.

% Here again, let us recall the definitions of the notions that we
% mention.  A Courant algebroid over a manifold $M$ is a vector bundle
% $\mathsf E$ over $M$ equipped with a fibrewise nondegenerate symmetric
% bilinear form $\langle\cdot\,,\cdot\rangle$, a bilinear bracket
% $\lb\cdot\,,\cdot\rb$ on the smooth sections $\Gamma(\mathsf E)$, and a
% vector bundle map $\rho\colon \mathsf E\to TM$, called the anchor,
% which satisfy the following conditions: $\lb e_1, \lb e_2, e_3\rb\rb = \lb\lb e_1,
% e_2\rb, e_3\rb + \lb e_2, \lb e_1, e_3\rb\rb$, $\rho(e_1 )\langle e_2, e_3\rangle =
% \langle\lb e_1, e_2\rb, e_3\rangle + \langle e_2, \lb e_1 , e_3\rb\rangle$,
% % \item $\lb e_1, e_1\rb = \frac{1}{2}\mathcal D\langle e_1 , e_1\rangle$,
%   and $\lb e_1, e_2\rb +\lb e_2, e_1\rb =\rho^*\dr\langle e_1 , e_2\rangle$, for
%   all $e_1, e_2, e_3\in\Gamma(\mathsf E)$
%   \cite{Courant90a,LiWeXu97,Roytenberg99}.  In the third axiom
%   $\mathsf E$ is identified with $\mathsf E^*$ using the pairing.

\subsubsection*{Metric double vector bundles}
A double vector bundle is a commutative square
\begin{equation*}
\begin{xy}
\xymatrix{
D \ar[r]^{\pi_B}\ar[d]_{\pi_A}& B\ar[d]^{q_B}\\
A\ar[r]_{q_A} & M}
\end{xy}
\end{equation*}
of vector bundles such that the structure maps of the vertical bundles
define morphisms of the horizontal bundles. Each of the vector bundles
$D\to A$ and $D\to B$ has very useful sections to work with; the
\emph{linear} and the \emph{core sections}.
% such that
% $(d_1+_Ad_2)+_B(d_3+_Ad_4)=(d_1+_Bd_3)+_A(d_2+_Bd_4) $ for all
% $d_1,d_2,d_3,d_4\in D$ with $\pi_A(d_1)=\pi_A(d_2)$,
% $\pi_A(d_3)=\pi_A(d_4)$, $\pi_B(d_1)=\pi_B(d_3)$ and
% $\pi_B(d_2)=\pi_B(d_4)$. 
Take a triple $A,B,C$ of vector bundles over a smooth manifold
$M$. Then the fibre-product $A\times_M B\times_M C$ has a vector
bundle structure over $A$ given by
$(a,b,c)+_A(a,b',c')=(a,b+b',c+c')$, and
similarly a vector bundle structure over $B$. This defines a
commutative square as above and thus a double vector bundle structure
on $A\times_M B\times_M C\to B$.  This type of double vector bundle is
called \emph{decomposed}. Any double vector bundle is non-canonically
isomorphic to a decomposed one.  A \textbf{VB-Courant algebroid} is a double
vector bundle $(D,B,A,M)$  with a Courant algebroid structure
on $D\to B$ that is linear, i.e.~compatible with the other vector
bundle structure on $D$. An \textbf{LA-Courant algebroid} is a VB-Courant
algebroid with an additional Lie algebroid structure on $D\to A$ that
is linear, i.e.~compatible with the vector bundle structure on $D\to
B$, and that is also compatible with the Courant algebroid structure
in a sense that is explained in detail in this paper.

Up to now, the homological vector fields and the symplectic and
Poisson structures corresponding to Courant algebroids and VB-Courant
algebroids have never been written in coordinates as we do it in
\eqref{Q_1_in_coor} for Lie $1$-algebroids, and even the construction
of these structures as derived structures from the homological vector
field was not straightforward in practice. We remedy to this by giving
a simple description of the VB-Courant algebroid that is defined by a
Lie $2$-algebroid, and vice-versa. More precisely, this paper explains
in detail how to construct a \emph{decomposed} VB-Courant algebroid
from a \emph{split} Lie $2$-algebroid.  In order to do this, we give a
simplified definition of split Lie $2$-algebroids.

Our main result, advertised in the title of this paper, is an explicit
equivalence between the category of degree $2$ N-manifolds with a
category of double vector bundles with a linear non-degenerate pairing
over one of their sides.  We call the latter \textbf{metric double vector
bundles}.  From this theorem follow many enlightening results on
geometric structures on degree $2$ N-manifolds and on their
counterparts on metric double vector bundles, reflecting 
equivalences between the category of Poisson degree $2$
N-manifolds and the category of metric VB-algebroids, between the
category of Lie $2$-algebroids and the category of VB-Courant
algebroids, and between the category of Poisson Lie $2$-algebroids and
the category of LA-Courant algebroids.

\subsubsection*{Self-dual representations up to homotopy, Dorfman $2$-representations, etc.}
We prove along the way that split Poisson N-manifolds of degree $2$
are equivalent to \emph{self-dual $2$-term representations up to
  homotopy}, that split Lie $2$-algebroids are dual to \emph{Dorfman
  $2$-representations}, which resemble very much $2$-term
representations up to homotopy and demonstrate how Lie $2$-algebroids
can really be understood as Lie algebroids up to homotopy. We then
find a new way of checking if a Poisson bracket of degree $-2$ on an
N-manifold of degree $2$ is invariant under a homological vector
field, making the N-manifold into a Poisson Lie $2$-algebroid. In any
splitting of the underlying $N$-manifold, the self-dual representation
up to homotopy and the Dorfman $2$-representation defined by the
Poisson structure and the homological vector field, respectively, have
to form a matched pair, the definition of which resembles the one of
matched pairs of $2$-term representations up to homotopy
\cite{GrJoMaMe14}. This provides on the classical side a new and much
simpler way of defining LA-Courant algebroids, and using the result in
\cite{GrJoMaMe14}, shows immediately that LA-Dirac structures in
LA-Courant algebroids are double Lie algebroids.

We exhibit several new examples of (Poisson) Lie $2$-algebroids, and
show in particular that the bicrossproduct of a matched pair of
$2$-term representations up to homotopy is a split Lie $2$-algebroid
(see Theorem \ref{double_2_rep}), just as the bicrossproduct of a
matched pair of Lie algebroid representations is a Lie algebroid.
This result is independently interesting, because it finally unifies
in a natural framework the two notions of double of a matched pair of
representations.  A matched pair of representations of two Lie
algebroids $A$ and $B$ over $M$ defines a Lie algebroid structure on
$A\oplus B$ \cite{Mokri97}, sometimes called the double of the matched
pair, but called here the bicrossproduct of the matched pair.  The
matched pair defines also a double Lie algebroid $A\times B$ with
sides $A$ and $B$ and trivial core. This double Lie algebroid is
called the double of the matched pair.  Similarly, we know now that a
matched pair of $2$-term representations up to homotopy has a split
Lie $2$-algebroid as bicrossproduct, and a decomposed double Lie
algebroid as double. The split Lie $2$-algebroid is exactly the
N-geometric counterpart of the VB-Courant algebroid that is equivalent
to the double Lie algebroid.  The case of matched pair of
representations and their bicrossproducts and doubles are in fact a
special (degenerate) case of this more general equivalence; namely the
one of a linear Courant algebroid over a trivial base.

We find as well that a matched pair of $2$-term representations up to
homotopy defines in a less conventional, but still very natural manner
two Poisson Lie $2$-algebroids.

In Figure \ref{fig:classic} is a table of all the classical
differential geometric objects that we encounter in this paper.  The
ordinary arrows are the obvious forgetful functors. The snake arrows
correspond to constructions in \S\ref{VB_courant_from_double} (and
\S\ref{matched_pair_2_rep_sec}) and
\S\ref{Poi_lie_2_def_by_matched_ruth}. We do not discuss these arrows
as functors, but it would be a good exercise to do
so.  Note that the triangle on the
bottom left does not commute; we exhibit two different constructions
of a VB-Courant algebroid and of an LA-Courant algebroid from a double
Lie algebroid.  The hooked arrows are embeddings. In Figure \ref{fig:super} is a very similar table,
with the N-geometric counterparts of the objects in Figure
\ref{fig:classic}.

\begin{figure}[h]
\caption{Diagrammatic table of the (classical, double) geometric
  objects in this paper. }\label{fig:classic}
\includegraphics[scale=0.8]{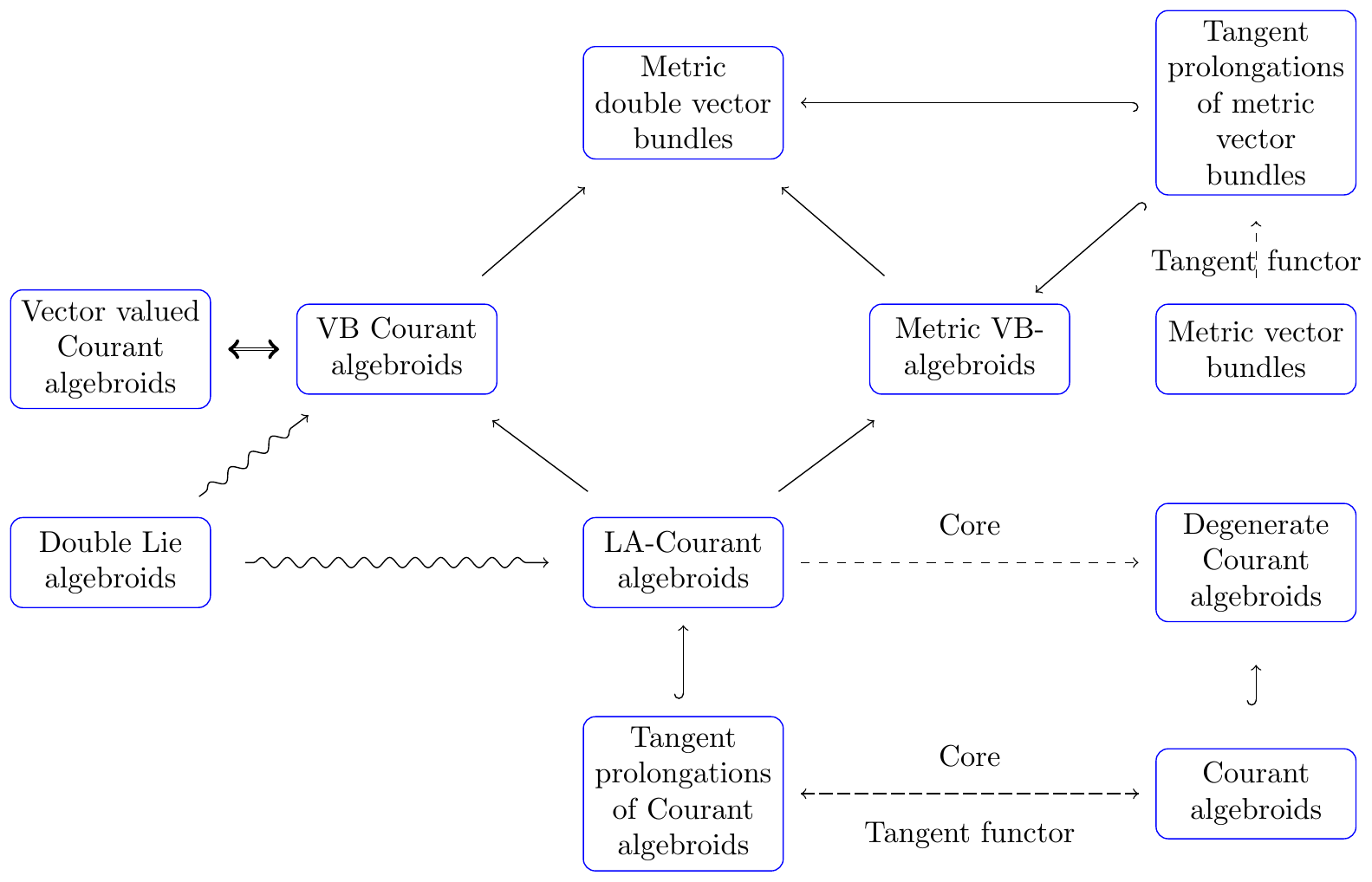}
\end{figure}

\begin{figure}
\caption{Diagrammatic table of the supergeometric objects in this
  paper. }\label{fig:super}
\includegraphics[scale=0.8]{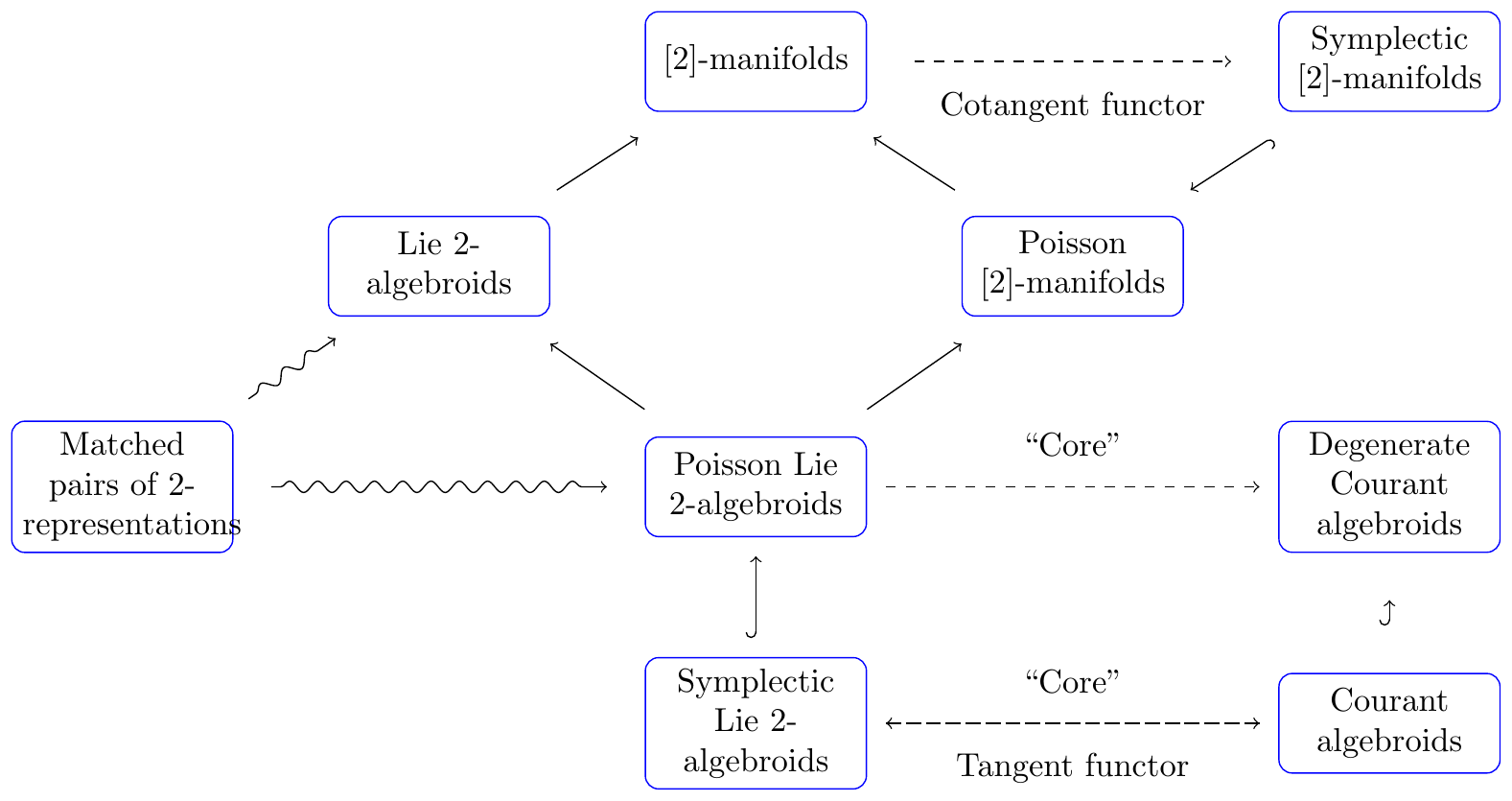}
\end{figure}

The range of applications of our results culminates in Theorem
\ref{culminate} with a new method exhibiting explicitly the Courant
algebroid structure arising from a symplectic Lie $2$-algebroid, and
thus contradicting the common wisdom that the Courant bracket can only
be obtained as a derived bracket from a symplectic Lie $2$-algebroid.
The core of the LA-Courant algebroid corresponding to a Poisson Lie
$2$-algebroid inherits the structure of a degenerate Courant
algebroid, which looks like the correction of a Dorfman connection
(one of the ingredients of the Dorfman $2$-representation) by the
Poisson bracket. In the symplectic case, the Courant algebroid is
non-degenerate, hence a Courant algebroid in the ordinary sense, and
exactly the one that is found following Roytenberg's approach to be
equivalent to the symplectic Lie $2$-algebroid. Unsurprisingly, the
ambient LA-Courant algebroid is nothing else than the tangent double
of the Courant algebroid.

This explains what might seem at first sight confusing in comparing
Li-Bland's work with Roytenberg's results. As symplectic manifolds are
special (non-degenerate) Poisson manifolds, symplectic Lie
$2$-algebroids form a special class of Poisson Lie
$2$-algebroids. Hence, the ``classical'' counterparts of symplectic
Lie $2$-algebroids should form a subcategory of the classical
counterparts of Poisson Lie $2$-algebroids.  We adopt the view (and
prove) that symplectic Lie $2$-algebroids are in fact equivalent to
tangent doubles of Courant algebroids, and that the Courant algebroids
are in fact the core structures emanating from those.

\subsubsection*{Original motivation}
Let us explain in more detail our methodology and our original
motivation.  A VB-Lie algebroid is a double vector bundle with one
side $D\to B$ endowed with a Lie algebroid bracket and an anchor that
are \emph{linear}. The side $A\to M$ then inherits a natural Lie
algebroid structure. Take for simplicity a decomposed VB-algebroid
$(A\times_M B\times_M C\to B, A\to M)$.  Here, linearity of the
bracket means that the bracket of two core sections\footnote{The
  sections of $A$, $B$ and $C\to M$ define very useful sections of the
  double vector bundle $A\times_M B\times_M C$: a section
  $a\in\Gamma(A)$ defines a linear section $a^l$ of $A\times_M
  B\times_M C\to B$: $a^l(b_m)=(a(m),b_m,0^C_m)$. The map $\tilde
  a\colon B\to A\times_M B\times_M C$ is then a vector bundle morphism
  over $a\colon M \to A$.  A section $c\in\Gamma(C)$ defines a core
  section $c^\dagger$ of $A\times_M B\times_M C\to B$:
  $c^\dagger(b_m)=(0^A_m,b_m,c(m))$. For
  $\phi\in\Gamma(\operatorname{Hom}(B,C))$, the linear section
  $\widetilde{\phi}$ sends $b_m$ to
  $(0^A_m,b_m,\phi(b_m))$.} is zero, the bracket of a core
section with a linear section is a core section, and the bracket of
two linear sections is again linear. In formulae
\begin{equation*} [c_1^\dagger, c_2^\dagger]=0, \quad [a^l,
  c^\dagger]=\nabla_ac^\dagger, \quad [a_1^l,
  a_2^l]=[a_1,a_2]^l-\widetilde{R(a_1,a_2)},
\end{equation*}
defining so a linear connection
$\nabla\colon\Gamma(A)\times\Gamma(C)\to\Gamma(C)$ and a vector valued
two-form $R\in\Omega^2(A,\operatorname{Hom}(B,C))$.
Linearity of the anchor means that the anchor of a linear section is a
linear vector field on the base $B$, and that the anchor of a core
section is a vertical vector field on $B$.  In formulae again
\begin{equation*} \rho_D(c^\dagger)=(\partial c)^\uparrow, \qquad 
\rho_D(a^l)=\widehat{\nabla_a},
\end{equation*}
defining so a linear connection $\nabla\colon
\Gamma(A)\times\Gamma(B)\to\Gamma(B)$ and a vector bundle morphism
$\partial\colon C\to B$. Gracia-Saz and Mehta prove in \cite{GrMe10a}
that the two connections, the bundle map and the vector-valued
$2$-form are the ingredients of a super-representation, also known as
$2$-term representation up to homotopy.

The definition of a VB-Courant algebroid is very similar to the one of
VB-algebroids.  The Courant bracket, the anchor and the non-degenerate
pairing all have to be linear. In 2012 we proved that the standard
Courant algebroid over a vector bundle can be decomposed into a
connection, a Dorfman connection, a curvature term and a vector bundle
map, in a manner that resembles very much the one in
\cite{GrMe10a}. There our original motivation was to prove that, as
linear splittings of the tangent space $TE$ of a vector bundle $E$ are
equivalent to linear connections on the vector bundle, linear
splittings of the Pontryagin bundle $TE\oplus T^*E$ over $E$ are
equivalent to a certain class of Dorfman connections \cite{Jotz13a}.

It was then very natural to show that this was in fact a very special
case of a general decomposition of VB-Courant algebroids, in the
spirit of Gracia-Saz and Mehta's result. For very long we worked with
general splittings of VB-Courant algebroids, exhibiting very promising
formulae and objects looking like the ``Courant counterpart'' to
$2$-term representations up to homotopy. Nonetheless these objects
were hard to grasp conceptually and to relate to Lie $2$-algebroids,
which were already known to correspond to VB-Courant algebroids. It
was only when we implemented the idea of working only with maximally
isotropic horizontal spaces in the metric double vector bundles
underlying VB-Courant algebroids, that the geometric objects that we
had found simplified to what could be called Lie derivatives up to
homotopy, or Lie $2$-derivatives, in duality to Lie
$2$-algebroids. For simplicity we called them Dorfman
$2$-representations, because Dorfman representations (skew-symmetric
and flat Dorfman connections) are exactly the Lie derivatives in
duality with Lie algebroids \cite{Jotz13a}. The next main consequence
of this new method was our discovery of the equivalence of metric
double vector bundles with N-manifolds of degree $2$, which is now the
core of this paper.

\subsection*{Outline, main results and applications}
This paper is organised as follows.
\paragraph{\textbf{Section \ref{sect:b+d}}} 
We start by quickly recalling how vector bundle morphisms are
equivalent to morphisms of the sheaves of sections of the dual
bundles.  We then discuss in detail the necessary background on double
vector bundles and their dualisation and splittings. We recall how
linear splittings of VB-algebroids induce $2$-term representations up
to homotopy \cite{GrMe10a} and how matched pairs of $2$-term
representations up to homotopy correspond in this manner to linear
splittings of double Lie algebroids \cite{GrJoMaMe14}.

\paragraph{\textbf{Section \ref{sec:metDVB}}}
In this section we recall the definition of $\N$-graded manifolds
(N-manifolds in short) and we recall the equivalence of $N$-manifolds
of degree $1$ with vector bundles.  Then we define metric double
vector bundles and their Lagrangian splittings, the existence of which we
prove. We describe the morphisms in the category of metric double
vector bundles, and we construct an equivalence of this category with the
category of $N$-manifolds of degree $2$.

\paragraph{\textbf{Section \ref{sec:Poisson}}}
We study Poisson structures of degree $-2$ on $N$-manifolds of degree
$2$.  We show how a Poisson structure of degree $-2$ on a split
$N$-manifold of degree $2$ is equivalent to a $2$-term representation
up to homotopy that is dual to itself. Then we give the geometrisation
of Poisson N-manifolds of degree $2$; namely linear Lie algebroids
structures on metric double vector bundles, that are compatible with
the metric.  We prove that the equivalence of categories established
in the previous section induces an equivalence of the category of
Poisson N-manifolds of degree $2$ with the category of metric
VB-algebroids. Finally, we discuss some examples of Poisson
N-manifolds of degree $2$, and the corresponding metric
VB-algebroids. We discuss in detail symplectic N-manifolds of degree
$2$, and how they correspond to tangent doubles of Euclidean vector
bundles.

\paragraph{\textbf{Section \ref{sec:split_lie_2}}}
In this section we start by recalling necessary background on Courant
algebroids, Dirac structures and Dorfman connections.  Then we
formulate in our own manner Sheng and Zhu's definition of split Lie
$2$-algebroids \cite{ShZh14}, before dualising it and obtaining the
notion of Dorfman $2$-representation. Then we write in coordinates the
homological vector field corresponding to a split Lie $2$-algebroid,
showing where the components of the Dorfman $2$-representation (or
equivalently of the split Lie $2$-algebroid) appear. In Section
\ref{examples_split_lie_2}, we give several classes of examples of
split Lie $2$-algebroids, introducing in particular the standard split
Lie $2$-algebroids defined by a vector bundle, and the bicrossproduct
Lie $2$-algebroid of a matched pair of $2$-term representations up to
homotopy.  Finally we describe morphisms of split Lie $2$-algebroids.

\paragraph{\textbf{Section \ref{sec:VB_cour}}}
In this section we give the definition of VB-Courant algebroids
\cite{Li-Bland12} and we relate Dorfman $2$-representations with
Lagrangian splittings of VB-Courant algebroids, in the spirit of
Gracia-Saz and Mehta's description of split VB-algebroids via $2$-term
representations up to homotopy \cite{GrMe10a}. Then we describe the
VB-Courant algebroids corresponding to the examples of split Lie
$2$-algebroids found in the previous section, and we prove that the
equivalence of categories established in Section \ref{sec:metDVB}
induces an equivalence of the category of VB-Courant algebroids
with the category of Lie $2$-algebroids.

\paragraph{\textbf{Section \ref{sec:LA-Cour}}}
We define the matching of self-adjoint $2$-representations with Dorfman
$2$-re\-pre\-sen\-ta\-tions. Then we show how split Poisson Lie
$2$-algebroids define such matched pairs, and how decomposed
LA-Courant algebroids define such matched pairs. We use this to deduce
an equivalence of the categories of Poisson Lie $2$-algebroids and
LA-Courant algebroids. Then we describe examples of LA-Courant
algebroids and Poisson Lie $2$-algebroids.

Next we focus on the core of an LA-Courant algebroid; we prove that it
inherits a natural structure of degenerate Courant algebroid
(``natural'' in the sense that this structure does not depend on any
Lagrangian splitting).  Symplectic Lie $2$-algebroids correspond via
the equivalence above to the tangent doubles of ordinary Courant
algebroids. The core structure in this class of LA-Courant algebroids
is just the underlying ordinary Courant algebroid. Hence, the Courant
bracket defined by a symplectic Lie $2$-algebroid can be recovered,
after any choice of splitting, as a kind of correction of a Dorfman
connection by the Poisson bracket.

\paragraph{\textbf{Section \ref{sec:dirac}}}
In the last section we discuss VB-Dirac structures (with support) in
VB-Courant algebroids, maximally isotropic subalgebroids in metric
VB-algebroids, and LA-Dirac structures (with support) in LA-Courant
algebroids.  We find in each case, after the choice of an adequate
splitting, conditions on the sides and core of the double subbundles
with the (matching) Dorfman $2$-representation and self-adjoint
$2$-representation for the double subbundle to be a VB-Dirac
structure, maximally isotropic subalgebroid or LA-Dirac structure. We
find that the integrability of a wide VB-Dirac structure is completely
encoded in the restriction to its side of the dull bracket defined by
any splitting adapted to the Dirac structure.  We prove that LA-Dirac
structures are automatically double Lie algebroids.

Next we explain in the language developed in this paper the notion of
pseudo-Dirac connection defined by Li-Bland in \cite{Li-Bland14}, and
we explain the equivalence that he finds between pseudo-Dirac
connections and VB-Dirac structures in tangent doubles of Courant
algebroids. We extend his result to an equivalence of ``quadratic''
pseudo-Dirac connections with LA-Dirac structures in the tangent
double of a Courant algebroid.

Finally we prove that an LA-Dirac structure defines a Manin pair over
its double base. The construction of the Courant algebroid resembles a
semi-direct product of a Lie algebroid with the degenerate Courant
algebroid on the core of the ambient LA-Courant algebroid.  These
Manin pairs will be the subject of future studies of the author.  In
particular, we will relate them to the infinitesimal description of
Dirac groupoids in \cite{LiSe11}.

\paragraph{\textbf{Appendices}}
To increase the readability of this paper, we give in the appendix the
proofs of its three most technical theorems.  In Section
\ref{appendix_proof_of_main} we prove the theorem describing
decomposed VB-Courant algebroids via Dorfman $2$-representations.  In
Section \ref{appendix_LACA} we prove that the Dorfman
$2$-representations and the self-adjoint $2$-representations encoding
the two sides of a split LA-Courant algebroid form a matched pair, and
vice-versa. This proof is very long and very technical, but contains
some interesting constructions.  In Section \ref{proof_of_manin_ap} we
prove the theorem on the Manin pair associated to an LA-Dirac
structure.

\subsection*{Acknowledgement}
The author warmly thanks Chenchang Zhu for giving her a necessary
insight at the origin of her interest in Lie $2$-algebroids, and Alan
Weinstein, David Li-Bland, Rajan Mehta and Dmitry Roytenberg for
interesting conversations. Thanks go also to Yunhe Sheng for his help
on a technical detail in his and Zhu's definition of a split Lie
$n$-algebroid. Finally the author thanks Rohan Jotz Lean for many very
useful comments on earlier versions of this work.

\subsection*{Notation and conventions}

We write $p_M\colon TM\to M$, $q_E\colon E\to M$ for vector bundle
mapsand $\pi_A\colon D\to A$ and $\pi_B\colon D\to B$ for the two
vector bundle projections of a double vector bundle.  For a vector
bundle $Q\to M$ we often identify without further mentioning the
vector bundle $(Q^*)^*$ with $Q$ via the canonical isomorphism. We
write $\langle\cdot\,,\cdot\rangle$ or $\langle\cdot\,,\cdot\rangle_M$
the canonical pairing of a vector bundle with its dual;
i.e.~$\langle a_m,\alpha_m\rangle=\alpha_m(a_m)$ for $a_m\in A$ and
$\alpha_m\in A^*$. We will use many different pairings; in general,
which pairing will be used will be clear from its arguments.  Given a
section $\xi$ of $E^*$, we will always write $\ell_\xi\colon E\to \R$
for the linear function associated to it, i.e. the function defined by
$e_m\mapsto \langle \xi(m), e_m\rangle$ for all $e_m\in E$.

Let $M$ be a smooth manifold. We will denote by $\mx(M)$ and
$\Omega^1(M)$ the spaces of (local) smooth sections of the tangent and
the cotangent bundle, respectively. For an arbitrary vector bundle
$E\to M$, the space of (local) sections of $E$ will be written
$\Gamma(E)$.  Let $f\colon M\to N$ be a smooth map between two smooth
manifolds $M$ and $N$.  Then two vector fields $X\in\mx(M)$ and
$Y\in\mx(N)$ are said to be \textbf{$f$-related} if $Tf\circ X=Y\circ
f$ on $\dom(X)\cap f\inv(\dom(Y))$.  We write then $X\sim_f Y$.

We write ``$[n]$-manifold'' for ``N-manifolds of degree $n$''. This
notation is to avoid confusions with $n$-manifolds, which are usually
understood as smooth manifolds of \emph{dimension} $n$.  We will say
$2$-representations for $2$-term representations up to homotopy.

%%% Local Variables: 
%%% mode: latex
%%% TeX-master: "Jotz14b_Ngeom"
%%% End: 
\section{Background and definitions on double vector bundles and VB-algebroids}
\label{sect:b+d}
We collect in this section background notions on ordinary vector
bundles and their morphisms, on double vector bundles and their linear
splittings and dualisations, on VB-algebroids and double Lie
algebroids, and on $2$-term representations up to homotopy and how
they encode split VB-algebroids. Further references will be given
throughout the text.

\subsection{Vector bundles and morphisms}\label{usual_VB_morphisms}
Let $A\to M$ and $B\to N$ be vector bundles and $\omega\colon A\to B$
a morphism of vector bundles over a smooth map $\omega_0\colon M\to
N$.  First we set a few notations. We will say that $a\in\Gamma_M(A)$
is $\omega$-related to $b\in\Gamma_N(B)$ if
\[ \omega(a(m))=b(\omega_0(m))
\]
for all $m\in M$. We will then write $a\sim_\omega b$.

We will write $\omega_0^*B\to M$ for the pullback of $B$ under
$\omega_0$; for $m\in M$, elements of $(\omega_0^*B)(m)$ are pairs
$(m,b_{\omega_0(m)})$ with $b_{\omega_0(m)}\in B(\omega_0(m))$. The
morphism $\omega$ induces then a morphism $\omega^!\colon A\to
\omega_0^*B$, $\omega^!(a_m)=(m,\omega(a_m))$ over the identity on
$M$. For a section $b\in\Gamma_N(B)$, we get in a similar manner a
section $\omega_0^!b\in\Gamma_M(\omega_0^*B)$; defined by
$(\omega_0^!b)(m)=(m,b(\omega_0(m)))$ for all $m\in M$.  We have then
$\omega^\star(b)=(\omega^!)^*( \omega_0^!b)$ for all
$b\in\Gamma_N(B)$.

 The dual of a morphism
$\omega\colon A\to B$ over $\omega_0\colon M\to N$ is in general not a morphism
of vector bundles, but a relation $R_{\omega^*}\subseteq A^*\times
B^*$ defined as follows:
\[
R_{\omega^*}=\{(\omega_m^*\beta_{\omega_0(m)},\beta_{\omega_0(m)})\mid
m\in M, \beta_{\omega_0(m)}\in B^*_{\omega_0(m)}\},
\]
where $\omega_m\colon A_m\to B_{\omega_0(m)}$ is the morphism of
vector spaces.  This relation induces a morphism $\omega^\star$ of
modules over $\omega_0^*\colon C^\infty(N)\to C^\infty(M)$:
\begin{equation}\label{dual_of_VB_map}
\omega^\star\colon \Gamma_N(B^*)\to \Gamma_M(A^*),
\qquad  \omega^\star(\beta)(m)=\omega_m^*\beta_{\omega_0(m)}
\end{equation}
for all $\beta\in\Gamma_N(B^*)$ and $m\in M$. That is,
$(\omega^\star(\beta)(m), \beta(\omega_0(m)))\in R_{\omega^*}$ for all
$\beta\in\Gamma_N(B^*)$ and $m\in M$.  We prove the following lemma.
\begin{lemma}\label{bundlemap_eq_to_morphism}
  The map $\cdot^\star$, that sends a morphism of vector bundles
  $\omega\colon A\to B$ over $\omega_0\colon M\to N$ to the morphism
  $\omega^\star\colon\Gamma_N(B^*)\to\Gamma_M(A^*)$ of modules over
  $\omega_0^*\colon C^\infty(N)\to C^\infty(M)$, is a bijection.
\end{lemma}

\begin{proof}
  First we have to check that $\omega^\star$ is well-defined, that is,
  that the image under $\omega^\star$ of a smooth section of $B^*$ is
  again smooth. Consider the pullback of $B$ under $\omega_0$,
  i.e.~the vector bundle $\omega_0^*B\to M$. The morphism $\omega$ of
  vector bundles induces a morphism $\omega^!\colon A\to \omega_0^*B$
  of smooth vector bundles over the identity on $M$: $\omega^!$ is
  defined by $\omega^!(a_m)=(m,\omega(a_m))$ for all $a_m$ in the
  fiber of $A$ over $m\in M$. Now we have
  $(\omega_0^*B)^*=\omega_0^*B^*$ and for each section
  $\beta\in\Gamma_N(B^*)$, we define
  $\beta^!\in\Gamma_M(\omega_0^*B^*)$ by
  $\beta^!(m)=(m,\beta(\omega_0(m)))$. The smoothness of
  $\omega^\star(\beta)$ follows from the equality
  $\omega^\star(\beta)=(\omega^!)^*\beta^!$: for each $m\in M$, and
  each $a_m\in A_m$, we have
\begin{align*}
  \langle((\omega^!)^*\beta^!)(m), a_m\rangle&=\langle\beta^!(m),
  \omega^!(a_m)\rangle
  =\langle(m,\beta(\omega_0(m))),(m,\omega(a_m))\rangle\\
  &=\langle\beta(\omega_0(m)),
  \omega(a_m)\rangle=\langle\omega^\star(\beta)(m),a_m\rangle.
\end{align*}

The map $\omega^\star$ is obviously a morphism of modules over
$\omega_0^*$: for $\beta\in\Gamma(B^*)$ and $f\in C^\infty(N)$, we
find
\[\omega^\star(f\beta)=\omega_0^*f\,\omega^\star(\beta).
\]

Next we need to show that a morphism $\mu^\star\colon
\Gamma_N(B^*)\to\Gamma_M(A^*)$ of modules over $\mu_0^*\colon
C^\infty(N)\to C^\infty(M)$, for $\mu_0\colon M\to N$ smooth, induces
a morphism $A\to B$ of vector bundles over $\mu_0\colon M\to N$.
Choose $a_m$ in the fiber of $A$ over $m$ and define $\mu(a_m)\in
B_{\mu_0(m)}$ by
\[\langle \beta(\mu_0(m)), \mu(a_m)\rangle=\langle\mu^\star(\beta)(m), a_m\rangle 
\]
for all $\beta\in \Gamma(B^*)$. To prove that $\mu$ is a vector bundle
morphism\footnote{The smoothness of $\mu$ is seen as follows: let
  $b_1,\ldots,b_n$ be local basis fields for $B$ and let
  $\beta_1,\ldots,\beta_n$ be the dual basis fields. Then for eah
  $a_m\in A$, $\mu(a_m)$ can be written
  $\sum_{i=1}^n\langle\mu(a_m),\beta_i(\mu_0(m))\rangle
  b_i(\mu_0(m))$.  Since each
  $\langle\mu(a_m),\beta_i(\mu_0(m))\rangle$ equals
  $\ell_{\mu^\star(\beta_i)}(a_m)$, we find that locally, $\mu$ can be
  written $\mu=\sum_{i=1}^n\ell_{\mu^\star(\beta_i)}\cdot
  (b_i\circ\mu_0\circ q_A)$. }, we need to check that $\langle
\beta(\mu_0(m)), \mu(a_m)\rangle$ only depends on the value of $\beta$
at $\mu_0(m)$, or in other words, that if $\beta$ vanishes at
$\mu_0(m)$, then $\langle \beta(\mu_0(m)), \mu(a_m)\rangle=0$. If
$\beta$ vanishes at $\mu_0(m)$, then $\beta$ can be written as $f\cdot
\beta'$ with $\beta'\in \Gamma(B^*)$ and $f\in C^\infty(N)$ with
$f(\mu_0(m))=0$.  But then
\[\langle \beta(\mu_0(m)), \mu(a_m)\rangle=\langle f(\mu_0(m))\mu^\star(\beta')(m), a_m\rangle=0.
\]
The morphism $\mu$ of vector bundles clearly induces $\mu^\star$ on
the sets of sections of the duals, and vice-versa.
\end{proof}
\medskip

\subsection{Double vector bundles}
We briefly recall the definitions of double vector bundles, of their
\emph{linear} and \emph{core} sections, and of their \emph{linear
  splittings} and \emph{lifts}. We refer to
\cite{Pradines77,Mackenzie05,GrMe10a} for more details.

A \textbf{double vector bundle} $(D;A,B;M)$ is a commutative square
\begin{equation*}
\begin{xy}
\xymatrix{
D \ar[r]^{\pi_B}\ar[d]_{\pi_A}& B\ar[d]^{q_B}\\
A\ar[r]_{q_A} & M}
\end{xy}
\end{equation*}
satisfying the following four conditions:
\begin{enumerate}
\item all four sides are vector bundles;
\item $\pi_B$ is a vector bundle morphism over $q_A$;
\item $+_B: D\times_B D \rightarrow D$ is a vector bundle
morphism over $+: A\times_M A \rightarrow A$, where $+_B$ is
the addition map for the vector bundle $D\rightarrow B$, and 
\item the scalar multiplication $\R\times D\to D$ in the bundle $D\to
  B$ is a vector bundle morphism over the scalar multiplication
  $\R\times A\to A$.
\end{enumerate}
The corresponding statements for the operations in the bundle $D\to A$
follow.  Note that the condition that each addition in $D$ is a
morphism with respect to the other is exactly
\begin{equation}\label{add_add} (d_1+_Ad_2)+_B(d_3+_Ad_4)=(d_1+_Bd_3)+_A(d_2+_Bd_4)
\end{equation}
for $d_1,d_2,d_3,d_4\in D$ with $\pi_A(d_1)=\pi_A(d_2)$,
$\pi_A(d_3)=\pi_A(d_4)$ and $\pi_B(d_1)=\pi_B(d_3)$,
$\pi_B(d_2)=\pi_B(d_4)$.

Given a double vector bundle $(D; A, B; M)$, the vector bundles $A$
and $B$ are called the \textbf{side bundles}. The \textbf{core} $C$ of
a double vector bundle is the intersection of the kernels of $\pi_A$
and of $\pi_B$. It has a natural vector bundle structure over
$M$, 
the projection of which we call $q_C\colon C \rightarrow M$. The
inclusion $C \hookrightarrow D$ is usually denoted by
$C_m \ni c \longmapsto \overline{c} \in \pi_A^{-1}(0^A_m) \cap \pi_B^{-1}(0^B_m)$.

Given a double vector bundle $(D;A,B;M)$, the space of sections
$\Gamma_B(D)$ is generated as a $C^{\infty}(B)$-module by two
distinguished classes of sections (see \cite{Mackenzie11}), the
\emph{linear} and the \emph{core sections} which we now describe.  For
a smooth section $c\colon M \rightarrow C$, the corresponding
\textbf{core section} $c^\dagger\colon B \rightarrow D$ is defined as
\begin{equation}\label{core_section}
c^\dagger(b_m) = \tilde{0}_{\vphantom{1}_{b_m}} +_A \overline{c(m)}, \,\, m \in M, \, b_m \in B_m.
\end{equation}
We denote the corresponding core section $A\to D$ by $c^\dagger$ also,
relying on the argument to distinguish between them. The space of core
sections of $D$ over $B$ will be written $\Gamma_B^c(D)$.

A section $\xi\in \Gamma_B(D)$ is called \textbf{linear} if $\xi\colon
B \rightarrow D$ is a bundle morphism from $B \rightarrow M$ to $D
\rightarrow A$ over a section $a\in\Gamma(A)$.  The space of linear
sections of $D$ over $B$ is denoted by $\Gamma^\ell_B(D)$. A section
$\psi\in \Gamma(B^*\otimes C)$ defines a linear section
$\tilde{\psi}\colon B\to D$ over the zero section $0^A\colon M\to A$
by
\begin{equation}\label{core_morf}
\widetilde{\psi}(b_m) = \tilde{0}_{b_m}+_A \overline{\psi(b_m)}
\end{equation}
for all $b_m\in B$.  We call $\widetilde{\psi}$ a \textbf{core-linear
  section}.

\begin{example}\label{trivial_dvb}
  Let $A, \, B, \, C$ be vector bundles over $M$ and consider
  $D=A\times_M B \times_M C$. With the vector bundle structures
  $D=q^{!}_A(B\oplus C) \to A$ and $D=q_B^{!}(A\oplus C) \to B$, one
  finds that $(D; A, B; M)$ is a double vector bundle called the
  \textit{decomposed double vector bundle with sides $A$ and $B$ and
    core $C$}. The core sections are given by
$c^\dagger\colon b_m \mapsto (0^A_m, b_m, c(m))$, where $m \in M$,  $b_m \in
B_m$, $c \in \Gamma(C)$,
and similarly for $c^\dagger\colon A\to D$.  The space of linear
sections $\Gamma^\ell_B(D)$ is naturally identified with
$\Gamma(A)\oplus \Gamma(B^*\otimes C)$ via
$$
(a, \psi): b_m \mapsto (a(m), b_m, \psi(b_m)), \text{ where } \psi \in
\Gamma(B^*\otimes C), \, a\in \Gamma(A).
$$

In particular, the fibered product $A\times_M B$ is a double vector
bundle over the sides $A$ and $B$ and its core is the trivial bundle
over $M$.
\end{example}

\subsubsection{Linear splittings and lifts}
\label{subsub:lsl}
A \textbf{linear splitting} of $(D; A, B; M)$ is an injective morphism
of double vector bundles $\Sigma\colon A\times_M B\hookrightarrow D$
over the identity on the sides $A$ and $B$.  That every double vector
bundle admits local linear splittings was proved by \cite{GrRo09}.
Local linear splittings are equivalent to double vector bundle
charts. Pradines originally defined double vector bundles as
topological spaces with an atlas of double vector bundle charts
\cite{Pradines74a} (see Definition \ref{double_atlas}). Using a
partition of unity, he proved that (provided the double base is a
smooth manifold) this implies the existence of a global double
splitting \cite{Pradines77}. Hence, any double vector bundle in the
sense of our definition admits a (global) linear splitting.

Note that a linear splitting of $D$ is equivalent to a
\textbf{decomposition} of $D$, i.e.~an isomorphism $\mathbb I\colon
A\times_MB\times_MC\to D$ of double vector bundles over the identities
on the sides and inducing the identity on the core. Given a linear
splitting $\Sigma$, the corresponding decomposition $\mathbb I$ is
given by $\mathbb I(a_m,b_m,c_m)=\Sigma(a_m,b_m)+_B(\tilde
0_{\vphantom{1}_{b_m}} +_A \overline{c_m})$.  Given a decomposition
$\mathbb I$, the corresponding linear splitting $\Sigma$ is given by
$\Sigma(a_m,b_m)=\mathbb I(a_m,b_m,0^C_m)$.

\medskip

A linear splitting $\Sigma$ of $D$ is also equivalent to a splitting
$\sigma_A$ of the short exact sequence of $C^\infty(M)$-modules
\begin{equation}\label{fat_seq_gamma}
0 \longrightarrow \Gamma(B^*\otimes C) \hookrightarrow \Gamma^\ell_B(D) 
\longrightarrow \Gamma(A) \longrightarrow 0,
\end{equation}
where the third map is the map that sends a linear section $(\xi,a)$
to its base section $a\in\Gamma(A)$.  The splitting $\sigma_A$ will be
called a \textbf{horizontal lift} or simply a \textbf{lift}. Given
$\Sigma$, the horizontal lift $\sigma_A\colon \Gamma(A)\to
\Gamma_B^\ell(D)$ is given by $\sigma_A(a)(b_m)=\Sigma(a(m), b_m)$ for
all $a\in\Gamma(A)$ and $b_m\in B$.  By the symmetry of a linear
splitting, we find that a lift $\sigma_A\colon
\Gamma(A)\to\Gamma_B^\ell(D)$ is equivalent to a lift $\sigma_B\colon
\Gamma(B)\to \Gamma_A^\ell(D)$.  Given a lift
$\sigma_A\colon\Gamma(A)\to\Gamma^\ell_B(D)$, the corresponding lift
$\sigma_B\colon\Gamma(B)\to\Gamma^\ell_A(D)$ is given by
$\sigma_B(b)(a(m))=\sigma_A(a)(b(m))$ for all $a\in\Gamma(A)$,
$b\in\Gamma(B)$.

Note finally that two linear splittings $\Sigma^1,\Sigma^2\colon
A\times_MB\to D$ differ by a section $\phi_{12}$ of $A^*\otimes
B^*\otimes C\simeq \operatorname{Hom}(A,B^*\otimes C)\simeq
\operatorname{Hom}(B,A^*\otimes C)$ in the following sense.  For each
$a\in\Gamma(A)$ the difference $\sigma_A^1(a)-_B\sigma_A^2(a)$ of
horizontal lifts is the core-linear section defined by
$\phi_{12}(a)\in\Gamma(B^*\otimes C)$. By symmetry,
$\sigma_B^1(b)-_A\sigma_B^2(b)=\widetilde{\phi_{12}(b)}$ for each
$b\in\Gamma(B)$.

\medskip

The space of linear sections is a locally free $C^{\infty}(M)$-module
(this follows from the existence of local splittings). Hence, there is
a vector bundle $\widehat{A}$ over $M$ such that $\Gamma^l_B(D)$ is
isomorphic to $\Gamma(\widehat{A})$ as $C^{\infty}(M)$-modules. The
vector bundle $\widehat{A}$ is sometimes called the \textbf{fat vector
  bundle} defined by $\Gamma^l_B(D)$. % Recall that for a
% linear section $\chi$, there exists a section $a_\chi\in\Gamma(A)$
% such that $\pi_A\circ \chi = a_\chi \circ q_B$. The map $\chi \mapsto a_\chi$
% induces a short exact sequence of vector bundles
The short exact sequence \eqref{fat_seq_gamma} induces a short exact
sequence of vector bundles
\begin{equation}\label{fat_seq}
0 \longrightarrow B^*\otimes C \hookrightarrow \widehat{A} \longrightarrow A \longrightarrow 0.
\end{equation}

\subsubsection{The tangent double of a vector bundle}\label{tangent_double}
Let $q_E\colon E\to M$ be a vector bundle.  Then the tangent bundle
$TE$ has two vector bundle structures; one as the tangent bundle of
the manifold $E$, and the second as a vector bundle over $TM$. The
structure maps of $TE\to TM$ are the derivatives of the structure maps
of $E\to M$. The space $TE$ is a double vector bundle with core bundle
$E \to M$. The map $\bar{}\,\colon E\to p_E^{-1}(0^E)\cap
(Tq_E)^{-1}(0^{TM})$ sends $e_m\in E_m$ to $\bar
e_m=\left.\frac{d}{dt}\right\an{t=0}te_m\in T_{0^E_m}E$.
\begin{equation*}
\begin{xy}
\xymatrix{
TE \ar[d]_{Tq_E}\ar[r]^{p_E}& E\ar[d]^{q_E}\\
 TM\ar[r]_{p_M}& M}
\end{xy}
\end{equation*}

The core vector field corresponding to $e \in \Gamma(E)$ is the
vertical lift $e^{\uparrow}\colon E \to TE$, i.e.~the vector field
with flow $\phi\colon E\times \R\to E$, $\phi_t(e'_m)=e'_m+te(m)$. An
element of $\Gamma^\ell_E(TE)=\mx^\ell(E)$ is called a \textbf{linear
  vector field}. It is well-known (see e.g.~\cite{Mackenzie05}) that a
linear vector field $\xi\in\mx^l(E)$ covering $X\in\mx(M)$ is
equivalent to a derivation $D_\xi^*\colon \Gamma(E^*) \to \Gamma(E^*)$
over $X\in \mx(M)$, and hence to the dual derivation
$D_\xi\colon\Gamma(E)\to\Gamma(E)$ over $X\in \mx(M)$. The precise
correspondence is given by\footnote{Since its flow is a flow of vector
  bundle morphisms, a linear vector field sends linear functions to
  linear functions and pullbacks to pullbacks.}
\begin{equation}\label{ableitungen}
\xi(\ell_{\varepsilon}) 
= \ell_{D_\xi^*(\varepsilon)} \,\,\,\, \text{ and }  \,\,\, \xi(q_E^*f)= q_E^*(X(f))
\end{equation}
for all $\varepsilon\in\Gamma(E^*)$ and $f\in C^\infty(M)$.  Here
$\ell_\varepsilon$ is the linear function $E\to\R$ corresponding to
$\varepsilon$. We will write $\widehat D$ for the linear vector field
in $\mx^l(E)$ corresponding in this manner to a derivation $D$ of
$\Gamma(E)$.  The choice of a linear splitting $\Sigma$ for $(TE; TM,
E; M)$ is equivalent to the choice of a connection on $E$: Since a
linear splitting gives us a linear vector field
$\sigma_{TM}(X)\in\mx^l(E)$ for each $X\in \mx(M)$, we can define
$\nabla\colon \mx(M)\times\Gamma(E)\to \Gamma(E)$ by
$\sigma_{TM}(X)=\widehat{\nabla_X}$ for all $X\in\mx(M)$. Conversely,
a connection $\nabla\colon \mx(M)\times\Gamma(E)\to\Gamma(E)$ defines
a lift $\sigma_{TM}^\nabla$ for $(TE; TM, E; M)$ and a linear
splitting $\Sigma^\nabla\colon TM\times_M E \to TE$.

We recall as well the relation between the connection and the Lie
bracket of vector fields on $E$.  Given $\nabla$, it is easy to see
using the equalities in \eqref{ableitungen} that, writing $\sigma$ for
$\sigma_{TM}^\nabla$:
\begin{equation}\label{Lie_bracket_VF}
  \left[\sigma(X), \sigma(Y)\right]=\sigma[X,Y]-\widetilde{R_\nabla(X,Y)},\qquad
  \left[\sigma(X), e^\uparrow\right]=(\nabla_Xe)^\uparrow,\qquad
  \left[e^\uparrow,e'^\uparrow\right]=0,
\end{equation}
for all $X,Y\in\mx(M)$ and $e,e'\in\Gamma(E)$.  That is, the Lie
bracket of vector fields on $M$ and the connection encode completely
the Lie bracket of vector fields on $E$.

\medskip

Now let us have a quick look at the other structure on the double
vector bundle $TE$. The lift
$\sigma_{E}^\nabla\colon\Gamma(E)\to\Gamma_{TM}^\ell(TE)$ is given by
\begin{equation*}
  \sigma_{E}^\nabla(e)(v_m) = T_me(v_m) +_{TM} (T_m0^E(v_m) -_E \overline{\nabla_{v_m} e}), \,\, v_m \in TM, \, e \in \Gamma(E).
\end{equation*}
Further, for $e\in\Gamma(E)$, the core section $e^\times\in\Gamma_{TM}(TE)$
is given by
\begin{equation}\label{cross_sect_def}
 e^\times(v_m)=T_m0^E(v_m)+_E\left.\frac{d}{dt}\right\an{t=0}te(m).
\end{equation}

\subsubsection{Dualisation and lifts}\label{dual}
Double vector bundles can be dualised in two distinct ways.  We denote
by $D\duer A$ the dual of $D$ as a vector bundle over $A$ and likewise
for $D\duer B$. The dual $D\duer A$ is again a double vector
bundle\footnote{The projection $\pi_{C^*}\colon D\duer A\to C^*$ is
  defined as follows: if $\Phi\in D\duer A$ projects to
  $\pi_A(\Phi)=a_m$, then $\pi_{C^*}(\Phi)\in C^*_m$ is defined by
\[\pi_{C^*}(\Phi)(c_m)=\Phi(0^D_{a_m}+_B\overline{c_m})
\]
for all $c_m\in C_m$. The addition in the fibers of the vector bundle
$D\duer A\to C^*$ is defined as follows: if $\Phi_1$ and $\Phi_2\in
D\duer A$ satisfy
\[\pi_{C^*}(\Phi_1)=\pi_{C^*}(\Phi_2)\qquad \pi_A(\Phi_1)=a^1_m\qquad
\pi_A(\Phi_2)=a^2_m,
\]
then $\Phi_1+_{C^*}\Phi_2$ is defined by
\[(\Phi_1+_{C^*}\Phi_2)(d_1+_Bd_2)=\Phi_1(d_1)+\Phi_2(d_2)
\]
for all $d_1,d_2\in D$ with $\pi_B(d_1)=\pi_B(d_2)$ and
$\pi_A(d_1)=a^1_m$, $\pi_A(d_2)=a^2_m$.  The core element
$\overline{\beta_m}\in D\duer A$ defined by $\beta_m\in B^*$ is
defined by $\overline{\beta_m}(d)=\beta_m(\pi_B(d))$ for all $d\in D$
with $\pi_A(d)=0^A_m$.  By playing with the vector bundle structures
on $D\duer A$ and \eqref{add_add}, one can show that each core element
of $D\duer A$ is of this form. We encourage the reader who is not
familiar with the dualisations of double vector bundles to check this,
and also to find out where the projection to $C^*$ is relevant in the
definition of the addition over $C^*$. See \cite{Mackenzie11}.}, with
side bundles $A$ and $C^*$ and core $B^*$
\cite{Mackenzie99,Mackenzie11}.
$$ 
{\xymatrix{
    D\ar[r]^{\pi_B}\ar[d]_{\pi_A}&   B\ar[d]^{q_{B}}\\
    A\ar[r]_{q_A}                   &  M\\
  }} \qquad\qquad {\xymatrix{
    D\duer A \ar[r]^{\pi_{C^*}}\ar[d]_{\pi_A}&   C^*\ar[d]^{q_{C^*}}\\
    A\ar[r]_{q_{A}}                   &  M\\
  }} \qquad\qquad {\xymatrix{
    D\duer B \ar[r]^{\pi_B}\ar[d]_{\pi_{C^*}}&   B\ar[d]^{q_B}\\
    C^*\ar[r]_{q_{C^*}}                   &  M\\
  }}
$$ 

By dualising again $D\duer A$ over $C^*$, we get
\[\xymatrix{
D\duer A\duer C^* \ar[r]^{\pi_{C^*}}\ar[d]_{\pi_B}&   C^*\ar[d]^{q_{C^*}}      \\
B\ar[r]_{q_{B}}             &  M,\\
}\]
with core $A^*$. In the same manner, we get a double vector bundle
$D\duer B\duer C^*$ with sides $A$ and $C^*$ and core $B^*$.

The vector bundles $D\duer B\to C^*$ and $D\duer A\to C^*$ are, up to
a sign, naturally in duality to each other \cite{Mackenzie05}. The
pairing
\[ \nsp{\cdot\,}{\cdot} \colon (D\duer A)\times_{C^*} (D\duer B)\to
\mathbb R
\]  
is defined as follows: for $\Phi\in D\duer A$ and $\Psi\in D\duer B$
projecting to the same element $\gamma_m$ in $C^*$, choose $d\in D$
with $\pi_A(d)=\pi_A(\Phi)$ and $\pi_B(d)=\pi_B(\Psi)$.  Then $\langle
\Phi, d\rangle_A-\langle \Psi,d\rangle_B$ does not depend on the
choice of $d$ and we set $\nsp{\Phi}{\Psi}=\langle \Phi,
d\rangle_A-\langle \Psi,d\rangle_B$.
 
This implies in particular that $D\duer A$ is canonically (up to a
sign) isomorphic to $D\duer B\duer C^*$ and $D\duer B$ is isomorphic
to $D\duer A\duer C^*$. % We will use this below.

\medskip

Each linear section $\xi\in\Gamma_B(D)$ over $a\in\Gamma(A)$ induces a
linear section $\xi^\sqcap\in \Gamma_{C^*}^\ell(D\duer B\duer C^*)$
over $a$. Namely $\xi$ induces a function $\ell_\xi\colon D\duer
B\to\R$ which is fibrewise linear over $B$ and, using the definition
of the addition in $D\duer B\to C^*$,
it follows that $\ell_\xi$ is also linear over $C^*$. The
corresponding linear section of $D\duer B\duer C^*\to C^*$ is denoted
$\xi^\sqcap$ \cite{Mackenzie11}.  Thus
\begin{equation}
\label{eq:xisqcap-new}
\langle\xi^\sqcap(\gamma), \Phi\rangle _{C^*}
= \ell_\xi(\Phi)
= \langle\Phi,\, \xi(b)\rangle_B
\end{equation}
for $\Phi\in D\duer B$ such that $\pi_B(\Phi)=b$ and
$\pi_{C^*}(\Phi)=\gamma$.

Given a linear splitting $\Sigma\colon A\times_M B\to D$ of $D$, we
get hence a linear splitting $\Sigma^{\star,B}\colon C^*\times_M A\to
D\duer B\duer C^*$, defined by the horizontal lift
$\sigma_{A}^{\star,B}\colon \Gamma(A)\to\Gamma_{C^*}^\ell(D\duer
B\duer C^*)$:
\begin{equation}\label{eq:intermediate_lift}
\sigma_{A}^{\star,B}(a)=(\sigma_A(a))^\sqcap
\end{equation}
for all $a\in\Gamma(A)$.

Now we use the (canonical up to a sign) isomorphism of $D\duer A$ with
$D\duer B\duer C^*$ to construct a canonical linear splitting of
$D\duer A$ given a linear splitting of $D$.  We identify $D\duer A$
with $D\duer B\duer C^*$ using $-\nsp{\cdot\,}{\cdot}$. Thus we define
the horizontal lift $\sigma^\star_A\colon
\Gamma(A)\to\Gamma_{C^*}^\ell(D\duer A)$ by
\begin{equation}\label{iso_sign}
\nsp{\sigma_A^\star(a)}{\cdot} = -\sigma_A^{\star,B}(a)
\end{equation}
for all $a\in\Gamma(A)$.  The choice of sign in \eqref{iso_sign} is
necessary for $\sigma_A^\star(a)$ to be a linear section of $D\duer A$
over $a$ (and not over $-a$).

By (skew-)symmetry, given the lift $\sigma_B\colon\Gamma(B)\to
\Gamma^\ell_A(D)$, we identify $D\duer B$ with $D\duer A\duer C^*$
using $\nsp{\cdot}{\cdot}$ and define the lift $\sigma^\star_B\colon
\Gamma(B)\to\Gamma_{C^*}^\ell(D\duer B)$ by
$\nsp{\sigma_B^\star(b)}{\cdot} = \sigma_B^{\star,A}(b)$ for all
$b\in\Gamma(B)$.  (This time, we do not need the minus sign.)  As a
summary, we have the equations:
\begin{equation}\label{equal_fat}
  \nsp{\sigma_A^\star(a)}{\sigma_B^\star(b)}=0, \qquad
  \nsp{\sigma_A^\star(a)}{\alpha^\dagger} = -q_{C^*}^*\langle\alpha, a\rangle,\qquad
  \nsp{\beta^\dagger}{\sigma_B^\star(b)} = q_{C^*}^*\langle\beta, b\rangle,
\end{equation}
for all $a\in\Gamma(A)$, $b\in\Gamma(B)$, $\alpha\in\Gamma(A^*)$ and
$\beta\in\Gamma(B^*)$.  Furthermore, the following Lemma shows that
the horizontal lift
$\sigma^\star_A\colon\Gamma(A)\to\Gamma_{C^*}^l(D\duer A)$ is very
natural.
\begin{lemma}\label{lemma_dual_splitting}
  Given a horizontal lift $\sigma_A\colon\Gamma(A)\to\Gamma_B^l(D)$,
  the ``dual'' horizontal lift
  $\sigma^\star_A\colon\Gamma(A)\to\Gamma_{C^*}^l(D\duer A)$ can
  alternatively be defined by
\[\langle \sigma_A^\star(a)(\gamma_m),\sigma_A(a)(b_m)\rangle_A=0, \qquad
\langle
\sigma_A^\star(a)(\gamma_m),c^\dagger(a(m))\rangle_A=\langle\gamma_m,
c(m)\rangle
\]
for all $a\in\Gamma(A)$, $c\in\Gamma(C)$, $b_m\in B$ and $\gamma_m\in C^*$.
\end{lemma}
\begin{proof}
By \eqref{eq:intermediate_lift} and \eqref{eq:xisqcap-new}, we have 
\begin{equation}\label{first_eq}\nsp{\sigma_A^\star(a)(\gamma_m)}{\Phi}=-\langle (\sigma_A(a))^\sqcap(\gamma_m), \Phi\rangle
  =-\langle  \Phi, \sigma_A(a)(b_m)\rangle
\end{equation}
for $\Phi\in D\duer B$ with $\pi_B(\Phi)=b_m$ and
$\pi_{C^*}(\Phi)=\gamma_m$.  Since for $b\in\Gamma(B)$ with
$b(m)=b_m$, $d=\sigma_A(a)(b_m)=\sigma_B(b)(a(m))\in D$ is an element
with $\pi_A(d)=a(m)$ and $\pi_B(d)=b(m)=b_m$, the pairing on the left
can also be written
\begin{equation}\label{second_eq}
\nsp{\sigma_A^\star(a)(\gamma_m)}{\Phi}=\langle
\sigma_A^\star(a)(\gamma_m), \sigma_A(a)(b(m))\rangle_B- \langle
\Phi, \sigma_A(a)(b(m))\rangle.
\end{equation}
\eqref{first_eq} and \eqref{second_eq} together show that $\langle
\sigma_A^\star(a)(\gamma_m), \sigma_A(a)(b(m))\rangle=0$.
Further, we find for $\alpha\in\Gamma(B^*)$ and any $c\in\Gamma(C)$:
\[-\langle \alpha,
a\rangle(m)=\nsp{\sigma_A^\star(a)(\gamma_m)}{\alpha^\dagger(\gamma_m)}=\langle\sigma_A^\star(a)(\gamma_m),
c^\dagger(a(m))\rangle
-\langle\alpha^\dagger(\gamma_m),c^\dagger(a(m))\rangle.
\]
In order to pair $\alpha^\dagger(\gamma_m)$ with $c^\dagger(a(m))$ over $0^B_m\in B$, we write them as 
\[\alpha^\dagger(\gamma_m)= 0^{D\duer B}_{\gamma_m}+_B\overline{\alpha(m)}, \qquad 
c^\dagger(a(m))=0^D_{a(m)}+_B\overline{c(m)}.
\]
 We get 
\[-\langle \alpha, a\rangle (m)=\langle\sigma_A^\star(a)(\gamma_m),
c^\dagger(a(m))\rangle
-\langle \gamma_m,c(m)\rangle-\langle a, \alpha\rangle(m),
\]
and so $\langle\sigma_A^\star(a)(\gamma_m), c^\dagger(a(m))\rangle
=\langle \gamma_m,c(m)\rangle$ for all $a\in\Gamma(A)$,
$c\in\Gamma(C)$ and $\gamma_m\in C^*$.
\end{proof}

\subsection{VB-algebroids and double Lie algebroids}
\label{subsect:VBa}

% What we are here calling VB-algebroids were defined in
% \cite{Mackenzie98x,Mackenzie11} and called \LAvbs.  \footnote{The terminology
% `\LAvb' followed that of \LAgpds, which were defined in \cite[\S4]{Mackenzie92}, 
% on the model of Pradines' \cite{Pradines88} concept of \VBgpd. In \cite{Mackenzie98x} 
% and \cite[3.3]{Mackenzie11} \LAvbs were seen as a special case of \LAgpds. 
% The terminology `VB-algebroid' of \cite{GrMe10a} distingushes the equivalent
% formulation in terms of bracket conditions on the linear and core sections.}

Let $(D; A, B; M)$ be a double vector bundle
$$
\xymatrix{
D \ar[r]^{\pi_B}\ar[d]_{\pi_A}&   B\ar[d]^{q_B}\\
A\ar[r]_{q_A}                   &  M\\
}
$$ with core $C$.
Then $(D \to B; A \to M)$ is a \textbf{VB-algebroid}
(\cite{Mackenzie98x}; see also \cite{GrMe10a}) if $D \to B$ has a Lie
algebroid structure the anchor of which is a bundle morphism
$\Theta_B\colon D \to TB$ over $\rho_A\colon A \to TM$ and such that
the Lie bracket is linear:
\begin{equation*} [\Gamma^\ell_B(D), \Gamma^\ell_B(D)] \subset
  \Gamma^\ell_B(D), \qquad [\Gamma^\ell_B(D), \Gamma^c_B(D)] \subset
  \Gamma^c_B(D), \qquad [\Gamma^c_B(D), \Gamma^c_B(D)]= 0.
\end{equation*}
The vector bundle $A\to M$ is then also a Lie algebroid, with anchor
$\rho_A$ and bracket defined as follows: if $\xi_1,
\xi_2\in\Gamma^\ell_B(D)$ are linear over $a_1,a_2\in\Gamma(A)$, then
the bracket $[\xi_1,\xi_2]$ is linear over $[a_1,a_2]$.  We also say
that the Lie algebroid structure on $D\to B$ is linear over the Lie
algebroid $A\to M$.

Since the anchor $\Theta_B$ is linear, it sends a core section
$c^\dagger$, $c\in\Gamma(C)$ to a vertical vector field on $B$.  This
defines the \textbf{core-anchor} $\partial_B\colon C\to B$; for
$c\in\Gamma(C)$ we have $\Theta_B(c^\dagger)=(\partial_Bc)^\uparrow$
(see \cite{Mackenzie92}).

\begin{example}\label{td}
  It is easy to see from the considerations in \S\ref{tangent_double}
  that the tangent double $(TE;E,TM;M)$ of a vector bundle $E\to M$
  has a VB-algebroid structure $(TE\to E, TM\to M)$. (The Lie
  algebroid structures are the tangent Lie algebroid structures.)
\end{example}

If $D$ is a VB-algebroid with Lie algebroid structures on $D\to B$ and
$A\to M$ the dual vector bundle $D\duer B\to B$ has a
\emph{Lie-Poisson structure} (a linear Poisson structure), and the
structure on $D\duer B$ is also Lie-Poisson with respect to $D\duer
B\to C^*$ \cite[3.4]{Mackenzie11}. Dualising this bundle gives a Lie
algebroid structure on $D\duer B\duer C^*\to C^*$. This equips the
double vector bundle $(D\duer B\duer C^*; C^*,A;M)$ with a
VB-algebroid structure. Using the isomorphism defined by
$-\nsp{\cdot}{\cdot}$, the double vector bundle $(D\duer A\to C^*;A\to
M)$ is also a VB-algebroid. In the same manner, if $(D\to A, B\to M)$
is a VB-algebroid then $(D\duer B\to C^*;B\to M)$ is a VB-algebroid.

A \textbf{double Lie algebroid} \cite{Mackenzie11} is a double vector
bundle $(D;A,B;M)$ with core denoted $C$, and with Lie algebroid
structures on each of $A\to M$, $B\to M$, $D\to A$ and $D\to B$ such
that each pair of parallel Lie algebroids gives $D$ the structure of a
VB-algebroid, and such that
the pair $(D\duer A, D\duer B)$ with the induced Lie algebroid
structures on base $C^*$ and the pairing $\nsp{\cdot}{\cdot}$, is a
Lie bialgebroid.

\subsection{Representations up to homotopy and VB-algebroids}
\label{subsect:ruths}
Let $A\to M$ be a Lie algebroid and consider an $A$-connection
$\nabla$ on a vector bundle $E\to M$.  Then the space
$\Omega^\bullet(A,E)$ of $E$-valued Lie algebroid forms has an induced
operator $\dr_\nabla$ given by:
\begin{equation*}
\begin{split}
  \dr_\nabla\omega(a_1,\ldots,a_{k+1})=&\sum_{i<j}(-1)^{i+j}\omega([a_i,a_j],a_1,\ldots,\hat a_i,\ldots,\hat a_j,\ldots, a_{k+1})\\
  &\qquad +\sum_i(-1)^{i+1}\nabla_{a_i}(\omega(a_1,\ldots,\hat
  a_i,\ldots,a_{k+1}))
\end{split}
\end{equation*}
for all $\omega\in\Omega^k(A,E)$ and $a_1,\ldots,a_{k+1}\in\Gamma(A)$.
The connection is flat if and only if $\dr_\nabla=0$.

Let now $\mathcal E= \bigoplus_{k\in \mathbb{Z}} E_k[k]$ be a graded
vector bundle. Consider the space $\Omega(A,\mathcal E)$ with grading
given by
$$
\Omega(A,\mathcal E)[k] = \bigoplus_{i+j=k}\Omega^i(A, E_j).
$$

\begin{definition}\cite{ArCr12}
  A \emph{representation up to homotopy of $A$ on $\mathcal E$} is a
  map \linebreak$\mathcal D\colon \Omega(A, \mathcal E) \to \Omega(A,\mathcal
  E)$ with total degree $1$ and such that $\mathcal D^2=0$ and
$$
\mathcal D(\alpha \wedge \omega) = \dr_A\alpha \wedge \omega +
(-1)^{|\alpha|} \alpha \wedge \mathcal D(\omega), \text{ for } \alpha
\in \Gamma(\wedge A^*), \, \omega \in \Omega(A,\mathcal E),
$$
where $\dr_A\colon \Gamma(\wedge A^*) \to \Gamma(\wedge A^*)$ is the
Lie algebroid differential.
\end{definition}
Note that Gracia-Saz and Mehta
\cite{GrMe10a} defined this concept 
independently and called them
``superrepresentations''. 

Let $A$ be a Lie algebroid. The representations up to homotopy which
we will consider are always on graded vector bundles $\mathcal E=
E_0[0]\oplus E_1[1]$ concentrated on degrees 0 and 1, so called
\emph{$2$-term graded vector bundles}.  These representations are
equivalent to the following data (see \cite{ArCr12,GrMe10a}):
\begin{enumerate}
\item [(1)] a map $\partial\colon E_0\to E_1$,
\item [(2)] two $A$-connections, $\nabla^0$ and $\nabla^1$ on $E_0$
  and $E_1$, respectively, such that $\partial \circ \nabla^0 =
  \nabla^1 \circ \partial$, \item [(3)] an element $R \in \Omega^2(A,
  \Hom(E_1, E_0))$ such that $R_{\nabla^0} = R\circ \partial$,
  $R_{\nabla^1}=\partial \circ R$ and $\dr_{\nabla^{\Hom}}R=0$,
  where $\nabla^{\Hom}$ is the connection induced on $\Hom(E_1,E_0)$
  by $\nabla^0$ and $\nabla^1$.
\end{enumerate}
For brevity we will call such a 2-term representation up to homotopy a
\textbf{2-re\-pre\-sen\-ta\-tion}.
\medskip

Let $(D\to B, A\to M)$ be a VB-Lie algebroid and choose a linear splitting
$\Sigma\colon A\times_MB\to D$. Since the anchor of a linear section
is linear, for each $a\in \Gamma(A)$ the vector field
$\Theta_B(\sigma_A(a))$ defines a derivation of $\Gamma(B)$ with
symbol $\rho(a)$ (see \S \ref{tangent_double}). This defines a linear
connection $\nabla^{AB}\colon \Gamma(A)\times\Gamma(B)\to\Gamma(B)$:
\[\Theta_B(\sigma_A(a))=\widehat{\nabla_a^{AB}}\]
for all $a\in\Gamma(A)$.  Since the bracket of a linear section with a
core section is again a core section, we find a linear connection
$\nabla^{AC}\colon\Gamma(A)\times\Gamma(C)\to\Gamma(C)$ such
that \[[\sigma_A(a),c^\dagger]=(\nabla_a^{AC}c)^\dagger\] for all
$c\in\Gamma(C)$ and $a\in\Gamma(A)$.  The difference
$\sigma_A[a_1,a_2]-[\sigma_A(a_1), \sigma_A(a_2)]$ is a core-linear
section for all $a_1,a_2\in\Gamma(A)$.  This defines a vector valued
Lie algebroid form $R\in\Omega^2(A,\operatorname{Hom}(B,C))$ such that
\[[\sigma_A(a_1), \sigma_A(a_2)]=\sigma_A[a_1,a_2]-\widetilde{R(a_1,a_2)},
\]
for all $a_1,a_2\in\Gamma(A)$. See \cite{GrMe10a} for more details on
these constructions.  The following theorem is proved in
\cite{GrMe10a}.
\begin{theorem}\label{rajan}
  Let $(D \to B; A \to M)$ be a VB-algebroid and choose a linear
  splitting $\Sigma\colon A\times_MB\to D$.  The triple
  $(\nabla^{AB},\nabla^{AC},R)$ defined as above is a
  $2$-representation of $A$ on the complex $\partial_B\colon C\to B$,
  where $\partial_B$ is the core-anchor.

  Conversely, let $(D;A,B;M)$ be a double vector bundle such that $A$
  has a Lie algebroid structure and choose a linear splitting
  $\Sigma\colon A\times_MB\to D$. Then if
  $(\nabla^{AB},\nabla^{AC},R)$ is a 2-representation of $A$ on a
  complex $\partial_B\colon C\to B$, then the three equations above
  and the core-anchor $\partial_B$ define a VB-algebroid structure on
  $(D\to B; A\to M)$.

\end{theorem}

 In the situation of the previous theorem, we have
\begin{equation*}
\left[\sigma_A(a),\widetilde\phi\right]=\widetilde{\nabla_a^{\rm Hom}\phi}
\qquad \text{ and }\qquad 
\left[c^\dagger,\widetilde\phi\right]=(\phi(\partial_Bc))^\dagger
\end{equation*}
for all $a\in\Gamma(A)$, $\phi\in\Gamma(\operatorname{Hom}(B,C))$ and
$c\in\Gamma(C)$, see for instance \cite{GrJoMaMe14}.

\begin{remark}\label{change}
  If $\Sigma_1,\Sigma_2\colon A\times_M B\to D$ are two linear
  splittings of a VB-algebroid $(D\to B, A\to M)$ and
  $\phi_{12}\in\Gamma(A^*\otimes B^*\otimes C)$ is the change of
  splitting, then the two corresponding 2-representations are related
  by the following identities \cite{GrMe10a}.
\[\nabla^{B,2}_a=\nabla^{B,1}_a+\partial_B\circ \phi_{12}(a), 
\quad \nabla^{C,2}_a=\nabla^{C,1}_a+\phi_{12}(a)\circ\partial_B\]
and
\begin{equation*}
\begin{split}
R^2(a_1,a_2)=&R^1(a_1,a_2)+(\dr_{\nabla^{\operatorname{Hom}(B,C)}}\phi_{12})(a_1,a_2)\\
&\qquad+\phi_{12}(a_1)\partial_B\phi_{12}(a_2)-\phi_{12}(a_2)\partial_B\phi_{12}(a_1)
\end{split}
\end{equation*}
for all $a,a_1,a_2\in\Gamma(A)$.
\end{remark}

\begin{example}\label{double_ruth}
  Choose a linear connection
  $\nabla\colon\mx(M)\times\Gamma(E)\to\Gamma(E)$ and consider the
  corresponding linear splitting $\Sigma^\nabla$ of $TE$ as in Section
  \ref{tangent_double}.  The description of the Lie bracket of vector
  fields in \eqref{Lie_bracket_VF} shows that the 2-representation
  induced by $\Sigma^\nabla$ is the 2-representation of $TM$ on
  $\Id_E\colon E\to E$ given by $(\nabla,\nabla,R_\nabla)$.
\end{example}

\begin{example}[The tangent of a Lie algebroid]\label{tangent_double_al}
Let $(A\to M,\rho,[\cdot\,,\cdot])$ be a Lie algebroid. Then the
tangent $TA\to TM$ has a Lie algebroid structure with bracket defined
by $[Ta_1, Ta_2]=T[a_1,a_2]$, $[Ta_1, a_2^\dagger]=[a_1,a_2]^\dagger$
and $[a_1^\dagger, a_2^\dagger]=0$ for all $a_1,a_2\in\Gamma(A)$. The
anchor of $Ta$ is $\widehat{[\rho(a),\cdot]}\in\mx(TM)$ and the anchor
of $a^\dagger$ is $\rho(a)^\uparrow$ for all $a\in\Gamma(A)$.  This
defines a VB-algebroid structure $(TA\to TM; A\to M)$ on
$(TA;TM,A;M)$.

Given a $TM$-connection on $A$, and so a linear splitting
$\Sigma^\nabla$ of $TA$ as in Section \ref{tangent_double}, the
2-representation of $A$ on $\rho\colon A\to TM$ encoding this
VB-algebroid is the \textbf{adjoint $2$-representation}
$(\nabla^{\rm bas},\nabla^{\rm bas}, R_\nabla^{\rm bas})$, where the
connections are defined by
\begin{equation*}
\begin{split}
\nabla^{\rm bas}&\colon \Gamma(A)\times\mx(M)\to\mx(M), \qquad \nabla^{\rm
    bas}_aX=[\rho(a), X]+\rho(\nabla_Xa)
\end{split}
\end{equation*} 
and 
\begin{equation*}
\begin{split}
  \nabla^{\rm bas}&\colon
  \Gamma(A)\times\Gamma(A)\to\Gamma(A), \qquad \nabla^{\rm bas}_{a_1}a_2=[a_1,a_2]+\nabla_{\rho(a_2)}a_1,
\end{split}
\end{equation*}  and $R_\nabla^{\rm
    bas}\in \Omega^2(A,\operatorname{Hom}(TM,A))$ is given by
  \[R_\nabla^{\rm
    bas}(a_1,a_2)X=-\nabla_X[a_1,a_2]+[\nabla_Xa_1,a_2]+[a_1,\nabla_Xa_2]+\nabla_{\nabla^{\rm
      bas}_{a_2}X}a_1-\nabla_{\nabla^{\rm bas}_{a_1}X}a_2
\]
for all $X\in\mx(M)$, $a, a_1,a_2\in\Gamma(A)$.
\end{example}

\subsubsection{Dualisation and $2$-representations}\label{dual_and_ruths}
Let $(D\to B, A\to M)$ be a VB-algebroid.  Let $\Sigma\colon
A\times_MB\to D$ be a linear splitting of $D$ and denote by
$(\nabla^B,\nabla^C,R)$ the 2-representation of the Lie algebroid $A$
on $\partial_B\colon C\to B$. We have seen above that $(D\duer A\to
C^*, A\to M)$ has an induced VB-algebroid structure, and we have shown
that the linear splitting $\Sigma$ induces a linear splitting
$\Sigma^\star\colon A\times_M C^*\to D\duer A$ of $D\duer A$.  The
2-representation of $A$ that is associated to this splitting is then
$({\nabla^C}^*,{\nabla^B}^*,-R^*) $ on the complex $\partial_B^*\colon
B^*\to C^*$. This is easy to verify, and proved in the
appendix\footnote{The construction of the ``dual'' linear splitting of
  $D\duer A$, given a linear splitting of $D$, is done in
  \cite{DrJoOr13} by dualising the decomposition and taking its
  inverse. The resulting linear splitting of $D\duer A$ is the same.}
of \cite{DrJoOr13}.  For the proof we only need to recall that, by
construction, $\ell_{\sigma_A^\star(a)}$ equals $\ell_{\sigma_A(a)}$
as a function on $D\duer B$.

\subsubsection{Double Lie algebroids and matched pairs of $2$-representations}\label{matched_pair_2_rep_sec}

\begin{definition}\cite{GrJoMaMe14}\label{matched_pair_2_rep}
  Let $(A\to M, \rho_A, [\cdot\,,\cdot])$ and and $(B\to M, \rho_B,
  [\cdot\,,\cdot])$ be two Lie algebroids and assume that $A$ acts on
  $B\oplus C$ up to homotopy via $(\partial_B\colon C\to B,
  \nabla^{AB},\nabla^{AC}, R_{AB})$ and $B$ acts on $A\oplus C$ up to
  homotopy via \linebreak
$(\partial_A\colon C\to A, \nabla^{BA},\nabla^{BC},
  R_{BA})$\footnote{For the sake of simplicity, we write in this
    definition $\nabla$ for all the four connections. It will always be
    clear from the indexes which connection is meant. We write
    $\nabla^A$ for the $A$-connection induced by $\nabla^{AB}$ and
    $\nabla^{AC}$ on $\wedge^2 B^*\otimes C$ and $\nabla^B$ for the
    $B$-connection induced on $\wedge^2 A^*\otimes C$.  }.  Then we
  say that the two representations up to homotopy form a matched pair
  if
\begin{enumerate}
%\item $\rho_A\circ \partial_A=\rho_B\circ\partial_B$,
\item
  $\nabla_{\partial_Ac_1}c_2-\nabla_{\partial_Bc_2}c_1=-\nabla_{\partial_Ac_2}c_1+\nabla_{\partial_Bc_1}c_2$,
\item $[a,\partial_Ac]=\partial_A(\nabla_ac)-\nabla_{\partial_Bc}a$,
\item $[b,\partial_Bc]=\partial_B(\nabla_bc)-\nabla_{\partial_Ac}b$,
%\item $[\rho_A(a),\rho_B(b)]=\rho_B(\nabla_ab)-\rho_A(\nabla_ba)$,
\item 
$\nabla_b\nabla_ac-\nabla_a\nabla_bc-\nabla_{\nabla_ba}c+\nabla_{\nabla_ab}c=
R_{BA}(b,\partial_Bc)a-R_{AB}(a,\partial_Ac)b$,
\item
  $\partial_A(R_{AB}(a_1,a_2)b)=-\nabla_b[a_1,a_2]+[\nabla_ba_1,a_2]+[a_1,\nabla_ba_2]+\nabla_{\nabla_{a_2}b}a_1-\nabla_{\nabla_{a_1}b}a_2$,
\item
  $\partial_B(R_{BA}(b_1,b_2)a)=-\nabla_a[b_1,b_2]+[\nabla_ab_1,b_2]+[b_1,\nabla_ab_2]+\nabla_{\nabla_{b_2}a}b_1-\nabla_{\nabla_{b_1}a}b_2$,
\end{enumerate}
for all $a,a_1,a_2\in\Gamma(A)$, $b,b_1,b_2\in\Gamma(B)$ and
$c,c_1,c_2\in\Gamma(C)$, and
\begin{enumerate}\setcounter{enumi}{6}
\item $\dr_{\nabla^A}R_{BA}=\dr_{\nabla^B}R_{AB}\in \Omega^2(A,
  \wedge^2B^*\otimes C)=\Omega^2(B,\wedge^2 A^*\otimes C)$, where
  $R_{AB}$ is seen as an element of $\Omega^1(A, \wedge^2B^*\otimes
  C)$ and $R_{AB}$ as an element of $\Omega^1(B, \wedge^2A^*\otimes
  C)$.
\end{enumerate}
\end{definition}

\begin{remark}\label{lie_bracket_on_C}
  From these equations follow
  $\rho_A\circ \partial_A=\rho_B\circ\partial_B$ and
  $[\rho_A(a),\rho_B(b)]=\rho_B(\nabla_ab)-\rho_A(\nabla_ba)$ for all
  $a\in\Gamma(A)$ and $b\in\Gamma(B)$. 

  The vector bundle $C$ inherits a Lie algebroid structure with anchor
  $\rho_A\circ  \partial_A=\rho_B\circ\partial_B$   and  with  bracket
  given                            by                           $[c_1,
  c_2]=\nabla_{\partial_Ac_1}c_2-\nabla_{\partial_Bc_2}c_1$   for  all
  $c_1,c_2\in\Gamma(C)$.  The proof of  the Jacobi  identity is  a not
  completely straightforward computation; it follows from (2), (3) and
  (4) and can be done just as the proof of Theorem \ref{core_courant}.
\end{remark}

Consider a double vector bundle $(D;A,B;M)$ with core $C$ and a VB-Lie
algebroid structure on each of its sides.  After the choice of a
splitting $\Sigma\colon A\times_M B\to D$, the Lie algebroid
structures on the two sides of $D$ are described as above by two
$2$-representations; the Lie algebroid $D\to B$ is
described by
$(\partial_B,\nabla^{AB},\nabla^{AC},R_{AB}\in\Omega^2(A,\operatorname{Hom}(B,C)))$. The
Lie algebroid structure of $D\to A$ is described by $(\partial_A,
\nabla^{BA},\nabla^{BC},R_{BA}\in\Omega^2(B,\operatorname{Hom}(A,C)))$,
where $\partial_A\colon C\to A$,
$\nabla^{BA}\colon\Gamma(B)\times\Gamma(A)\to\Gamma(A)$,
$\nabla^{BC}\colon\Gamma(B)\times\Gamma(C)\to \Gamma(C)$ are
connections and $R_{BA}$ is the curvature term.

We prove in \cite{JoMa14} that $(D;A,B,M)$ is a double Lie algebroid
if and only if, for any decomposition of $D$, the two induced
2-representations above form a matched pair.

\section{{[2]}-manifolds and metric double vector bundles}\label{sec:metDVB}
In this section we recall the definition of N-manifolds of degree
$2$. Then we introduce linear metrics on double vector bundles, and we
show how the category of N-manifolds of degree $2$ is equivalent to
the category of metric double vector bundles.
\subsection{N-manifolds}
Here we give the definitions of N-manifolds. We are particularly
interested in N-manifolds of degree $2$. We refer to \cite{BoPo13} for
more details.

\begin{definition}\label{n_man} An \textbf{ N-manifold} or \textbf{$\N$-graded
    manifold} $\mathcal M$ of degree $n$ and dimension $(p;
  r_1,\ldots, r_n)$ is a smooth $p$-dimensional manifold $M$ endowed
  with a sheaf $\mathcal A=C^\infty(\mathcal M)$ of $\N$-graded
  commutative associative unital $\R$-algebras, whose degree 0 term is
  $\mathcal A^0=C^\infty(M)$ and which is locally freely generated as
  a sheaf of $C^\infty(M)$-algebras by $r_1+\ldots+ r_n$ graded
  commutative generators
  $\xi_1^{1},\ldots,\xi_1^{r_1}$,$\xi_2^1,\ldots,\xi_2^{r_2},\ldots$, $\xi_n^1,\ldots,\xi_n^{r_n}$
  with $\xi_i^j$ of degree $i$ for $i\in\{1,\ldots,n\}$ and
  $j\in\{1,\ldots,r_i\}$.

  A morphism of $N$-manifolds $\mu:\mathcal M\to \mathcal N$ is a
  smooth map $\mu_0\colon M\to N$ of the underlying smooth manifolds
  together with a morphism $\mu^\star\colon C^\infty(\mathcal N)\to
  C^\infty(\mathcal M)$ of sheaves of graded algebras such that
  $\mu^\star(f)=\mu_0^*f$ for all $f\in C^\infty(N)$.
\end{definition}

We will call \textbf{$[n]$-manifold} an N-manifold of degree $n\in
\N$. We will write $\mathcal A^i$ for the elements of degree $i$ in
$\mathcal A$, and we will write $|\xi|$ for the degree of a
homogeneous element $\xi\in\mathcal A$, i.e.~an element which can be
written as a sum of functions of the same degree.  Note that a
\textbf{$[1]$-manifold} over a smooth manifold $M$ is equivalent to a
locally finitely generated sheaf of $C^\infty(M)$-modules.

Let us quickly introduce the notion of vector field on an N-manifold. Let $\mathcal
M$ be an $[n]$-manifold and write as before $\mathcal
A=C^\infty(\mathcal M)$.  A \textbf{vector field} of degree $j$ on $\mathcal M$ 
is a degree $j$ derivation $\phi$ of $\mathcal A$
such that
\[|\phi(\xi)|=j+|\xi|
\]
for a homogeneous element $\xi\in\mathcal A$.  As in \cite{BoPo13}, we
write $\der^\bullet\mathcal A$ for the sheaf of graded derivations of
$\mathcal A$.  % We have $\der^\bullet\mathcal
% A=\oplus_{l\geq -n}\der^l\mathcal A$.

The vector fields on $\mathcal M$ and their Lie bracket defined 
by $[\phi,\psi]=\phi\psi-(-1)^{|\phi||\psi|}\psi\phi$
 satisfy the following 
conditions:
\begin{enumerate}
\item $\phi(\xi\eta)=\phi(\xi)\eta+(-1)^{|\phi||\xi|}\xi\phi(\eta)$,
\item $[\phi,\psi]=(-1)^{1+|\phi||\psi|}[\psi,\phi]$,
\item $[\phi,\xi\psi]=\phi(\xi)\psi+(-1)^{|\phi||\xi|}\xi[\phi,\psi]$,
\item $(-1)^{|\phi||\gamma|}[\phi,[\psi,\gamma]]+(-1)^{|\psi||\phi|}[\psi,[\gamma,\phi]]
+(-1)^{|\gamma||\psi|}[\gamma,[\phi,\psi]]=0$
\end{enumerate}
for $\phi,\psi,\gamma$ homogeneous elements of $\der^\bullet\mathcal
A$ and $\xi,\eta$ homogeneous elements of $\mathcal A$.  For instance,
the derivation $\partial_{\xi^i_j}$ of $\mathcal A(U)$ sends $\xi^i_j$
to $1$ and the other local generators to $0$. It is hence a derivation
of degree $-j$. Locally, $\der^\bullet\mathcal A(U)$ is generated as a
$\mathcal A(U)$-module by $\partial_{x_k}$, $k=1,\ldots,p$ and
$\partial_{\xi^i_j}$, $j=1,\ldots,n$, $i=1,\ldots,r_j$.

\bigskip Our goal in this chapter is to prove that $[2]$-manifolds are
equivalent to double vector bundles endowed with a linear metric
(Theorem \ref{main_crucial}).  We begin with a few observations on the
equivalence of sheaves of locally finite $C^\infty$-modules with
smooth vector bundles. Theorem \ref{main_crucial} will in a sense
generalise this result.

\subsubsection{Vector bundles and $[1]$-manifolds}\label{classical_eq}
Here we recall the equivalence of categories between degree
$[1]$-manifolds (or locally finitely generated sheaves of
$C^\infty$-modules) and smooth vector bundles (see for instance
\cite[Theorem II.1.13]{Wells08}). This section can be seen as
introductory to the methods in Section \ref{eq_2_manifolds}.

Let $\operatorname{VB}$ be the category of smooth vector bundles, with
the following morphisms. Let $E\to M$ and $F\to N$ be vector
bundles. Then a morphism $\Phi\colon F\dashrightarrow E$ of vector
bundles is a vector bundle map $\phi\colon F^*\to E^*$ over
$\phi_0\colon N\to M$. Recall from Lemma
\ref{bundlemap_eq_to_morphism} that this is equivalent to a map
$\phi^\star\colon \Gamma(E)\to\Gamma(F)$ defined as in
\eqref{dual_of_VB_map} and satisfying
\[\phi^\star(f\cdot e)=\phi_0^*f\cdot\phi^\star(e)
\]
for all $f\in C^\infty(M)$ and $e\in \Gamma(E)$.

Let $\operatorname{[1]-Man}$ be the category of $[1]$-manifolds, or
equivalently the category of locally finitely generated sheaves of
$C^\infty(M)$-modules, for smooth manifolds $M$.  The morphisms in
this category are defined as follows.  Let $\mathcal A_M$ and
$\mathcal A_N$ be two sheaves of $C^\infty(M)$, respectively
$C^\infty(N)$-modules, for two smooth manifolds $M$ and $N$.  A
morphism $\mu\colon\mathcal A_N\dashrightarrow \mathcal A_M$ is a pair
of a morphism $\mu^\star\colon\mathcal A_M\to\mathcal A_N$ of sheaves
of modules and a smooth map $\mu_0\colon N\to M$, such that
\[\mu^\star(f\cdot a)=\mu_0^*f\cdot \mu^\star(a)
\]
for all $f\in C^\infty(U)$ and $a\in\mathcal A_M(U)$, $U$ open in $M$.

\medskip

We now establish the equivalence between $\operatorname{VB}$ and
$\operatorname{[1]-Man}$.  The functor \linebreak$\Gamma(\cdot)\colon
\operatorname{VB}\to\operatorname{[1]-Man}$ sends a vector bundle
$E\to M$ to its set of sections $\Gamma(E)$, a locally finitely
generated sheaf of $C^\infty(M)$-modules. $\Gamma(\cdot)$ sends a
morphism $\Phi=(\phi,\phi_0)\colon F\dashrightarrow E$ as above to the
morphism $\phi^\star\colon \Gamma(E)\to \Gamma(F)$ over
$\phi_0^*\colon C^\infty(M)\to C^\infty(N)$. % The functor
% $\Gamma(\cdot)$ is contravariant.

\medskip Next choose a $[1]$-manifold $\mathcal A$ over a smooth
manifold $M$.  There exists a maximal covering $\{U_\alpha\}$ of $M$
such that $\mathcal A(U_\alpha)$ is finitely generated by generators
$\xi^\alpha_1,\ldots,\xi^\alpha_m$. For two indices $\alpha,\beta$
such that $U_\alpha\cap U_\beta\neq\emptyset$, we can write each
generator in an unique manner as
$\xi^\beta_j=\sum_{i=1}^m\psi_{\alpha\beta}^{ij}\xi^\alpha_i$ with
smooth functions $\psi_{\alpha\beta}^{ij}\in C^\infty(U_\alpha\cap
U_\beta)$.  We define $A_{\alpha\beta}\in C^\infty(U_\alpha\cap
U_\beta, \operatorname{Gl}(\R^m))$ by
$A_{\alpha\beta}(x)=(\psi_{\alpha\beta}^{ij})_{i,j=1,\ldots,m}$.  We
have then immediately \begin{equation} \label{cocycle}
  A_{\gamma\alpha}\cdot A_{\alpha\beta}=A_{\gamma\beta},
\end{equation}
 where $\cdot$ is the pointwise
multiplication of matrices.  Next we consider the disjoint union
$\tilde E=\bigsqcup_{\alpha}U_\alpha\times \R^m$
and identify for $x\in U_\alpha\cap U_\beta\neq \emptyset$
\[(x,v)\in U_\beta\times \R^m \quad \text{ with }\quad (x, A_{\alpha\beta}(x)v)\in U_\alpha\times \R^m .
\]
By \eqref{cocycle}, this defines an equivalence relation on $\tilde E$
and the quotient $E=E(\mathcal A^1)$ has a smooth vector bundle
structure with vector bundle charts given by the inclusions
$U_\alpha\times \R^m\hookrightarrow E$, and changes of charts the
cocycles $A_{\alpha\beta}$.  Note that the maps $e_i^\alpha\colon
U_\alpha\to U_\alpha\times\R^m$, $x\mapsto (x,e_i)$ define smooth
local sections of $E$ and
$e_i^\beta=\sum_{j=1}^n\psi_{\alpha\beta}^{ji}e_j^\alpha$ for
$\alpha,\beta$ such that $U_\alpha\cap U_\beta\neq \emptyset$. Hence,
we can identify $\xi_i^\alpha$ with the section $e_i^\alpha$ and we
see that a morphism $\mu\colon\mathcal A^1_N\dashrightarrow \mathcal
A^1_M$ over $\mu_0\colon N\to M$ defines a morphism
$E(\mu)^\star\colon \Gamma(E(\mathcal A^1_M))\to\Gamma(E(\mathcal
A^1_N))$ of modules over $\mu_0^*\colon C^\infty(M)\to C^\infty(N)$,
and so by Lemma \ref{bundlemap_eq_to_morphism} a vector bundle
morphism $E^*(\mathcal A^1_N)\to E^*(\mathcal A^1_M)$ over
$\mu_0\colon N\to M$.  Hence we have constructed a functor
$E(\cdot)\colon \operatorname{[1]-Man}\to\operatorname{VB}$.

\medskip

Next we show that the two functors build together an equivalence of
categories.  The functor $\E(\cdot)\circ \mathcal A(\cdot)\colon
\operatorname{VB}\to\operatorname{VB}$ sends a vector bundle to the
abstract vector bundle defined by its trivialisations and cocycles.
There is an obvious natural isomorphism between this functor and the
identity functor $\operatorname{VB}\to\operatorname{VB}$.

The functor $\mathcal A(\cdot)\circ E(\cdot)\colon
\operatorname{[1]-Man}\to\operatorname{[1]-Man}$ sends a
$[1]$-manifold over $M$ with local generators $\xi^\alpha_i$ and
cocycles $A_{\alpha\beta}$ to the sheaf of sections of $E(\mathcal
A)$, with local basis sections $e_i^\alpha$ and cocycles
$A_{\alpha\beta}$. There is an obvious natural isomorphism between this functor and the
identity functor $\operatorname{[1]-Man}\to\operatorname{[1]-Man}$.

\medskip

Finally note that it would be more natural to define a
(contravariant!)  functor from the category of $[1]$-manifolds to the
category of vector bundles with the usual notion of vector bundle
morphism, by sending a vector bundle $E\to M$ to the sheaf of sections
of its dual $E^*$ (i.e.~the sheaf over $M$ of linear functions on
$E$), and by sending a $[1]$-manifold $\mathcal A$ to the dual of
$E(\mathcal A)$ constructed above.
In the remainder of this section we will extend
 the less natural equivalence to an equivalence of metric double 
vector bundles with $[2]$-manifolds. In that case, the equivalence 
of categories will be more natural in this manner.

\subsubsection{Split N-manifolds}
Next we quickly discuss split N-manifolds and we recall how each
N-manifold is noncanonically isomorphic to a split N-manifold of the
same degree and of the same dimension.
\begin{example} 
\begin{enumerate}
\item  Let $E$ be a smooth vector bundle of rank $r$ over a smooth manifold
  $M$ of dimension $p$ and assign the degree $n$ to the fiber coordinates
  of $E$. This defines $E[-n]$, an $[n]$-manifold of dimension
  $(p;r_1=0,\ldots,r_{i-1}=0,r_i=r)$ with
  $\mathcal A^n=\Gamma(E^*)$.

\item Now let $E_{-1},E_{-2},\ldots,E_{-n}$ be smooth vector bundles
  of finite ranks $r_1,\ldots,r_n$ over $M$ and assign the degree $i$
  to the fiber coordinates of $E_{-i}$, for each $i=1,\ldots,n$.  The
  direct sum $E=E_{-1}\oplus \ldots\oplus E_{-n}$ is a graded vector
  bundle with grading concentrated in degrees $-1,\ldots,-n$.  The
  $[n]$-manifold $E_{-1}[-1]\oplus\ldots\oplus E_{-n}[-n]$
  has local basis sections of ${E_{-i}}^*$ as local generators of
  degree $i$, for $i=1,\ldots,n$, and so dimension $(p;r_1,\ldots,r_n)$. %  The sheaf $\mathcal
  % A=C^\infty(\mathcal M)$ is hence the sheaf of sections of the graded
  % symmetric tensor algebra of the dual bundle $E^*=
  % {E_{-1}}^*\oplus\ldots\oplus {E_{-n}}^*$.
  An $[n]$-manifold
  $\mathcal M=E_{-1}[-1]\oplus\ldots\oplus E_{-n}[n]$ defined in this
  manner by a graded vector bundle is called a \textbf{split
    $[n]$-manifold}.\end{enumerate}
\end{example}
In this paper, we are exclusively interested in the cases $n=2$ and
$n=1$.  Choose two vector bundles $E_{-1}$ and $E_{-2}$ of ranks $r_1$
and $r_2$ over a smooth manifold $M$. Set $E=E_{-1}\oplus E_{-2}$ and
consider $\mathcal M=E_{-1}[-1]\oplus E_{-2}[-2]$.  We find $\mathcal
A^0=C^\infty(M)$, $\mathcal A^1=\Gamma({E_{-1}}^*)$ and $\mathcal
A^2=\Gamma({E_{-2}}^*\oplus \wedge^2{E_{-1}}^*)$.

  A morphism $\mu\colon E_{-1}[-1]\oplus E_{-2}[-2]\to
  F_{-1}[-1]\oplus F_{-2}[-2]$ of split $[2]$-manifolds over the bases $M$ and $N$, respectively, consists of a
  smooth map $\mu_0\colon M\to N$, % where $M$ is the base of the vector
  % bundles $E_{-1},E_{-2}$ and $N$ is the base of the vector bundles
  % $F_{-1},F_{-2}$,
  three vector bundle morphisms $\mu_1\colon
  E_{-1}\to F_{-1}$, $\mu_2\colon E_{-2}\to F_{-2}$ and
  $\mu_{12}\colon \wedge^2E_{-1}\to F_{-2}$ over $\mu_0$. The map
  $\mu^\star$ sends a degree $1$ function $\xi\in\Gamma(F_{-1})$ to
  \begin{equation}\label{morphism_dec_1}
{\mu_1}^\star\xi\in\Gamma(E_{-1}), \qquad
  ({\mu_1}^\star\xi)(m)={\mu_1}_m^*(\xi(\mu_0(m)))\quad \text{ for all
  }m\in M,
\end{equation}
 and a degree $2$-function $\xi\in\Gamma(F_{-2}^*)$ to
  \begin{equation}\label{morphism_dec_2}
{\mu_2}^\star\xi+\mu_{12}^\star\xi\in\Gamma({E_{-2}}^*\oplus
  \wedge^2{E_{-1}}^*).
\end{equation}
Any N-manifold is non-canonically diffeomorphic to a split
N-manifold. In other words, the categories of split N-manifolds and of
N-manifolds are equivalent.  This is proved for instance in
\cite{BoPo13}. The super-version of this theorem is due to Batchelor
\cite{Batchelor80} and is known as Batchelor's theorem.

\begin{theorem}\label{split_N} Any $[n]$-manifold
  is non-canonically diffeomorphic to a split $[n]$-manifold.
\end{theorem}

We give here the proof by \cite{BoPo13} in the case $n=2$. We are
especially interested in the morphism of split $[2]$-manifolds induced
by a change of splitting of a $[2]$-manifold and will emphasize this
in the proof.
\begin{proof}[Sketch of Proof, \cite{BoPo13}]
  Consider a $[2]$-manifold $\mathcal M$ over a smooth base manifold
  $M$. Since ${\mathcal A}^0=C^\infty(M)$ and ${\mathcal A}^0{\mathcal
    A}^1\subset {\mathcal A}^1$, the sheaf ${\mathcal A}^1$ is a
  locally free sheaf of $C^\infty(M)$-modules and there exists a
  vector bundle $E\to M$ such that ${\mathcal A}^1\simeq
  \Gamma(E)$. Set $E_{-1}^*=E$. Now let ${\mathcal A}_1$ be the
  subalgebra of $\mathcal A$ generated by ${\mathcal
    A}^0\oplus{\mathcal A}^1$. We find easily that $\mathcal A_1\simeq
  \Gamma(\wedge^\bullet E_{-1}^*)$ and ${\mathcal A}_{1}\cap {\mathcal
    A}^2=(\mathcal A^1)^2$ is a proper $\mathcal A^0$-submodule of
  $\mathcal A^2$. Since the quotient $\mathcal A^2/(\mathcal A^1)^2$
  is a locally free sheaf of $C^\infty(M)$-modules, we have ${\mathcal
    A}^2/({\mathcal A}^1)^2\simeq \Gamma(E_{-2}^*)$, where $E_{-2}$ is
  a vector bundle over $M$. The short exact sequence \begin{equation*}
    0\rightarrow ({\mathcal A}^1)^2\hookrightarrow {\mathcal
      A}^2\rightarrow \Gamma(E_{-2}^*)\rightarrow 0\end{equation*} of
  ${\mathcal A}^0$-modules is non canonically split. Let us choose a
  splitting and identify $\Gamma(E_{-2}^*)$ with a submodule of
  ${\mathcal A}^2$:
$${\mathcal A}^2 \simeq ({\mathcal A}^1)^2\oplus\Gamma(E_{-2}^*)= \Gamma(\wedge^2E^*_{-1}\oplus
E^*_{-2})\;.$$ Hence, the considered $[2]$-manifold is diffeomorphic,
modulo the chosen splitting, to the split $[2]$-manifold
$E_{-1}[-1]\oplus E_{-2}[-2]$.

Note finally that a change of splitting is equivalent to a section
$\phi$ of\linebreak $\operatorname{Hom}(E_{-2}^*,\wedge^2E_{-1}^*)$ and
induces an isomorphism of split $[2]$-manifolds over the identity on
$M$: $\mu^\star(\xi)=\xi+\phi\circ\xi\in\Gamma(E_{-2}^*\oplus
\wedge^2E_{-1}^*)$ for all $\xi\in\Gamma(E_{-2}^*)$ and
$\mu^\star(\xi)=\xi$ for all $\xi\in\Gamma(E_{-1}^*)$.
\end{proof}

Note that $[1]$-manifolds are automatically split. As we have seen in
\S \ref{classical_eq}, $[1]$-manifolds are just vector bundles with a
degree shifting in the fibers, i.e.~$\mathcal M=E[-1]$ for some vector
bundle $E\to M$ and $C^\infty(\mathcal M)=\Gamma(\bigwedge^\bullet
E^*)$, the exterior algebra of $E$.  We finally give the definition of
wide $[1]$-submanifolds of $[2]$-manifolds.
\begin{definition}\label{def_1_subman}
  Let $\mathcal M$ be a $[2]$-manifold of dimension $(p;r_1,r_2)$. A
  wide $[1]$-submanifold $\mathcal N$ of $\mathcal M$ is a
  $[1]$-manifold of dimension $(p;r)$, $r\leq r_1$, over the same
  smooth base $M$, together with a morphism of $N$-manifolds
  $\mu\colon \mathcal N\to \mathcal M$ over the identity on $M$ and
  such that locally, $\mu_U^\star\colon C^\infty_U(\mathcal M)\to
  C^\infty_U(\mathcal N)$, $\mu^\star_U(\xi_2^j)=0$ for
  $j=1,\ldots,r_2$, $\mu^\star_U(\xi_1^j)=0$ for $j=r+1,\ldots,r_1$ and
  $\mu^\star_U(\xi_1^j)=\eta_1^j$ for $j=1,\ldots,r$.  Here,
  $\xi_1^j,\xi_2^k$ are local generators of $C^\infty(\mathcal M)$ and
  $\eta_1^j$ are local generators for $C^\infty(\mathcal N)$.
\end{definition}

More explicitly, if $\mathcal M$ splits as $Q[-1]\oplus B^*[-2]$,
then $\mathcal N$ can be understood as $U[-1]$ for a subbundle
$U\subseteq Q$.  The map $\mu^\star$ is then locally described as
follows.  Take a trivialising open set $V\subseteq M$ for both $B$ and
$Q$.  Choose local basis fields $u_1,\ldots,u_r$ for $U$ over $V$, and
complete this list to local basis fields
$u_1,\ldots,u_r,q_{r+1},\ldots,q_{r_1}$ for $Q$. The dual basis fields
$\tau_1,\ldots,\tau_{r_1}$ for $Q^*$ satisfy then
$\tau_{r+1},\ldots,\tau_{r_1}\in\Gamma(U^\circ)$. The morphism
$\mu^\star_U$ sends $\tau_j$ to $0$ for $j=r+1,\ldots,r_1$ and to
$\bar\tau_j=\tau_j+U^\circ\in\Gamma(Q^*/U^\circ)\simeq\Gamma(U^*)$ for
$j=1,\ldots,r$.

\medskip

Finally note that if an $[n]$-manifold $\mathcal M$ splits as
$E_1[-1]\oplus E_2[-2]\oplus\ldots\oplus E_n[-n]$, then each section
$e$ of $E_j$ defines a derivation $\hat{e}$ of degree $-j$ on
$\mathcal M$: $\hat{e}(f)=0$,
$\hat{e}(\xi_j^i)=\langle e, \xi_j^j\rangle$, and $\hat{e}(\xi_k^i)=0$
for $k\neq j$.  We find $\hat{e_j^i}=\partial_{\xi_j^i}$ if
$\{e_j^1, \ldots,e_j^{r_j}\}$ is a local basis of $E_j$ and
$\{\xi_j^1,\ldots,\xi_j^{r_j}\}$ is the dual basis of $E_j^*$.

Further, a derivation $\phi$ of degree $0$ on $\mathcal M$ can be
written as a sum
\begin{equation*}
X+D_1+D_2+\ldots+D_n,
\end{equation*}
with $X\in \mx(M)$ and each $D_i$ a derivation of $E_i^*$ with symbol
$X\in\mx(M)$.  The derivation $X+D_1+\ldots+D_n$ can be written in coordinates as
\[\sum_{i=1}^pX(x_i)\partial_{x_i}+\sum_{i,j=1}^{r_1}\langle e^i_1,
D_1(\xi^j_1)\rangle\xi^i_1\partial_{\xi_1^j}+\ldots
+\sum_{i,j=1}^{r_1}\langle e^i_n,
D_n(\xi^j_n)\rangle\xi^i_n\partial_{\xi_n^j}.
\]
In particular, if for each $j$ the map $D^j\colon \mx(M)\to\der(E_j)$ is a morphism of
$C^\infty(M)$-modules 
that sends a vector field $X$ to a derivation $D^j(X)$ over $X$, 
then 
\begin{equation}\label{der_deg_0}
\{X+D^1(X)+\ldots+D^n(X)\mid X\in\mx(M)\}\cup\{\hat{e}\mid
e\in\Gamma(E_j) \text{ for some } j\} 
\end{equation}
span $\der(C^\infty(\mathcal M))$ as a $C^\infty(\mathcal M)$-module.

\subsection{Metric double vector bundles} 
Next we introduce linear metrics on double vector bundles.
\begin{definition}
A \textbf{metric double vector bundle} is a double vector bundle
$(\mathbb E, Q; B, M)$ equipped with a \textbf{linear symmetric
  non-degenerate pairing on $\mathbb E\times_B\mathbb E\to \R$},
i.e.~such that the map
\begin{equation}
\begin{xy}
  \xymatrix{\mathbb{E}\ar[rr]^{\pi_B}\ar[dd]_{\pi_Q}&&B\ar[dd]^{q_B}&
    & &\mathbb{E}\duer B\ar[rr]\ar[dd]&&B\ar[dd]\\
    &C\ar[dr]&   &\ar[r]^{\Beta}& &&Q^*\ar[dr]&\\
    Q\ar[rr]_{q_Q}&&M&&&C^*\ar[rr]&&M }
\end{xy}
\end{equation}
defined by the pairing $\langle\cdot,\cdot\rangle$ is an isomorphism
of double vector bundles.  In particular, the core $C\to M$ of
$\mathbb{E}$ is canonically isomorphic to $Q^*\to M$.
 \end{definition}
 Note that equivalently, a linear symmetric non-degenerate pairing
 $\langle\cdot\,,\cdot\rangle$ on $\mathbb E\to B$ is linear if the
 core of $\mathbb E$ is isomorphic to $Q^*$ and, via this isomorphism,
\begin{enumerate}
\item $\langle \tau_1^\dagger, \tau_2^\dagger\rangle=0$ for
  $\tau_1,\tau_2\in\Gamma(Q^*)$,
\item $\langle \chi, \tau^\dagger\rangle=q_B^*\langle q,\tau\rangle$
  for $\chi\in\Gamma_B^l(\mathbb E)$ linear over $q\in\Gamma(Q)$ and
  $\tau\in\Gamma(Q^*)$ and
\item $\langle\chi_1,\chi_2\rangle$ is a linear function on $B$ for
  $\chi_1,\chi_2\in \Gamma_B^l(\mathbb E)$.
\end{enumerate}
In the following, we will always identify with $Q^*$ the core of a
metric double vector bundle $(\mathbb E, Q; B, M)$.

Note also that the \textbf{opposite} $(\overline{\mathbb E};Q;B,M)$ of
a metric double vector bundle $(\overline{\mathbb E};B;Q,M)$ is the
metric double vector bundle with
\[\langle\cdot\,,\cdot\rangle_{\overline{\mathbb
    E}}=-\langle\cdot\,,\cdot\rangle_{\mathbb E}.
\]

\subsubsection{Lagrangian decompositions of a metric double vector
  bundle}\label{Lagr_dec}

\begin{definition}
  Let $(\mathbb E, B; Q, M)$ be a metric double vector bundle.  A
  linear splitting $\Sigma\colon Q\times_MB\to \mathbb E$ (or
  equivalently a decomposition of $\mathbb E$) is said to be
  \textbf{Lagrangian} if its image is maximal isotropic in $\mathbb
  E\to B$.  The corresponding horizontal lift
  $\sigma_Q\colon\Gamma(Q)\to\Gamma^l_B(\mathbb E)$ is then also said
  to be \textbf{Lagrangian}.
\end{definition}
Note that by definition, a horizontal lift $\sigma_Q\colon
\Gamma(Q)\to \Gamma^l_B(\mathbb E)$ is Lagrangian if and only if
$\langle \sigma_Q(q_1), \sigma_Q(q_2)\rangle=0$ for all $q_1,q_2\in
\Gamma(Q)$.

Recall from Section \ref{dual} that given a linear splitting
$\Sigma\colon Q\times_M B\to \mathbb E$, one can construct a linear
splitting $\Sigma^\star\colon Q^{**}\times_M B\to \mathbb E\duer B$.
\begin{lemma}\label{lem_Lagr_split_sections}
  Let $(\mathbb E; Q, B; M)$ be a metric double vector bundle and
  choose a linear splitting $\Sigma$ of $\mathbb E$.  Then $\Sigma$ is
  Lagrangian if and only if the linear map $\Beta\colon \mathbb E\to
  \mathbb E\duer B$ sends $\sigma_B(b)$ to $\sigma^\star_{B}(b)$ for
  all $b\in\Gamma(B)$.
\end{lemma}

\begin{proof}
  Recall that Lemma \ref{lemma_dual_splitting} states that given a
  horizontal lift $\sigma_B\colon\Gamma(B)\to\Gamma_Q^l(\mathbb E)$,
  the dual horizontal lift
  $\sigma^\star_B\colon\Gamma(B)\to\Gamma_{Q^{**}}^l(\mathbb E\duer
  B)$ can be defined by
\[\langle \sigma_B^\star(b)(p_m),\sigma_B(b)(q_m)\rangle_B=0, \qquad
\langle \sigma_B^\star(b)(p_m),\tau^\dagger(b(m))\rangle_B=\langle p_m,
\tau(m)\rangle
\]
for all $b\in\Gamma(B)$, $\tau\in\Gamma(Q^*)$, $q_m\in Q$ and $p_m\in
Q^{**}\simeq Q$.

On the other hand, if $\Sigma\colon B\times_MQ\to\mathbb E$ is a
Lagrangian splitting, we have \begin{equation*}
\begin{split}
\langle \Beta(\sigma_B(b)(p(m))),
\sigma_B(b)(q(m))\rangle_B&= \langle \sigma_B(b)(p(m)),
\sigma_B(b)(q(m))\rangle_{\mathbb E}\\
&=\langle
\sigma_Q(p),\sigma_Q(q)\rangle_{\mathbb E}(b(m))=0
\end{split}
\end{equation*}
 for all $q,p\in\Gamma(Q)$ and
$b\in\Gamma(B)$, and
\begin{equation*}
\begin{split}
  \langle \Beta(\sigma_B(b)(p(m))), \tau^\dagger(b(m))\rangle_B&=
  \langle \sigma_B(b)(p(m)), \tau^\dagger(b(m))\rangle_{\mathbb E}\\
  &=\langle \sigma_Q(p)(b(m)), \tau^\dagger(b(m))\rangle_{\mathbb
    E}=\langle p,\tau\rangle(m)
\end{split}
\end{equation*}
for all $\tau\in\Gamma(Q^*)$. This proves that $\Beta$ sends the
linear section $\sigma_B(b)\in\Gamma_Q^l(\mathbb E)$ to
$\sigma^\star_B(b)\in \Gamma^l_{Q^{**}}(\mathbb E\duer B)$.  It is
easy to see from the four equalities above that this condition is
necessary for $\Sigma$ to be Lagrangian.
\end{proof}

Let $\sigma_Q\colon\Gamma(Q)\to\Gamma^l_B(\mathbb E)$ be an arbitrary
horizontal lift. We have seen that by definition of a linear metric on
$\mathbb E\to B$, the pairing of two linear sections is a linear
function on $B$. This implies with
\[ \sigma_Q(fq)=q_B^*f\cdot\sigma_Q(q) \text{ and }
\ell_{f\beta}=q_B^*f\cdot\ell_\beta \text{ for all } f\in C^\infty(M),
q\in \Gamma(Q) \text{ and } \beta\in\Gamma(B^*)
\]
the existence of a symmetric tensor $\Lambda\in S^2(Q,B^*)$ such that
  \begin{equation}\label{lambda_def}
    \langle \sigma_Q(q_1), \sigma_Q(q_2)\rangle_{\mathbb E}=\ell_{\Lambda(q_1,q_2)}.
\end{equation}
In particular, $\Lambda(q,\cdot): Q\to B^*$ is a morphism of vector
bundles for each $q\in\Gamma(Q)$.  Define a new horizontal lift
$\sigma_Q'\colon\Gamma(Q)\to \Gamma^l_B(\mathbb E)$ by
$\sigma_Q'(q)=\sigma_Q(q)-\frac{1}{2}\widetilde{\Lambda(q,\cdot)^*}$ for
all $q\in\Gamma(Q)$. 
Since for $\phi\in\Gamma(\operatorname{Hom}(B,Q^*))$, $\langle
\widetilde{\phi}, \chi\rangle=\ell_{\phi^*(q)}$ if
$\chi\in\Gamma_B^l(\mathbb E)$ is linear over $q\in\Gamma(Q)$, we find then
\[
\langle \sigma_Q'(q_1),\sigma_Q'(q_2)\rangle_{\mathbb E}
=\langle\sigma_Q(q_1),\sigma_Q(q_2)\rangle_{\mathbb E}-\frac{1}{2}\ell_{\Lambda(q_1,q_2)}-\frac{1}{2}\ell_{\Lambda(q_2,q_1)}
=0
\] for all $q_1,q_2\in\Gamma(Q)$. This proves
the following result.
\begin{theorem}\label{symmetrization}
  Let $(\mathbb E, B; Q, M)$ be a metric double vector bundle.  Then
  there exists a Lagrangian splitting of $\mathbb E$.
\end{theorem}

Next we show that a change of Lagrangian splitting corresponds to a
skew-symmetric element of $\Gamma(Q^*\otimes B^*\otimes Q^*)$.
\begin{proposition}\label{lagrangian_change_of_split}
  Let $(\mathbb E, B; Q, M)$ be a metric double vector bundle and
  choose a Lagrangian horizontal lift
  $\sigma_Q^1\colon\Gamma(Q)\to\Gamma_B^l(\mathbb E)$.  Then a second
  horizontal lift $\sigma_Q^2\colon\Gamma(Q)\to\Gamma_B^l(\mathbb E)$
  is Lagrangian if and only if the change of lift
  $\phi_{12}\in\Gamma(Q^*\otimes B^*\otimes Q^*)$ satisfies the
  following equality:
\[
\langle \phi_{12}(q),q'\rangle=-\langle\phi_{12}(q'),
q\rangle\in\Gamma(B^*)\] for all $q,q'\in\Gamma(Q)$, i.e.~if and only
if $\phi_{12}\in\Gamma(Q^*\wedge Q^*\otimes B^*)$
\end{proposition}

\begin{proof}
  For $q\in\Gamma(Q)$ we have $\langle\widetilde{\phi_{12}(q)},
  \chi\rangle=\ell_{\langle\phi_{12}(q), q'\rangle}$ for any linear
  section $\chi\in\Gamma_B^l(\mathbb E)$ over $q'\in\Gamma(Q)$.  Hence
  we find
\begin{equation*}
\begin{split}
  &\langle \sigma_Q^1(q), \sigma_Q^1(q')\rangle_{\mathbb E}-\langle \sigma_Q^2(q),
  \sigma_Q^2(q')\rangle_{\mathbb E} \\
&=\langle \sigma_Q^1(q)-\sigma_Q^2(q),
  \sigma_Q^1(q')\rangle_{\mathbb E}
  +\langle \sigma_Q^2(q), \sigma_Q^1(q')-\sigma_Q^2(q')\rangle_{\mathbb E}\\
  &=\ell_{\langle \phi_{12}(q),q'\rangle}+\ell_{\langle q,
    \phi_{12}(q')\rangle}
\end{split}
\end{equation*}
and we can conclude.
\end{proof}

\begin{remark}
  It is interesting to see that the last proposition implies that not
  any linear section of $\mathbb E$ over $B$ can be obtained as the
  Lagrangian horizontal lift of a section of $Q$. This is easy to
  understand in the following example.
\end{remark}

\begin{example}\label{metric_connections}
  Let $E\to M$ be a metric vector bundle, i.e.~a vector bundle endowed
  with a symmetric non-degenerate pairing
  $\langle\cdot\,,\cdot\rangle\colon E\times_M E\to \R$.  Then
  $E\simeq E^*$ and the tangent double is a metric double vector
  bundle $(TE,E;TM,M)$ with pairing $TE\times_{TM}TE\to \R$ the
  tangent of the pairing $E\times_M E\to \R$. In particular, we have
\[\langle Te_1, Te_2\rangle_{TE}=\ell_{\dr\langle e_1,e_2\rangle}, 
\quad \langle Te_1, e_2^\dagger\rangle_{TE}=p_M^* \langle e_1,e_2\rangle
\quad \text{ and }\langle e_1^\dagger, e_2^\dagger\rangle_{TE}=0\] for
$e_1,e_2\in\Gamma(E)$.

Recall from \S\ref{tangent_double} that linear splittings of $TE$ are
equivalent to linear connections $\nabla\colon
\mx(M)\times\Gamma(E)\to\Gamma(E)$.  We have then for all $e_1,e_2\in\Gamma(\mathsf E)$:
  \[\left\langle \sigma_{\mathsf E}^\nabla(e_1), e_2^\dagger\right\rangle=\left\langle
    Te_1-\widetilde{\nabla_\cdot e_1},
    e_2^\dagger\right\rangle=p_M^*\langle e_1, e_2\rangle\] and
  \[\left\langle \sigma_{\mathsf E}^\nabla (e_1), \sigma_{\mathsf E}^\nabla (e_2)\right\rangle=\left\langle
    Te_1-\widetilde{\nabla_\cdot e_1}, Te_2-\widetilde{\nabla_\cdot
    e_2}\right\rangle=\ell_{\dr\langle e_1, e_2\rangle-\langle
    e_2, \nabla_\cdot e_1\rangle-\langle e_1, \nabla_\cdot e_2\rangle}.\] 
The Lagrangian splittings of $TE$
are hence exactly the linear splittings that correspond to \textbf{metric}
connections, i.e.~linear connections $\nabla\colon
\mx(M)\times\Gamma(E)\to\Gamma(E)$ that preserve the metric:
\[\langle\nabla_\cdot e_1, e_2\rangle+\langle e_1,\nabla_\cdot e_2\rangle=\dr\langle e_1, e_2\rangle
\quad \text{ for } e_1,e_2\in\Gamma(E).
\]
\end{example}

\begin{example}\label{met_TET*E}
Let $q_E\colon E\to M$ be a vector bundle and consider the double vector bundle
\begin{equation*}
\begin{xy}
  \xymatrix{
    TE\oplus T^*E\ar[rr]^{\Phi_E:=({q_E}_*, r_E)}\ar[d]_{\pi_E}&& TM\oplus E^*\ar[d]\\
    E\ar[rr]_{q_E}&&M }
\end{xy}
\end{equation*} with sides $E$ and $TM\oplus E^*\to M$, and 
with core $E\oplus T^*M\to M$.  The projection
$r_E\colon T^*E\to E$ is defined by
\[r_E(\theta_{e_m})\in E^*_m, \qquad \langle r_E(\theta_{e_m}),
e_m'\rangle =\left\langle \theta_{e_m},
  \left.\frac{d}{dt}\right\an{t=0}e_m+te_m'\right\rangle,
 \]
 and is a fibration of vector bundles over the projection $q_E\colon
 E\to M$.  The core elements are identified in the following manner with elements of $E\oplus
 T^*M\to M$. For $m\in M$ and
 $(e_m,\theta_m)\in E_m\times T_m^*M$, the pair
\[\left(\left.\frac{d}{dt}\right\an{t=0}te_m, (T_{0_m^E}q_E)^*\theta_m
\right)\] projects to $(0^{TM}_m,0^{E^*}_m)$ under $\Phi_E$ and to
$0^E_m$ under $\pi_E$. Conversely, any element of $TE\oplus T^*E$ in
the double kernel can be written in this manner.  Next recall that
$TE\oplus_E T^*E\to E$ has a natural pairing given by
\begin{equation}\label{standard_pairing}
\langle (v^1_{e_m},\theta^1_{e_m}),  (v^2_{e_m},\theta^2_{e_m})\rangle=\theta^1_{e_m}(v^2_{e_m})+\theta^2_{e_m}(v^1_{e_m}),
\end{equation}
the natural pairing underlying the standard Courant algebroid
structure on $TE\oplus_E T^*E\to E$.

We prove in \cite{Jotz13a} that linear splittings of
$TE\oplus_E T^*E$  are in bijection with dull
brackets on sections of $TM\oplus E^*$, and so also with Dorfman
connections $\Delta\colon \Gamma(TM\oplus E^*)\times\Gamma(E\oplus
T^*M)\to\Gamma(E\oplus T^*M)$. Choose such a Dorfman connection.  For
any pair $(X,\epsilon)\in\Gamma(TM\oplus E^*)$, the horizontal lift
$\sigma:=\sigma_{TM\oplus E^*}^\Delta\colon\Gamma(TM\oplus
E^*)\to\Gamma_E(TE\oplus T^*E)=\mx(E)\times\Omega^1(E)$ is given by
\[\sigma(X,\epsilon)(e_m)=\left(T_me X(m), \dr
    \ell_\epsilon(e_m)\right)-\Delta_{(X,\epsilon)}(e,0)^\dagger(e_m)
\]
for all $e_m\in E$.

The vector bundle $TM\oplus E^*$ is anchored by the morphism
$\pr_{TM}\colon TM\oplus E^*\to TM$.  As aconsequence, the $TM$-part
of $\lb q_1, q_2\rb_\Delta +\lb q_2, q_1\rb_\Delta$ is trivial and
this sum be seen as an element of $\Gamma(E^*)$.  We proved the
following result in \cite{Jotz13a}.
\begin{theorem}\label{Lagrangin1}
  Choose $q, q_1,q_2\in\Gamma(TM\oplus E^*)$ and $\tau,\tau_1,
  \tau_2\in\Gamma(E\oplus T^*M)$.  The natural pairing on fibres of
  $TE\oplus T^*E\to E$ is given by
 \begin{enumerate}
\item $\left\langle \sigma(q_1), \sigma(q_2)\right\rangle=\ell_{\lb q_1, q_2\rb_\Delta +\lb q_2, q_1\rb_\Delta}$,
\item $\left\langle \sigma(q),
    \tau^\dagger\right\rangle=q_E^*\langle q, \tau\rangle$.
% \item $\left\langle \tau_1^\dagger,
%     \tau_2^\dagger\right\rangle=0$.
\end{enumerate}
\end{theorem}
As a consequence, the natural pairing on fibres of $TE\oplus T^*E\to
E$ is a linear metric on $(TE\oplus T^*E;TM\oplus E^*,E;M)$ and the
Lagrangian splittings are equivalent to skew-symmetric dull brackets
on sections of the anchored vector bundle $(TM\oplus E^*,\pr_{TM})$.
\end{example}

\subsubsection{The category of metric double vector bundles}\label{morphisms_of_met_DVB}
We define a morphism $\Omega$ of metric double vector bundles
$(\mathbb E; Q,B;M)$ and $(\mathbb F;P,A;M)$ as the dual of a genuine
morphism $\mathbb E\duer Q\to \mathbb F\duer P$ of double vector
bundles.
\begin{definition}
A morphism 
\begin{equation*}
\Omega\colon \mathbb F\dashrightarrow\mathbb E 
\end{equation*}
of metric double vector bundles
is an isotropic relation $\Omega\subseteq \overline{\mathbb
  F}\times\mathbb E$ that is the dual of a double vector bundle
morphism
\begin{equation*}
  \omega\colon\begin{xy}
    \xymatrix{\mathbb F\duer P\ar[r]\ar[d]& P\ar[d]\\
      P^{**}\ar[r]&N}
\end{xy}
\rightarrow\begin{xy}
\xymatrix{\mathbb E\duer Q\ar[r]\ar[d]& Q\ar[d]\\
Q^{**}\ar[r]&M}
\end{xy}
\end{equation*}
over $\omega_0\colon N\to M$.

We write $\operatorname{MDVB}$ for the obtained category of metric double vector bundles.
\end{definition} 
Let $(\mathbb E, B; Q, M)$ be a metric double vector bundle. 
Choose a Lagrangian splitting $\Sigma\colon Q\times_MB\to \mathbb E$
and set
\[\mathcal A^2(\mathbb
E):=\sigma_B(\Gamma(B))+
\{\tilde\omega\mid \omega\in\Gamma(Q^*\wedge Q^*)\}.
\]
In other words, $\mathcal A^2(\mathbb E)$ is the $C^\infty(M)$-module
generated by all Lagrangian horizontal lifts of sections of $B$. Note
that by Proposition \ref{lagrangian_change_of_split}, $\mathcal
A^2(\mathbb E)$ does not depend on the choice of Lagrangian splitting.

\begin{theorem}\label{met_maps}
  Let $(\mathbb E;B,Q;M)$ and $\mathbb F;A,P;N)$ be two metric double
  vector bundles.  A morphism $\Omega\colon \mathbb
  F\dashrightarrow\mathbb E$ is equivalent to a triple of maps
\begin{align*}
\omega^\star \colon \mathcal A^2(\mathbb E)\to \mathcal A^2(\mathbb F),\quad 
\omega_P^\star \colon \Gamma(Q^*)\to \Gamma(P^*),\quad\text{ and } \quad 
\omega_0 \colon N\to M
\end{align*}
such that
\begin{enumerate}
\item $\omega^\star\left(\widetilde{\tau_1\wedge\tau_2}\right)=\widetilde{\omega_P^\star\tau_1\wedge\omega_P^\star\tau_2}$,
\item
  $\omega^\star(q_Q^*f\cdot\chi)=q_P^*(\omega_0^*f)\cdot\omega^\star(\chi)$
  and
\item $\omega_P^\star(f\cdot \tau)=\omega_0^*f\cdot \omega_P^\star\tau$
\end{enumerate}
for all $\tau,\tau_1,\tau_2\in\Gamma(Q^*)$, $f\in C^\infty(M)$ and
$\chi \in \mathcal A^2(\mathbb E)$.
\end{theorem}

\begin{proof}%[Proof of Theorem \ref{met_maps}]
By definition, a morphism $\Omega\colon \mathbb
  F\dashrightarrow\mathbb E$ is a morphism 
\[\omega\colon \mathbb F\duer P \to \mathbb E\duer Q\]
of double vector bundles with some further properties.  Let
$\omega_P\colon P\to Q$ be the induced vector bundle map on the side
$P$ of $\mathbb F\duer P$, and let $\omega_0\colon N\to M$ be the
induced smooth map on the double base.  The morphism of double vector bundles
induces a morphism $\omega^\star\colon \Gamma_Q(\mathbb E)\to
\Gamma_P(\mathbb F)$ of modules over $\omega_P^*\colon C^\infty(Q)\to
C^\infty(P)$.  Since $\omega_P$ is a vector bundle map, the pullback
$\omega_P^*\colon C^\infty(Q)\to C^\infty(P)$ is completely determined
by its value on linear functions and on pullbacks under $q_Q$ of
functions on $M$;
\[\omega_P^*(q_Q^*f)=q_P^*(\omega_0^*f)
\]
for all $f\in C^\infty(M)$ and
\[\omega_P^*(\ell_{\tau})=\ell_{\omega^\star_P(\tau)}
\]
for all $\tau\in\Gamma(Q^*)$.  The map $\omega^\star$ sends linear
sections to linear sections and core sections to core sections, and it
is completely determined by its images on these two sets of sections.
We denote by $\omega^\star\colon \Gamma^l_Q(\mathbb E)\to
\Gamma^l_P(\mathbb F)$ the induced map. We need to check that the induced
map on core sections is given by
$\omega^\star(\tau^\dagger)=(\omega_P^\star(\tau))^\dagger$ for all
$\tau\in \Gamma(Q^*)$ and 
that 
$\omega^\star$ restricts further to a map
\[\omega^\star\colon \mathcal A(\mathbb E) \to \mathcal A(F).
\]

The morphism $\omega$ of double vector bundles induces a vector bundle
map $\omega_{A^*}\colon A^*\to B^*$ on the cores, and so a map
$\omega_{A^*}^\star\colon \Gamma(B)\to \Gamma(A)$ of modules over
$\omega_0^*\colon C^\infty(M)\to C^\infty(N)$.  If
$\chi\in\Gamma_Q^l(\mathbb E)$ is linear over $b\in \Gamma(B)$, then
$\omega^\star(\chi)\in\Gamma_P^l(\mathbb F)$ is linear over
$\omega_{A^*}^\star(b)\in\Gamma(A)$. In particular, choose Lagrangian
splittings $\Sigma^e\colon Q\times_MB\to \mathbb E$, $\Sigma^f\colon
P\times_N A\to \mathbb F$.  Then the image of $\sigma^e_B(b)$ is
$\sigma_A^f(\omega_{A^*}^\star(b))+\widetilde{\psi}$ for some
$\psi\in\Gamma(\operatorname{Hom}(P,P^*))$.  We have for all
$p^1,p^2\in \Gamma(P)$, $n\in N$:
\begin{align*}
  0&=\langle \sigma_B^e(b)(\omega_P(p^1_n)),
  \sigma_B^e(b)(\omega_P(p^2_n))\rangle\\
  &=\langle \sigma_A^f(\omega_{A^*}^\star(b))(p^1_n)+\widetilde{\psi}(p^1_n), \sigma_A^f(\omega_{A^*}^\star(b))(p^2_n)+\widetilde{\psi}(p^2_n)\rangle\\
  &=\langle
  \sigma_P^f(p^1)(\omega_{A^*}^\star(b)(n))+(\psi(p^1))^\dagger(\omega_{A^*}^\star(b)(n)),
  \sigma_P^f(p^2)(\omega_{A^*}^\star(b)(n))\\
  &\qquad \qquad \qquad\qquad \qquad \qquad \qquad \qquad \qquad \qquad \qquad \qquad 
 +(\psi(p^2))^\dagger(\omega_{A^*}^\star(b)(n)) \rangle\\
  &=0+\langle \psi(p^1(n)),p^2(n)\rangle+\langle
  \psi(p^2(n)),p^1(n)\rangle+0.
\end{align*}
This shows that $\psi\in\Gamma(P^*\wedge P^*)$, and so that
$\omega^\star(\sigma_B^e(b))\in\mathcal A^2(\mathbb F)$.

Finally note that for $\nu\in\Gamma(P^*)$ and $p_1,p_2\in \Gamma(P)$, we have 
\[\nu^\dagger(p_1(n))=\nu^\dagger(0^A(n))+_A\sigma_P(p_1)(0^A(n))
\]
and so
\begin{equation*}
\begin{split}
\langle \nu^\dagger(p_1(n)), \sigma_A(0^A)(p_2(n))\rangle
&=\langle\nu^\dagger(0^A(n))+_A\sigma_P(p_1)(0^A(n)) , \sigma_P(p_2)(0^A(n))\rangle\\
&=\langle\nu,p_2\rangle(n).
\end{split}
\end{equation*}
Hence, since for $\tau\in\Gamma(Q^*)$, $\omega^\star(\tau^\dagger)$ is
a core section of $\mathbb F$ over $P$ and
\begin{equation*}
\begin{split}
  \langle \omega^\star(\tau^\dagger)(p^1(n)),
  \omega^\star(\sigma_B^e(0^B))(p^2(n))\rangle
  &=\langle\tau^\dagger(\omega_P(p^1(n))),\sigma_B^e(0^B)(\omega_P(p^2(n)))\rangle\\
  &=\langle \tau(\omega_0(n)),\omega_P(p^2(n))\rangle\\
  &=\langle \omega_P^*\tau(\omega_0(n)),p^2(n)\rangle=\langle
  (\omega_P^\star\tau)(n), p^2(n)\rangle,
\end{split}
\end{equation*}
we find that $\omega^\star(\tau^\dagger)=\left(\omega_P^\star\tau\right)^\dagger$.

Now we check the 3 conditions. (2) and (3) follow immediately from the definition
of $\omega_P^\star$ and $\omega^\star$. To see (1), write $\widetilde{\tau_1\wedge\tau_2}$ as 
$\ell_{\tau_1}\tau_2^\dagger-\ell_{\tau_2}\tau_1^\dagger$. Then 
\[\omega^\star\left(\widetilde{\tau_1\wedge\tau_2}\right)
=\ell_{\omega_P^\star\tau_1}(\omega_P^\star\tau_2)^\dagger-\ell_{\omega_P^\star\tau_2}(\omega_P^\star\tau_1)^\dagger
=\widetilde{\omega_P^\star\tau_1\wedge\omega_P^\star\tau_2}.
\]
\end{proof}

\begin{remark}
\begin{enumerate}
\item The morphism
\[\omega_{A^*}^\star\colon \Gamma(B)\to \Gamma(A)
\]
of modules over $\omega_0^*\colon C^\infty(M)\to C^\infty(N)$,
i.e.~the vector bundle morphism $\omega_{A^*}\colon A^*\to B^*$ is
induced as follows by the three maps in the theorem.  If
$\chi\in\Gamma_Q^l(\mathbb E)$ is linear over $b\in \Gamma(B)$, then
$\omega^\star(\chi)$ is linear over $\omega_{A^*}^\star(b)$. To see
that $\omega_{A^*}^\star$ is well-defined
(i.e.~$\omega_{A^*}^\star(b)$ does not depend on the choice of $\chi$
over $b$), use (1) in the theorem.  To see that
$\omega_{A^*}^\star(f\cdot b)=\omega_0^*f\cdot\omega_{A^*}^\star(b)$
for $f\in C^\infty(M)$, use that $q_Q^*f\cdot\chi$ is linear over
$f\cdot b$ and (2) in the theorem.
\item A morphism $\Omega\colon
  B_2\times_{M_2}Q_2\times_{M_2}Q_2^*\dashrightarrow
  B_1\times_{M_1}Q_1\times_{M_1}Q_1^* $ of decomposed metric double
  vector bundles is consequently described by $\omega_Q\colon Q_1\to Q_2$,
  $\omega_{B}\colon B_1^*\to B_2^*$ and $\omega_{12}\colon Q_1\wedge
  Q_1\to B_2^*$, all morphisms of vector bundles over a smooth map
  $\omega_0\colon M_1\to M_2$.  

  For $b\in\Gamma(B_2)$ the section $b^l\in\Gamma_{Q_2}^l(
  B_2\times_{M_2}Q_2\times_{M_2}Q_2^*)$,
  $b^l(q_m)=(b(m),q_m,0_m^{Q^*})$, is sent by $\omega^\star$ to
  $(\omega_B^\star(b))^l+\widetilde{\omega_{12}^\star(b)}\in\Gamma_{Q_1}^l(B_1\times_{M_1}Q_1\times_{M_1}Q_1^*)$,
  where for $\phi\in\Gamma(\operatorname{Hom}(Q_1,Q_1^*))$,
  $\widetilde\phi\in\Gamma_{Q_1}^l(B_1\times_{M_1}Q_1\times_{M_1}Q_1^*)$
  is defined by $\widetilde{\phi}(q_m)=(0^B_m,q_m,\phi(q_m))$.  For
  $\tau\in\Gamma_{M_2}(Q_2^*)$, the core section $\tau^\dagger\in \Gamma_{Q_2}^c(
  B_2\times_{M_2}Q_2\times_{M_2}Q_2^*)$, 
$\tau^\dagger(q_m)=(0^B_m,q_m,\tau(m))$, 
is sent to
  $(\omega_Q^\star\tau)^\dagger\in \Gamma_{Q_1}^c(
  B_1\times_{M_1}Q_1\times_{M_1}Q_1^*)$.
\end{enumerate}
\end{remark}

\subsection{Equivalence of [2]-manifolds and metric double vector
  bundles}\label{eq_2_manifolds}
In this section our goal is to prove the following theorem.
\begin{theorem}\label{main_crucial}
  The category of metric double vector bundles and the category of
  $[2]$-manifolds are equivalent.
\end{theorem}
% Let as before $(\mathbb E; B,Q;M)$ be a metric double vector bundle.

\subsubsection{The functor $\mathcal M(\cdot)\colon \operatorname{MDVB}\to \operatorname{[2]-Man}$.}
Let $(\mathbb E, B; Q, M)$ be a metric double vector bundle.  We
define a $[2]$-manifold $\mathcal M(\mathbb E)$ as follows. The sheaf
$\mathcal A(\mathbb E)=C^\infty(\mathcal M(\mathbb E))$ of $\N$-graded
commutative associative unital $\R$-algebras is generated by $\mathcal
A^0(\mathbb E)=C^\infty(M)$, $\mathcal A^1(\mathbb E)=
% \Gamma_Q^c(\mathbb E)\simeq 
\Gamma(Q^*)$ and the sheaf of degree $2$-functions is $\mathcal
A^2(\mathbb E)$. The product of $\tau_1$ with $\tau_2\in\Gamma(Q^*)$
is $\widetilde{\tau_1\wedge\tau_2}\in\mathcal A^2(\mathbb E)$.  The
product of $f\in C^\infty(M)$ with $\xi\in \mathcal A^2(\mathbb E)$ is
$q_Q^*f\cdot\xi\in \mathcal A^2(\mathbb E)$ and the product of
elements of $\mathcal A^0$ with elements of $\mathcal A^1$ is obvious
since $\Gamma(Q^*)$ is a sheaf of $C^\infty(M)$-modules.  This proves
that $\mathcal A(\mathbb E)$ is well-defined.
  
Next we check that $\mathcal A(\mathbb E)$ is locally freely generated
over $C^\infty(M)$.  Choose $m\in M$ and a coordinate neighborhood
$U\ni m$ such that $Q$ and $B$ are trivialized on $U$ by the basis
frames $(q_1,\ldots, q_k)$ and $(b_1,\ldots,b_l)$. Let
$(\tau_1,\ldots,\tau_k)$ be the dual frame to $(q_1,\ldots, q_k)$,
i.e. a trivialization for $Q^*$. Recall also that after the choice of
a Lagrangian splitting $\Sigma\colon Q\times_M B\to \mathbb E$, each
element $\xi\in\mathcal A^2(\mathbb E)$ over $b\in\Gamma(B)$ can be
written $\sigma_B(b)+\widetilde{\phi}$ with
$\phi\in\Gamma(Q^*\wedge Q^*)$. Since
$\sigma_B\left(\sum_{i=1}^lf_ib_i\right)=\sum_{i=1}^l
q_Q^*f_i\cdot\sigma_B(b_i)$ for $f_1,\ldots,f_l\in C^\infty(M)$, we
conclude that $\mathcal A(\mathbb E)$ is generated on $U$ by
$\{\tau_1,\ldots,\tau_k,\sigma_B(b_1),\ldots,\sigma_B(b_l)\}$ over
$C^\infty(M)$.

  \medskip We have constructed a map $\mathcal M(\cdot)$ sending
  metric double vector bundles to $[2]$-manifolds.  By
  Theorem \ref{met_maps} a morphism
  $\Omega\colon \mathbb F\dashrightarrow \mathbb E$ of metric double
  vector bundles is the same as a triple of maps
  \[\omega_0\colon N\to M \quad \Leftrightarrow \quad \omega_0^*\colon
  C^\infty(M)\to C^\infty(N),\]
\[\omega^\star\colon \mathcal A^2(\mathbb E)\to \mathcal A^2(\mathbb F) \quad 
\text{ and } \quad \omega_P^\star\colon \Gamma(Q^*)\to \Gamma(P^*)
\]
with
\[
\omega^\star\left(\widetilde{\tau_1\wedge\tau_2}\right)=\widetilde{\omega_P^\star\tau_1\wedge\omega_P^\star\tau_2},\qquad
q_P^*\omega_0^*f\cdot \omega^\star(\chi)=\omega^\star(q_Q^*f\cdot\chi)\]
 and
\[\omega_0^*f\cdot\omega_P^\star(\tau)=\omega^\star(f\cdot\tau)
\]
for $f\in C^\infty(M)=\mathcal A^0(\mathbb E)$, $\tau\in
\Gamma(Q^*)=\mathcal A^1(\mathbb E)$ and $\chi\in \mathcal A^2(\mathbb
E)$. Hence we find that the triple
$(\omega^\star,\omega_P^\star,\omega_0^*)$ defines a morphism
$\mathcal M(\Omega)\colon \mathcal M(\mathbb F)\to \mathcal M(\mathbb
E)$ of $[2]$-manifolds. 
 \medskip

 We have so defined a functor $\mathcal M(\cdot)\colon
 \operatorname{MDVB}\to \operatorname{[2]-Man}$ from the category of
 metric double vector bundles to the category of $[2]$-manifolds.

\subsubsection{The functor $\mathcal G\colon\operatorname{[2]-Man} \to \operatorname{MDVB}$.}
Conversely, we construct explicitly a metric double vector bundle
associated to a given $[2]$-manifold $\mathcal M$. The idea is to
adapt the proof of the equivalence between locally free sheaves of
$C^\infty(M)$-modules with vector bundles over $M$ (see \S\ref{classical_eq}).

First we give Pradines'
original definition of a double vector bundle \cite{Pradines77} (in
the smooth and finite-dimensional case).
\begin{definition}\cite[C. \S 1]{Pradines77}\label{double_atlas}
  Let $M$ be a smooth manifold and $D$ a set with a map $\Pi\colon
  D\to M$. A \textbf{double vector bundle chart} is a quintuple
  $c=(U,\Theta,V_1,V_2,V_0)$, where $U$ is an open set in $M$,
  $V_1,V_2,V_3$ are three vector spaces and $\Theta\colon \Pi\inv(U)\to
  U\times V_1\times V_2\times V_0$ is a bijection such that
  $\Pi=\pr_1\circ\Theta$.

  Two double vector bundle charts $c$ and $c'$ are \textbf{compatible}
  if the ``change of chart'' $\Theta'\circ\Theta\inv$ over $U\cap U'$
  has the following form:
\[(x,v_1,v_2,v_0)\mapsto (x,A_1(x)v_1,A_2(x)v_2,A_0(x)v_0+\omega(x)(v_1,v_2))
\]
with $x\in U\cap U'$, $v_i\in V_i$, $A_i\in C^\infty(M,
\operatorname{Gl}(V_i))$ for $i=0,1,2$ and \linebreak$\omega\in
C^\infty(M,\operatorname{Hom}(V_1\otimes V_2,V_0))$.

A \textbf{double vector bundle atlas} $\lie A$ on $D$ is a set of
double vector bundle charts of $D$ that are pairwise compatible and
such that the set of underlying open sets in $M$ is a covering of $M$.
As usual, two double vector bundle atlases $\lie A_1$ and $\lie A_2$
are \textbf{equivalent} if their union is an atlas.  A double vector
bundle structure on $D$ is an equivalence class of double vector
bundle atlases on $D$.
\end{definition}

Given a $[2]$-manifold $\mathcal M$, we interpret its functions as the
components of a double vector bundle atlas, and show that the
obtained double vector bundle has a natural metric structure.

Let $M$ be the smooth manifold underlying $\mathcal M$ and assume that
$\mathcal M$ has dimension $(l;m,n)$.  Choose a maximal open covering
$\{U_\alpha\}$ of $M$ such that $\mathcal A(U_\alpha)$ is freely
generated by $\xi_1^\alpha,\ldots,\xi_m^\alpha$ (in degree $1$) and
$\eta^\alpha_1,\ldots,\eta^\alpha_n$ (degree $2$ generators).  Choose
now $\alpha,\beta$ such that $U_\alpha\cap U_\beta\neq
\emptyset$. Then each generator $\xi_i^\beta$ can be written in a
unique manner as $\sum_{j=1}^m\omega_{\alpha\beta}^{ji}\xi^\alpha_j$ with $\omega^{ji}\in
C^\infty(U_\alpha\cap U_\beta)$.  Each generator $\eta^\beta_i$ can
be written \[
\eta^\beta_i=\sum_{j=1}^n\psi^{ji}_{\alpha\beta}\cdot\left(\eta^\alpha_j+\sum_{1\leq k<l\leq
  m}\rho^{jkl}_{\alpha\beta}\cdot\xi^\alpha_k\wedge\xi^\alpha_l\right)\] with $\psi_{\alpha\beta}^{ij},
\rho_{\alpha\beta}^{ikl}\in C^\infty(U_\alpha\cap U_\beta)$.  Set
$A_1^{\alpha\beta}=(\omega^{ij}_{\alpha\beta})_{i,j}\in
C^\infty(M,\operatorname{Gl}({\R^m}^*))$,
$A_2^{\alpha\beta}=(\psi^{ij}_{\alpha\beta})_{i,j}\in
C^\infty(M,\operatorname{Gl}(\R^n))$. Define $\nu^{\alpha\beta}\in C^\infty(M, \operatorname{Hom}(\R^m\otimes
\R^n,{\R^m}^*))$ by \linebreak
$\nu^{\alpha\beta}(e_i,e_j)(e_l)=\rho^{jil}$
for $1\leq i<l\leq m$ and $j=1,\ldots,n$. 
Then by construction
\begin{equation}\label{2_cocycle}
\begin{split}
   A_1^{\gamma\alpha}\cdot A_1^{\alpha\beta}
  &=A_1^{\gamma\beta},\qquad A_2^{\gamma\alpha}\cdot A_2^{\alpha\beta}=A_2^{\gamma\beta}\quad\text{ and }\\
  \nu^{\gamma\beta}({A_1^{\beta\gamma}}^*(e_i),
  A_2^{\gamma\beta}(e_j))(e_l)&=\nu^{\gamma\alpha}({A_1^{\beta\gamma}}^*(e_i),
  A_2^{\gamma\beta}(e_j), e_l)\\
&\qquad \qquad   +\nu^{\alpha\beta}({A_1^{\beta\alpha}}^*(e_i), A_2^{\alpha\beta}(e_j))( {A_1^{\gamma\alpha}}^*e_l).
\end{split}
\end{equation}

Set $\tilde{\mathbb E}=\bigsqcup_\alpha
U_\alpha\times \R^m\times \R^n\times {\R^m}^*$ (the disjoint union)
and identify \[ (x,v_1,v_2,l_0)\in U_\beta\times\R^m\times \R^n\times {\R^m}^* \] with
\[\left(x, (A^{\beta\alpha}_1(x))^*(v_1), A^{\alpha\beta}_2(x)(v_2),A_1^{\alpha\beta}(x)(l_0)
  +\nu^{\alpha\beta}(x)((A_1^{\beta\alpha}(x))^*(v_1),
  A^{\alpha\beta}_2(x)(v_2))\right)\] in $U_\alpha\times\R^m\times
\R^n\times {\R^m}^*$ for $x\in U_\alpha\cap U_\beta$. The cocycle
equations \eqref{2_cocycle} imply that this defines an equivalence
relation on $\tilde{\mathbb E}$.  The quotient space is $\mathbb E$, a
double vector bundle that does not depend anymore on the choice of
charts covering $M$. The map $\Pi\colon \mathbb E\to M$,
$(x,v_1,v_2,l_0)\mapsto x$ is well-defined and, by construction, the
charts $c=(U_\alpha,\Theta_\alpha=\Id,\R^m,\R^n,{\R^m}^*)$ define a
double vector bundle atlas on $\mathbb E$. Since the covering was
chosen to be maximal, the obtained double atlas of $\mathbb E$ does
not depend on any choices.

Recall from the proof of Theorem \ref{split_N} that there are two
vector bundles $E_{-1}$ and $E_{-2}$ associated canonically to a
$[2]$-manifold $\mathcal M$ (only the inclusion of $\Gamma(E_{-2}^*)$
in $\mathcal A^2$ is non-canonical). $E_{-1}$ and $E_{-2}^*$ are the
sides of $\mathbb E$ and $E_{-1}^*$ is the core of $\mathbb E$.  The
vector bundle $E_{-1}^*$ can be defined by the transition functions
$A_1^{\alpha\beta}$ and $\tilde
E_{-1}=\bigsqcup_{\alpha}U_\alpha\times \R^n$, $(x,v)\sim
(x,A_1^{\alpha\beta}(x)(v))$ for $x\in U_\alpha\cap U_\beta$. The
vector bundle $E_{-2}$ can be defined in the same manner using
$A_2^{\alpha\beta}$ and the model space $\R^m$.

We use this to define a linear metric on $\mathbb E$. Over a chart domain $U_\alpha$ we set 
\[\langle (x, v_1,v_2,l_0), (x,v_1',v_2,l_0')\rangle=l_0(v_1')+l_0'(v_1).\]
By construction, this does not depend on the choice of $\alpha$ with
$x\in U_\alpha$.

\medskip

Again by definition of the morphisms in the category of
$[2]$-manifolds and in the category of metric double vector bundles,
this defines a functor $\mathcal G\colon \operatorname{[2]-Man}\to
\operatorname{MDVB}$ between the two categories.

\subsubsection{Equivalence of categories.}

Finally we need to prove that the two obtained functors define an
equivalence of categories.  The functor $\mathcal G\circ\mathcal
M(\cdot)$ is the functor that sends a metric double vector bundle to
its maximal double vector bundle atlas, hence it is naturally
isomorphic to the identity functor.

The functor $\mathcal M(\cdot)\circ \mathcal G\colon
\operatorname{[2]-Man}\to\operatorname{[2]-Man}$ sends a
$[2]$-manifold $\mathcal M$ over $M$ with degree 1 local generators $\xi_\alpha^i$ and
cocycles $A^1_{\alpha\beta}$ and degree 2 generators $\eta_\alpha^i$ and cocycles
$A^2_{\alpha\beta}$ and $\nu_{\alpha\beta}$
to the sheaf of core and Lagrangian linear sections of $\mathcal G(\mathcal M)$. 
There is an obvious natural isomorphism between this functor and the
identity functor $\operatorname{[2]-Man}\to\operatorname{[2]-Man}$.

\subsubsection{Correspondence of splittings.}\label{cor_splittings}
Decomposed metric double vector bundles 
$B\times_MQ\times_MQ^*$ are equivalent to split $[2]$-manifolds.

Choose a metric double vector bundle $(\mathbb E;Q,B;M)$ and the
corresponding $[2]$-manifold $\mathcal M$.  Each choice of a
Lagrangian decomposition $\mathbb I$ of $\mathbb E$ is equivalent to a
choice of splitting $\mathcal S$ of the corresponding $[2]$-manifold,
such that the following diagram commutes
\begin{equation*}
\begin{xy}
  \xymatrix{
    \mathbb E\ar[r]^{\mathbb I\qquad }\ar[d]_{\mathcal M(\cdot)}
&Q\times_MB\times_MQ^*\ar[d]^{\mathcal M(\cdot)}\\
    \mathcal M(\mathbb E)\ar[r]_{\mathcal S\qquad }&Q[-1]\oplus B^*[-2].  }
\end{xy}
\end{equation*}

Note also that the category of split $[2]$-manifolds is equivalent to
the category of $[2]$-manifolds, and the category of decomposed metric
double vector bundles is equivalent to the category of metric double
vector bundles.  We will often use this in the following sections.

\section{Poisson $[2]$-manifolds and metric VB-algebroids.}\label{sec:Poisson}
In this section we study $[2]$-manifolds endowed with a Poisson
structure of degree $-2$. We show how split Poisson $[2]$-manifolds
are equivalent to a special family of 2-representations. Then we prove
that Poisson $[2]$-manifolds are equivalent to metric double vector
bundles endowed with a linear Lie algebroid structure that is
compatible with the metric.

\begin{definition}
  A Poisson $[2]$-manifold is a $[2]$-manifold endowed with a Poisson
  structure of degree $-2$. A morphism of Poisson $[2]$-manifolds 
is a morphism of $[2]$-manifolds that preserves the Poisson structure. 
\end{definition}

Note that a Poisson bracket of degree $-2$ on a $[2]$-manifold
$\mathcal M$ is an $\R$-bilinear map $\{\cdot\,,\cdot\}\colon
C^\infty(\mathcal M)\times C^\infty(\mathcal M)\to C^\infty(\mathcal
M)$ of the graded sheaves of functions, such that\footnote{We will
  also write $|\xi|$ for the degree of a homogeneous element $\xi\in
  C^\infty_{\mathcal M}(U)$.}  $\deg\{\xi,\eta\}=\deg\xi+\deg\eta-2$
for homogeneous elements $\xi,\eta\in C^\infty_{\mathcal M}(U)$. The
bracket is graded skew-symmetric;
$\{\xi,\eta\}=-(-1)^{|\xi|\,|\eta|}\{\eta,\xi\}$ and satisfies the
graded Leibniz and Jacobi identities
 \begin{equation}\label{graded_leibniz}
\{\xi_1,\xi_2\cdot\xi_3\}=\{\xi_1,\xi_2\}\cdot\xi_3+(-1)^{|\xi_1|\,|\xi_2|}\xi_2\cdot\{\xi_1,\xi_3\}
\end{equation}
and
\begin{equation}\label{graded_jacobi}
\{\xi_1,\{\xi_2,\xi_3\}\}=\{\{\xi_1,\xi_2\},\xi_3\}+(-1)^{|\xi_1|\,|\xi_2|}\{\xi_2,\{\xi_1,\xi_3\}\}
\end{equation}
for homogeneous $\xi_1,\xi_2,\xi_3\in C^\infty_{\mathcal M}(U)$.

A morphism $\mu\colon\mathcal M_1\to\mathcal M_2$ of Poisson
$[2]$-manifolds
satisfies \[\mu^\star\{\xi_1,\xi_2\}=\{\mu^\star\xi_1,\mu^\star\xi_2\}
\]
for all $\xi_1,\xi_2\in C^\infty_{\mathcal M_2}(U)$, $U$ open in $M_2$.

\subsection{Split Poisson $[2]$-manifolds and self-dual $2$-representations}\label{sec:Poisson2man}
We begin by defining self-dual $2$-representations.
\begin{definition}\label{saruth}
  Let $(A,\rho,[\cdot\,,\cdot])$ be a Lie algebroid. A
  2-representation $(\nabla^Q,\nabla^{Q^*}, R)$ of $A$ on a complex
  $\partial_Q\colon Q^*\to Q$ is said to be \textbf{self-dual} if it
  equals its dual, i.e.
  \begin{enumerate}
\item $\partial_Q=\partial_Q^*$,
\item $\nabla^Q$ and $\nabla^{Q^*}$ are dual to each other,
\item and 
  $R^*=-R\in\Omega^2(A,\operatorname{Hom}(Q,Q^*))$.
\end{enumerate}
\end{definition}

Then we prove the following result.
\begin{theorem}\label{poisson_is_saruth}
  There is a bijection between split Poisson 2-manifolds and
  self-dual 2-representations.
\end{theorem}

\begin{proof}
  First let us consider a split 2-manifold $\mathcal M=Q[-1]\oplus
  B^*[-2]$. That is, $Q$ and $B$ are vector bundles over $M$ and the
  functions of degree $0$ on $\mathcal M$ are the elements of
  $C^\infty(M)$, the functions of degree $1$ are sections of $Q^*$ and
  the functions of degree $2$ are sections of $B\oplus Q^*\wedge Q^*$.
  Let us now take any Poisson bracket $\{\cdot\,,\cdot\}$ of degree
  $-2$ on $C^\infty(\mathcal M)$.  In the following, we consider
  arbitrary $f,f_1,f_2\in C^\infty(M)$,
  $\tau,\tau_1,\tau_2\in\Gamma(Q^*)$, and
  $b,b_1,b_2\in\Gamma(B)$.

The brackets $\{f_1,f_2\}$,
  $\{f,\tau\}$ have degree $-2$ and $-1$, respectively, and
  must hence vanish. The bracket $\{\tau_1,\tau_2\}$ is a function
  on $M$ because it has degree $0$.
% Set $\{\tau_1,\tau_2\}=\langle \tau_1,\tau_2\rangle\in C^\infty(M)$.
  Since $\{f,\tau\}=0$ for all $f\in C^\infty(M)$ and
  $\tau\in\Gamma(Q^*)$, this defines a vector bundle morphism
  $\partial_Q\colon Q^*\to Q$ by \eqref{graded_leibniz}:
  $\langle\tau_2, \partial_Q(\tau_1)\rangle=\{\tau_1,\tau_2\}$.  Since
  $\{\tau_1,\tau_2\}=-(-1)^{|\tau_2|}\{\tau_2,\tau_1\}=\{\tau_2,\tau_1\}$,
  we find that $\partial_Q^*=\partial_Q$. The Poisson bracket
  $\{b,f\}$ has degree $0$ and is hence an element of
  $C^\infty(M)$. Again by \eqref{graded_leibniz}, this defines a derivation
  $\{b,\cdot\}\an{C^\infty(M)}$ of $C^\infty(M)$, hence a vector field
  $\rho_B(b)\in \mx(M)$; $\{b,f\}=\rho_B(b)(f)$. By the Leibnitz
  identity \eqref{graded_leibniz} for the Poisson bracket and the
  equality $\{f_1,f_2\}=0$ for all $f_1,f_2\in C^\infty(M)$, we get in
  this manner a vector bundle morphism (an anchor) $\rho_B\colon B\to
  TM$. The bracket $\{b,\tau\}$ has degree $1$ and is hence a section
  of $Q^*$.  Since
  $\{b,f\tau\}=f\{b,\tau\}+\{b,f\}\tau=f\{b,\tau\}+\rho_B(b)(f)\tau$
  and $\{f b,\tau\}=f\{b,\tau\}+\{f,\tau\}b=f\{b,\tau\}$, we find a
  linear $B$-connection $\nabla$ on $Q^*$ by setting
  $\nabla_b\tau=\{b,\tau\}$.  Let us finally look at the bracket
  $\{b_1,b_2\}$. This function has degree $2$ and is hence the sum of
  a section of $B$ and a section of $Q^*\wedge Q^*$.  We write
  $\{b_1,b_2\}=[b_1,b_2]-R(b_1,b_2)$ with $ [b_1,b_2]\in\Gamma(B)$ and
  $R(b_1,b_2)\in\Gamma(Q^*\wedge Q^*)$. By a similar reasoning as
  before, we find that this defines a skew-symmetric bracket
  $[\cdot\,,\cdot]$ on $\Gamma(B)$ that satisfies a Leibniz equality
  with respect to $\rho_B$, and an element
  $R\in\Omega^2(B,\operatorname{Hom}(Q,Q^*))$ such that $R^*=-R$.
  Note also here that the bracket $\{b,\phi\}$ for
  $\phi\in\Gamma(Q^*\wedge Q^*)\subseteq
  \Gamma(\operatorname{Hom}(Q,Q^*))$ is just
  $\nabla^{\operatorname{Hom}}_b\phi$, where
  $\nabla^{\operatorname{Hom}}$ is the $B$-connection induced on $
  \operatorname{Hom}(Q,Q^*)$ by $\nabla$ and $\nabla^*$.  \medskip

  Now we will show that the dull algebroid structure on $B$ is in
  reality a Lie algebroid structure, and that $(\nabla,\nabla^*,R)$ is
  a self-dual 2-representation of $B$ on $\partial_Q\colon Q\to
  Q^*$.  In order to do this, we only need to recall that the
  Poisson structure $\{\cdot\,,\cdot\}$ satisfies the Jacobi identity.
  The Jacobi identity for the three functions $b_1,b_2,f$ yields the
  compatibility of the anchor on $B$ with the bracket on
  $\Gamma(B)$. The Jacobi identity for $b,\tau_1,\tau_2$ yields
  $\partial_Q\circ \nabla=\nabla^*\circ\partial_Q$, and the Jacobi
  identity for $b_1,b_2,\tau$ yields $R_\nabla=R\circ\partial_Q$. The
  equality $R_{\nabla^*}=\partial_Q\circ R$ follows using
  $\partial_Q=\partial_Q^*$, $R^*=-R$ and $R_\nabla^*=-R_{\nabla^*}$.
  The Jacobi identity for $b_1,b_2,b_3\in\Gamma(B)$ yields in a
  straightforward manner the Jacobi identity for $[\cdot\,,\cdot]$ on
  sections of $\Gamma(B)$ and the equation
  $\dr_{\nabla^{\operatorname{Hom}}}=0$.  \medskip

  Take conversely a self dual 2-representation of a Lie algebroid $B$
  on a 2-term complex $\partial_Q\colon Q^*\to Q$ and consider the
  $[2]$-manifold $\mathcal M=Q[-1]\oplus B^*[-2]$. Then the self-dual
  2-representation defines as described above a Poisson bracket of
  degree $-2$ on $C^\infty(\mathcal M)$.
\end{proof}

Next take two split $[2]$-manifolds $Q_1\oplus B_1^*$ and $Q_2\oplus B_2^*$
over $M_1$ and $M_2$, respectively, endowed with two Poisson
structures of degree $-2$.  By the preceding theorem, we have hence
two self-dual 2-representations: $B_1$ acts on $\partial_{Q_1}\colon
Q_1^*\to Q_1$ via $\nabla^1\colon
\Gamma(B_1)\times\Gamma(Q_1)\to\Gamma(Q_1)$ and
$R_1\in\Omega^2(B_1,Q_1^*\wedge Q_1^*)$, and $B_2$ acts on
$\partial_{Q_2}\colon Q_2^*\to Q_2$ via $\nabla^2\colon
\Gamma(B_2)\times\Gamma(Q_2)\to\Gamma(Q_2)$ and
$R_2\in\Omega^2(B_2,Q_2^*\wedge Q_2^*)$.

A computation shows that a morphism $\mu\colon Q_1\oplus B_1^*\to
Q_2\oplus B_2^*$ preserves the Poisson structures if and only if its
decomposition $\mu_0\colon M_1\to M_2$, $\mu_Q\colon Q_1\to
Q_2$, $\mu_B\colon B_1^*\to B_2^*$ and
$\mu_{QB}\in\Omega^2(Q_1,\omega_0^*B_2)$ define a \emph{morphism of
  self-dual 2-representations}. That is:\label{morphism_of_saruth}
\begin{enumerate}
\item $[\mu_B^\star(b_1), \mu_B^\star(b_2)]_1=\mu_B^\star[b_1,b_2]_2$ for all $b_1,b_2\in\Gamma(B_2)$,
\item $\rho_{B_1}(\mu_B^\star(b))\sim_{\mu_0}\rho_{B_2}(b)$ for all $b\in\Gamma(B)$,
\item $\langle
  \mu_Q^\star(\tau_1),\partial_{Q_1}\mu_Q^\star(\tau_2)\rangle=\langle
  \tau_1,\partial_{Q_2}\tau_2\rangle$ for all
  $\tau_1,\tau_2\in\Gamma(Q_2^*)$,
\item
  $\mu_Q^\star(\nabla^2_b\tau)=\nabla^1_{\mu_B^\star(b)}\mu_Q^\star(\tau)-\mu_{QB}^\star(b)\circ\partial_{Q_1}\mu_Q^\star(\tau)$
  for all $b\in\Gamma(B_2)$ and $\tau\in\Gamma(Q_2^*)$, and
\item $R_1(\mu_B^\star(b_1),\mu_B^\star(b_2))-\mu_Q^\star R_2(b_1,b_2)
  =-\mu_{QB}^\star[b_1,b_2]+\nabla_{\mu_B^\star(b_1)}\mu_{QB}^\star(b_2)\linebreak-\nabla_{\mu_B^\star(b_2)}\mu_{QB}^\star(b_1)
  +\mu_{QB}^\star(b_2)\circ\partial_{Q_1}\mu_{QB}^\star(b_1)-\mu_{QB}^\star(b_1)\circ\partial_{Q_1}\mu_{QB}^\star(b_2)$.
\end{enumerate}
In the fourth and fifth equation, $\mu_{QB}^\star(b)$ is seen as a
section of $\operatorname{Hom}(Q,Q^*)$.  Note that if $M_1=M_2=M$ and
$\omega_0=\Id_M$, then (1)--(3) translate to
\begin{enumerate}
\item $\mu_B^*\colon B_2\to B_1$ is a Lie algebroid morphism,
\item $\mu_Q\circ\partial_{Q_1}\circ\mu_Q^*=\partial_{Q_2}$.
\end{enumerate}

\subsubsection{Symplectic $[2]$-manifolds}\label{symplectic}
Note that an ordinary Poisson manifold $(M,\{\cdot\,,\cdot\})$ is
symplectic if and only if the vector bundle morphism $\sharp\colon
T^*M\to TM$ defined by $\dr f\mapsto X_f$ is surjective, where
$X_f\in\mx(M)$ is the derivation $\{ f,\cdot\}$.  Alternatively, we
can say that the Poisson manifold is symplectic if the image of the
map $\sharp\colon C^\infty(M)\to \mx(M)$, $f\mapsto \{f,\cdot\}$
generates $\mx(M)$ as a $C^\infty(M)$-module.

In the same manner, if $(\mathcal M, \{\cdot\,,\cdot\})$ is a Poisson
$[n]$-manifold, the map $\sharp\colon C^\infty(\mathcal M)\to
\der(C^\infty(\mathcal M))$ sends $\xi$ to $\{\xi, \cdot\}$. Then
$(\mathcal M, \{\cdot\,,\cdot\})$ is a \textbf{symplectic
  $[n]$-manifold} if the image of this map generates $\der(C^\infty(\mathcal M))$
as a $C^\infty(\mathcal M)$-module.

\medskip
Let $(q_E\colon E\to M, \langle\cdot\,,\cdot\rangle)$ be an Euclidean
vector bundle, i.e.~a vector bundle endowed with a nondegenerate
fiberwise pairing
$\langle\cdot\,,\cdot\rangle\colon E\times_M E\to \R$. 
Choose a metric connection $\nabla\colon
\mx(M)\times\Gamma(E)\to\Gamma(E)$.
Then, identifing $E$
with $E^*$ via the metric, we find that the $2$-representation 
$(\Id_E\colon E\to E, \nabla, \nabla, R_\nabla)$ is self-dual  (an easy calculation shows that if
$\nabla$ is metric, then $\langle
R_\nabla(X_1,X_2)e_1,e_2\rangle=-\langle
R_\nabla(X_1,X_2)e_2,e_1\rangle$ for all $e_1,e_2\in\Gamma(E)$ and
$X_1,X_2\in\mx(M)$).
Consider the split Poisson $[2]$-manifold $E[-1]\oplus T^*M[-2]$, with
the Poisson bracket given by the metric
connection
$\nabla\colon \mx(M)\times\Gamma(E)\to\Gamma(E)$. That is, 
the Poisson bracket is given by
\[ \{f_1,f_2\}=0,\quad \{f,\Beta(e)\}=0,\quad
\{\Beta(e_1),\Beta(e_2)\}=\Beta(e_2)(e_1)=\langle e_1,
e_2\rangle,\]
\[\{X,\xi\}=\Beta(\nabla_Xe), \quad \{X,f\}=X(f)\]
and 
\[\{X_1,X_2\}=[X_1,X_2]-R_\nabla(X_1,X_2).
\]

Recall from \eqref{der_deg_0} the special derivations that we found on
split $[n]$-manifolds.  The function
$\sharp\colon C^\infty(E[-1]\oplus T^*M[-2])\to
\operatorname{Der}(C^\infty(E[-1]\oplus T^*M[-2]))$
sends a function $f$ of degree $0$ to $\hat{\dr f}$, a
derivation of degree $-2$. $\sharp$ sends $\Beta(e)$ to
$\hat{e}+\Beta(\nabla_\cdot e)$, which is a derivation of degree
$-1$. Note that $\Beta(\nabla_\cdot e)$ can be written as a sum
$\sum_i \Beta(e_i)\hat{\dr f_i}$ with some sections $e_i\in\Gamma(E)$ and
functions $f_i\in C^\infty(M)$.  Finally $\sharp$ sends $X$ to
$X+\Beta\circ\nabla_X\circ\Beta\inv+[X,\cdot]-R(X,\cdot)$, which is a
derivation of degree $0$.  Note that $R(X,\cdot)$ can be written as
$\sum\Beta(e_i)\Beta(e_j)\hat{\dr f_{ij}}$ for some sections
$e_i,e_j\in\Gamma(E)$ and some functions $f_{ij}$ in $C^\infty(M)$.
Hence, since the derivations $\hat{\dr f}$, $\hat e$ and
$X+\Beta\circ\nabla_X\circ\Beta\inv+[X,\cdot]$ for $f\in C^\infty(M)$,
$e\in\Gamma(E)$ and $X\in\mx(M)$, span
$\operatorname{Der}(C^\infty(E[-1]\oplus T^*M[-2]))$ as a
$C^\infty(E[-1]\oplus T^*M[-2])$-module, we find as a consequence that
$E[-1]\oplus T^*M[-2]$ is a symplectic $[2]$-manifold.

\medskip 

More generally, take a split Poisson $[2]$-manifold $Q[-1]\oplus
B^*[-2]$, hence a self-dual $2$-representation $(\partial_Q\colon
Q^*\to Q, \nabla, \nabla^*,R)$ of a Lie algebroid $B$. Then $\sharp
f=\rho_B^*\dr f$ for all $f\in C^\infty(M)$,
$\sharp\tau=\hat{\partial_Q\tau}-\nabla^*_\cdot\tau$ and $\sharp
b=\rho_B(b)+\nabla_b^*+[b,\cdot]-R(b,\cdot)$. A discussion as the one
above shows that the Poisson structure is symplectic if and only if
$\rho_B\colon B\to TM$ is injective and surjective, hence an
isomorphism and $\partial_Q\colon Q^*\to Q$ is surjective, hence an
isomorphism. The isomorphism $\partial_Q$ identifies then $Q$ with its
dual and $Q$ becomes so an Euclidean bundle with the pairing $\langle
q_1, q_2\rangle_Q=\langle \partial_Q\inv(q_1),
q_2\rangle=\{\partial_Q\inv q_1,\partial_Q\inv q_2\}$. Via the
identification $\partial_Q\colon Q^*\overset{\sim}{\to} Q$, the linear
connection $\nabla$ is then automatically a metric connection and the
self-dual $2$-representation is $(\Id_Q\colon Q\to
Q,\nabla,\nabla,R_\nabla)$.

\medskip

We have hence found that split symplectic $[2]$-manifolds are
equivalent to self-dual $2$-representation 
$(\Id_E\colon E\to E, \nabla, \nabla, R_\nabla)$ defined by an
Euclidean vector bundle $E$ together with a metric
connection $\nabla$, see also \cite{Roytenberg02}.

\subsection{Metric VB-algebroids}\label{sec:metVBa}
Next we introduce the notion of metric VB-algebroids.
\begin{definition}\label{metric_VBLA}
  Let $(D;Q,B;M)$ be a metric double vector bundle (with core $Q^*$)
  and assume that $(D\to Q, B\to M)$ is a VB-algebroid. Then
  $(D\to Q, B\to M)$ is a \textbf{metric VB-algebroid} if the isomorphism
  $\Beta\colon D\to D\duer B$ is an isomorphism of VB-algebroids.

  A morphism $\Omega\colon \mathbb E_2\to\mathbb E_1$ of metric
  VB-algebroids is a morphism of the underlying metric double vector
  bundles, such that $\Omega\subseteq \overline{\mathbb
    E_2}\times\mathbb E_1$ is a subalgebroid.
\end{definition}

Recall from Theorem \ref{rajan} that linear splittings of
VB-algebroids define $2$-re\-pre\-sen\-ta\-tions.  We will prove that
Lagrangian splittings of metric VB-algebroids correspond to self-dual
$2$-re\-pre\-sen\-ta\-tions.

\begin{proposition}\label{metric_VBLA_via_sa_ruth}
  Let $(\mathbb E\to Q, B\to M)$ be a VB-algebroid with core $Q^*$ and
  assume that $\mathbb E$ is endowed with a linear metric.  Choose a
  Lagrangian decomposition of $\mathbb E$ and consider the
  corresponding 2-representation of $B$ on $\partial_Q\colon Q^*\to
  Q$. This 2-representation is self-dual if and only if $(\mathbb
  E\to Q, B\to M)$ is a metric VB-algebroid.
\end{proposition}

\begin{proof}
  It is easy to see that $\Beta\colon \mathbb E\to \mathbb E\duer B$
  sends core sections $\tau^\dagger\in\Gamma^c_Q(\mathbb E)$ to core
  sections $\tau^\dagger\in\Gamma^c_Q(\mathbb E\duer B)$.  (As always,
  we identify $Q^{**}$ with $Q$ via the canonical isomorphism.)  Let
  $\Sigma\colon B\times_MQ\to \mathbb E$ be a Lagrangian splitting of
  $\mathbb E$.  We have seen in Section \ref{dual} that the map
  $\sigma_B\colon \Gamma(B)\to\Gamma_Q^l(\mathbb E)$ induces a
  horizontal lift $\sigma_B^\star\colon \Gamma(B)\to\Gamma_Q^l(\mathbb
  E\duer B)$.  Recall from Lemma \ref{lem_Lagr_split_sections} that
  $\Beta$ sends also the linear sections $\sigma_B(b)$ to
  $\sigma_B^\star(b)$, for all $b\in\Gamma(B)$.

\medskip

The double vector bundle $\mathbb E\duer B$ has a VB-algebroid
structure $(\mathbb E\duer B\to Q^{**}, B\to M)$ (see
\S\ref{subsect:VBa}).  Given the splitting $\Sigma^\star\colon
B\times_M Q^{**}\to \mathbb E\duer B$ defined by a Lagrangian
splitting $\Sigma\colon B\times_M Q\to \mathbb E$, the VB-algebroid
structure is given by the dual of the $2$-representation
$(\partial_Q\colon Q^*\to Q, \nabla^Q, \nabla^{Q^*},
R\in\Omega^2(B,\Hom(Q,Q^*)))$, i.e.
\begin{equation*}
\begin{split}
\rho_{\mathbb E\duer B}(\tau^\dagger)&=(\partial_Q^*\tau)^\uparrow \in \mx^c(Q^{**}),\\
\rho_{\mathbb E\duer B}(\sigma_B^\star(b))&=\widehat{{\nabla^{Q^*}}^*}\in \mx^l(Q^{**}),\\
\left[\sigma_B^\star(b), \tau^\dagger\right]&=({\nabla^Q}^*_b\tau)^\dagger, \text{ and }\\
\left[\sigma_B^\star(b_1), \sigma_B^\star(b_2)\right]&=\sigma_B^\star[b_1,b_2]+\widetilde{R(b_1,b_2)^*}
\end{split}
\end{equation*}
(see \S\ref{dual_and_ruths}).
This shows immediately that $\Beta$ is an isomorphism of VB-algebroids
over the canonical isomorphism $Q\to Q^{**}$ if and only if the $2$-representation
$(\partial_Q\colon Q^*\to Q, \nabla^Q, \nabla^{Q^*},
R\in\Omega^2(B,\Hom(Q,Q^*)))$ is self-dual.
\end{proof}

Recall that a morphism $\Omega\subseteq \overline{\mathbb
  E_2}\times\mathbb E_1$ of metric double vector bundles has four
components: a smooth map $\omega_0\colon M_1\to M_2$ of the double
bases, two vector bundle morphisms $\omega_Q\colon Q_1\to Q_2$ and
$\omega_B\colon B_1^*\to B_2^*$ over $\omega_0$ and a vector valued
$2$-form $\omega\in\Omega^2(Q_1,\omega_0^*B_2)$. Choose two Lagrangian
splittings $\Sigma^1\colon Q_1\times_{M_1}B_1\to \mathbb E_1$ and
$\Sigma^2\colon Q_2\times_{M_2}B_2\to\mathbb E_2$ and the
corresponding self-dual 2-representations. Using Section
\ref{morphisms_of_met_DVB}, we note that $\Omega$ is spanned over
$(\operatorname{graph}(\omega_Q)\subseteq Q_1\times Q_2)$ by sections
$\tau^\dagger\colon \operatorname{graph}(\omega_Q)\to \Omega$,
$\tau^\dagger(q_m,\omega_Q(q_m))=(\omega_Q^\star(\tau)^\dagger(q_m),
\tau^\dagger(\omega_Q(q_m)))$ for all $\tau\in\Gamma(Q_2^*)$ and
$\sigma(b)\colon \operatorname{graph}(\omega_Q)\to \Omega$,
$\sigma(b)(q_m,\omega_Q(q_m))=(\sigma_{B_1}^1(\omega_B^\star(b))(q_m)+\widetilde{\omega^\star(b)}(q_m),
\sigma_{B_2}^2(b)(\omega_Q(q_m)))$ for all $b\in\Gamma(B_2)$.

A straightforward computation shows that $\Omega$ is a subalgebroid over
$\operatorname{graph}(\omega_Q)$ of $\overline{\mathbb E_2}\times
\mathbb E_1$ if and only if $\omega_0,\omega_Q,\omega_B,\omega$ and
the two self-dual 2-representations satisfy Conditions (1) to (5)
on Page \pageref{morphism_of_saruth}:
\begin{enumerate}
\item $[\omega_B^\star(b_1), \omega_B^\star(b_2)]_1=\omega_B^\star[b_1,b_2]_2$ for all $b_1,b_2\in\Gamma(B_2)$,
\item $\rho_{B_1}(\omega_B^\star(b))\sim_{\omega_0}\rho_{B_2}(b)$ for all $b\in\Gamma(B)$,
\item $\langle
  \omega_Q^\star(\tau_1),\partial_{Q_1}\omega_Q^\star(\tau_2)\rangle=\langle
  \tau_1,\partial_{Q_2}\tau_2\rangle$ for all
  $\tau_1,\tau_2\in\Gamma(Q_2^*)$,
\item
  $\omega_Q^\star(\nabla^2_b\tau)=\nabla^1_{\omega_B^\star(b)}\omega_Q^\star(\tau)-\omega^\star(b)\circ\partial_{Q_1}\omega_Q^\star(\tau)$
  for all $b\in\Gamma(B_2)$ and $\tau\in\Gamma(Q_2^*)$, and
\item $R_1(\omega_B^\star(b_1),\omega_B^\star(b_2))-\omega_Q^\star R_2(b_1,b_2)
  =-\omega^\star[b_1,b_2]+\nabla_{\omega_B^\star(b_1)}\omega^\star(b_2)\linebreak -\nabla_{\omega_B^\star(b_2)}\omega^\star(b_1)
  +\omega^\star(b_2)\circ\partial_{Q_1}\omega^\star(b_1)-\omega^\star(b_1)\circ\partial_{Q_1}\omega^\star(b_2)$.
\end{enumerate}

\subsection{Equivalence of Poisson $[2]$-manifolds with metric VB-algebroids.}
The functors found in Section \ref{eq_2_manifolds} between the
category of metric double vector bundles and the category of
$[2]$-manifolds restrict to functors between the category of metric
VB-algebroids and the category of Poisson $[2]$-manifolds.
\begin{theorem}
  The category of Poisson $[2]$-manifolds is equivalent to the
  category of metric VB-algebroids.
%  \[\begin{xy}
%    \xymatrix{D\ar[r]^{\pi_Q}\ar[d]_{\pi_B}&B\ar[d]^{q_B}\\
%      Q\ar[r]_{q_Q}&M }
% \end{xy}
% \]
% with core $Q^*$ and side algebroid $B$.
\end{theorem}

\begin{proof}
  Let $(\mathcal M, \{\cdot\,,\cdot\})$ be a Poisson $[2]$-manifold
  and consider the corresponding double vector bundle $\mathbb
  E_{\mathcal M}$.  Choose a splitting $\mathcal M\simeq Q[-1]\oplus
  B^*[-2]$ of $\mathcal M$ and consider the corresponding Lagrangian
  splitting $\Sigma$ of $\mathbb E_{\mathcal M}$.  

  As we have seen in Theorem \ref{poisson_is_saruth}, the split
  Poisson 2-manifold $Q[-1]\oplus B^*[-2]$ is equivalent to a
  self-dual 2-representation of a Lie algebroid structure on $B$ on
  a morphism $\partial_Q\colon Q^*\to Q$. This 2-representation
  defines a VB-algebroid structure on the decomposition of $\mathbb
  E_{\mathcal M}$ and so by isomorphism on $\mathbb E_{\mathcal
    M}$. Proposition \ref{metric_VBLA_via_sa_ruth} implies then that
  the Lie algebroid structure is compatible with the metric. The
  metric Lie algebroid structure on $\mathbb E_{\mathcal M}$ does not
  depend on the choice of splitting of $\mathcal M$.  Hence, the
  functor $\mathcal G$ restricts to a functor $\mathcal G_{Poi}$ from
  the category of Poisson $[2]$-manifolds to metric VB-algebroids.

The discussions at the end of Sections \ref{sec:Poisson2man} and \ref{sec:metVBa}
show that morphisms of split Poisson $[2]$-manifolds are 
sent by $\mathcal G$ to morphisms of decomposed metric VB-algebroids.

The functor $\mathcal F$ restricts in a similar manner to a functor
$\mathcal F_{LA}$ from the category of metric VB-algebroids to the
category of Poisson $[2]$-manifolds. The natural transformations found
in the proof of Theorem \ref{main_crucial} restrict to natural
transformations $\mathcal F_{LA}\mathcal G_{Poi}\simeq \Id$ and
$\mathcal G_{Poi}\mathcal F_{LA}\simeq \Id$.
\end{proof}

\subsection{Examples}

We conclude by discussing three important classes of examples.
\subsubsection{Tangent doubles of Euclidean bundles vs symplectic $[2]$-manifolds }\label{tangent_euclidean}
Consider an Euclidean vector bundle $E\to M$ and a metric connection 
$\nabla\colon\mx(M)\times\Gamma(E)\to\Gamma(E)$. The double tangent
\begin{equation*}
\begin{xy}
\xymatrix{
TE \ar[d]_{Tq_E}\ar[r]^{p_E}& E\ar[d]^{q_E}\\
 TM\ar[r]_{p_M}& M}
\end{xy}
\end{equation*}
has a VB-algebroid structure $(TE\to E;TM\to M)$ and a linear metric\linebreak
$\langle\cdot\,,\cdot\rangle\colon TE\times_{TM}TE\to \R$ defined as
in Example \ref{metric_connections}.  % Choose a metric
% connection $\nabla\colon\mx(M)\times\Gamma(E)\to\Gamma(E)$ and hence a
% Lagrangian splitting $\Sigma^\nabla\colon E\times_MTM\to TE$.

Recall that Lagrangian linear splittings of $TE$ are equivalent to metric
connections $\nabla\colon \mx(M)\times\Gamma(E)\to\Gamma(E)$, i.e.~
connections that preserve the pairing: $\dr\langle
e_1,e_2\rangle=\langle\nabla_\cdot e_1,e_2\rangle+\langle
e_1,\nabla_\cdot e_2\rangle$ for all $e_1,e_2\in\Gamma(E)$.  In other
words, $\nabla=\nabla^*$ when $E^*$ is identified with $E$ via the
non-degenerate pairing. The $2$-representation $(\Id_E\colon E\to E,
\nabla, \nabla, R_\nabla)$ defined by the Lagrangian splitting
$\Sigma^\nabla\colon E\times_M TM\to TE$ and the VB-algebroid $(TE\to
E, TM\to M)$ is then self-dual (see also \S\ref{symplectic}).

The Poisson $[2]$-manifold $\mathcal M(TE)$ associated to $TE$ is
given as follows. The functions of degree $0$ are elements of
$C^\infty(M)$, the functions of degree $1$ are sections of $E$ ($E$ is
identified with $E^*$ via the isomorphism $\Beta\colon E\to E^*$
defined by the pairing) and the functions of degree $2$ are the vector
fields $\tilde X$ on $E$ that preserve the pairing,
i.e.~$\tilde X=\widehat{D}\in\mx(E)$ for a derivation $D_{\tilde X}$
over $X\in\mx(M)$, that preserves the pairing. The Poisson bracket is
given by $\{\tilde X,\tilde Y\}=\widetilde{[X,Y]}$,
$\{\tilde X, e\}=D_{\tilde X}(e)$ and $\{\tilde X, f\}=X(f)$,
$\{e_1,e_2\}=\langle e_1, e_2\rangle$, and $\{e, f\}=\{f_1,f_2\}=0$
for all $e,e_1,e_2\in\Gamma(E)$, $f,f_1,f_2\in C^\infty(M)$ and
$\tilde X, \tilde Y\in \mx^{\langle\cdot\,,\cdot\rangle, l}(E)$.  The
Poisson $[2]$-manifold $\mathcal M(TE)$ splits as the split Poisson
$[2]$-manifold described in \S\ref{symplectic}. It is hence
symplectic.  Thus, we have found that the equivalence found in this
section restricts to an equivalence of symplectic $[2]$-manifolds with
tangent doubles of Euclidean vector bundles.

\subsubsection{The metric double of a VB-algebroid}\label{metric_double_VB_alg}
Take a VB-algebroid $(D\to A, B \to M)$ with core $C$ and a linear
splitting $\Sigma\colon A\times_M B\to D$.  Let $(\partial_A\colon
C\to A, \nabla^A,\nabla^C, R\in \Omega^2(B,\operatorname{Hom}(A,C)))$
be the $2$-representation of the Lie algebroid $B$ that is induced by
$\Sigma$ and the VB-algebroid $(D\to A, B\to M)$.  Recall from Section
\S\ref{dual_and_ruths} that $(D\duer B\to C^*, B\to M)$ has an induced
VB-algebroid structure and from Lemma \ref{lemma_dual_splitting} that
the splitting $\Sigma$ induces a splitting $\Sigma^\star\colon
B\times_M C^*\to D\duer B$. The $2$-representation that is defined by
this splitting and this VB-algebroid is the $2$-representation
$(\partial_A^*\colon A^*\to C^*, {\nabla^A}^*,{\nabla^C}^*, -R^*\in
\Omega^2(B,\operatorname{Hom}(C^*,A^*)))$.

The direct sum $D\oplus_B(D\duer B)$ over $B$ \begin{equation*}
\begin{xy}
\xymatrix{
D\oplus_B(D\duer B) \ar[d]\ar[r]& B\ar[d]\\
A\oplus C^*\ar[r]& M}
\end{xy}
\end{equation*}
has then a VB-algebroid structure $(D\oplus_B(D\duer B)\to A\oplus
C^*, B\to M)$ with core $C\oplus A^*$. It is easy to see that $\Sigma$
and $\Sigma^\star$ define a linear splitting $\tilde\Sigma\colon
B\times_M (A\oplus C^*)\to D\oplus_B(D\duer B)$, $\tilde \Sigma(b_m,
(a_m,\gamma_m))=(\Sigma(a_m,b_m),\Sigma^\star(b_m,\gamma_m))$.
The induced $2$-representation is 
\[(\partial_A\oplus\partial_A^*\colon C\oplus A^*\to A\oplus C^*, \nabla^A\oplus{\nabla^C}^*,\nabla^C\oplus{\nabla^A}^*,
R\oplus(-R^*)),
\]
a self-dual $2$-representation of the Lie algebroid $B$. This gives us
a new class of examples of (split) Poisson $2$-manifolds induced from
ordinary $2$-representations or VB-algebroids.  Note that the
splittings of $D\oplus_B(D\duer B)$ obtained as above are not the only
Lagrangian splittings, and that the Example of $(TA\oplus T^*A\to
TM\oplus A^*, A\to M)$ discussed in the next example and in \cite{Jotz13a} is a
special case.

\subsubsection{The Pontryagin algebroid over a Lie algebroid}\label{pont_VB_LA}
If $A$ is a Lie algebroid, then since $TA\duer A=T^*A$, the double
vector bundle $T^*A$ has a VB-algebroid structure $(T^*A\to A^*, A\to
M)$ with core $T^*M$.  As a consequence, the fibered product
$TA\oplus_AT^*A$ has a VB-algebroid structure $(TA\oplus_AT^*A\to
TM\oplus A^*, A\to M)$. Recall from Example \ref{met_TET*E} that
$(TA\oplus T^*A;TM\oplus A^*,A;M)$ has also a natural linear metric,
which is given by \eqref{standard_pairing}.

Recall from Example \ref{met_TET*E} that linear splittings of
$TA\oplus_AT^*A$ are in bijection with dull brackets on sections of
$TM\oplus A^*$, and so also with Dorfman connections $\Delta\colon
\Gamma(TM\oplus A^*)\times\Gamma(A\oplus T^*M)\to\Gamma(A\oplus
T^*M)$. We give in \cite{Jotz13a} the $2$-representation
$((\rho,\rho^*)\colon A\oplus T^*M\to TM\oplus A^*, \nabla^{\rm
  bas},\nabla^{\rm bas}, R_\Delta^{\rm bas})$ of $A$ that is defined
by the VB-algebroid $(TA\oplus_AT^*A\to TM\oplus A^*, A\to M)$ side
and any such Dorfman connection: The connections $\nabla^{\rm
  bas}\colon \Gamma(A)\times\Gamma(A\oplus T^*M) \to\Gamma(A\oplus
T^*M)$ and $\nabla^{\rm bas}\colon \Gamma(A)\times\Gamma(TM\oplus A^*)
\to\Gamma(TM\oplus A^*)$ are
\begin{equation*}
\nabla^{\rm bas}_a(X,\alpha)=(\rho,\rho^*)(\Omega_{(X,\alpha)}a)+\ldr{a}(X,\alpha)
\end{equation*}
and
 \begin{equation*}\nabla^{\rm
  bas}_a(b,\theta)=\Omega_{(\rho,\rho^*)(b,\theta)}a+\ldr{a}(b,\theta), 
\end{equation*}
where $\Omega\colon \Gamma(TM\oplus A^*)\times\Gamma(A)\to\Gamma(A\oplus T^*M)$ is defined by
\[\Omega_{(X,\alpha)}a=\Delta_{(X,\alpha)}(a,0)-(0,\dr\langle\alpha, a\rangle)
\] and for $a\in\Gamma(A)$, the derivations $\ldr{a}$ over $\rho(a)$
are defined by:
\[\ldr{a}\colon \Gamma(A\oplus T^*M)\to\Gamma(A\oplus T^*M),\quad \ldr{a}(b,\theta)=([a,b], \ldr{\rho(a)}\theta)\]
and 
\[\ldr{a}\colon \Gamma(TM\oplus A^*)\to\Gamma(TM\oplus A^*), \quad \ldr{a}(X,\alpha)=([\rho(a),X], \ldr{a}\alpha).\]
We prove in \cite{Jotz13a} that the two connections above are dual to
each other if and only if the dull bracket dual to $\Delta$ is
skew-symmetric. Hence, the two connections are dual to each other if
and only if the chosen linear splitting is Lagrangian (see Example
\ref{met_TET*E}).  The basic curvature
\[R_\Delta^{\rm bas}\colon \Gamma(A)\times\Gamma(A)\times\Gamma(TM\oplus A^*)\to\Gamma(A\oplus T^*M)
\]
is given by 
\begin{align*}
R_\Delta^{\rm bas}(a,b)(X,\xi)=&-\Omega_{(X,\xi)}[a,b] +\ldr{a}\left(\Omega_{(X,\xi)}b\right)-\ldr{b}\left(\Omega_{(X,\xi)}a\right)\\
&\qquad                                                     + \Omega_{\nabla^{\rm bas}_b(X,\xi)}a-\Omega_{\nabla^{\rm bas}_a(X,\xi)}b.
\end{align*}
Assume that the linear splitting is Lagrangian.  A relatively long but
straightforward computation shows that ${R_\Delta^{\rm
    bas}}^*=-R_\Delta^{\rm bas}$, and so that the $2$-representation
is self-dual. Hence $(TA\oplus_AT^*A\to TM\oplus A^*, A\to M)$ is a
metric VB-algebroid.

\section{Split Lie 2-algebroids and Dorfman 2-representations}\label{sec:split_lie_2}
In this section we recall the notions of Courant algebroids, Dirac
structures, dull algebroids and Dorfman connections. Then we discuss
(split) Lie $2$-algebroids and the dual Dorfman
$2$-representations. We give several classes of examples of split Lie
$2$-algebroids.
\subsection{Preliminaries}\label{background_courant_notions}
We introduce in this section a slights generalisations of the notion
of Courant algebroid, namely \emph{degenerate Courant algebroids} and
\emph{Courant algebroids with pairing in a vector bundle.}  Degenerate
Courant algebroids will appear naturally in our study of LA-Courant
algebroids, and we will see that the fat bundle associated to a
VB-Courant algebroid will carry a natural Courant algebroid structure
with pairing in a vector bundle.

In the following, an anchored vector bundle is a vector bundle $Q\to
M$ endowed with a vector bundle morphism $\rho_Q\colon Q\to TM$ over
the identity.  An anchored vector bundle $(Q\to M, \rho_Q)$ and a
vector bundle $B\to M$ are said to be paired if there exists a
fibrewise pairing $\langle\cdot\,,\cdot\rangle\colon Q\times_M B\to
\R$ and a map $\dr_B\colon C^\infty(M)\to \Gamma(B)$ such that
\begin{equation}\label{compatibility_anchor_pairing}
  \langle q, \dr_Bf\rangle=\rho_Q(q)(f)
\end{equation}
for all $q\in\Gamma(Q)$ and $f\in C^\infty(M)$.  The triple
$(B,\dr_B, \langle\cdot\,,\cdot\rangle)$ will be called a
\textbf{pre-dual} %\index{pre-dual}
of $Q$ and $Q$ and $B$ are said to be \textbf{paired by
  $\langle\cdot\,,\cdot\rangle$}.

Consider an anchored vector bundle $(\mathsf E\to M, \rho)$ and a
vector bundle $V$ over the same base $M$ together with a map
$\tilde\rho\colon\Gamma(\mathsf E)\to \operatorname{Der}(V)$, such
that the symbol of $\tilde\rho(e)$ is $\rho(e)\in \mx(M)$ for all
$e\in \Gamma(\mathsf E)$.  Assume that $\mathsf E$ is paired with
itself via a pairing $\langle\cdot\,,\cdot\rangle\colon \mathsf
E\times_M \mathsf E\to V$ with values in $V$ and that there exists a
map $\mathcal D\colon \Gamma(V)\to \Gamma(\mathsf E)$ such that
$\langle \mathcal Dv, e\rangle=\tilde\rho(e)(v)$ for all
$v\in\Gamma(V)$.

Then $\mathsf E\to M$ is a \textbf{degenerate Courant algebroid with
  pairing in $V$} over the manifold $M$ if $\mathsf E$ is in addition
equipped with an $\R$-bilinear bracket $\lb\cdot\,,\cdot\rb$ on the
smooth sections $\Gamma(\mathsf E)$ such that the following conditions
are satisfied:
\begin{enumerate}
\item[(CA1)] $\lb e_1, \lb e_2, e_3\rb\rb= \lb \lb e_1, e_2\rb, e_3\rb+ \lb
  e_2, \lb e_1, e_3\rb\rb$,
\item[(CA2)] $\tilde\rho(e_1 )\langle e_2, e_3\rangle= \langle\lb e_1,
  e_2\rb, e_3\rangle + \langle e_2, \lb e_1 , e_3\rb\rangle$,
%\item $[e_1, e_1] = \frac{1}{2}\mathcal D\langle e_1 , e_1\rangle$,
\item[(CA3)] $\lb e_1, e_2\rb+\lb e_2, e_1\rb =\mathcal D\langle e_1 ,
  e_2\rangle$,
\item[(CA4)]     $\rho(\lb e_1, e_2\rb ) = [\rho(e_1), \rho(e_2)]$,
\item[(CA5)]      $\lb e_1, f e_2\rb= f \lb e_1 , e_2\rb+ (\rho(e_1 )f )e_2$
\end{enumerate}
for all $e_1, e_2, e_3\in\Gamma(\mathsf E)$ and $f\in
C^\infty(M)$.  If the pairing $\langle\cdot\,,\cdot\rangle$ is
nondegenerate, then $(\mathsf E\to M, \tilde\rho,
\langle\cdot\,,\cdot\rangle, \lb\cdot\,,\cdot\rb)$ is a
\textbf{Courant algebroid with pairing in V}. 
If $V=\mathbb R\times M\to M$ is the trivial bundle and the pairing is
nondegenerate, then $\mathcal D = \beta\inv\circ\rho^*\circ\dr \colon
C^\infty(M)\to\Gamma(\mathsf E)$, where $\beta$ is the isomorphism
$\mathsf E\to \mathsf E^*$ given by
$\Beta(e)=\langle e,\cdot\rangle$
for all $e\in\mathsf E$.  The quadruple $(\mathsf E\to M, \rho,
\langle\cdot\,,\cdot\rangle, \lb\cdot\,,\cdot\rb)$ is then a
\textbf{Courant algebroid} \cite{LiWeXu97,Roytenberg99} and Conditions (CA4) and
(CA5) follow then from (CA1), (CA2) and (CA3) (see \cite{Uchino02} and also \cite{Jotz13a}
for a quicker proof).

\medskip

In our study of VB-Courant algebroids, we will need the following two lemmas. 
\begin{lemma}[\cite{Roytenberg02}]\label{roytenberg_useful}
  Let $(\mathsf E\to M, \rho, \langle\cdot\,,\cdot\rangle,
  \lb\cdot\,,\cdot\rb)$ be a Courant algebroid.  For all $\theta\in
  \Omega^1(M)$ and $e\in\Gamma(\mathsf E)$, we have:
  \[ \lb e,
  \beta\inv\rho^*\theta\rb=\beta\inv\rho^*(\ldr{\rho(e)}\theta),
  \qquad \lb\beta\inv\rho^*\theta,
  e\rb=-\beta\inv\rho^*(\ip{\rho(e)}\dr\theta)
\]
and 
\begin{equation}\label{anchor_with_D}
  \rho(\beta\inv\rho^*\theta)=0.
\end{equation}
\end{lemma}
In particular, it follows from \eqref{anchor_with_D} that 
\begin{equation}\label{anchor_with_D_2}
\rho\circ \mathcal D=0.
\end{equation}

\begin{lemma}[\cite{Li-Bland12}]\label{useful_lemma}
  Let $\mathsf {E}\to M$ be a vector bundle, $\rho\colon \mathsf{E}\to
  TM$ be a bundle map, $\langle\cdot,\cdot\rangle$ a bundle metric on
  $\mathsf{E}$, and let $\mathcal S\subseteq\Gamma(\mathsf{E})$ be a
  subspace of sections which generates $\Gamma(\mathsf{E})$ as a
  $C^\infty(M)$-module. Suppose that $\lb\cdot\,,\cdot\rb:\mathcal
  S\times\mathcal S\to \mathcal S$ is a bracket which satisfies
\begin{enumerate}
\item $\lb s_1,\lb s_2,s_3\rb\rb=\lb\lb s_1,s_2\rb,s_3\rb
+\lb s_2,\lb s_1,s_3\rb\rb$, 
\item $\rho(s_1)\langle s_2,s_3\rangle=\langle
  \lb s_1,s_2\rb, s_3\rangle+\langle s_2, \lb
  s_1,s_3\rb\rangle$,
\item $\lb s_1,s_2\rb+\lb s_2,s_1\rb=\rho^*\dr
  \langle s_1,s_2\rangle$,
\item $\rho\lb s_1,s_2\rb=[\rho(s_1),\rho(s_2)]$,
\end{enumerate}
for any $s_i\in \mathcal S$, and that $\rho\circ\rho^*=0$. Then
there is a unique extension of $\lb\cdot\,,\cdot\rb$ to a 
bracket on all of $\Gamma(\mathsf{E})$ such that $(\mathsf E, \rho,
\langle\cdot\,,\cdot\rangle, \lb\cdot\,,\cdot\rb)$ is a Courant algebroid.
\end{lemma}

\medskip

A \textbf{Dirac structure} with support in a Courant algebroid $\mathsf E\to M$ 
is a subbundle $D\to S$ over a sub-manifold $S$ of $M$, such that
$D(s)$ is maximal isotropic in $\mathsf E(s)$ for all $s\in S$ and 
\[ e_1\an{S}\in\Gamma_S(D), e_2\in\Gamma_S(D) \quad \Rightarrow\quad \lb e_1, e_2\rb\an{S}\in\Gamma_S(D)
\]
for all $e_1,e_2\in\Gamma(\mathsf E)$.

We will use the following lemma in two of our technical proofs
involving Dirac structures with support. We leave the proof to the
reader.
\begin{lemma}\label{useful_for_dirac_w_support}
  Let $\mathsf E\to M$ be a Courant algebroid and $D\to S$ a
  subbundle; with $S$ a sub-manifold of $M$. Assume that $D\to S$ is
  spanned by the restrictions to $S$ of a family $\mathcal S\subseteq
  \Gamma(\mathsf E)$ of sections of $\mathsf E$. Then $D$ is a Dirac
  structure with support $S$ if and only if
\begin{enumerate}
\item $\rho_{\mathsf E}(e)(s)\in T_sS$ for all $e\in\mathcal S$ and $s\in S$,
\item $D_s$ is Lagrangian in $\mathbb E_s$ for all $s\in S$ and
\item $\lb e_1,e_2\rb\an{S}\in\Gamma_S(D)$ for all $e_1,e_2\in\mathcal
  S$.
\end{enumerate}
\end{lemma}

Next we recall the notion of Dorfman connection \cite{Jotz13a}.
\begin{definition}\label{the_def}
  Let $(Q\to M,\rho_Q)$ be an anchored vector bundle and let $(B\to M,
  \dr_B, \langle\cdot\,,\cdot\rangle)$ be paired with $(Q,\rho_Q)$.  A
  \textbf{Dorfman ($Q$-)connection on $B$} is an $\R$-linear map
\begin{equation*}
  \Delta\colon \Gamma(Q)\to \operatorname{Der}(B)
\end{equation*} 
such that 
\begin{enumerate}
\item $\Delta_q$ is a derivation over $\rho_Q(q)\in\mx(M)$,
\item $\Delta_{f q}b=f\Delta_qb+\langle q, b\rangle \cdot \dr_B f$ and
  % \item $\rho(q)\langle q', b\rangle=\langle[q,q']_Q,
  %   b\rangle+\langle q', \Delta_qb\rangle$
\item $\Delta_{q}\dr_Bf=\dr_B(\rho_Q(q)f)$
\end{enumerate}
for all $f\in C^\infty(M)$, $q,q'\in\Gamma(Q)$, $b\in\Gamma(B)$.
\end{definition}

In the last definition, the map $\langle q, \cdot\rangle
\dr_Bf\colon B\to B$ is seen as a section of $\operatorname{Hom}(B,B)$,
i.e.~a derivation over $0\in\mx(M)$.

\begin{remark}
Note that if the pairing $\langle\cdot\,,\cdot\rangle:Q\times_MB\to
\R$ is nondegenerate, then $B\simeq Q^*$ 
and the map $\dr_{B}=\dr_{Q^*}\colon
C^\infty(M)\to\Gamma(Q^*)$ is \emph{defined} by
\eqref{compatibility_anchor_pairing}: we have then
$\dr_{Q^*}f=\rho_Q^*\dr f$
for all $f\in C^\infty(M)$.

The map $\Delta^*:\Gamma(Q)\times\Gamma(Q)\to \Gamma(Q)$ that is dual
to $\Delta$ in the sense of dual derivations, i.e. $\langle
\Delta^*_{q_1}q_2, \tau\rangle=\rho_Q(q_1)\langle q_2,\tau\rangle
-\langle q_2, \Delta_{q_1}\tau\rangle$ for all $q_1,q_2\in
\Gamma(Q)$ and $\tau\in\Gamma(Q^*)$ defines then a \emph{dull bracket} on
$\Gamma(Q)$:
\[\lb q_1, q_2\rb_\Delta=\Delta^*_{q_1}q_2
\]
in the sense of the following definition.
\end{remark}

\begin{definition}
  A \textbf{dull algebroid} is an anchored vector bundle $(Q\to M,
  \rho_Q)$ with a bracket $\lb\cdot\,,\cdot\rb $ on $\Gamma(Q)$ such that
\begin{equation}\label{anchor_preserves_bracket}
  \rho_Q\lb q_1, q_2\rb=[\rho_Q(q_1),\rho_Q(q_2)]
\end{equation}
and  (the Leibniz identity)
\begin{equation*} \lb f_1 q_1, f_2
  q_2\rb=f_1f_2\lb q_1,
  q_2\rb+f_1\rho_Q(q_1)(f_2)q_2-f_2\rho_Q(q_2)(f_1)q_1
\end{equation*}
for all $f_1,f_2\in C^\infty(M)$, $q_1, q_2\in\Gamma(Q)$.
\end{definition}
In other words, a dull algebroid is a \textbf{Lie algebroid} if its
bracket is in addition skew-symmetric and satisfies the Jacobi
identity. Note that a skew symmetric dull bracket can be constructed
as follows from an arbitrary dull bracket on $Q$; the
skew-symmetrisation $\lb\cdot\,,\cdot\rb'$ of $\lb\cdot\,,\cdot\rb$ is
defined by $\lb q_1,q_2\rb'=\frac{1}{2}\left(\lb q_1,q_2\rb-\lb q_2,
  q_1\rb\right)$ for all $q_1,q_2\in\Gamma(Q)$.

% In general, an $\R$-bilinear map
% $[\cdot\,,\cdot]_Q\colon\Gamma(Q)\times\Gamma(Q)\to\Gamma(Q)$ (a
% bracket on $\Gamma(Q)$) is said to be ``dual'' to a Dorfman $Q$-connection
% on $B$ if
% \[\rho_Q(q_1)\langle
% q_2, b\rangle =\langle [q_1, q_2]_Q, b\rangle+\langle q_2,
% \Delta_{q_1}b\rangle
% \]
% for all $q_1,q_2\in\Gamma(Q)$ and $b\in\Gamma(B)$.
% We then immediately have
% \begin{align*}
%   \langle b, f[q_1,q_2]_Q+\rho_Q(q_1)(f)q_2-[q_1,f
%   q_2]_Q\rangle=0\\
%   \langle b, f[q_1,q_2]_Q-\rho_Q(q_2)(f)q_1-[f q_1,
%   q_2]_Q\rangle=0
% \end{align*}
% for all $b\in\Gamma(B)$ and $q_1,q_2\in\Gamma(Q)$.
Assume that $\Delta\colon \Gamma(Q)\to\Gamma(Q^*)\to \Gamma(Q^*)$ is a
Dorfman connection and let $\lb\cdot\,,\cdot\rb_\Delta$ be the dual
dull bracket.  Note that \eqref{anchor_preserves_bracket} is
equivalent to (3) in Definition \ref{the_def} and the Leibniz identity
corresponds to (2) and (3).

The \emph{curvature} of a general Dorfman connection $\Delta\colon
\Gamma(Q)\times\Gamma(B)\to\Gamma(B)$ is the map
\[R_\Delta\colon \Gamma(Q)\times\Gamma(Q)\to\Gamma(B^*\otimes B),\]
defined on $q,q'\in\Gamma(Q)$ by
$R_\Delta(q,q'):=\Delta_q\Delta_{q'}-\Delta_{q'}\Delta_q-\Delta_{[q,q']_Q}$. If
$B=Q^*$ and the pairing is the natural one, the curvature is
equivalent to the Jacobiator of the dull bracket:
\begin{equation}\label{curv_dual_Jac}
\langle \tau, \Jac_{\lb\cdot\,,\cdot\rb_\Delta}(q_1,q_2,q_3)\rangle =\langle R_\Delta(q_1,q_2)\tau, q_3\rangle
\end{equation}
for $q_1,q_2,q_3\in\Gamma(Q)$ and $b\in\Gamma(B)$, where 
\[ \Jac_{\lb\cdot\,,\cdot\rb_\Delta}(q_1,q_2,q_3)=\lb \lb
q_1,q_2\rb_\Delta,q_3\rb_\Delta+\lb q_2,\lb q_1,q_3\rb_\Delta-\lb
q_1,\lb q_2,q_3\rb_\Delta\rb_\Delta
\]
is the Jacobiator of $\lb\cdot\,,\cdot\rb_\Delta$ in Leibniz
form. Hence, the Dorfman connection is flat if and only if the
corresponding dull bracket satisfies the Jacobi identity in Leibniz
form $\lb q_1,\lb q_2,q_3\rb_\Delta\rb_\Delta=\lb \lb
q_1,q_2\rb_\Delta,q_3\rb_\Delta+\lb q_2,\lb q_1,q_3\rb_\Delta$ for all
$q_1,q_2,q_3\in\Gamma(Q)$. A flat Dorfman connection is called a
\textbf{Dorfman representation}, if the dual dull bracket is in
addition skew-symmetric. The dull bracket $\lb\cdot\,,\cdot\rb_\Delta$
and the anchor $\rho_Q$ define then a Lie algebroid structure on
$\Gamma(Q)$.  Conversely, given a Lie algebroid $A$, then the Lie
derivative $\ldr{}^A\colon \Gamma(A)\times\Gamma(A^*)\to\Gamma(A^*)$
is a Dorfman representation. Hence, Lie algebroids are dual to Dorfman
representations.

\bigskip

Assume that $(Q,\rho_Q)$ is an anchored vector bundle with a Dorfman
connection $\Delta:\Gamma(Q)\times\Gamma(Q^*)\to\Gamma(Q^*)$ and let
$\nabla:\Gamma(Q)\times\Gamma(B)\to\Gamma(B)$ be a linear
$Q$-connection on a vector bundle\footnote{Note that the notation
  slightly changes here. $Q$ is paired with its dual $Q^*$ and $B$ is
  just a vector bundle over the same base.} $B$.  For each
$q\in\Gamma(Q)$, $\nabla_q$ and $\Delta_q$ define a derivation
$\lozenge_q$ of $\Gamma(\operatorname{Hom}(B,Q^*))$: for $\Phi\in
\Gamma(\operatorname{Hom}(B,Q^*))$ and $b\in\Gamma(B)$, we have
\begin{equation}
(\lozenge_q\Phi)(b)=\Delta_q(\Phi(b))-\Phi(\nabla_qb).
\end{equation}
The map $\lozenge:\Gamma(Q)\times \Gamma(\operatorname{Hom}(B,Q^*))\to
\Gamma(\operatorname{Hom}(B,Q^*))$ satisfies\footnote{One could define this
 as a  Dorfman connection over a vector bundle valued pairing: here the
  pairing $Q\times_M\operatorname{Hom}(B,Q^*)\to B^*$ would be defined by
  $(q,\Phi)\mapsto \Phi^*(q)$.}
\[\lozenge_{f q}\Phi=f
\lozenge_q\Phi+\Phi^*(q)\cdot\rho_Q^*\dr f
\]
and $\lozenge_{q}(\xi\cdot\rho_Q^*\dr f)=\nabla_q^*\xi\cdot\rho_Q^*\dr
f+\xi\cdot\rho_Q^*\dr(\ldr{\rho_Q(q)}f) $ for all $q\in\Gamma(Q)$,
$f\in C^\infty(M)$ and $\xi\in\Gamma(B^*)$.  If the dull bracket
$\lb\cdot\,,\cdot\rb_\Delta$ dual to $\Delta$ is skew-symmetric, then
the complex $\Omega(Q,\operatorname{Hom}(B,Q^*))$ has an induced
operator $\dr_\lozenge\colon
\Omega^q(Q,\operatorname{Hom}(B,Q^*))\to\Omega^{q+1}(Q,\operatorname{Hom}(B,Q^*))$
given by the Koszul formula
\begin{equation*}
\begin{split}
  \dr_\lozenge\eta(q_1,\ldots,q_{k+1})=&\sum_{i<j}(-1)^{i+j}\omega(\lb q_i,q_i\rb_\Delta,\ldots,\hat q_i, \ldots, \hat q_j, \ldots, q_{k+1})\\
  &\qquad +\sum_i(-1)^{i+1}\lozenge_{q_i}(\omega(q_1, \ldots, \hat q_i,
  \ldots, q_{k+1}))
\end{split}
\end{equation*}
for all $\eta\in\Omega^k(Q,\operatorname{Hom}(B,Q^*))$ and $q_1,\ldots,q_{k+1}\in\Gamma(Q)$.

\subsection{Dorfman 2-representations and split Lie 2-algebroids}\label{dorfman_eq_split}
We now define the \emph{Dorfman 2-representations} and show that they
are the dual derivations to split Lie 2-algebroids. 

Recall from \S\ref{symplectic} the definition of a graded derivation
on an $[n]$-manifold.
A \textbf{homological} vector field $\chi$ on $\mathcal M$ is a
derivation of degree $1$ of $C^\infty(\mathcal M)$ such that
$\chi^2=\frac{1}{2}[Q,Q]$ vanishes.  A
homological vector field on a $[1]$-manifold $\mathcal M=E[-1]$ is the
Cartan differential $\dr_E$ associated to a Lie algebroid structure on
$E$. This result is due to Vaintrob \cite{Vaintrob97} and was
explained in our introduction. A \textbf{Lie n-algebroid} is an
$[n]$-manifold endowed with a homological vector field (an
\emph{NQ-manifold} of degree $n$).  

A \textbf{split Lie n-algebroid} is a split $[n]$-manifold endowed
with a homological vector field.  Split Lie n-algebroids were studied
by Sheng and Zhu \cite{ShZh14} and described as vector bundles endowed
with a bracket that satisfies the Jacobi identity up to some
correction terms, see also \cite{BoPo13}.  Let us first give in our
own words their definition of a split Lie $2$-algebroid.
\begin{definition}\label{split_Lie2}
  A split Lie 2-algebroid $Q\oplus B^*\to M$ is a pair of an anchored vector bundle
  \footnote{The names that we choose for the vector bundles will
    become natural in a moment.}  $(Q\to M, \rho_Q)$ and a vector
  bundle $B\to M$, together with
\begin{enumerate}
\item a vector bundle map $l_1\colon B^*\to Q$,
\item a
  skew-symmetric dull bracket $\lb\cdot\,,\cdot\rb\colon
  \Gamma(Q)\times\Gamma(Q)\to\Gamma(Q)$,
\item a linear connection $\nabla\colon\Gamma(Q)\times\Gamma(B)\to\Gamma(B)$ and 
\item a vector valued $3$-form $l_3\in \Omega^3(Q,B^*)$,
\end{enumerate}
such that
\begin{enumerate}
\item[(i)] $\nabla_{l_1(\beta_1)}^*\beta_2+\nabla_{l_1(\beta_2)}^*\beta_1=0$ for
  all $\beta_1,\beta_2\in\Gamma(B^*)$,
\item[(ii)] $\lb q, l_1(\beta)\rb=l_1(\nabla_q^*\beta)$ for
  $q\in\Gamma(Q)$ and $\beta\in\Gamma(B^*)$,
\item [(iii)]$\Jac_{\lb \cdot,\cdot\rb}=-l_1\circ l_3\in \Omega^3(Q,Q)$,
\item[(iv)] $R_{\nabla^*}(q_1,q_2)\beta=l_3(q_1,q_2,l_1(\beta))$
  for $q\in\Gamma(Q)$ and $\beta_1,\beta_2\in\Gamma(B^*)$, and 
\item[(v)] $\dr_{\nabla^*}l_3=0$.
\end{enumerate}
\end{definition}
From (iii) follows the identity $\rho_Q\circ l_1=0$. To get the
definition that was first given in \cite{ShZh14}, consider the skew
symmetric bracket $l_2\colon \Gamma(Q\oplus B^*)\times\Gamma(Q\oplus
B^*)\to \Gamma(Q\oplus B^*)$,
\begin{equation}\label{l2}
l_2((q_1,\beta_1),(q_2,\beta_2))=(\lb q_1, q_2\rb, \nabla_{q_1}^*\beta_2-\nabla_{q_2}^*\beta_1)
\end{equation}
for $q_1,q_2\in\Gamma(Q)$ and $\beta_1,\beta_2\in\Gamma(B^*)$. Note
that this bracket satisfies a Leibniz identity with anchor
$\rho_Q\circ \pr_Q\colon Q\oplus B^*\to TM$ and that the Jacobiator of
this bracket is then given by
\[\Jac_{l_2}((q_1,\beta_1), (q_2,\beta_2), (q_3,\beta_3))
=(-l_1(l_3(q_1,q_2,q_3)), l_3(q_1,q_2,l_1(\beta_3))+\rm{c.p.}
\]
Since $\dr_{\nabla^*}l_3=0$, one could say that a split Lie
2-algebroid is a Lie algebroid ``up to homotopy''.  Conversely, all
the geometric objects in the previous definition can easily be
constructed from the original definition of split Lie 2-algebroids.

As we have seen above, flat Dorfman connections with skew-symmetric
dual dull brackets are in duality with Lie algebroids (or equivalently
``split Lie 1-algebroids'').  As we will see below, we define in fact
Dorfman 2-representations as the Lie derivatives that are dual to
split Lie 2-algebroids.  The notion of Dorfman 2-representation
defined below resembles the notion of 2-representation. The meaning of
this analogy will become clearer in our study of VB-Courant
algebroids.
\begin{definition}\label{def_dorfman_2_conn}
  Let $(Q\to M, \rho_Q)$ be an anchored vector bundle.  A $(Q,
  \rho_Q)$-Dorfman 2-representation is a quadruple
  $(\partial_B,\Delta,\nabla,R)$ where $\partial_B\colon Q^*\to B$ is
  a vector bundle morphism, $\Delta$ is a Dorfman connection
\[\Delta\colon \Gamma(Q)\times\Gamma(Q^*)\to\Gamma(Q^*),
\]
$\nabla$ is a linear connection
\[\nabla\colon \Gamma(Q)\times\Gamma(B)\to\Gamma(B),\]
and $R$ is an element of $\Omega^2(Q,\operatorname{Hom}(B,Q^*))$ such
that
\begin{enumerate}
\item[(D1)]$\partial_B\circ \Delta_q=\nabla_q\circ \partial_B$,
\item[(D2)] $\lb q_1, q_2\rb_\Delta=-\lb q_2, q_1\rb_\Delta$,
\item[(D3)]
  $\nabla^*_{\partial_B^*\xi_1}\xi_2+\nabla^*_{\partial_B^*\xi_2}\xi_1=0$,
\item[(D4)] $\partial_B\circ R(q_1,q_2)=R_\nabla(q_1,q_2)$ and
  $R(q_1,q_2)\circ \partial_B=R_\Delta(q_1,q_2)$,
\item[(D5)] $R(q_1,q_2)^*q_3=-R(q_1,q_3)^*q_2$ and
\item[(D6)]  $\dr_\lozenge R(q_1,q_2,q_3)=\nabla_\cdot^*\left(R(q_1,q_2)^*q_3\right)$
\end{enumerate}
for all 
 $\xi_1,\xi_2\in\Gamma(B^*)$ and  $q, q_1, q_2, q_3\in\Gamma(Q)$ and $f\in C^\infty(M)$.
\end{definition}

Note that (D5) is equivalent to $R$ defining an element $\omega_R$ of
$\Omega^3(Q,B^*)$ by $\omega_R(q_1,q_2,q_3)=R(q_1,q_2)^*q_3$.
Axiom (D6) is equivalent to $\dr_{\nabla^*}\omega_R=0$ and (D4) is
equivalent to $(R_{\nabla}(q_1,q_2))^*\xi=-R_{\nabla^*}(q_1,q_2)\xi=\omega_R(q_1,q_2,\partial_B^*\xi)$ for
$q_1,q_2\in\Gamma(Q)$ and $\xi\in\Gamma(B^*)$ and 
$\Jac_{\lb\cdot\,,\cdot\rb_\Delta}=\partial_B^*\omega_R$.

One can see quite easily that the definition of a Dorfman
2-representations is just a rephrasing of the definition of a Lie
2-algebroid. % \footnote{Note also that (D1) implies $\partial_B\circ\rho_Q^*=0$: take
% $q\in\Gamma(Q)$ and $\tau\in\Gamma(Q^*)$ with non-vanishing pairing
% at $m\in M$, and $f\in C^\infty(M)$. Then
% $f\nabla_q(\partial_B\tau)=\nabla_{f
%   q}(\partial_B\tau)=\partial_B(\Delta_{f
%   q}\tau)=\partial_B(f\Delta_q\tau+\langle
% q,\tau\rangle\rho_Q^*\dr f)=f\nabla_q(\partial_B\tau)
% +\langle q,\tau\rangle\partial_B(\rho_Q^*\dr f) $. Hence,
% $\partial_B(\rho_Q^*\dr f)=0$ at $m\in M$. Since $m$ and
% $f$ were arbitrary, we have proved that
% $\partial_B\circ\rho_Q^*=0$.}
Set $\omega_R=l_3$ and $\partial_B^*=-l_1$.  Then (D1) is
(ii) in Definition \ref{split_Lie2} and the other axioms have already
been explained above.  Note that the vector bundle $Q\oplus B^*$ is
anchored by $\rho_Q$ and paired with $Q^*\oplus B$ by the natural
pairing and the map $C^\infty(M)\to \Gamma(Q^*\oplus B)$,
$f\mapsto (\rho_Q^*\dr f, 0)$. Hence we can define 
a new Dorfman connection 
\[\Delta^2\colon \Gamma(Q\oplus B^*)\times \Gamma(Q^*\oplus B)\to\Gamma(Q^*\oplus B)\]
by 
\[\Delta^2_{(q,\beta)}(\tau,b)=(\Delta_q\tau, \nabla_qb),
\]
then $\Delta^2$ is the Dorfman connection that is dual to the bracket
$l_2$ defined in \eqref{l2}.  Hence, we can think of Dorfman
2-representations as ``Lie 2-derivatives'', or ``Lie derivatives up to
homotopy''. In other words, as the duals of Lie algebroids are Dorfman
representations, the duals of Lie $2$-algebroids are Dorfman
$2$-representations.

\subsection{Split Lie-$2$-algebroids as split $[2]$Q-manifolds}
Before we go on with the study of examples, we briefly describe how to
construct from the objects in Definitions \ref{split_Lie2} and
\ref{def_dorfman_2_conn} the corresponding homological vector fields
on split $[2]$-manifolds.

Consider a split $[2]$-manifold $\mathcal M=Q[-1]\oplus B^*[-2]$.
% together with a homological vector field $\chi$. 
On an open chart
$U\subseteq M$ trivialising both $Q$ and $B$, we have coordinates
$(x_1,\ldots,x_p)$ and we can choose a local frame $(q_1,\ldots,q_{r_1})$ of
sections of $Q$, and a local frame $(\beta_1,\ldots,\beta_{r_2})$ of sections
of $B^*$. We denote by
$(\tau_1,\ldots,\tau_{r_1})$ and $(b_1,\ldots,b_{r_2})$ the dual frames
for $Q^*$ and $B$, respectively. 
As functions on $\mathcal M$, 
the coordinates functions $x_1,\ldots,x_p$ have degree $0$, the functions $\tau_1,
\ldots,\tau_{r_1}$ have degree $1$ and the functions $b_1,\ldots, b_{r_2}$
have degree $2$.  The vector fields $\partial_{x_1},
\ldots, \partial_{x_p},\partial_{\tau_1},\ldots,\partial_{\tau_{r_1}},\partial_{b_1},\ldots,\partial_{b_{r_2}}$
have degree $0,-1,-2$, respectively, and generate (locally) the set of
vector fields on $\mathcal M$ as a $C^\infty_{\mathcal M}$-module.
% Then since it has degree $1$ we can write $\chi$ as
% \begin{equation}\label{Q_in_coordinates}
% \begin{split}
%   \chi=&\sum_{i,j}A^{ij}\tau_i\partial_{x_j}
% -\sum_{i<j}\sum_kB^{ijk}\tau_i\tau_j\partial_{\tau_k}\\
%   &+\sum_{r,k}C^{rk} b_r\partial_{\tau_k}
% -\sum_{i<j<k}\sum_lD^{ijkl}\tau_i\tau_j\tau_k\partial_{b_l}-\sum_{ijk}E^{ijl}\tau_ib_j\partial_{b_k}
% \end{split}
% \end{equation}
% with $A^{ij}$, $B^{ijk}$, $C^{rk}$, $D^{ijkl}$ and $E^{ijl}$ smooth functions on $U$.
% Now we \emph{define} a vector bundle morphism $\rho_Q\colon Q\an{U}\to TU$ by $A^{ij}=
Assume that $Q$ is endowed with an anchor $\rho_Q$ and take a Dorfman
2-representation $(\partial_B\colon Q^*\to B,\Delta,\nabla,R)$ of $Q$
on $B\oplus Q^*$. Define a vector field $\mathcal Q$ of degree $1$ on
$\mathcal M$ by the following formula in local coordinates:
\begin{equation}\label{Q_in_coordinates}
\begin{split}
  \mathcal Q=&\sum_{i,j}\rho_Q(q_i)(x_j)\tau_i\partial_{x_j}
-\sum_{i<j}\sum_k\langle \lb q_i, q_j\rb_\Delta,\tau_k\rangle\tau_i\tau_j\partial_{\tau_k}\\
  &+\sum_{r,k}\langle\partial_B^*\beta_r,\tau_k\rangle b_r\partial_{\tau_k}
-\sum_{i<j<k}\sum_l\omega_R(q_i,q_j,q_k)(b_l)\tau_i\tau_j\tau_k\partial_{b_l}\\
  &-\sum_{ijl}\langle\nabla_{q_i}^*\beta_j,b_l\rangle\tau_ib_j\partial_{b_l}.
\end{split}
\end{equation}
This vector field satisfies $[\mathcal Q,\mathcal Q]=0$. To see this,
we compute $\mathcal Q\circ \mathcal Q$ in local coordinates.

For $f\in C^\infty(M)$ we have $\mathcal Q(f)=\rho_Q^*\dr f$ and a
relatively easy computation shows that $\mathcal Q(\mathcal Q(f))=0$
for all $f\in C^\infty(M)$ if and only if $\rho_Q\circ\partial_B^*=0$
and $\rho_Q\lb q_1,q_2\rb_\Delta=[\rho_Q(q_1), \rho_Q(q_2)]$ for all
$q_1,q_2\in\Gamma(Q)$.

A longer, but straightforward computation shows that $\mathcal
Q(\mathcal Q(\tau_k))=0$ for $k=1,\ldots,r_1$ if and only if
$\nabla_{q_i}(\partial_B\tau_k)=\partial_B(\Delta_{q_i}\tau_k)$ for
all $i=1,\ldots,n$ and $\lb \lb q_i,q_j\rb_\Delta,q_l\rb_\Delta+\lb
q_j, \lb q_i,q_l\rb_\Delta \rb_\Delta-\lb q_i,\lb
q_j,q_l\rb_\Delta\rb_\Delta=\partial_B^*R(q_i,q_j)^*q_l=\partial_B^*\omega_R(q_i,q_j,q_l)$
for all $i,j,l=1,\ldots,n$.

Finally, a very long but also straightforward computation shows that
$\mathcal Q^2(b_k)=0$ for $k=1,\ldots,l$ if and only if
$\nabla_{\partial_B^*\beta_r}\beta_j+\nabla_{\partial_B^*\beta_j}\beta_r=0$
for $j,r=1,\ldots,l$, $\partial_B(R(q_i,q_j)b_k)=R_\nabla(q_i,q_j)b_l$
for $i,j=1,\ldots,n$ and $k=1,\ldots,l$ and
$\dr_{\nabla^*}\omega_R=0$.

Conversely, write locally the homological vector field in
coordinates and define $\rho_Q$, $\omega_R$ and $\lb\cdot\,,\cdot\rb_\Delta$
on basis sections $q_1,\ldots,q_n$ of $Q$, $\partial_B^*$ on basis
sections of $B^*$ and $\nabla^*$ on basis sections of $B^*$ and
$Q$. The global objects can then be defined locally using Leibniz
identities and tensoriality of these objects. These local
definitions will then be compatible with changes of trivialisations
and define globally the structure components of a split Lie
2-algebroid.

Hence, we have found an explicit way of writing the homological vector
field that is equivalent to a split Lie 2-algebroid.  Note that by
definition, a $[2]$-manifolds is always locally split. Hence, a Lie
$2$-algebroid always defines local Dorfman $2$-representation, and the
homological vector field can always be locally written as in
\eqref{Q_in_coordinates}. Going from an open set to an other will
involve a change of local basis and a change of splitting, but we will show later that changes
of splittings of Lie $2$-algebroids are easy to describe

\subsection{Examples of Dorfman 2-representations and split Lie 2-algebroids}\label{examples_split_lie_2}
We describe here five classes of examples of split Lie
2-algebroids. Later we will discuss their geometric meanings. We do
not verify in detail the axioms for Dorfman 2-representations of for
split Lie 2-algebroids. The computations in order to do this for
Examples \ref{standard}, \ref{adjoint} and \ref{double} are long, but
straightforward. Note that, alternatively, the next section will
provide a geometric proof of the fact that the following objects are
Dorfman $2$-representations (and split Lie $2$-algebroids), since we
will find them to be equivalent to special classes of VB-Courant
algebroids.

\subsubsection{Dorfman 2-representation associated to a Lie algebroid representation} 
Let $(Q\to M,\rho,[\cdot\,,\cdot])$ be a Lie algebroid and
$\nabla\colon\Gamma(Q)\times\Gamma(B)\to\Gamma(B)$ a representation of $Q$
on a vector bundle $B$.  Then $(\partial_B=0\colon Q^*\to B, \ldr{}^Q,
\nabla, R=0)$ is a Dorfman 2-representation.

 The corresponding Lie 2-algebroid is a
semi-direct extension of the Lie algebroid $Q$ (and a special case of
the double Lie 2-algebroids defined later). Here $l_1=0$ and $l_2$ is given by
$l_2(q_1+\beta_1,q_2+\beta_2)=[q_1,q_2]+(\nabla_{q_1}^*\beta_2-\nabla_{q_2}^*\beta_1)$
for $q_1,q_2\in\Gamma(Q)$ and $\beta_1,\beta_2\in\Gamma(B^*)$, and
$l_3=0$. Hence $(Q\oplus B^*\to M, \rho=\rho_Q\circ\pr_Q, l_2)$ is simply a Lie algebroid.

\subsubsection{Standard Dorfman 2-representations}\label{standard}
Let $E\to M$ be a vector bundle, set $\partial_E=\pr_E\colon E\oplus
T^*M\to E$ and consider a skew-symmetric dull bracket
$\lb\cdot\,,\cdot\rb$ on $\Gamma(TM\oplus E^*)$, with $TM\oplus E^*$
anchored by $\pr_{TM}$ and let $\Delta\colon \Gamma(TM\oplus
E^*)\times\Gamma(E\oplus T^*M)\to\Gamma(E\oplus T^*M)$ be the dual
Dorfman connection. This defines as follows a split Lie 2-algebroid
structure on the vector bundles $(TM\oplus E^*,\pr_{TM})$ and $E^*$,
or a Dorfman 2-representation of $(TM\oplus E^*,\pr_{TM})$ on $E\oplus
T^*M$ and $E^*$.

Let $\nabla\colon\Gamma(TM\oplus E^*)\times\Gamma(E)\to\Gamma(E)$ be
the ordinary linear connection\footnote{To see that
  $\nabla=\pr_E\circ\Delta\circ\iota_E$ is an ordinary connection,
  recall that since $TM\oplus E^*$ is anchored by $\pr_{TM}$, the map
  $\dr_{E\oplus T^*M}=\pr_{TM}^*\dr\colon C^\infty(M)\to
  \Gamma(E\oplus T^*M)$ sends $f\to (0,\dr f)$.}  defined by
$\nabla=\pr_E\circ\Delta\circ\iota_E$, where $\Delta$ is dual to
$\lb\cdot\,,\cdot\rb$.  The vector bundle map $l_1=-\pr_E^*\colon
E^*\to TM\oplus E^*$ is just the opposite of the canonical inclusion.
Finally define $l_3$ by
$l_3(v_1,v_2,v_3)=\Jac_{\lb\cdot\,,\cdot\rb}(v_1,v_2,v_3)$ and
accordingly $R=R_\Delta\circ\iota_E\in\Omega^2(TM\oplus E^*,
\Hom(E,E\oplus T^*M))$. (Note that since $TM\oplus E^*$ is anchored by
$\pr_{TM}$, the tangent part of the dull bracket must just be the Lie
bracket of vector fields. The Jacobiator $\Jac_{\lb\cdot\,,\cdot\rb}$
can hence really be seen as an element of $\Omega^3(TM\oplus E^*,
E^*)$.)

A straightforward verification of the axioms shows that
$(\partial_E,\Delta,\nabla,R)$ is a $(TM\oplus E^*,\pr_{TM})$-Dorfman
2-representation, or equivalently, that $l_1$, $\lb\cdot\,,\cdot\rb$,
$\nabla^*$, $l_3$ define a split Lie 2-algebroid. For reasons that
will become clearer later, we call \emph{standard} this type of split
Lie 2-algebroid and Dorfman 2-representation.

\subsubsection{Adjoint Dorfman 2-representations}\label{adjoint}
Let $\mathsf E\to M$ be a Courant algebroid with anchor $\rho_{\mathsf
  E}$ and bracket $\lb\cdot\,,\cdot\rb$ and choose a metric linear
connection $\nabla\colon\mx(M)\times\Gamma(\mathsf E)\to\Gamma(\mathsf
E)$, i.e.~a linear connection that preserves the pairing. Set
$\partial_{TM}=\rho_{\mathsf E}\colon \mathsf E\to TM$, identifying
$\mathsf E$ with its dual via the pairing.  The map $\Delta\colon
\Gamma(\mathsf E)\times\Gamma(\mathsf E)\to\Gamma(\mathsf E)$,
\[\Delta_ee'=\lb e,e'\rb+\nabla_{\rho(e')}e\]
is a Dorfman connection, which we call the \emph{basic Dorfman connection associated to
$\nabla$}. The dual skew-symmetric (!) dull bracket is given by $
\lb e,e'\rb_\Delta=\lb e,e'\rb-\Beta\inv\rho^*\langle\nabla_\cdot
e,e'\rangle $ for all $e,e'\in\Gamma(\mathsf E)$.  The map
$\nabla^{\rm bas}\colon\Gamma(\mathsf E)\times\mx(M)\to\mx(M)$,
\[\nabla^{\rm bas}_eX=[\rho(e),X]+\rho(\nabla_{X}e)\]
is a linear connection, the \emph{basic connection associated to $\nabla$}.

We now define the \emph{basic curvature} $R_\Delta^{\rm
  bas}\in\Omega^2(\mathsf E,\Hom(TM,\mathsf E))$
by\footnote{Alternatively, $R_\Delta^{\rm bas}$ can be defined by $
  R_\Delta^{\rm bas}(e_1,e_2)X= -\nabla_X\lb e_1,e_2\rb+\lb
  \nabla_Xe_1,e_2\rb+\lb e_1,\nabla_Xe_2\rb +\nabla_{\nabla^{\rm
      bas}_{e_2}X}e_1-\nabla_{\nabla^{\rm
      bas}_{e_1}X}e_2-\beta\inv\langle \nabla_{\nabla^{\rm bas}_\cdot
    X}e_1,e_2\rangle$. Using $-R_\nabla^*=R_{\nabla^*}=R_\nabla$
  (where we identify $\mathsf E$ with its dual using
  $\langle\cdot\,,\cdot\rangle$), the identity $R_\Delta^{\rm
    bas}(e_1,e_2)=-R_\Delta^{\rm bas}(e_2,e_1)$ is obvious using the
  first definition. It is an easy computation using the second one.}
\begin{align*}
  R_\Delta^{\rm bas}(e_1,e_2)X=& -\nabla_X\lb
  e_1,e_2\rb_\Delta+\lb \nabla_Xe_1,e_2\rb_\Delta+\lb e_1,\nabla_Xe_2\rb_\Delta\\
  &+\nabla_{\nabla^{\rm bas}_{e_2}X}e_1-\nabla_{\nabla^{\rm
      bas}_{e_1}X}e_2-\beta\inv\rho^*\langle R_\nabla(X,\cdot)e_1,e_2\rangle
\end{align*}
for all $e_1,e_2\in\Gamma(\mathsf E)$ and $X\in\mx(M)$.
Then $(\rho, \Delta,\nabla^{\rm bas}, R_\Delta^{\rm bas})$ is a Dorfman 2-representation.
Note the similarity of this construction with the one of the adjoint 
representation up to homotopy in Example \ref{tangent_double_al}.

 The corresponding \emph{adjoint} split Lie 2-algebroid can be described as follows. 
The map $l_1$ is $-\Beta\inv\circ
 \rho_{\mathsf E}^*\colon T^*M\to\mathsf E$ and 
 $l_2(e_1+\theta_1,e_2+\theta_2)$ is
\[(\lb e_1,
e_2\rb-\beta\inv\rho_{\mathsf E}\langle \nabla_\cdot e_1,
e_2\rangle)+(\ldr{\rho(e_1)}\theta_2-\ldr{\rho(e_2)}\theta_1+\langle
\rho^*\theta_1,\nabla_\cdot e_2\rangle-\langle
\rho^*\theta_2,\nabla_\cdot e_1\rangle)
\]
for $e_1,e_2\in\Gamma(\mathsf E)$ and
$\theta_1,\theta_2\in\Omega^1(M)$. The form $l_3\in\Omega^3(\mathsf
E,T^*M)$ is given by $l_3(e_1,e_2,e_3)=\langle R_\Delta^{\rm
  bas}(e_1,e_2), e_3\rangle$ and corresponds to the
tensor $\Psi$ defined in \cite[Definition 4.1.2]{Li-Bland12} (the left-hand side of \eqref{alternative_curvature}).  
We will see later that the adjoint split Lie 2-algebroids are exactly the \emph{split symplectic 
Lie 2-algebroids}, and correspond hence to splittings of the tangent doubles of 
Courant algebroids. They are usually considered as the super-geometric objects corresponding 
to Courant algebroids, but we will explain this correspondence in more detail later.

\subsubsection{Dorfman 2-representation defined by a 2-representation}\label{semi-direct}
Let $(\partial_B\colon C\to B,\nabla,\nabla,R)$ be a representation up
to homotopy of a Lie algebroid $A$ on $B\oplus C$.  Then the quadruple
$(\partial_B\circ \pr_C\colon C\oplus A^*\to B, \Delta,\nabla,R)$
defined by
\begin{equation}
\begin{split}
\Delta\colon&\Gamma(A\oplus C^*)\times\Gamma(C\oplus A^*)\to\Gamma(C\oplus A^*)\\
&\Delta_{(a,\gamma)}(c,\alpha)=(\nabla_ac,\ldr{a}\alpha+\langle\nabla^*_\cdot\gamma,c\rangle),
\end{split}
\end{equation}
\begin{equation}
\begin{split}
  \nabla\colon&\Gamma(A\oplus C^*)\times\Gamma(B)\to\Gamma(B), \qquad \nabla_{(a,\gamma)}b=\nabla_ab
\end{split}
\end{equation}
with $A\oplus C^*$ anchored by $\rho_A\circ\pr_A$,
and 
$R\in \omega^2(A\oplus C^*, \operatorname{Hom}(B,C\oplus A^*))$,
\begin{equation}
R((a_1,\gamma_1), (a_2,\gamma_2))=\left(R(a_1,a_2), \langle\gamma_2,R(a_1,\cdot)\rangle
+\langle\gamma_1,R(\cdot,a_2)\rangle\right)
\end{equation}
is a Dorfman 2-representation.

The vector bundle map $l_1$ is here 
$l_1=\iota_{C^*}\circ\partial_B^*$, where $\iota_{C^*}\colon C^*\to
A\oplus C^*$ is the canonical inclusion, and the bracket $l_2$ is given
by \[l_2((a_1,\gamma_1)+\alpha_1,(a_2,\gamma_2)+\alpha_2)=([a_1,a_2],\nabla_{a_1}^*\gamma_2-\nabla_{a_2}^*\gamma_1)+(\nabla_{a_1}^*\alpha_2-\nabla_{a_2}^*\alpha_1)\]
for $a_1,a_2\in\Gamma(A)$, $\gamma_1,\gamma_2\in\Gamma(C^*)$ and
$\alpha_1,\alpha_2\in\Gamma(B^*)$. The tensor $l_3$ is finally given
by \[l_3((a_1,\gamma_1), (a_2,\gamma_2), (a_3,\gamma_3))=\langle
R(a_1,a_2), \gamma_3\rangle+\text{c.p.}
\]

Note that if we work with the dual $A$-representation up to
homotopy\linebreak $(\partial_B^*\colon B^*\to C^*, \nabla^*,
\nabla^*, -R^*)$, then we get the Lie 2-algebroid defined in
\cite[Proposition 3.5]{ShZh14} as the semi-direct product of a 2-
representation and a Lie algebroid.  This is then also a special case
of the bicrossproduct of a matched pair of 2-representations (see the
next section).  We will explain later the slightly peculiar choice
that we make here.

\subsection{The bicrossproduct of a matched pair of
  2-representations}\label{double}
In this section we describe a class of examples, which is of great
independent interest: we show that the bicrossproduct of a
matched pair of 2-representations is a split Lie 2-algebroid.

We construct a split Lie 2-algebroid $(A\oplus B)\oplus C$ induced by
a matched pair of 2-representations as in Definition
\ref{matched_pair_2_rep}.  The vector bundle $A\oplus B\to M$ is
anchored by $\rho_A\circ\pr_A+\rho_B\circ\pr_B$ and paired with $A^*\oplus B^*$ as
follows:
\[\langle (a,b),(\alpha,\beta)\rangle=\alpha(a)-\beta(b)
\]
for all $a\in\Gamma(A)$, $b\in\Gamma(B)$, $\alpha\in\Gamma(A^*)$ and
$\beta\in\Gamma(B^*)$. The morphism $A^*\oplus B^*\to C^*$ is
$\partial_A^*\circ\pr_{A^*}+\partial_B^*\circ\pr_{B^*}$. The $A\oplus B$-Dorfman connection
on $A^*\oplus B^*$ is defined by
\[\Delta_{(a,b)}(\alpha, \beta)=(\nabla^*_b\alpha+\ldr{a}\alpha-\langle\nabla_\cdot
b,\beta\rangle, \nabla^*_a\beta+\ldr{b}\beta-\langle \nabla_\cdot
a,\alpha\rangle).
\]
The dual dull bracket on $\Gamma(A\oplus B)$ is
\begin{equation}\label{bicross_bracket}
\lb (a,b), (a',b')\rb=([a,a']+\nabla_ba'-\nabla_{b'}a,
[b,b']+\nabla_ab'-\nabla_{a'}b).
\end{equation}
The $A\oplus B$-connection on $C^*$ is simply given by
$\nabla_{(a,b)}^*\gamma=\nabla_a^*\gamma+\nabla_b^*\gamma$ and the
dual connection is $\nabla\colon \Gamma(A\oplus
B)\times\Gamma(C)\to\Gamma(C)$, 
\begin{equation}\label{nabla_bicross_C}
\nabla_{(a,b)}c=\nabla_ac+\nabla_bc.
\end{equation}
Finally, the form $\omega_R=l_3\in\Omega^3(A\oplus B,C)$ is given by
\begin{equation}\label{omega_matched}
\begin{split}
  \omega_R((a_1,b_1),(a_2,b_2),(a_3,b_3))&=R(a_1,a_2)b_3+R(a_2,a_3)b_1+R(a_3,a_1)b_2\\
  &\qquad -R(b_1,b_2)a_3-R(b_2,b_3)a_1-R(b_3,b_1)a_2.
\end{split}
\end{equation}
The quadruple $(\partial_A^*+\partial_B^*\colon A^*\oplus B^*\to C^*,
\nabla, \Delta,R)$ is a Dorfman 2-representation.  Equivalently the
vector bundle $(A\oplus B)\oplus C\to M$ with the anchor $\rho_A\circ\pr_A+\rho_B\circ\pr_B\colon A\oplus B\to C$,
$l_1=\partial_B\circ\pr_B-\partial_A\circ\pr_A$, $l_3=\omega_R$ and the skew-symmetric
dull bracket \eqref{bicross_bracket} define a split Lie
2-algebroid.

Moreover, we prove the following theorem:
\begin{theorem}\label{double_2_rep}
  The double of a matched pair of 2-representations is a split Lie
  2-algebroid with the structure given above.
  Conversely if $(A\oplus B)\oplus C$ has a split Lie 2-algebroid
  structure such that 
\begin{enumerate}
\item $\lb (a_1,0),(a_2,0)\rb=([a_1,a_2],0)$
with a section $[a_1,a_2]\in\Gamma(A)$ for all $a_1,a_2\in\Gamma(A)$
and in the same manner $\lb (0,b_1),(0,b_2)\rb=(0, [b_1,b_2])$
with a section $[b_1,b_2]\in\Gamma(B)$ for all $b_1,b_2\in\Gamma(B)$, and
\item $l_3((a_1,0), (a_2,0), (a_3,0))=0$ and $l_3((0, b_1), (0,b_2), (0,b_3))=0$
for all $a_1,a_2,a_3$ in $\Gamma(A)$ and $b_1,b_2,b_3$ in $\Gamma(B)$,
\end{enumerate} then
$A$ and $B$ are Lie subalgebroids of $(A\oplus B)\oplus C$ and 
  $(A\oplus B)\oplus C$ is the double of a matched pair of
  $2$-representations of $A$ on $B\oplus C$ and of $B$ on $A\oplus
  C$. The 2-representation of $A$ is given by
\begin{equation}\label{2_rep_A}
\begin{split}
\partial_B(c)&=\pr_B(l_1(c)),\,\, 
\nabla_ab=\pr_B\lb (a,0), (0,b)\rb,\,\,
\nabla_ac=\nabla_{(a,0)}c,\\
 R_{AB}(a_1,a_2)b&=l_3(a_1,a_2,b)
\end{split}
\end{equation}
and the $B$-representation is given by 
\begin{equation}
\begin{split}\label{2_rep_B}
\partial_A(c)&=-\pr_A(l_1(c)),\,\,
\nabla_ba=\pr_A\lb(0,b),(a,0)\rb,\,\,\nabla_bc=\nabla_{(0,b)}c,\\
 R_{BA}(b_1,b_2)a&=-l_3(b_1,b_2,a).
\end{split}
\end{equation}
\end{theorem}

\begin{proof}
  Assume first that $(A\oplus B)\oplus C$ is a split Lie 2-algebroid
  with (1) and (2). The bracket
  $[\cdot\,,\cdot]\colon\Gamma(A)\times\Gamma(A)\to\Gamma(A)$ defined
  by $\lb (a_1,0), (a_2,0)\rb=([a_1,a_2],0)$ is obviously
  skew-symmetric and $\R$-bilinear. Define an anchor $\rho_A$ on $A$
  by $\rho_A(a)=\rho_{A\oplus B}(a,0)$. Then we get immediately
  \[([a_1,fa_2],0)=\lb (a_1,0),
  f(a_2,0)\rb=f([a_1,a_2],0)+\rho_{A\oplus B}(a_1,0)(f)(a_2,0),\]
  which shows that $[a_1,fa_2]=f[a_1,a_2]+\rho_A(a_1)(f)a_2$ for all
  $a_1,a_2\in\Gamma(A)$. Further, we find
  $\Jac_{[\cdot\,,\cdot]}(a_1,a_2,a_3)=\pr_A(\Jac_{\lb\cdot\,,\cdot\rb}((a_1,0),(a_2,0),(a_3,0)))=-(\pr_A\circ
  l_1\circ l_3 )((a_1,0),(a_2,0),(a_3,0)))=0$ since $l_3$ vanishes on
  sections of $A$.  Hence $A$ is a wide subalgebroid of the split Lie
  $2$-algebroid. In a similar manner, we find a Lie algebroid
  structure on $B$. Next we prove that \eqref{2_rep_A} defines a
  $2$-representation of $A$.
Using (ii) in Definition \ref{split_Lie2} we find for $a\in\Gamma(A)$ and $c\in\Gamma(C)$ that 
\begin{equation*}
\begin{split}
\partial_B(\nabla_ac)&=(\pr_B\circ l_1)(\nabla_{(a,0)}c)\overset{\rm (ii)}{=}\pr_B\lb (a,0),l_1(c)\rb\\
&=\pr_B\lb (a,0),(0,\pr_B(l_1(c)))\rb=\nabla_a(\partial_Bc).
 \end{split}
\end{equation*}
In the third equation we have used Condition (1) and in the last
equation the definitions of $\partial_B$ and
$\nabla_a\colon\Gamma(B)\to\Gamma(B)$.
In the following, we will write for simplicity $a$ for $(a,0)\in\Gamma(A\oplus B)$, etc.
We get easily
\begin{equation*}
\begin{split}
R_{AB}(a_1,a_2)\partial_Bc=l_3(a_1,a_2,\pr_B(l_1(c)) )=l_3(a_1,a_2,l_1(c))\overset{\rm (iv)}{=}R_\nabla(a_1,a_2)c
 \end{split}
\end{equation*}
and 
\[\partial_BR_{AB}(a_1,a_2)b=(\pr_B\circ l_1\circ l_3)(a_1,a_2,b)\overset{\rm
  (iii)}{=}-\pr_B(\Jac_{\lb\cdot\,,\cdot\rb}(a_1,a_2,b))
\]
for all $a_1,a_2\in\Gamma(A)$, $b\in\Gamma(B)$ and $c\in\Gamma(C)$.
By Condition (1) and the definition of $\nabla_a\colon\Gamma(B)\to\Gamma(B)$,
we find \begin{equation*}
\begin{split}
R_\nabla(a_1,a_2)b&=\pr_B\lb a_1,\lb a_2,b\rb\rb-\pr_B\lb a_2\lb a_1, b\rb\rb
-\pr_B\lb\lb a_1,a_2\rb, b\rb\\
&=-\pr_B(\Jac_{\lb\cdot\,,\cdot\rb}(a_1,a_2,b)).
 \end{split}
\end{equation*}
Hence, $\partial_BR_{AB}(a_1,a_2)b=R_\nabla(a_1,a_2)b$. Finally, an easy
computation along the same lines shows that
\begin{equation}\label{dl3_van1}
\langle (\dr_{\nabla^{\operatorname{Hom}}}R_{AB})(a_1,a_2,a_3), b\rangle=(\dr_{\nabla}l_3)(a_1,a_2,a_3,b)
\end{equation}
for $a_1,a_2,a_3\in\Gamma(A)$ and $b\in\Gamma(B)$.  Since
$\dr_{\nabla}l_3=0$, we find
$\dr_{\nabla^{\operatorname{Hom}}}R_{AB}=0$. In a similar manner, we prove
that \eqref{2_rep_B} defines a $2$-representation of $B$.  Further, by
construction of the $2$-representations, the split Lie $2$-algebroid
structure on $(A\oplus B)\oplus C)$ must be defined as in
\eqref{bicross_bracket}, \eqref{nabla_bicross_C} and
\eqref{omega_matched}, with the anchor
$\rho_A\circ\pr_A+\rho_B\circ\pr_B$ and
$l_1=\partial_B\circ\pr_B-\partial_A\circ\pr_A$. Hence, to conclude
the proof, it only remains to check that the split Lie $2$-algebroid
conditions for these objects are equivalent to the seven conditions in
Definition \ref{matched_pair_2_rep} for the two $2$-representations.

First, we find immediately that (M1) is equivalent to (i).
Then we find by construction
\[[a,\partial_Ac]+\nabla_{\partial_Bc}a=-[a,\pr_A(l_1(c))]+\nabla_{\pr_B(l_1(c))}a=\pr_A\lb l_1(c),a \rb=-\pr_A \lb a, l_1(c)\rb.
\]
Hence, we find (M2) if and only if 
$\pr_A \lb a, l_1(c)\rb=\pr_A\circ l_1(\nabla_ac)$. But since 
\begin{equation*}
\begin{split}
\lb a, l_1c\rb&=(\pr_A \lb a, l_1(c)\rb, \nabla_a\pr_Bl_1(c))=(\pr_A \lb a, l_1(c)\rb, \nabla_a\partial_B(c))\\
&=(\pr_A \lb a, l_1(c)\rb,\partial_B\nabla_ac)=(\pr_A \lb a, l_1(c)\rb,\pr_B(l_1(\nabla_ac))),
\end{split}
\end{equation*}
we have $\pr_A \lb a, l_1(c)\rb=\pr_A\circ l_1(\nabla_ac)$ if and only
if $\lb a, l_1c\rb=l_1(\nabla_ac)$. Hence (M2) is satisfied if and
only if $\lb a, l_1c\rb=l_1(\nabla_ac)$ for all $a\in\Gamma(A)$ and
$c\in\Gamma(C)$.  In a similar manner, we find that (M3) is equivalent
to $\lb b, l_1c\rb=l_1(\nabla_bc)$ for all $b\in\Gamma(B)$ and
$c\in\Gamma(C)$. This shows that (M2) and (M3) together are equivalent to (ii).

Next, a simple computation shows that (M4) is equivalent to
$R_\nabla(b,a)c=l_3(b,a,l_1(c))$. Since
$R_\nabla(a,a')c=R_{AB}(a,a')\partial_Bc=l_3(a,a',\pr_B(l_1(c)))=l_3(a,a',l_1(c))$
and $R_\nabla(b,b')c=l_3(b,b',l_1(c))$, we get that (M4) is equivalent
to (iv).

Two straightwforward computations show that (M5) is equivalent to
\[\pr_A(\Jac_{\lb\cdot\,,\cdot\rb}(a_1,a_2,b))=-\pr_A(l_1l_3(a_1,a_2,b))\]
and that (M6) is equivalent to
\[\pr_B(\Jac_{\lb\cdot\,,\cdot\rb}(b_1,b_2,a))=-\pr_A(l_1l_3(b_1,b_2,a)).\]
But since $\pr_B(\Jac_{\lb\cdot\,,\cdot\rb}(a_1,a_2,b))=-R_\nabla(a_1,a_2)b$
by construction and $R_\nabla(a_1,a_2)b=\partial_BR_{AB}(a_1,a_2)b=\pr_B(l_1l_3(a_1,a_2,b))$,
we find \[\pr_B(\Jac_{\lb\cdot\,,\cdot\rb}(a_1,a_2,b))=-\pr_B(l_1l_3(a_1,a_2,b)),\] and in a similar 
manner \[\pr_A(\Jac_{\lb\cdot\,,\cdot\rb}(b_1,b_2,a))=-\pr_A(l_1l_3(b_1,b_2,a)).\]
Since $\Jac_{\lb\cdot\,,\cdot\rb}(a_1,a_2,a_3)=0$, $\Jac_{\lb\cdot\,,\cdot\rb}(b_1,b_2,b_3)=0$, and 
$l_3$ vanishes on sections of $A$, and respectively on sections of $B$,
we conclude that (M5) and (M6) together are equivalent to (iii).

Finally, a slightly longer, but still straightforward computation
shows that
\[(\dr_{\nabla^A}R_{BA})(a_1,a_2)(b_1,b_2)-(\dr_{\nabla^B}R_{AB})(b_1,b_2)(a_1,a_2)=(\dr_\nabla l_3)(a_1,a_2,b_1,b_2)
\]
for all $a_1,a_2\in\Gamma(A)$ and $b_1,b_2\in\Gamma(B)$. This,
\eqref{dl3_van1} and a similar identity for $R_{BA}$, and the
vanishing of $l_3$ on sections of $A$, and respectively on sections of
$B$, show that (M7) is equivalent to (v).
\end{proof}

 If $C=0$, then $R_{AB}=0$ and $R_{BA}=0$ and
the matched pair of 2-representations is just a matched pair of Lie
algebroids.  The double is then concentrated in degree $0$, with
$l_3=0$ and $l_2$ is the bicrossproduct Lie algebroid structure on
$A\oplus B$ with anchor $\rho_A+\rho_B$ \cite{Lu97,Mokri97}. Hence, in
that case the split Lie 2-algebroid is just the bicrossproduct of a
matched pair of representations and the dual (flat) Dorfman connection
is the corresponding Lie derivative. The Lie $2$-algebroid is in that
case a genuine Lie $1$-algebroid.

It would be interesting to define the notion
of matched pair of higher representations up to homotopy, and show
that the induced doubles are split Lie n-algebroids.
This will be studied in a subsequent project.

In the case where $B$ is a trivial Lie algebroid and acts trivially
up to homotopy on $\partial_A=0\colon C\to A$, the
double is the semi-direct product Lie 2-algebroid found in
\cite[Proposition 3.5]{ShZh14} (see our preceding example).

\subsection{Morphisms of (split) Lie 2-algebroids}\label{mor_lie_2_alg}
In this section we quickly discuss morphisms of split Lie
$2$-algebroids.
\begin{definition}
  A morphism $\mu\colon (\mathcal M_1,\mathcal Q_1)\to (\mathcal
  M_2,\mathcal Q_2)$ of Lie 2-algebroids is a morphism $\mu\colon
  \mathcal M_1\to \mathcal M_2$ of the underlying $[2]$-manifolds,
  such that
\begin{equation}\label{morphism_of_Q}
\mu^\star\circ\mathcal Q_2=\mathcal Q_1\circ\mu^\star\colon
C^\infty(\mathcal M_2)\to C^\infty(\mathcal M_1).
\end{equation}
\end{definition}

Assume that the two $[2]$-manifolds $\mathcal M_1$ and $\mathcal M_2$
are split $[2]$-manifold $\mathcal M_1=Q_1[-1]\oplus B_1^*[-2]$ and
$\mathcal M_2=Q_2[-1]\oplus B_2^*[-2]$.  Then the homological vector
fields $\mathcal Q_1$ and $\mathcal Q_2$ can be written in coordinates
as in \eqref{Q_in_coordinates}, with two Dorfman $2$-representations;
$(\partial_{1}\colon Q_1^*\to B_1, \Delta^1, \nabla^1, R_1)$ of
$(Q_1\to M_1,\rho_{1}\colon Q_1\to TM_1)$ and $(\partial_{2}\colon
Q_2^*\to B_2, \Delta^2, \nabla^2, R_2)$ of $(Q_2\to M_2,\rho_{2}\colon
Q_2\to TM_2)$.

Since $\mathcal M_1$ and $\mathcal M_2$ are split, the morphism
$\mu^\star\colon C^\infty(\mathcal M_2)\to C^\infty(\mathcal M_1)$
over $\mu_0^*\colon C^\infty(M_2)\to C^\infty(M_1)$ decomposes as well
as $\mu_Q\colon Q_1\to Q_2$, $\mu_B\colon B_1^*\to B_2^*$ and
$\mu_{12}\in \Omega^2(Q_1,\mu_0^*B_2^*)$ as in \eqref{morphism_dec_1}
and \eqref{morphism_dec_2}.

A study of \eqref{morphism_of_Q} in coordinates, which we leave to the
reader, shows that \eqref{morphism_of_Q} is equivalent to
\begin{enumerate}
\item $\mu_Q\colon Q_1\to Q_2$ over $\mu_0\colon M_1\to M_2$ is compatible 
with the anchors $\rho_1\colon Q_1\to TM_1$ and $\rho_2\colon Q_2\to TM_2$:
\[T_m\mu_0(\rho_1(q_m))=\rho_2(\mu_Q(q_m))
\]
for all $q_m\in Q_1$,
\item $\partial_1\circ\mu_Q^\star=\mu_B^\star\circ\partial_2$ 
as maps from $\Gamma(Q_2^*)$ to $\Gamma(B_1)$, or in other words $\mu_Q\circ\partial_1^*=\partial_2^*\circ\mu_B$.
\item $\mu_Q$ preserves the dull brackets up to $\partial_2^*\mu_{12}$:
i.e.~if $q^1\sim_{\mu_Q}r^1$ and $q^2\sim_{\mu_Q}r^2$, then
\[ \lb q^1,r^1\rb_1\sim_{\mu_Q}\lb q^2,r^2\rb_2-\partial_2^*\mu_{12}(q^1,r^1).
\]
\item $\mu_B$ and $\mu_Q$ intertwines the connections 
$\nabla^1$ and $\nabla^2$ up to $\partial_1\circ\mu_{12}$:
\[\mu_B^\star((\mu_Q^\star\nabla^2)_{q}b)=\nabla^1_{q}(\mu_B^\star(b))-\partial_1\circ\langle\mu_{12}(q,\cdot), b\rangle
\in\Gamma(B_1)
\]
for all $q_m\in Q_1$ and $b\in\Gamma(B^2)$, and 
\item $\mu_Q^\star\omega_{R_2}-\mu_B\circ\omega_{R_1}=-\dr_{(\mu_Q^\star\nabla^2)}\mu_{12}\in\Omega^3(Q_1,\mu_0^*B_2^*)$.
\end{enumerate}

In the equalities above we have used the following constructions.
Recall that $\mu_{12}$ is an element of $\Omega^2(Q_1,\mu_0^*B_2^*)$, 
and  $\omega_{R_i}\in\Omega^3(Q_i,B_i^*)$ for $i=1,2$.
The tensors $\mu_Q^\star\omega_{R_2}\in\Omega^2(Q_1,\mu_0^*B_2^*)$ and 
$\mu_B\circ \omega_{R_1}\in\Omega^2(Q_1,\mu_0^*B_2^*)$
can be defined as follows:
\[(\mu_Q^\star\omega_{R_2})(q_1(m),q_2(m),q_3(m))=\omega_{R_2}(\mu_Q(q_1(m)),\mu_Q(q_2(m)),\mu_Q(q_3(m)))
\]
in $B_2^*(\mu_0(m))$, and 
\[(\mu_B\circ \omega_{R_1})(q_1(m),q_2(m),q_3(m))=\mu_B(\omega_{R_1})(q_1(m),q_2(m),q_3(m)))\in B_2^*(\mu_0(m))
\]
for all $q_1,q_2,q_3\in\Gamma(Q_1)$. The linear connection
\[\mu_Q^\star\nabla^2\colon\Gamma(Q_1)\times\Gamma(\mu_0^*B_2^*)\to\Gamma(\mu_0^*B_2^*)
\]
is defined by
\[ (\mu_Q^\star\nabla^2)_q(\mu_0^!\beta)(m)=\nabla^2_{\mu_Q(q(m))}\beta\in B_2^*(\mu_0(m))
\]
for all $q\in\Gamma(Q_1)$ and $\beta\in\Gamma(B_2^*)$. 

\medskip
\begin{definition}\label{morphism_Lie_2}
  We call a triple $(\mu_Q,\mu_B,\mu_{12})$ over $\mu_0$ satisfying
  the five conditions above a \textbf{morphism of split Lie
    $2$-algebroids}.
\end{definition}
In particular, if $\mathcal M_1=\mathcal M_2$, $\mu_0=\Id_M\colon M\to
M$, $\mu_Q=\Id_Q\colon Q\to Q$ and $\mu_B=\Id_{B^*}\colon B^*\to B^*$,
then $\mu_{12}\in\Gamma(Q^*\wedge Q^*\otimes B^*)$ is just a change
of splitting The five conditions above simplify to
\begin{enumerate}
\item The dull brackets are related by $\lb q, q'\rb_{2}=\lb q,
  q'\rb_1 +\partial_B^*\mu_{12}(q,q')$.
% \item The symmetric pairings with values in $B^*$
% are related by 
% $\Lambda^2   (q,q')=\Lambda^1(q,q')-\Phi_{12}(q)^*(q')
% -\Phi_{12}(q')^*(q)$.   
\item The connections are related by
  $\nabla^2_q=\nabla^1_q-\partial_B\circ\mu_{12}(q)$.
\item The curvature terms are related by 
$\omega_{R^2}-\omega_{R^1}=-\dr_{{\nabla^2}^*}\mu_{12}$,
where the Cartan differential $\dr_{{\nabla^2}^*}$ is defined with the
dull bracket $\lb\cdot\,,\cdot\rb_1$ on $\Gamma(Q)$.
\end{enumerate}

\section{VB-Courant algebroids and Lie 2-algebroids}\label{sec:VB_cour}
In this section we describe and prove in detail the equivalence
between VB-Courant algebroids and Lie 2-algebroids. In short, a
homological vector field on a $[2]$-manifold defines an anchor and a
Courant bracket on the corresponding metric double vector bundle. This
Courant bracket and this anchor are automatically compatible with the
metric and define so a linear Courant algebroid structure on the
double vector bundle.  Be aware that a correspondence of Lie
2-algebroids and VB-Courant algebroids has already been discussed by
Li-Bland \cite{Li-Bland12}.  Our goal is to make this result more
constructive, to deduce it from the two preceding sections, and to
illustrate it with several (partly new) examples.
%  Our goal here is
% to make this equivalence more concrete, i.e.~in terms of sections of
% the VB-algebroids and functions on the corresponding Q-[2]-manifolds.

\subsection{Definition and observations}
We will work with the following definition of a VB-Courant
algebroid, which is basically due to Li-Bland \cite{Li-Bland12}.
\begin{definition}
A VB-Courant algebroid is a metric double vector bundle 
\[\begin{xy}
\xymatrix{\mathbb{E}\ar[r]^{\pi_B}\ar[d]_{\pi_Q}&B\ar[d]^{q_B}\\
Q\ar[r]_{q_Q}&M
}
\end{xy}
\]
with core $Q^*$ such that $\mathbb E\to B$ is a Courant algebroid and
the following conditions are satisfied.
\begin{enumerate}
\item The anchor map $\Theta\colon \mathbb E\to TB$ is linear. That is, 
\begin{equation}
\begin{xy}
  \xymatrix{\mathbb{E}\ar[rr]^{\pi_B}\ar[dd]_{\pi_Q}&&B\ar[dd]^{q_B}&
    & &TB\ar[rr]^{p_B}\ar[dd]_{Tq_B}&&B\ar[dd]^{q_B}\\
    &C\ar[dr]&   &\ar[r]^\Theta& &&B\ar[dr]&\\
    Q\ar[rr]_{q_Q}&&M&&&TM\ar[rr]_{p_M}&&M }
\end{xy}
\end{equation}
is a morphism of double vector bundles.
\item The Courant bracket is linear. That is 
\begin{align*}
  \left\lb\Gamma^l_B(\mathbb{E}),\Gamma^l_B(\mathbb{E})\right\rb&\subseteq\Gamma^l_B(\mathbb{E}),\qquad 
  \left\lb\Gamma^l_B(\mathbb{E}),\Gamma^c_B(\mathbb{E})\right\rb\subseteq\Gamma^c_B(\mathbb{E}),\\
 & \left\lb\Gamma^c_B(\mathbb{E}),\Gamma^c_B(\mathbb{E})\right\rb=0.
\end{align*}
\end{enumerate}
\end{definition}

We make the following observations. Let $\rho_Q\colon Q\to TM$ be the
side map of the anchor, i.e.~if $\pi_Q(\chi)=q$ for $\chi\in \mathbb
E$, then $Tq_B(\Theta(\chi))=\rho_Q(q)$.  Let $\partial_B\colon Q^*\to
B$ be the core map defined as follows by the anchor $\Theta$:
\begin{equation}\label{anchor_on_pullbacks}
\Theta(\sigma^\dagger)=(\partial_B\sigma)^\uparrow
\end{equation}
for all $\sigma\in\Gamma(Q^*)$ ($\partial_B$ is sometimes called the
``core-anchor'').  Then the operator $\mathcal D=\Theta^*\dr\colon
C^\infty(B)\to \Gamma_B(\mathbb E)$ satisfies
$\mathcal D(q_B^*f)=(\rho_Q^*\dr f)^\dagger$
for all $f\in C^\infty(M)$ and \eqref{anchor_with_D_2} yields
immediately 
\begin{equation}\label{rho_delta}
\partial_B\circ \rho_Q^*=0,\quad  \text{ which is equivalent to } \quad \rho_Q\circ \partial_B^*=0.
\end{equation}
Recall finally that if $\chi \in\Gamma^l_B(\mathbb E) $ is linear over
$q\in\Gamma(Q)$, then $\langle \chi, \tau^\dagger\rangle=q_B^*\langle
q, \tau\rangle$ for all $\tau\in\Gamma(Q^*)$ and $\rho(\chi)$ is
linear over $\rho_Q(q)$.

\subsection{The fat Courant algebroid}
Recall from \eqref{fat_seq} that there exists a vector bundle
$\widehat{\mathbb E}\to M$, which sheaf of sections is the sheaf of
$C^\infty(M)$-modules $\Gamma^l_B(\mathbb E)$, the linear sections of
$\mathbb E$ over $B$.  Gracia-Saz and Mehta show in \cite{GrMe10a}
that if $\mathbb E$ is endowed with a linear Lie algebroid structure
over $B$, then $\widehat{\mathbb E}\to M$ inherits a Lie algebroid
structure, which is called the ``fat Lie algebroid''. For
completeness, we describe here quickly the counterpart of this in the
case of a linear Courant algebroid structure on $\mathbb E\to B$.

Note that the restriction of the pairing on $\mathbb E$ to linear
sections of $\mathbb E$ defines a nondegenerate pairing on
$\widehat{\mathbb E}$ with values in $B^*$.  Since the Courant bracket
of linear sections is again linear, we get the following theorem.
\begin{theorem}\label{fat}
The vector bundle $\widehat{\mathbb E}$
inherits a Courant algebroid structure with pairing in $B^*$. 
\end{theorem}
We will come back later to this structure.  Recall that for
$\phi\in\Gamma(\operatorname{Hom}(B,Q^*))$, the core-linear section
$\widetilde{\phi}$ of $\mathbb E\to B$ is defined by
\[\widetilde{\phi}(b_m)=0_{b_m}+_B\overline{\phi(b_m)}.
\]
Note that $\widehat{\mathbb E}$ is also naturally paired with $Q^*$:
$\langle \chi(m), \sigma(m)\rangle=\langle \pi_Q(\chi(m)),
\sigma(m)\rangle$ for all $\chi\in\Gamma^l_B(\mathbb
E)=\Gamma(\widehat{\mathbb E})$ and $\sigma\in\Gamma(Q^*)$.  This
pairing is degenerate since it restricts to $0$ on
$\operatorname{Hom}(B,Q^*)\times_MQ^*$.  The following proposition can
easily be proved.
\begin{proposition}
\begin{enumerate}
\item The map 
\[\Delta\colon \Gamma(\widehat{\mathbb E})\times \Gamma(Q^*)\to
\Gamma(Q^*)\quad 
\text{defined by } \quad \left(\Delta_\chi \tau\right)^\dagger=\lb \chi,
\tau^\dagger\rb
\]
is a flat Dorfman connection, where $\widehat{\mathbb E}$ is endowed
with the anchor $\rho_Q\circ \pi_Q$ and paired with $Q^*$ as above.
\item The map
\[\nabla^*\colon \Gamma(\widehat{\mathbb E})\times \Gamma(B^*)\to
\Gamma(B^*)\quad 
\text{defined by } \quad
\ell_{\nabla^*_\chi\beta}=\Theta(\chi)(\ell_\beta)\]
for all $\beta\in\Gamma(B^*)$ is a flat connection.
\end{enumerate}
\end{proposition}

We dualise the connection $\nabla^*$ to a flat connection
$\nabla\colon\Gamma(\widehat{\mathbb E})\times \Gamma(B)\to \Gamma(B)$.
\begin{proposition}
The following hold for $\Delta$ and $\nabla$:
\begin{enumerate}
\item $\partial_B\circ \Delta=\nabla\circ \partial_B$ and
\item $\left\lb \chi, \widetilde{\phi}\right\rb_{\widehat{\mathbb E}}=\widetilde{\phi\circ
  \nabla_\chi-\Delta_\chi\circ \phi}$
\end{enumerate}
for $\chi\in\Gamma(\widehat{\mathbb E})$ and
$\phi\in\Gamma(\operatorname{Hom}(B,Q^*))$.
\end{proposition}

\begin{proof}
\begin{enumerate}
\item Choose $\chi\in\Gamma^l_B(\mathbb E)$ and $\tau\in\Gamma(Q^*)$. Then 
\begin{equation*}
  (\partial_B\circ\Delta_\delta\tau)^\uparrow=
\Theta(\Delta_\chi\tau^\dagger)=\Theta\left(\left\lb \chi,\tau^\dagger\right\rb\right)
  =\left[\Theta(\chi), (\partial_B\tau)^\uparrow\right]=(\nabla_\chi(\partial_B\tau))^\uparrow.
\end{equation*}
\item The second equation is easy to check by writing $\widetilde{\phi}=\sum_{i=1}^n\ell_{\beta_i}\cdot\tau_i^\dagger$ 
with $\beta_i\in\Gamma(B^*)$ and $\tau_i\in\Gamma(Q^*)$.
\end{enumerate}
\end{proof}

\begin{lemma}\label{formulas}
  For $\phi,\psi\in\Gamma(\operatorname{Hom}(B,Q^*))$ and
  $\tau\in\Gamma(Q^*)$, we have
\begin{enumerate}
\item $\left\lb \tau^\dagger, \widetilde{\phi}\right\rb
  =(\phi(\partial_B \tau))^\dagger= -\left\lb \widetilde{\phi},
    \tau^\dagger\right\rb$ and
\item $\left\lb \widetilde{\phi},
    \widetilde{\psi}\right\rb=\widetilde{\psi\circ\partial_B\circ\phi-\phi\circ\partial_B\circ\psi}
  $.
\end{enumerate}
\end{lemma}

\begin{remark}
  Note that by the second equality, $\operatorname{Hom}(B,Q^*)$ has
  the structure of a Lie algebra bundle.
\end{remark}

\begin{proof} We write $\phi=\sum_{i=1}^n\beta_i\cdot \tau_i$ and
  $\psi=\sum_{j=1}^n\beta'_j\cdot\tau_j$ with
  $\beta_1,\ldots,\beta_n,\beta'_1,\ldots,\beta'_n\in\Gamma(B^*)$ and
  $\tau_1,\ldots,\tau_n\in\Gamma(Q^*)$.  Hence, we have
  $\widetilde{\phi}=\sum_{i=1}^n\ell_{\beta_i}\tau_i^\dagger$ and
  $\widetilde{\psi}=\sum_{j=1}^n\ell_{\beta'_j}\tau_j^\dagger$.  First we
  compute \begin{equation*}
\begin{split}
\left\lb \tau^\dagger,
    \sum_{i=1}^n\ell_{\beta_i}\tau_i^\dagger\right\rb
  &=\sum_{i=1}^n(\partial_B\tau)^\uparrow(\ell_{\beta_i})\tau_i^\dagger\\
  &=\sum_{i=1}^nq_B^*\langle \partial_B\tau, \beta_i\rangle
  \tau_i^\dagger=\left(\sum_{i=1}^n \langle \partial_B\tau,
    \beta_i\rangle \tau_i\right)^\dagger
\end{split}
\end{equation*}
and we get (1).  Since $\langle
\tau^\dagger,\widetilde{\phi}\rangle=0$, the second equality follows.

Then we have 
\begin{align*}
  \left\lb \sum_{i=1}^n\ell_{\beta_i}\tau_i^\dagger,
    \sum_{j=1}^n\ell_{\beta'_j}\tau_j^\dagger\right\rb
  &=\sum_{i=1}^n\sum_{j=1}^n\ell_{\beta_i}(\partial_B\tau_i)^\uparrow(\ell_{\beta'_j})\tau_j^\dagger-\ell_{\beta'_j}(\partial_B\tau_j)^\uparrow(\ell_{\beta_i})\tau_i^\dagger\\
  % &=\sum_{i=1}^n\sum_{j=1}^n\ell_{\beta_i}q_B^*\langle \partial_B\tau_i,
  % \beta'_j\rangle\tau_j^\dagger-\ell_{\beta'_j}q_B^*\langle\partial_B\tau_j,\beta_i\rangle\tau_i^\dagger\\
  &=\left(\sum_{i=1}^n\sum_{j=1}^n\langle \partial_B\tau_i,
    \beta'_j\rangle\cdot\beta_i\cdot\tau_j-\langle\partial_B\tau_j,\beta_i\rangle\cdot\beta'_j\cdot\tau_i
  \right)^\dagger.
\end{align*}
Thus, we get (2).

\end{proof}

\subsection{Dorfman 2-representations and Lagrangian decompositions of VB-algebroids}
In this section, we study in detail the structure of a VB-Courant
algebroid, using Lagrangian decompositions of the underlying metric double
vector bundle.  Our goal is the following theorem. Note the similarity
of this result with Theorem \ref{rajan} in the VB-algebroid case.
\begin{theorem}\label{main}
  Let $(\mathbb E; Q,B;M)$ be a VB-Courant algebroid and choose a
  Lagrangian splitting $\Sigma\colon Q\times_MB\to \mathbb E$. Then there exists a Dorfman
 2-representation $(\Delta,\nabla, R)$ of
  $(Q,\rho_Q)$ on the core-anchor $\partial_B\colon Q^*\to B$ such that
\begin{equation}\label{VB_def1}
\begin{split}
  \Theta(\sigma_Q(q))&=\widehat{\nabla_q}\in\mx(B), \\
  \lb \sigma_Q(q_1), \sigma_Q(q_2)\rb &=\sigma_Q(\lb q_1,
    q_2\rb_\Delta)-\widetilde{R(q_1,q_2)},\\
  \lb \sigma_Q(q), \tau^\dagger\rb&=(\Delta_q\tau)^\dagger
\end{split}
\end{equation}
for all $q,q_1,q_2\in\Gamma(Q)$ and $\tau\in\Gamma(Q^*)$, where
$\lb\cdot\,,\cdot\rb_\Delta\colon\Gamma(Q)\times\Gamma(Q)\to\Gamma(Q)$ is
the dull bracket that is dual to $\Delta $.

Conversely, a Lagrangian splitting $\Sigma\colon Q\times B^*\to
\mathbb E$ of the metric double vector bundle $\mathbb E$ together
with a Dorfman 2-representation define a linear Courant algebroid
structure on $\mathbb E$ by \eqref{VB_def1}.
\end{theorem}

% Note that given a Lagrangian splitting
% $\sigma_Q\colon\Gamma(Q)\to\Gamma_B^l(\mathsf E)$, we have immediately the
% following identities:
% \begin{equation*}
% \begin{split}
%   \Theta(\sigma^\dagger)&=(\partial_B\sigma)^\dagger\in\mx(B),\quad
%   \langle \sigma_Q(q_1), \sigma_Q(q_2)\rangle =0, \quad \langle \sigma_Q(q),
%   \sigma^\dagger\rangle=q_B^*\langle q, \sigma\rangle \quad \text{
%     and } \langle \sigma_1^\dagger, \sigma_2^\dagger\rangle=0
% \end{split}
% \end{equation*}
% for $\sigma, \sigma_1,\sigma_2\in\Gamma(Q^*)$ and
% $q,q_1,q_2\in\Gamma(Q)$.

First we will construct the objects $\lb\cdot\,,\cdot\rb_\Delta,
\Delta, \nabla,R$ as in the theorem, and then we will prove in the appendix that they
satisfy the axioms of a Dorfman 2-representation.

\subsubsection{Construction of the Dorfman 2-representation and outline of the proof}\label{construction_of_objects}
The objects $\lb\cdot\,,\cdot\rb_\Delta$, $\nabla$, $\Delta$ and $R$
are defined as in the theorem. Let us be more precise.

First recall that, by definition, the Courant bracket of two linear
sections of $\mathbb E\to B$ is again linear. Hence, we can denote by
$\lb q_1, q_2\rb_\sigma$ the section of $Q$ such that
\begin{equation}
  \pi_Q\circ \lb \sigma_Q(q_1), \sigma_Q(q_2)\rb =\lb q_1, q_2\rb_\sigma\circ
  q_B.
\end{equation}

Since for each $q\in\Gamma(Q)$, the anchor $\Theta(\sigma_Q(q))$ is a linear vector field on
  $B$ over $\rho_Q(q) \in\mx(M)$, there
  exists a derivation $D_q\colon\Gamma(B^*)\to\Gamma(B^*)$ over
  $\rho_Q(q) $ such that $\Theta(\sigma_Q(q))(\ell_\beta)=\ell_{D_q\beta}$ for
  all $\beta\in\Gamma(B^*)$, and $\Theta(\sigma_Q(q))(q_B^*f)=q_B^*(\rho_Q(q)(f))$ for all $f\in
  C^\infty(M)$.  This defines a linear $Q$-connection
  $\nabla\colon\Gamma(Q)\times\Gamma(B)\to\Gamma(B)$:
\[ \nabla_qb=D_q^*b
\]
for all $b\in\Gamma(B)$. Then by definition, $\Theta(\sigma_Q(q))=\widehat{\nabla_q}\in \mx^l(B)$.

 For $q\in\Gamma(Q)$ and $\tau\in\Gamma(Q^*)$, the bracket
  $\left\lb \sigma_Q(q), \tau^\dagger\right\rb$ is a core section. It
  is easy to check that the map
  $\Delta\colon\Gamma(Q)\times\Gamma(Q^*)\to\Gamma(Q^*)$ defined by
  \[ \left\lb \sigma_Q(q),
    \tau^\dagger\right\rb=(\Delta_q\tau)^\dagger
\]
is a Dorfman connection.\footnote{ Note that Condition $(C3)$ then implies that 
$\lb \tau^\dagger, \sigma_Q(q)\rb=\left(-\Delta_q\tau+\rho_Q^*\dr\langle \tau,
  q\rangle\right)^\dagger$.}

The difference of the two linear sections $\lb \sigma_Q(q_1),
  \sigma_Q(q_2)\rb-\sigma_Q(\lb q_1, q_2\rb_\sigma)$ is again a linear section, which projects to $0$ under $\pi_Q$. 
Hence, there
  exists a vector bundle morphism $R(q_1,q_2)\colon B\to Q^*$ such that
  $\sigma_Q(\lb q_1, q_2\rb_\sigma)-\lb \sigma_Q(q_1), \sigma_Q(q_2)\rb=\widetilde{R(q_1,q_2)}$.
  This defines a map
  $R\colon\Gamma(Q)\times\Gamma(Q)\to\Gamma(\operatorname{Hom}(B,Q^*))$.
We show in the appendix that $(\partial_B, \nabla, \Delta, R)$ is a
Dorfman 2-representation, and that $\lb\cdot\,,\cdot\rb_\Delta=\lb\cdot\,,\cdot\rb_\sigma$.

Conversely, choose a Lagrangian splitting $\Sigma\colon Q\times_MB$ of
a metric double vector bundle $(\mathbb E, Q; B, M)$ with core $Q^*$
and let $\mathcal S\subseteq\Gamma_B(\mathbb E)$ be the subset
$\{\tau^\dagger\mid \tau\in\Gamma(Q^*)\}\cup \{\sigma_Q(q)\mid
q\in\Gamma(Q)\}\subseteq \Gamma(\mathbb E)$. Choose a Dorfman
2-representation $(\partial_B\colon Q^*\to B, \nabla, \Delta, R)$ of $(Q,\rho_Q)$. 
Define then a vector bundle map $\Theta\colon\mathbb E\to TB$ over the
identity on $B$ by $\Theta(\sigma_Q(q))=\widehat{\nabla_q}$ and
$\Theta(\tau^\dagger)=(\partial_B\tau)^\dagger$ and a bracket
$\lb\cdot\,,\cdot\rb$ on $\mathcal S$ by
\begin{equation*}
\begin{split}
  \lb \sigma_Q(q_1), \sigma_Q(q_2)\rb=\sigma_Q(\lb q_1,
    q_2\rb_\Delta)-\widetilde{R(q_1,q_2)}, \quad \lb \sigma_Q( q),
  \tau^\dagger\rb=(\Delta_q\tau)^\dagger, \\\quad \lb
  \tau^\dagger, \sigma_Q(q)\rb=\left(-\Delta_q\tau+\rho_Q^*\dr\langle
    \tau, q\rangle\right)^\dagger, \quad \lb \tau_1^\dagger,
  \tau_2^\dagger\rb=0.
\end{split}
\end{equation*}
We show in the appendix that this bracket, the pairing and the anchor satisfy the
conditions of Lemma \ref{useful_lemma}, and so that $(\mathbb E, B; Q, M)$ with this structure 
is a VB-Courant algebroid.

\subsubsection{Change of Lagrangian decomposition}
Next we study how the Dorfman 2-representation $(\partial_B\colon
Q^*\to B, \nabla, \Delta, R, \Lambda)$ associated to a Lagrangian
decomposition of a VB-Courant algebroid changes when one changes the
Lagrangian decomposition.

The proof of the following proposition is straightforward and left to
the reader. Compare this result with the equations subsequent to
Definition \ref{morphism_Lie_2}, that describe a change of splittings
of split Lie $2$-algebroid.
 
\begin{proposition}\label{change_of_lift}
  Let $\Sigma^1,\Sigma^2\colon B\times_MQ\to\mathbb E$ be two
  Lagrangian splittings and let $\phi_{12}\in\Gamma(Q^*\otimes Q^*\otimes
  B^*)$ be the change of lift.
\begin{enumerate}
\item The Dorfman connections are related by
$\Delta_q^2\sigma=\Delta_q^1\sigma+\phi_{12}(q)(\partial_B\sigma)$
\item and the dull brackets consequently by $\lb q, q'\rb_{2}=\lb q,
  q'\rb_1 -\partial_B^*\phi_{12}(q)^*(q')$.
% \item The symmetric pairings with values in $B^*$
% are related by 
% $\Lambda^2   (q,q')=\Lambda^1(q,q')-\Phi_{12}(q)^*(q')
% -\Phi_{12}(q')^*(q)$.   
\item The connections are related by
  $\nabla^2_q=\nabla^1_q+\partial_B\circ\phi_{12}(q)$.
\item The curvature terms are related by 
$\omega_{R^2}-\omega_{R^1}=\dr_{{\nabla^2}^*}\phi_{12}$,
where the Cartan differential $\dr_{{\nabla^2}^*}$ is defined with the
dull bracket $\lb\cdot\,,\cdot\rb_1$ on $\Gamma(Q)$.
\end{enumerate}
\end{proposition}

% In particular, if
% $h_2(q)=h_1(q)+\frac{1}{2}\Lambda^2(q,\cdot)^\uparrow$ for all
% $q\in\Gamma(Q)$, then $\Lambda^2=0$, $\lb\cdot\,,\cdot\rb_2$ is the
% skew-symmetrization of $\lb\cdot\,,\cdot\rb_1$ and the third formula
% provides a formula for the skew-symmetrization of
% $\omega_{R^1}\colon\Gamma(Q)^3\to\Gamma(B^*)$.
As an application, we get the following corollary
of Theorem \ref{main} and Theorem \ref{fat}.

\begin{corollary}
  Let $(Q\oplus B^*\to M, \rho_Q,
  l_1, l_2, l_3)$ be a split Lie 2-algebroid and
  $(\partial_B=-l_1^*,\Delta,\nabla,R)$ the dual Dorfman
  2-representation. Then the vector bundle $\mathsf
  E:=Q\oplus\operatorname{Hom}(B,Q^*)$ is a
  Courant algebroid with pairing in $B^*$ given by $\langle
  (q_1,\phi_1), (q_2,\phi_2)\rangle=\phi_1^*(q_2)+\phi_2^*(q_1)$, with
  the anchor $\tilde\rho\colon \mathsf E\to
  \widehat{\operatorname{Der}(B)}$,
  $\tilde\rho(q,\phi)=\nabla_q^*+\phi^*\circ l_1$ over $\rho(q)$ and
  the bracket given by
\begin{equation*}
\begin{split}
  \lb (q_1,\phi_1), (q_2, \phi_2)\rb=\Bigl(\lb q_1,
  q_2\rb_\Delta+l_1^*(\phi_1^*(q_2)),
  &\lozenge_{q_1}\phi_2-\lozenge_{q_2}\phi_1+\nabla_\cdot^*(\phi_1^*(q_2))
  \\
  +&\phi_2\circ l_1^*\circ\phi_1-\phi_1\circ
  l_1^*\circ\phi_2+R(q_1,q_2)\Bigr).
\end{split}
\end{equation*}
The map $\mathcal D\colon \Gamma(B^*)\to \Gamma(\mathsf E)$ sends
$q $ to $(l_1(q),\nabla^*_\cdot q)$. The bracket does not depend on the choice of splitting.
\end{corollary}

\subsection{Examples of VB-algebroids and the corresponding Dorfman 2-re\-pre\-sen\-ta\-tions}
We give here some examples of VB-Courant algebroids, and we compute
the corresponding classes of Dorfman $2$-representations. We find the
Dorfman $2$-representations described in Section
\ref{examples_split_lie_2}. In each of the examples below, it is easy
to check that the Courant algebroid structure is linear. Hence, it is
easy to check geometrically that the objects described in
\ref{examples_split_lie_2} are indeed split Lie $2$-algebroids. This
is why we omitted the detailed computations in that section.

\subsubsection{The standard Courant algebroid over a vector bundle}\label{standard_VB_Courant_ex}
We have discussed this example in great detail in \cite{Jotz13a}, but
not in the language of Dorfman 2-representations and Lie 2-algebroids.
In \cite{Jotz13a}, we worked with general, not necessarily Lagrangian,
linear splittings.

Let $q_E\colon E\to M$ be a vector bundle and consider the VB-Courant
algebroid
\begin{equation*}
\begin{xy}
  \xymatrix{
    TE\oplus T^*E\ar[rr]^{\Phi_E:=({q_E}_*, r_E)}\ar[d]_{\pi_E}&& TM\oplus E^*\ar[d]\\
    E\ar[rr]_{q_E}&&M }
\end{xy}
\end{equation*} with base $E$
and side $TM\oplus E^*\to M$, and 
with core $E\oplus T^*M\to M$, or in other words the standard (VB-)Courant
algebroid over a vector bundle $q_E\colon E\to M$. Recall that $TE\oplus T^*E$ was found in Example \ref{met_TET*E}
to have a linear metric. Recall also from this example that linear splittings of
$TE\oplus_ET^*E$ are in bijection with dull brackets on sections of
$TM\oplus E^*$, and so also with Dorfman connections $\Delta\colon
\Gamma(TM\oplus E^*)\times\Gamma(E\oplus T^*M)\to\Gamma(E\oplus
T^*M)$, and that Lagrangian splittings of $TE\oplus_ET^*E$ are in bijection with
skew-symmetric  dull brackets on sections of
$TM\oplus E^*$.

The anchor $\Theta=\pr_{TE}\colon TE\oplus T^*E\to TE$ restricts to
the map $\partial_{E}=\pr_{E}\colon E\oplus T^*M\to E$ on the cores,
and defines an anchor $\rho_{TM\oplus E^*}=\pr_{TM}\colon TM\oplus
E^*\to TM$ on the side.  In other words, the anchor of
$(e,\theta)^\dagger$ is $e^\uparrow\in \mx^c(E)$ and if $\widetilde{(X,\epsilon)}$
is a linear section of $TE\oplus T^*E\to E$ over
$(X,\epsilon)\in\Gamma(TM\oplus E^*)$, the anchor
$\Theta(\widetilde{(X,\epsilon)})$ is linear over $X$. %  The vector
% bundle $TM\oplus E^*$ is hence anchored by $\pr_{TM}\colonTM\oplus E^*\to
% TM$. 

% Recall from Example \ref{met_TET*E} that the linear splitting is
% Lagrangian if and only if $\lb\cdot\,,\cdot\rb_\Delta$ is
% skew-symmetric. Further, l
Let $\iota_E\colon E\to E\oplus T^*M$ be the canonical inclusion. In
\cite{Jotz13a} we prove the following result (for general linear
splittings).
\begin{theorem}\label{Lagrangin}
  Choose $q, q_1,q_2\in\Gamma(TM\oplus E^*)$ and $\tau,\tau_1,
  \tau_2\in\Gamma(E\oplus T^*M)$.  
The Courant-Dorfman bracket on sections of $TE\oplus T^*E\to E$ is
given by
\begin{enumerate}\setcounter{enumi}{2}
% \item  $\left\lb\tau_1^\dagger, \tau_2^\dagger\right\rb=0$,
\item $\left\lb \sigma(q),
    \tau^\dagger\right\rb=\left(\Delta_{q}\tau\right)^\dagger$,
 \item $ \left\lb \sigma(q_1), \sigma(q_2)\right\rb =\sigma(\lb q_1,
   q_2\rb_\Delta)-\widetilde{R_\Delta(q_1, q_2)\circ \iota_E}$.
\end{enumerate}
The anchor $\rho$ is described by
\begin{enumerate}\setcounter{enumi}{4}
\item $\Theta(\sigma(q))=\widehat{\nabla_q^*}\in\mx(E)$,
\end{enumerate}
where $\nabla\colon \Gamma(TM\oplus E^*)\times\Gamma(E)\to \Gamma(E)$
is defined by $\nabla_q=\pr_E\circ\Delta_q\circ\iota_E$ for all $q\in\Gamma(TM\oplus E^*)$.
\end{theorem}
Hence, if we choose a Lagrangian splitting of $TE\oplus_ET^*E$, we
find the Dorfman 2-representation of Example \ref{standard}. 

\subsubsection{VB-Courant algebroid defined by a VB-Lie algebroid}\label{VB_Courant_alg}
More generally, let
\begin{equation*}
\begin{xy}
  \xymatrix{
    D\ar[r]^{\pi_B}\ar[d]_{\pi_A}& B\ar[d]^{q_B}\\
    A\ar[r]_{q_A}&M }
\end{xy}
\end{equation*}
with core $C$, be endowed with a VB-Lie algebroid structure $(D\to B, A\to
M)$. Then the pair $(D,D\duer B)$ of vector bundles over $B$ is a Lie
bialgebroid, with $D\duer B$ endowed with the trivial Lie algebroid
structure.  We get a linear Courant algebroid $D\oplus_B (D\duer B)$
over $B$ with side $A\oplus C^*$
\begin{equation*}
\begin{xy}
  \xymatrix{
   D\oplus_B (D\duer B)\ar[r]\ar[d]&B\ar[d]\\
   A\oplus C^* \ar[r]&M }
\end{xy}
\end{equation*}
and core $C\oplus A^*$.  We check that the Courant algebroid structure
is linear. Let $\Sigma\colon A\times_M B\to D$ be a linear splitting
of $D$. Recall from Lemma \ref{lemma_dual_splitting} that we can
define a linear splitting of $D\duer B$ by $\Sigma^\star\colon
B\times_M C^*\to D\duer B$, $\langle\Sigma^\star(b_m,\gamma_m),
\Sigma(a_m,b_m)\rangle=0$ and $\langle\Sigma^\star(b_m,\gamma_m),
c^\dagger(b_m)\rangle=\langle \gamma_m,c(m)\rangle$ for all $b_m\in
B$, $a_m\in A$, $\gamma_m\in C^*$ and $c\in\Gamma(C)$. The linear
splitting $\tilde\Sigma\colon B\times_M(A\oplus C^*)\to D\oplus_B
(D\duer B)$,
$\tilde\Sigma(b_m,(a_m,\gamma_m))=(\Sigma(a_m,b_m),\Sigma^\perp(b_m,\gamma_m))$
as in \S\ref{metric_double_VB_alg} is then a Lagrangian splitting by
Lemma \ref{lemma_dual_splitting}.  A computation shows that the
Courant bracket on $\Gamma_B(D\oplus_B(D\duer B))$ is given by
\begin{equation*}
\begin{split}
  &\left\lb \tilde\sigma_{A\oplus C^*}(a_1,\gamma_1),
    \tilde\sigma_{A\oplus
      C^*}(a_2,\gamma_2)\right\rb\\
&=([\sigma_A(a_1),\sigma_A(a_2)],
  \ldr{\sigma_A(a_1)}\sigma^\star_{C^*}(\gamma_2)-\ip{\sigma_A(a_2)}\dr\sigma^\star_{C^*}(\gamma_1))\\
  &=\left(\sigma_A([a_1,a_2])-\widetilde{R(a_1,a_2)},
    \sigma^\star_{C^*}(\nabla_{a_1}^*\gamma_2-\nabla_{a_2}^*\gamma_1)-\widetilde{\langle\gamma_2,
      R(a_1, \cdot)\rangle}
    +\widetilde{\langle\gamma_1, R(a_2, \cdot)\rangle}\right)\\
  &\left\lb \tilde\sigma_{A\oplus C^*}(a,\gamma), (c,\alpha)^\dagger\right\rb
=\left(\nabla_ac^\dagger,(\ldr{a}\alpha +\langle\nabla_\cdot^*\gamma,c\rangle)^\dagger\right)\\
  &\left\lb (c_1,\alpha_1)^\dagger, (c_2,\alpha_2)^\dagger\right\rb=0,
\end{split}
\end{equation*}
and the anchor of $D\oplus_B(D\duer B)$ is defined by
\begin{equation*}
\begin{split}
\Theta(\tilde\sigma_{A\oplus C^*}(a,\gamma))=\Theta(\sigma_{A}(a))=\widehat{\nabla_a}\in \mx^l(B), \quad 
\Theta((c,\alpha)^\dagger)=(\partial_Bc)^\uparrow\in \mx^c(B),
\end{split}
\end{equation*}
where $(\partial_B\colon C\to B,
\nabla\colon\Gamma(A)\times\Gamma(B)\to\Gamma(B),\nabla\colon\Gamma(A)\times\Gamma(C)\to\Gamma(C),R)$
is the 2-representation of $A$ associated to the splitting
$\Sigma\colon A\times_MB\to D$ of the VB-algebroid $(D\to B, A\to M)$.
Hence, we have found the split Lie 2-algebroid described in Example \ref{semi-direct}.

\subsubsection{The tangent Courant algebroid}\label{tangent_Courant}
We consider here a Courant algebroid\linebreak $(\mathsf
E,\rho_{\mathsf E},\lb\cdot\,,\cdot\rb,\langle \cdot\,,\cdot\rangle)$.
In this example, $\mathsf E$ will always be anchored by the Courant
algebroid anchor map $\rho_{\mathsf E}$ and paired with itself by
$\langle \cdot\,,\cdot\rangle$ and $\mathcal
D=\Beta\inv\circ\rho_{\mathsf E}^*\circ\dr\colon C^\infty(M)\to
\Gamma(\mathsf E)$.  Note that $\lb\cdot\,,\cdot\rb$ is not a dull bracket.

We show that, after the choice of a metric connection on $\mathsf E$
and so of a Lagrangian splitting $\Sigma^\nabla\colon
TM\times_M\mathsf E\to T\mathsf E$ (see Example
\ref{metric_connections} and \S \ref{tangent_double}), the VB-Courant
algebroid structure on $(T\mathsf E\to TM, \mathsf E\to M)$ is
equivalent to the Dorfman $2$-representation defined by $\nabla$ as in
Example \ref{adjoint}.

\begin{theorem}\label{tangent_courant_double}
  Choose a linear connection $\nabla\colon\mx(M)\times\Gamma(\mathsf
  E)\to\Gamma(\mathsf E)$ that preserves the pairing on $\mathsf E$.
  The Courant algebroid structure on $T\mathsf E\to TM$ can be
  described as follows:
\begin{enumerate}
\item The pairing is given by 
\[\left\langle e_1^\dagger, e_2^\dagger\right\rangle=0, 
\quad \left\langle \sigma_{\mathsf E}^\nabla(e_1), e_2^\dagger\right\rangle=p_M^*\langle e_1,
e_2\rangle, \quad \text{ and } \left\langle \sigma_{\mathsf E}^\nabla (e_1),
  \sigma_{\mathsf E}^\nabla (e_2)\right\rangle=0,\]
\item the anchor is given by \[\Theta(\sigma_{\mathsf E}^\nabla (e))=\widehat{\nabla^{\rm bas}_e\cdot}
  \text{ and } \Theta(e^\dagger)=(\rho_{\mathsf E}(e))^\uparrow,\]
\item and the bracket is given by
\[\left\lb e_1^\dagger, e_2^\dagger\right\rb=0, \qquad 
\left\lb \sigma_{\mathsf E}^\nabla (e_1), e_2^\dagger\right\rb=(\Delta_{e_1}
e_2)^\dagger\] and    \[\left\lb \sigma_{\mathsf E}^\nabla (e_1), \sigma_{\mathsf E}^\nabla (e_2)\right\rb=\sigma_{\mathsf E}^\nabla (\lb e_1,
  e_2\rb_\Delta)-\widetilde{R_\Delta^{\rm bas}(e_1, e_2)}\]
\end{enumerate}
for all $e, e_1,e_2\in\Gamma(\mathsf E)$.
\end{theorem}

\begin{proof} We use the characterisation of the tangent Courant
  algebroid in \cite{BoZa09} (see also \cite{Li-Bland12}): the pairing
has already been discussed in Example \ref{metric_connections} and \S\ref{tangent_euclidean}.
It is given by $\langle Te_1, Te_2\rangle=\ell_{\dr\langle e_1,
    e_2\rangle}$ and $\langle Te_1, e_2^\dagger\rangle=p_M^*\langle
  e_1, e_2\rangle$.  The anchor is given by
  $\Theta(Te)=\widehat{\ldr{\rho_{\mathsf E}(e)}}\in\mx(TM)$ and
  $\Theta(e^\dagger)=(\rho_{\mathsf E}(e))^\uparrow\in \mx(TM)$. The
  bracket is given by $\lb Te_1, Te_2\rb =T\lb e_1, e_2\rb$ and $\lb
  Te_1, e_2^\dagger\rb=\lb e_1, e_2\rb^\dagger$ for all $e, e_1,
e_2\in \Gamma(\mathsf E)$.

(1) has already been checked in Example \ref{metric_connections}.  We
check (2), i.e.~that the anchor satisfies
$\Theta(\sigma^\nabla_{\mathsf E}(e))=\widehat{\nabla_e^{\rm bas}}$.
We have for $\theta\in\Omega^1(M)$ and $v_m\in TM$: $
\Theta(\sigma^\nabla_{\mathsf E} (e)(v_m))(\ell_\theta)=
\ell_{\ldr{\rho(e)}\theta}(v_m)-\langle \theta_m, \rho_{\mathsf
  E}(\nabla_{v_m}e)\rangle=\ell_{ {\nabla^{\rm bas}}^*_e\theta}(v_m)$
and for $f\in C^\infty(M)$: $\Theta(\sigma^\nabla_{\mathsf E}
(e))(p_M^*f)=p_M^*(\rho_{\mathsf E}(e)f)$.  This proves the equality.

Then we compute the brackets of our linear and core sections.  Choose
sections $\phi,\phi'$ of $\operatorname{Hom}(TM,\mathsf E)$. Then
$\left\lb T e, \widetilde{\phi}\right\rb=\widetilde{\ldr{e}\phi}$,
with $\ldr{e}\phi\in\Gamma(\Hom(TM,\mathsf E))$ defined by
$(\ldr{e}\phi)(X)=\lb e, \phi(X)\rb-\phi([\rho_{\mathsf E}(e),X])$ for
all $X\in\mx(M)$. The equality $\left\lb\widetilde{\phi}, T
  e\right\rb=-\widetilde{\ldr{e}\phi}+\mathcal
D\ell_{\langle\phi(\cdot),e\rangle}$ follows. For
$\theta\in\Omega^1(M)$, we compute $\left\langle \mathcal
  D\ell_\theta,
  e^\dagger\right\rangle=\Theta(e^\dagger)(\ell_\theta)=p_M^*\langle
\rho_{\mathsf E}(e), \theta\rangle$.  Thus, $\mathcal
D\ell_\theta=T(\Beta\inv\rho_{\mathsf E}^*\theta)+\widetilde{\psi}$
for a section $\psi\in\Gamma(\operatorname{Hom}(TM,\mathsf E))$ to be
determined.  Since $\left\langle
  \mathcal D\ell_\theta,
  Te\right\rangle=\Theta(Te)(\ell_\theta)=\ell_{\ldr{\rho_{\mathsf E}(e)}\theta}$, 
the bracket $\langle T(\Beta\inv\rho_{\mathsf E}^*\theta)+\widetilde{\psi},
Te\rangle=\ell_{\dr\langle\theta,\rho_{\mathsf E}(e)\rangle+\langle\psi(\cdot),
  e\rangle}$ must equal $\ell_{\ldr{\rho_{\mathsf E}(e)}\theta}$, and we find
$\langle\psi(\cdot), e\rangle=\ip{\rho_{\mathsf E}(e)}\dr\theta$. Because
$e\in\Gamma(\mathsf E)$ was arbitrary we find 
$\psi(X)=-\Beta\inv\rho_{\mathsf E}^*\ip{X}\dr\theta$ for
$X\in\mx(M)$.  We get in particular
\[\left\lb\widetilde{\phi}, T
  e\right\rb=-\widetilde{\ldr{e}\phi}+T(\Beta\inv\rho_{\mathsf
  E}^*\langle\phi(\cdot),e\rangle)-(\widetilde{\Beta\inv\rho_{\mathsf
    E}^*\ip{X}\dr\langle\phi(\cdot),e\rangle}.
\]
The formula $\left\lb \widetilde{\phi},
  \widetilde{\phi'}\right\rb=\widetilde{\phi'\circ\rho_{\mathsf
    E}\circ\phi-\phi\circ\rho_{\mathsf E}\circ\phi'}$
can easily be checked, as well as $\left\lb \widetilde{\phi},
  e^\dagger\right\rb=-\left\lb e^\dagger,
  \widetilde{\phi}\right\rb=-(\phi(\rho_{\mathsf E}(e)))^\dagger$.
Using this, we find now easily that
\begin{equation*}
\begin{split}\left\lb \sigma^\nabla_{\mathsf E} (e_1), \sigma^\nabla_{\mathsf E} (e_2)\right\rb&=\left\lb
    Te_1-\widetilde{\nabla_\cdot e_1}, Te_2-\widetilde{\nabla_\cdot
      e_2}\right\rb\\
  &=T\left\lb e_1, e_2\right\rb-\widetilde{\ldr{e_1}\nabla_\cdot e_2}
  +\widetilde{\ldr{e_2}\nabla_\cdot e_1}-T(\Beta\inv\rho_{\mathsf
    E}^*\langle\nabla_\cdot e_1,e_2\rangle)\\
  &\quad+\widetilde{\Beta\inv\rho_{\mathsf E}^*\dr\langle \nabla_\cdot
    e_1,e_2\rangle}+\widetilde{\nabla_{\rho_{\mathsf E}(\nabla_\cdot e_1)} e_2}-\widetilde{\nabla_{\rho_{\mathsf E}(\nabla_\cdot e_2)} e_1 }\\
  &=T\left\lb e_1,
    e_2\right\rb_\Delta-\widetilde{\ldr{e_1}\nabla_\cdot
    e_2}+\widetilde{\ldr{e_2}\nabla_\cdot
    e_1}+\widetilde{\Beta\inv\rho_{\mathsf E}^*\dr\langle \nabla_\cdot
    e_1,e_2\rangle}\\
  &\quad +\widetilde{\nabla_{\rho_{\mathsf E}(\nabla_\cdot e_1)}
    e_2}-\widetilde{\nabla_{\rho_{\mathsf E}(\nabla_\cdot e_2)} e_1 }.
\end{split}
\end{equation*}
Since for all $X\in \mx(M)$, we have 
\begin{equation*}
\begin{split}
  &-(\ldr{e_1}\nabla_\cdot e_2)(X)+(\ldr{e_2}\nabla_\cdot
  e_1)(X)+\beta\inv\rho_{\mathsf E}^*\ip{X}\dr\langle \nabla_\cdot
  e_1,e_2\rangle\\
  &=-\lb e_1, \nabla_X e_2\rb +\nabla_{[\rho_{\mathsf E}(e_1),
    X]}e_2+\lb e_2, \nabla_X e_1\rb-\nabla_{[\rho_{\mathsf E}(e_2),
    X]}e_1 +\beta\inv\rho_{\mathsf E}^*\ip{X}\dr\langle \nabla_\cdot
  e_1,e_2\rangle\\
% &+\nabla_{\rho_{\mathsf E}(\nabla_X e_1)}
%   e_2-\nabla_{\rho_{\mathsf E}(\nabla_X
%     e_2)} e_1 \\
  &=-\lb e_1, \nabla_X e_2\rb +\nabla_{[\rho_{\mathsf E}(e_1),
    X]}e_2-\lb \nabla_X e_1, e_2\rb-\nabla_{[\rho_{\mathsf E}(e_2),
    X]}e_1 +\beta\inv\rho_{\mathsf E}^*\ldr{X}\langle \nabla_\cdot
  e_1,e_2\rangle,
% &+\nabla_{\rho_{\mathsf E}(\nabla_X e_1)}
%   e_2-\nabla_{\rho_{\mathsf E}(\nabla_X
%     e_2)} e_1 ,
\end{split}
\end{equation*}
we find that $\left\lb \sigma^\nabla_{\mathsf E} (e_1),
  \sigma^\nabla_{\mathsf E} (e_2)\right\rb=T\lb e_1, e_2\rb_\Delta-\widetilde{R_\Delta^{\rm
  bas}(e_1,e_2)}$.  Finally we compute $ \left\lb \sigma^\nabla_{\mathsf E} (e_1),
  e_2^\dagger\right\rb=\left\lb Te_1-\widetilde{\nabla_\cdot e_1},
  e_2^\dagger\right\rb=\lb e_1, e_2\rb^\dagger+\nabla_{\rho_{\mathsf
    E}(e_2)}e_1^\dagger=\Delta_{e_1}e_2^\dagger$.
\end{proof}

\subsubsection{The VB-Courant algebroid associated to a double Lie
  algebroid}\label{VB_courant_from_double}
Consider a double vector bundle $(D;A,B;M)$ with core $C$ and a VB-Lie
algebroid structure on each of its sides.  Recall from \S
\ref{matched_pair_2_rep_sec} that $(D;A,B,M)$ is a double Lie
algebroid if and only if, for any linear splitting of $D$, the two
induced $2$-representations (denoted as in
\S\ref{matched_pair_2_rep_sec}) form a matched pair. By definition of
a double Lie algebroid, $(D\duer A, D\duer B)$ is then a Lie
bialgebroid over $C^*$, and so the double vector bundle
\begin{equation*}
\begin{xy}
  \xymatrix{D\duer A\oplus D\duer B\ar[r]\ar[d]&C^*\ar[d]\\
    A\oplus B\ar[r]&M }
\end{xy}
\end{equation*}
with core $A^*\oplus B^*$ has the structure of a VB-Courant algebroid
with base $C^*$ and side $A\oplus B$.

% We show that the matched pair of representations up to homopy is
% equivalent to a Courant representation of the direct sum Lie algebroid
% $A\oplus B\to M$.

Consider the splitting $\tilde\Sigma\colon (A\oplus B)\times_M C^*\to
D\duer A\oplus D\duer B$ given by
\linebreak$\tilde\Sigma((a(m),b(m)),\gamma_m)=(\sigma_A^\star(\gamma_m),
\sigma_B^\star(\gamma_m))$, where $\sigma_A^\star\colon \Gamma(A)\to
\Gamma_{C^*}^l(D\duer A)$ and $\sigma_B^\star\colon \Gamma(B)\to
\Gamma_{C^*}^l(D\duer B)$ are defined as in \eqref{iso_sign} or Lemma
\ref{lemma_dual_splitting}. By \eqref{equal_fat}, this is a Lagrangian
splitting.  Recall from \S\ref{dual_and_ruths} that the splitting
$\Sigma^\star\colon A\times_{M}C^*\to D\duer A$ of the VB-algebroid
$(D\duer A\to C^*, A\to M)$ corresponds to the $2$-representation
$({\nabla^C}^*,{\nabla^B}^*,-R^*) $ of $A$ on the complex
$\partial_B^*\colon B^*\to C^*$. In the same manner, the splitting
$\Sigma^\star\colon B\times_{M}C^*\to D\duer B$ of the VB-algebroid
$(D\duer B\to C^*, B\to M)$ corresponds to the $2$-representation
$({\nabla^C}^*,{\nabla^A}^*,-R^*) $ of $B$ on the complex
$\partial_A^*\colon A^*\to C^*$.

We quickly check that the split Lie 2-algebroid corresponding to the
linear splitting $\tilde\Sigma$ of $D\duer A\oplus D\duer B$ is the
bicrossproduct of the matched pair of 2-representations (see Example
\ref{double}). The equalities in \eqref{equal_fat} imply that we have
to consider $A\oplus B$ as paired with $A^*\oplus B^*$ in the non
standard way:
\[\langle (a,b),(\alpha,\beta)\rangle=\alpha(a)-\beta(b)
\]
for all $a\in\Gamma(A)$, $b\in\Gamma(B)$, $\alpha\in\Gamma(A^*)$ and
$\beta\in\Gamma(B^*)$.  The anchor of
$\tilde\sigma(a,b)=(\sigma^\star(a),\sigma^\star(b))$ is
$\widehat{\nabla_a^*}+\widehat{\nabla_b^*}\in\mx^l(C^*)$, and the
anchor of
$(\alpha,\beta)^\dagger=(\beta^\dagger,\alpha^\dagger)\in\Gamma_{C^*}^c(D\duer
A\oplus D\duer B)$ is
$(\partial_B^*\beta+\partial_A^*\alpha)^\uparrow\in\mx^c(C^*)$.  The
Courant bracket $\left\lb (\sigma^\star_A(a),\sigma_B^\star(b)),
  (\beta^\dagger,\alpha^\dagger)\right\rb $ is
\[\left([\sigma_A^\star(a),\beta^\dagger]+\ldr{\sigma_B^\star(b)}\beta^\dagger-\ip{\alpha^\dagger}\dr_{D\duer B}\sigma_A^\star(a),
  [\sigma_B^\star(b),\alpha^\dagger]+\ldr{\sigma_A^\star(a)}\alpha^\dagger-\ip{\beta^\dagger}\dr_{D\duer
    A}\sigma_B^\star(b) \right),
\]
where $\dr_{D\duer A}\colon\Gamma_{C^*}(\bigwedge^\bullet D\duer B)\to
\Gamma_{C^*}(\bigwedge^{\bullet+1} D\duer B)$ is defined as usual by
the Lie algebroid $D\duer A$, and similarly for $D\duer B$ (bear in
mind that some non standard signs arise from the signs in
\eqref{equal_fat}).  The derivation $\ldr{}\colon \Gamma(D\duer
A)\times \Gamma(D\duer B)\to \Gamma(D\duer B)$ is described by
\begin{equation*}
\begin{split}
\ldr{\beta^\dagger}\alpha^\dagger&=0,\quad 
\ldr{\beta^\dagger}\sigma_B^\star(b)=-\langle b,
\nabla^*_\cdot\beta\rangle^\dagger,\quad 
\ldr{\sigma_A^\star(a)}\alpha^\dagger=\ldr{a}\alpha^\dagger,\\
&\qquad 
\ldr{\sigma_A^\star(a)}\sigma_B^\star(b)=\sigma_B^\star(\nabla_ab)+\widetilde{R(a,\cdot)b},
\end{split}
\end{equation*}
in \cite[Lemma 4.8]{GrJoMaMe14}. Similar formulae hold for 
$\ldr{}\colon \Gamma(D\duer B)\times \Gamma(D\duer A)\to \Gamma(D\duer A)$.
We get
\[\left\lb (\sigma^\star_A(a),\sigma_B^\star(b)), (\beta^\dagger,\alpha^\dagger)\right\rb
=\left((\nabla_a^*\beta+\ldr{b}\beta-\langle\nabla_\cdot
  a,\alpha\rangle)^\dagger,
  (\nabla_b^*\alpha+\ldr{a}\alpha-\langle\nabla_\cdot
  b,\beta\rangle)^\dagger \right).
\]
In the same manner, we get
\begin{equation*}
\begin{split}
  &\left\lb (\sigma^\star_A(a_1),\sigma_B^\star(b_1)),
    (\sigma^\star_A(a_2),\sigma_B^\star(b_2))\right\rb\\
&  =\left(\sigma_A^\star([a,a']+\nabla_ba'-\nabla_{b'}a),
    \sigma_B^\star([b,b']+\nabla_ab'-\nabla_{a'}b)\right)\\
  &\qquad+\Bigl(-\widetilde{R(a_1,a_2)}+\widetilde{R(b_1,\cdot)a_2}-\widetilde{R(b_2,\cdot)a_1},
  -\widetilde{R(b_1,b_2)}+\widetilde{R(a_1,\cdot)b_2}-\widetilde{R(a_2,\cdot)b_1}\Bigr).
\end{split}
\end{equation*}

Hence we have found the following result.
\begin{theorem}\label{bij_da_matched_lie_2}
  There is a bijection between decomposed double Lie algebroids and Lie
  2-algebroids that are the bicrossproducts of matched pairs of
  2-representations.
\end{theorem}

Recall that if the vector bundle $C$ is trivial, the matched pair of
$2$-representations is just a matched pair of the Lie algebroids $A$
and $B$. The corresponding double Lie algebroid is the decomposed
double Lie algebroid $(A\times_MB,A,B,M)$ found in \cite{Mackenzie11}.
The corresponding VB-Courant algebroid is 
\begin{equation*}
\begin{xy}
  \xymatrix{A\times_MB^*\oplus A^*\times_MB\ar[r]\ar[d]&0\ar[d]\\
    A\oplus B\ar[r]&M }
\end{xy}
\end{equation*}
with core $B^*\oplus A^*$. In that case there is a natural Lagrangian
splitting and the corresponding Lie $2$-algebroid is just the
bicrossproduct Lie algebroid structure defined on $A\oplus B$ by the
matched pair, see also the end of \S\ref{double}. This shows that the
two notions of double of a matched pair of Lie algebroids; the
bicrossproduct Lie algebroid in \cite{Mokri97} and the double Lie
algebroid in \cite{Mackenzie11} are just the $N$-geometric and the
classical descriptions of the same phenomenon, and special cases of
Theorem \ref{bij_da_matched_lie_2}.

\subsection{Categorical equivalence of Lie 2-algebroids and VB-Courant
  algebroids}
In this section we quickly describe morphisms of VB-Courant
algebroids. Then we find an equivalence between the category of
VB-Courant algebroids and the category of Lie $2$-algebroids.
\subsubsection{Morphisms of VB-Courant algebroids}\label{mor_VB_Courant}
Recall from \S\ref{usual_VB_morphisms} and
\S\ref{morphisms_of_met_DVB} that a morphism $\Omega\colon \mathbb
E_1\dashrightarrow\mathbb E_2$ of metric double vector bundles is an
isotropic relation $\Omega\colon \overline{\mathbb E_1}\times\mathbb
E_2$. Assume that $\mathbb E_1$ and $\mathbb E_2$ have linear Courant
algebroid structures. Then $\Omega$ is a morphism of VB-Courant
algebroid if it is a Dirac structure (with support) in
$\overline{\mathbb E_1}\times\mathbb E_2$.

\medskip

Choose two Lagrangian splittings $\Sigma^1\colon Q_1\times
B_1\to\mathbb E_1$ and $\Sigma^2\colon Q_2\times B_2\to\mathbb
E_2$. Then, bY \S\ref{morphisms_of_met_DVB} there exists four
structure maps $\omega_0\colon M_1\to M_2$, $\omega_Q\colon Q_1\to
Q_2$, $\omega_B\colon B_1^*\to B_2^*$ and
$\omega_{12}\in\Omega^2(Q_1,\omega_0^*B_2^*)$ that define completely
$\Omega$. More precisely, $\Omega$ is spanned over
$\operatorname{Graph}(\omega_Q\colon Q_1\to Q_2)$ by sections
\[ \tilde b\colon \operatorname{Graph}(\omega_Q)\to \Omega,\]
\[\tilde b(q_m,\omega_Q(q_m))= \left(\sigma_{B_1}(\omega_B^\star
  b)(q_m)+\widetilde{\omega_{12}^\star(b)}(q_m),
  \sigma_{B_2}(b)(\omega_Q(q_m))\right)
\]
for all $b\in\Gamma_{M_2}(B_2)$,
and 
\[ \tau^\times\colon \operatorname{Graph}(\omega_Q)\to \Omega, \qquad
\tau^\times(q_m,\omega_Q(q_m))=
\left((\omega_Q^\star\tau)^\dagger(q_m),
  \tau^\dagger(\omega_Q(q_m))\right)
\]
for all $\tau\in\Gamma_{M_2}(Q_2^*)$.
% This shows that $\tilde\Sigma\colon
% \operatorname{Graph}(\omega_Q)\times_{M_1} \omega_0^*B_2\to \Omega$,
% \[\tilde\Sigma((q_m,\omega_Q(q_m)), b_{\omega_0(m)})=
% \]
% is a linear splitting of $\Omega$.
Note that $\Omega$ projects under $\pi_{B_1}\times\pi_{B_2}$ to
$R_{\omega_B^*}\subseteq B_1\times B_2$.  But if $q\in \Gamma(Q_1)$
then $\omega_Q^!q\in\Gamma_{M_1}(\omega_0^*Q_2)$ can be written as
$\sum_{i}f_i\omega_0^!q_i$ with $f_i\in C^\infty(M_1)$ and
$q_i\in\Gamma_{M_2}(Q_2)$. The pair $\left(\sigma_{B_1}(\omega_B^\star
  b)(q_m)+\widetilde{\omega_{12}^\star(b)}(q_m),
  \sigma_{B_2}(b)(\omega_Q(q_m))\right)$ can be written as
\[\left(\left(\sigma_{Q_1}(q)+\langle\omega_{12}(q,\cdot), b(\omega_0(m)))\rangle^\dagger\right)(\omega_B^\star
b(m)), \sum_if_i(m)\sigma_{Q_2}(q_i)(b(\omega_0(m)))\right).\] Hence,
$\Omega$ is spanned over $R_{\omega_B^*}$ by sections
\begin{equation}\label{section_morphism_1}
\left(\sigma_{Q_1}(q)\circ\pr_1+\langle\omega_{12}(q,\cdot),\pr_2\rangle^\dagger\circ\pr_1,
  \sum_i(f_i\circ q_{B_1}\circ \pr_1)\cdot(\sigma_{Q_2}(q_i)\circ\pr_2)\right)
\end{equation}
for all $q\in \Gamma_{M_1}(Q_1)$ and
\begin{equation}\label{section_morphism_2}
\left( (\omega_Q^\star\tau)^\dagger\circ\pr_1,\tau^\dagger\circ\pr_2\right)
\end{equation}
for all $\tau\in\Gamma(Q_2^*)$. Note also that
$\langle\omega_{12}(q,\cdot),\pr_2\rangle^\dagger\circ\pr_1$ can be written
\[\sum_{ijk}((f_{ijk}\chi_i(q))\circ q_{B_1}\circ\pr_1)\,\cdot\,
(\ell_{\beta_k}\circ\pr_2)\,\cdot\, (\tau_j^\dagger\circ\pr_1)\]
for some
$f_{ijk}\in C^\infty(M_1)$, $\chi_i,\tau_j\in\Gamma(Q_1^*)$ and
$\beta_k\in\Gamma(B_2^*)$.

\medskip

Checking all the conditions in Lemma \ref{useful_for_dirac_w_support}
on the two types of sections \eqref{section_morphism_1} and
\eqref{section_morphism_2} yield that $\Omega\to R_{\omega_B^*}$ is a
Dirac structure with support if and only if
\begin{enumerate}
\item $\omega_Q\colon Q_1\to Q_2$ over $\omega_0\colon M_1\to M_2$ is compatible 
with the anchors $\rho_1\colon Q_1\to TM_1$ and $\rho_2\colon Q_2\to TM_2$:
\[T_m\omega_0(\rho_1(q_m))=\rho_2(\omega_Q(q_m))
\]
for all $q_m\in Q_1$,
\item $\partial_1\circ\omega_Q^\star=\omega_B^\star\circ\partial_2$ 
as maps from $\Gamma(Q_2^*)$ to $\Gamma(B_1)$, or equivalently
$\omega_Q\circ\partial_1^*=\partial_2^*\circ\omega_B$,
\item $\omega_Q$ preserves the dull brackets up to
  $\partial_2^*\omega_{12}$: i.e.~if\footnote{For simplicity, we
    assume that sections $q\sim_{\omega_Q}r$ exist and span $Q_1$,
    respectively $Q_2$.  A general discussion without this assumption
    would be along the same lines, but much more technical (see the
    general definition of morphisms of Lie algebroids in
    \cite{Mackenzie05}.}  $q^1\sim_{\omega_Q}r^1$ and
  $q^2\sim_{\omega_Q}r^2$, then
\[ \lb q^1,r^1\rb_1\sim_{\omega_Q}\lb q^2,r^2\rb_2-\partial_2^*\omega_{12}(q^1,r^1).
\]
\item $\omega_B$ and $\omega_Q$ intertwines the connections 
$\nabla^1$ and $\nabla^2$ up to $\partial_1\circ\omega_{12}$:
\[\omega_B^\star((\omega_Q^\star\nabla^2)_{q}b)
=\nabla^1_{q}(\omega_B^\star(b))-\partial_1\circ\langle\omega_{12}(q,\cdot), b\rangle
\in\Gamma(B_1)
\]
for all $q_m\in Q_1$ and $b\in\Gamma(B^2)$, and 
\item
  $\omega_Q^\star\omega_{R_2}-\omega_B\circ\omega_{R_1}=-\dr_{(\omega_Q^\star\nabla^2)}\omega_{12}\in\Omega^3(Q_1,\omega_0^*B_2^*)$.
\end{enumerate}
Hence, $\Omega$ is a morphism of VB-Courant algebroid if and only if
it induces a morphism of Dorfman $2$-representations after any choice
of Lagrangian decompositions of $\mathbb E_1$ and $\mathbb E_2$.

\subsubsection{Equivalence of categories}
The functors found in Section \ref{eq_2_manifolds} between the
category of metric double vector bundles and the category of
$[2]$-manifolds restrict to functors between the category of 
VB-Courant algebroids and the category of Lie $[2]$-algebroids.
\begin{theorem}
  The category of Lie $2$-algebroids is equivalent to the
  category of VB-Courant algebroids.
%  \[\begin{xy}
%    \xymatrix{D\ar[r]^{\pi_Q}\ar[d]_{\pi_B}&B\ar[d]^{q_B}\\
%      Q\ar[r]_{q_Q}&M }
% \end{xy}
% \]
% with core $Q^*$ and side algebroid $B$.
\end{theorem}

\begin{proof}
  Let $(\mathcal M, \mathcal Q)$ be a Lie $2$-algebroid and consider
  the double vector bundle $\mathbb E_{\mathcal M}$ corresponding to
  $\mathcal M$.  Choose a splitting $\mathcal M\simeq Q[-1]\oplus
  B^*[-2]$ of $\mathcal M$ and consider the corresponding Lagrangian
  splitting $\Sigma$ of $\mathbb E_{\mathcal M}$.

  As we have seen in \S\ref{dorfman_eq_split}, the split Lie
  $2$-algebroid $(Q[-1]\oplus B^*[-2],\mathcal Q)$ is equivalent to a
  Dorfman $2$-representation. By Theorem \ref{main}, this Dorfman
  $2$-representation defines a VB-Courant algebroid structure on the
  decomposition of $\mathbb E_{\mathcal M}$ and so by isomorphism on
  $\mathbb E_{\mathcal M}$. Further, by Proposition
  \ref{change_of_lift} and \S\ref{cor_splittings}, the Courant
  algebroid structure on $\mathbb E_{\mathcal M}$ does not depend on
  the choice of splitting of $\mathcal M$, since a different choice of
  splitting will induce a change of Lagrangian splitting of $\mathbb
  E_{\mathcal M}$.  This shows that the functor $\mathcal G$ restricts
  to a functor $\mathcal G_{Q}$ from the category of Lie
  $2$-algebroids to the category of VB-Courant algebroids.

 Sections \ref{mor_lie_2_alg} and \ref{mor_VB_Courant}
show that morphisms of split Lie $2$-algebroids are 
sent by $\mathcal G$ to morphisms of decomposed VB-Courant algebroids.

The functor $\mathcal F$ restricts in a similar manner to a functor
$\mathcal F_{\rm VBC}$ from the category of VB-Courant algebroids to the
category of Lie $2$-algebroids. The natural transformations found
in the proof of Theorem \ref{main_crucial} restrict to natural
transformations $\mathcal F_{\rm VBC}\mathcal G_{Q}\simeq \Id$ and
$\mathcal G_{Q}\mathcal F_{\rm VBC}\simeq \Id$.
\end{proof}

\section{LA-Courant algebroids vs Poisson Lie 2-algebroids}\label{sec:LA-Cour}
In this section, we prove that a split Poisson Lie 2-algebroid is
equivalent to the \emph{matched pair} of a Dorfman 2-representation with a
self-dual 2-representation.

Take a double vector bundle 
\[\begin{xy}
\xymatrix{\mathbb{E}\ar[r]^{\pi_B}\ar[d]_{\pi_Q}&B\ar[d]^{q_B}\\
Q\ar[r]_{q_Q}&M
}
\end{xy}
\]
with core $Q^*$, with a VB-Lie algebroid structure on $(\mathbb E\to
Q, B\to M)$ and a VB-Courant algebroid structure on $(\mathbb E\to B,
Q\to M)$. In this section we show that the double vector bundle is an
LA-Courant algebroid if and only if the VB-algebroid is metric and the
self-dual 2-representation defined by any Lagrangian decomposition of
$\mathbb E$ and the VB-algebroid side forms a \emph{matched pair} with
the Dorfman 2-representation describing the Courant algebroid side.

We conclude by recovering in a constructive manner the equivalence
between LA-Courant algebroids and Poisson Lie 2-algebroids
\cite{Li-Bland12}.

We begin with the following definition.
\begin{definition}\label{matched_pairs}
  Let $(B\to M, \rho_B, [\cdot\,,\cdot])$ be a Lie algebroid and
  $(Q\to M, \rho_Q)$ an anchored vector bundle.  Assume that $B$ acts
  on $\partial_Q\colon Q^*\to Q$ up to homotopy via
  a self-dual 2-representation $(\nabla^Q,\nabla^{Q^*}, R_B)$, and let $(\partial_B\colon Q^*\to B,
  \Delta,\nabla, R_Q)$\footnote{For the sake of simplicity, we write
    in this definition $\nabla$ for three different connections,
    unless it is not clear from the indexes which connection is
    meant.} be a $Q$-Dorfman 2-representation.  Then we say that the
  2-representation and the Dorfman 2-representation form a matched
  pair if
  \begin{enumerate}
%\item[(LC4)]  $\rho_Q\circ \partial_Q=\rho_A\circ\partial_A$,
% \item[(LC5)]
%   $\left(\Delta_{\partial_Q\sigma_1}\sigma_2-\nabla_{\partial_A\sigma_2}\sigma_1\right)+\left(\Delta_{\partial_Q\sigma_2}\sigma_1-\nabla_{\partial_A\sigma_1}\sigma_2\right)
%   =\rho_Q^*\dr\langle \sigma_1, \partial_Q\sigma_2\rangle$,
% \item[(LC6)] $[\rho_Q(q), \rho_A(a)]=\rho_A(\nabla_qa)-\rho_Q(\nabla_aq)$,
\item[(M1)] $\partial_Q(\Delta_q\tau)=\nabla_{\partial_B\tau}q+\lb
  q, \partial_Q\tau\rb+\partial_B^*\langle \tau, \nabla_\cdot q\rangle$,
\item[(M2)] $\partial_B(\nabla_b\tau)=[b,\partial_B\tau]+\nabla_{\partial_Q\tau}b$,
\item[(M3)]
  $\partial_BR(b_1,b_2)q=-\nabla_q[b_1,b_2]+[\nabla_qb_1,b_2]+[b_1,\nabla_qb_2]+\nabla_{\nabla_{b_2}q}b_1-\nabla_{\nabla_{b_1}q}b_2$,
% \item[(LC10)]
%   $R(q,\partial_Q\tau)a-R(a,\partial_A\tau)q=\Delta_q\nabla_a\tau-\nabla_a\Delta_q\tau+\Delta_{\nabla_aq}\tau-\nabla_{\nabla_qa}\tau-\langle\nabla_{\nabla_\cdot
%     a}q, \tau\rangle$,
\item[(M4)] $\partial_QR(q_1,q_2)b=-\nabla_b\lb q_1,q_2\rb+\lb q_1,
  \nabla_{b}q_2\rb+\lb \nabla_bq_1,
  q_2\rb+\nabla_{\nabla_{q_2}b}q_1-\nabla_{\nabla_{q_1}b}q_2+\partial_B^*\langle
  R(\cdot, b)q_1, q_2\rangle$.
\item[(M5)] $\dr_{\nabla^B}\omega_R=\dr_{\nabla^Q} \omega_B\in \Omega^2(B,\wedge^3Q^*)=\Omega^3(Q,\wedge^2B^*)$, where 
$\omega_R$ is seen as an element of $\Omega^1(B,\wedge^3Q^*)$ and $\omega_B\in \Omega^2(Q,\wedge^2B^*)$ is defined 
by\linebreak $\omega_B(q_1,q_2)(b_1,b_2)=\langle R(b_1,b_2)q_1,q_2\rangle$. 
\end{enumerate}
\end{definition}

\begin{remark}\label{remark_simplifications}
\begin{enumerate}
\item (M5) is written out as \begin{equation*}
\begin{split}
  &\nabla_{b_2}R(q_1,q_2)b_1-\nabla_{b_1}R(q_1,q_2)b_2+R(q_1,q_2)[b_1,b_2]\\
  &+R(\nabla_{b_1}q_1,q_2)b_2+R(q_1,\nabla_{b_1}q_2)b_2-R(\nabla_{b_2}q_1,q_2)b_1-R(q_1,\nabla_{b_2}q_2)b_1\\
  &+\Delta_{q_1}R(b_1,b_2)q_2-\Delta_{q_2}R(b_1,b_2)q_1-R(b_1,b_2)\lb q_1,q_2\rb\\
  &-R(\nabla_{q_1}b_1,b_2)q_2-R(b_1,\nabla_{q_1}b_2)q_2+R(\nabla_{q_2}b_1,b_2)q_1-R(b_1,\nabla_{q_2}b_2)q_1\\
  =&\langle (R(b_1,\nabla_\cdot b_2)+R(\nabla_\cdot b_1, b_2))q_1,q_2\rangle -\rho_Q^*\dr\langle R(b_1,b_2)q_1, q_2\rangle
\end{split} 
\end{equation*}
for all $q_1,q_2\in \Gamma(Q)$ and $b_1,b_2\in\Gamma(B)$.
\item The equality $\rho_Q\circ \partial_Q=\rho_B\circ\partial_B$ follows easily from (M1)
if $Q$ has positive rank, and from (M2) if $B$ has positive rank. If both $Q$ and $B$ have rank zero, then
$\rho_Q\circ \partial_Q=\rho_B\circ\partial_B$ is trivially satisfied.
\item The equation $[\rho_Q(q), \rho_B(b)]=\rho_B(\nabla_qb)-\rho_Q(\nabla_bq)$ follows easily 
from (M3) if $B$ has positive rank, and from (M4) if $Q$ has positive rank.
If both $Q$ and $B$ have rank zero, then it is trivially satisfied.
\item If $\rho_Q\circ \partial_Q=\rho_B\circ\partial_B$, then (M1) is equivalent to
\begin{equation}\label{almost_C}
\left(\Delta_{\partial_Q\sigma_1}\sigma_2-\nabla_{\partial_B\sigma_2}\sigma_1\right)
+\left(\Delta_{\partial_Q\sigma_2}\sigma_1-\nabla_{\partial_B\sigma_1}\sigma_2\right)
=\rho_Q^*\dr\langle \sigma_1, \partial_Q\sigma_2\rangle
\end{equation}
for all $\sigma_1,\sigma_2\in\Gamma(Q^*)$.
\item If $[\rho_Q(q), \rho_B(b)]=\rho_B(\nabla_qb)-\rho_Q(\nabla_bq)$, then (M4) is equivalent to 
\begin{equation}\label{(LC10)}
R(q,\partial_Q\tau)b-R(b,\partial_B\tau)q=\Delta_q\nabla_b\tau-\nabla_b\Delta_q\tau
+\Delta_{\nabla_bq}\tau-\nabla_{\nabla_qb}\tau-\langle\nabla_{\nabla_\cdot  b}q, \tau\rangle
\end{equation}
for all $b\in\Gamma(B)$, $q\in\Gamma(Q)$ and $\tau\in\Gamma(Q^*)$.
\end{enumerate}
\end{remark}

\subsection{Poisson Lie 2-algebroids via matched pairs.}
We begin this subsection with the definition of a Poisson Lie
2-algebroid.
\begin{definition}\label{def_Poisson_lie_2}
  Let $\mathcal M$ be an Poisson $[2]$-manifold with algebra of
  functions $\mathcal A=C^\infty(\mathcal M)$ and Poisson bracket
  $\{\cdot\,,\cdot\}$. Assume that $\mathcal M$ has in addition a Lie
  2-algebroid structure, i.e.~ it is endowed with a homological vector
  field $Q\in\der^1\mathcal A$. Then $(\mathcal M, Q,
  \{\cdot\,,\cdot\})$ is a \textbf{Poisson Lie 2-algebroid} if the
  homological vector field preserves the Poisson structure, i.e.~ if
  \begin{equation}\label{Q_preserves_bracket}
Q\{\xi_1,\xi_2\}=\{Q(\xi_1), \xi_2\}+(-1)^{\deg\xi_1}\{\xi_1,Q(\xi_2)\}
\end{equation} for all
$\xi_1,\xi_2\in\mathcal A$.

A morphism of Poisson Lie 2-algebroids is a morphism of the underlying
$[2]$-manifold that is a morphism of Poisson $[2]$-manifolds
\emph{and} a morphism of Lie 2-algebroids.
\end{definition}

The main theorem of this section shows that matched pairs as in
Definition \ref{matched_pairs} are equivalent to split Poisson Lie
2-algebroids.
\begin{theorem}\label{Poisson_Lie_2_char}
  Let $\mathcal M=Q[-1]\oplus B^*[-2]$ be a split $[2]$-manifold
  endowed with a homological vector field $\mathcal Q$ and a Poisson
  bracket $\{\cdot\,,\cdot\}$ of degree $-2$. Let $(\partial_B\colon
  Q^*\to B,\Delta,\nabla,R_Q)$ be the Dorfman 2-representation of
  $(Q,\rho_Q)$ that encodes $\mathcal Q$ in coordinates, and let
  $(\partial_Q\colon Q^*\to Q, \nabla^*, \nabla, R_B)$ be the
  self-dual $2$-representation that is equivalent to the Poisson
  bracket.

  Then $(\mathcal M, \mathcal Q, \{\cdot\,,\cdot\})$ is a Poisson Lie
  2-algebroid if and only if the self dual 2-representation and the
  Dorfman 2-representation form a matched pair.
\end{theorem}

\begin{proof}
  The idea of this proof is very simple, but requires rather long
  computations in coordinates. We will leave some of the
  detailed verifications to the reader.

  We check \eqref{Q_preserves_bracket} in coordinates, by using the
  formulae found in \eqref{Q_in_coordinates} and Theorem
  \ref{poisson_is_saruth} for $\chi$ and $\{\cdot\,,\cdot\}$,
  respectively.  On an open chart $U\subseteq M$ trivialising both $Q$
  and $B$, we have coordinates $(x_1,\ldots,x_p)$ and we can choose a
  local frame $(q_1,\ldots,q_{r_1})$ of sections of $Q$, and a local frame
  $(\beta_1,\ldots,\beta_{r_2})$ of sections of $B^*$. We denote by
  $(\tau_1,\ldots,\tau_{r_1})$ and $(b_1,\ldots,b_{r_2})$ the dual frames for
  $Q^*$ and $B$, respectively.  As functions on $\mathcal M$, the
  coordinates functions $x_1,\ldots,x_p$ have degree $0$, the
  functions $\tau_1, \ldots,\tau_{r_1}$ have degree $1$ and the functions
  $b_1,\ldots, b_{r_2}$ have degree $2$.

  First we have
  $\mathcal Q(f)=\sum_i\rho_Q(q_i)(x_k)\sigma_i\partial_{x_k}(f)=\rho_Q^*\dr
  f\in\Gamma(Q^*)$ and $\{f,g\}=0$ for $f,g\in C^\infty(M)$. This
  yields
\[\{\mathcal Q(f),g\}+\{f,\mathcal Q(g)\}=\left\{\rho_Q^*\dr f, g\right\}
+\left\{f, \rho_Q^*\dr g\right\}=0=\mathcal Q\{f,g\}
\]
by the graded skew-symmetry and $\{\tau,f\}=0$ for
$\tau\in\Gamma(Q^*)$ and $f\in C^\infty(M)$.  
Then we have 
\begin{equation*}
\begin{split}
 & \{\mathcal Q(\tau_k),f\}-\{\tau_k,\mathcal Q(f)\}\\
&=\left\{-\sum_{i<j}\langle \lb q_i, q_j\rb,\tau_k\rangle\tau_i\tau_j+\sum_{r}\langle\partial_B^*\beta_r,\tau_k\rangle b_r,f\right\}-\{\tau_k,\rho_Q^*\dr f\}\\
  &=\sum_{r}\langle\partial_B^*\beta_r,\tau_k\rangle \rho_B(b_r)(f)-\langle \partial_Q\tau_k,\rho_Q^*\dr f\rangle\\
  &=\rho_B(\partial_B\tau_k)(f)-\rho_Q(\partial_Q\tau_k)(f).
\end{split}
\end{equation*}
But we also have $\mathcal Q\{\tau_k,f\}=\mathcal Q(0)=0$. Hence,
$\{\mathcal Q(\tau_k),f\}-\{\tau_k,\mathcal Q(f)\}=\mathcal Q\{\tau_k,f\}$ is equivalent to
$\rho_B(\partial_B\tau_k)(f)=\rho_Q(\partial_Q\tau_k)(f)$.

In a similar manner, we have $\mathcal Q\{b_l,f\}=\rho_Q^*\dr(\rho_B(b_l)(f))$ and $\{\mathcal Q(b_l),f\}+\{b_l,\mathcal Q(f)\}$ is 
\begin{equation*}
\begin{split}
  &-\left\{\sum_{i<j<k}\omega_R(q_i,q_j,q_k)(b_l)\tau_i\tau_j\tau_k
    +\sum_{ij}\langle\nabla_{q_i}^*\beta_j,b_l\rangle\tau_ib_j,
    f\right\}+\{b_l,\rho_Q^*\dr f\}\\
  =&-\sum_{ij}\langle\nabla_{q_i}^*\beta_j,b_l\rangle\tau_i\rho_B(b_j)(f)+\nabla^*_{b_l}(\rho_Q^*\dr
  f)=\sum_{ij}\langle\rho_B^*\dr
  f,\nabla_{q_i}b_l\rangle\tau_i+\nabla^*_{b_l}(\rho_Q^*\dr f).
\end{split}
\end{equation*}
Hence, $\mathcal Q\{b_l,f\}=\{\mathcal Q(b_l),f\}+\{b_l,\mathcal Q(f)\}$ if and only if 
\[\rho_Q(q)\rho_B(b_l)(f)=\langle\rho_B^*\dr f,\nabla_{q}b_l\rangle+\rho_B(b_l)\rho_Q(q)(f)-\rho_Q(\nabla_{b_l}q)(f)\]
for all $q\in\Gamma(Q)$. This is 
\[ [\rho_Q(q),\rho_B(b)]=\rho_B(\nabla_{q}b_l)(f)-\rho_Q(\nabla_{b_l}q)(f).
\]
Then we have $\mathcal Q\{b_l,\tau_k\}=\mathcal Q(\nabla_{b_l}\tau_k)$, which is
\begin{equation*}
\begin{split}
  \sum_{i,j}\rho_Q(q_i)\langle\nabla_{b_l}\tau_k,q_j\rangle\tau_i\tau_j-\sum_{i<j}\langle\lb
  q_i,q_j\rb,\nabla_b\tau_k\rangle\tau_i\tau_j
  +\sum_r\langle\nabla_{b_l}\tau_k,\partial_B^*\beta_r\rangle b_r
\end{split}
\end{equation*}
in coordinates. The Poisson bracket $\{\mathcal Q(b_l),\tau_k\}$ is computed to
be
\begin{equation*}
\begin{split}
  &-\sum_{i<j}\langle
  R(\partial_Q\tau_k,q_i)b_l,q_j\rangle\tau_i\tau_j
  -\sum_{r}\langle\nabla_{\partial_Q\tau_k}^*\beta_r,b_l\rangle
  b_r+\sum_{i,j}\langle\nabla_{\nabla_{q_i}b_l}\tau_k,q_j\rangle\tau_i\tau_j
\end{split}
\end{equation*}
in coordinates, and the Poisson bracket $\{b_l, \mathcal Q(\tau_k)\}$
is \begin{equation*}
\begin{split}
  &-\sum_{i<j}\rho_B(b_l)\langle\lb
  q_i,q_j\rb,\tau_k\rangle\tau_i\tau_j
  +\sum_{i,j}\sum_r\langle\lb q_r,q_j\rb,\tau_k\rangle\langle\tau_r,\nabla_{b_l}q_i\rangle\tau_i\tau_j\\
  &+\sum_{r}\langle[b_l,\partial_B\tau_k],\beta_r\rangle
  b_r-\sum_{i<j}\langle
  R(b_l,\partial_B\tau_k)q_i,q_j\rangle\tau_i\tau_j.
\end{split}
\end{equation*}
By comparing the coefficients in these three equations, we find that 
\[\mathcal Q\{b_l,\tau_k\}=\{\mathcal Q(b_l),\tau_k\}+\{b_l, \mathcal Q(\tau_k)\}
\]
if and only if 
\[\partial_B\nabla_{b_l}\tau_k=\nabla_{\partial_Q\tau_k}b_l+[b_l,\partial_B\tau_k],
\]
which is (M2)
and 
\begin{equation*}
\begin{split}
  &\cancel{\rho_Q(q_i)\langle\nabla_{b_l}\tau_k,q_j\rangle}-\cancel{\rho_Q(q_j)\langle\nabla_{b_l}\tau_k,q_i\rangle}
  -\langle\lb q_i,q_j\rb,\nabla_b\tau_k\rangle\\
  =&- \langle
  R(\partial_Q\tau_k,q_i)b_l,q_j\rangle+\langle\nabla_{\nabla_{q_i}b_l}\tau_k,q_j\rangle
  -\langle\nabla_{\nabla_{q_j}b_l}\tau_k,q_i\rangle\\
  &-\rho_B(b_l)\langle\lb q_i,q_j\rb,\tau_k\rangle+\langle\lb
  \nabla_{b_l}q_i,q_j\rb,\tau_k\rangle
  +\cancel{\rho_Q(q_j)\langle\tau_k,\nabla_{b_l}q_i\rangle} \\
  &-\langle\lb \nabla_{b_l}q_j,q_i\rb,\tau_k\rangle
  -\cancel{\rho_Q(q_i)\langle\tau_k,\nabla_{b_l}q_j\rangle} -\langle
  R(b_l,\partial_B\tau_k)q_i,q_j\rangle
\end{split}
\end{equation*}
which simplifies to (M4).  Next we study the condition
$\mathcal Q\{\tau_k,\tau_l\}=\{\mathcal Q(\tau_k),\tau_l\}-\{\tau_k,\mathcal Q(\tau_l)\}$.  The
left hand side is
$\rho_Q^*\dr\langle\tau_k,\partial_Q\tau_l\rangle=\sum_i\rho_Q(q_i)\langle\tau_k,\partial_Q\tau_l\rangle\tau_i$.
To get the right hand side we note that
$-\{\tau_k,\mathcal Q(\tau_l)\}=\{\mathcal Q(\tau_l),\tau_k\}$.  A short computation
yields
\[\{\mathcal Q(\tau_k),\tau_l\}=\sum_{i,j}\langle\lb
q_i,q_j\rb,\tau_k\rangle\langle\tau_l,\partial_Q\tau_i\rangle\tau_j
+\sum_r\sum_j\langle\partial_B^*\beta_r,\tau_k\rangle\langle\nabla_{b_r}\tau_l,q_j\rangle\tau_j.
\]
We get easily that
$\mathcal Q\{\tau_k,\tau_l\}=\{\mathcal Q(\tau_k),\tau_l\}-\{\tau_k,\mathcal Q(\tau_l)\}$ is
equivalent to \eqref{almost_C}. Recall from Remark
\ref{remark_simplifications} that together with
$\rho_B\circ\partial_B=\rho_Q\circ\partial_Q$, this is equivalent to
(M1).

The study of the equation $\mathcal Q\{b_l,b_k\}=\{\mathcal Q(b_l),b_k\}+\{b_l,\mathcal Q(b_k)\}$
is long, but relatively straightforward. We get $C^\infty(M)$-linear
combinations of $\tau_i\tau_j\tau_k$ and of $\tau_ib_r$.  Comparing
the factors of $\tau_i\tau_j\tau_k$ yields (M5) and comparing the
factors of $\tau_ib_r$ yields (M3).
\end{proof}

\subsection{LA-Courant algebroids and equivalence of categories}
Li-Bland's definition of an LA-Courant algebroid is quite technical
and requires the consideration of triple vector bundles. We give it
and study it in Appendix \ref{appendix_LACA}, where we prove the
following theorem.

\begin{theorem}\label{LA-Courant}
  Let $(\mathbb E;Q,B;M)$ be a double vector bundle with a VB-Courant
  algebroid structure on $(\mathbb E\to B, Q\to M)$ and a VB-Lie
  algebroid structure on $(\mathbb E\to Q, B\to M)$. Then in
  particular, $\mathbb E$ is a metric double vector bundle with the
  linear metric underlying the linear Courant algebroid structure on
  $\mathbb E\to B$.  Choose a Lagrangian decomposition $\Sigma\colon
  B\times_MQ\to \mathbb E$ of $\mathbb E$.  Then $(\mathbb E;Q,B;M)$
  is an LA-Courant algebroid if and only if
\begin{enumerate}
\item the linear Lie algebroid structure on $\mathbb E\to Q$ is
  compatible in the sense of Definition \ref{metric_VBLA} with the
  linear metric, and
\item the self-dual 2-representation and the Dorfman
  2-representation obtained from the Lagrangian splitting form a
  matched pair as in Definition \ref{matched_pairs}.
\end{enumerate}
\end{theorem}
The proof of this theorem is very long and technical, showing that the
definition of an LA-Courant algebroid is hard to handle.  Hence our
result provides a new definition of LA-Courant algebroids, that is
much simpler to articulate and probably also easier to use.

Further, we now explain how this theorem shows that LA-Courant
algebroid are equivalent to Poisson Lie 2-algebroids. (Note that this
has already been found by Li-Bland in \cite{Li-Bland12}.) First,
morphisms of LA-Courant algebroids are morphisms of metric double
vector bundles that preserve the Courant algebroid structure and the
Lie algebroid structure \cite{Li-Bland12}.  Hence, the category of
LA-Courant algebroids is a full subcategory of the intersection of the
category of metric VB-algebroids and the category of VB-Courant
algebroids.

On the other hand, Definition \ref{def_Poisson_lie_2} shows that the
category of Poisson Lie 2-algebroids is a full subcategory of the
intersection of the categories of Poisson $[2]$-manifolds and of Lie
$2$-algebroids.

This, Theorem \ref{Poisson_Lie_2_char} and Theorem \ref{LA-Courant}
show that the equivalences of the categories of metric VB-algebroids
and of Poisson $[2]$-manifolds and of the categories of VB-Courant
algebroids and Lie $2$-algebroids restrict to an equivalence of the
category of LA-Courant algebroids with the category of Poisson Lie
$2$-algebroids.

% from a Lagrangian decomposition of a
%   VB-Courant algebroid, and how to get a split Poisson 2-manifold from
%   a decomposition of a metric VB-Lie algebroid.  We now consider a
%   double vector bundle with a VB-Courant algebroid structure on one
%   side and a metric VB-Lie algebroid structure on the other side.  Let
%   $\mathcal M$ be the $[2]$-manifold defined by the underlying metric
%   double vector bundle structure on $\mathbb E$.  We study when the
%   homological vector field on $\mathcal M$ that is equivalent to the
%   Courant structure preserves the Poisson bracket on $\mathcal
%   A=C^\infty(\mathcal M)$ that is equivalent to the metric Lie
%   structure. Let us first point out that
%   this result can be obtained by a tedious analysis of the classical
%   definition of LA-Courant algebroid, that was given in
%   \cite{Li-Bland12}.  This approach is much more complicated and
%   technical, but slightly more constructive in the sense that the
%   computations done there yield Equations (LC1)-(M5) in a
%   straightforward manner.
% \todo{reformat, do this section}

\subsection{Examples of LA-Courant algebroids and Poisson Lie 2-algebroids}
Next we discuss some classes of Examples of LA-Courant algebroids, and
the corresponding Poisson Lie $2$-algebroids.
\subsubsection{The tangent double of a Courant algebroid}\label{TCourant_LACourant}
Let $\mathsf E\to M$ be a Courant algebroid and choose a metric
connection $\nabla\colon\mx(M)\times\Gamma(\mathsf E)\to\Gamma(\mathsf
E)$.  We have seen in Examples \ref{double_ruth},
\ref{metric_connections} and \S\ref{tangent_euclidean} that the triple
$(\nabla,\nabla,R_\nabla)$ is then the $TM$-representation up to
homotopy describing $(T\mathsf E\to \mathsf E, TM\to M)$ after the
choice of the splitting $\Sigma^\nabla\colon\mathsf E\times_M M\to
T\mathsf E$.  We have also seen in Section \ref{tangent_Courant} that
the $\mathsf E$-Dorfman 2-representation encoding the Courant
algebroid side $(T\mathsf E\to TM, \mathsf E\to TM)$ is
$(\rho_{\mathsf E}\colon\mathsf E\to TM, \Delta,\nabla^{\rm bas},
R^{\rm bas}_\Delta)$. A straightforward computation resembling the one
in \cite[Section 3.2]{JoMa14} for the tangent double of a Lie
algebroid shows that this $2$-representation and this Dorfman
$2$-representation are compatible, and so that $T\mathsf E$ is
an LA-Courant algebroid (see also \cite{Li-Bland12}).

Example \ref{tangent_euclidean} and \S\ref{symplectic} show that this
class of LA-Courant algebroids is equivalent to the \emph{symplectic}
Lie $2$-algebroids.

\subsubsection{The standard Courant algebroid over a Lie algebroid}\label{PontLA_LACourant}
Let $A$ be a Lie algebroid. Then $TA\oplus T^*A$ has a VB-Courant
algebroid structure $(TA\oplus_AT^*A; TM\oplus A^*, A;M)$ by
\S\ref{standard_VB_Courant_ex} and a metric VB-algebroid structure
$(TA\oplus_AT^*A\to TM\oplus A^*, A\to M)$ by Example \ref{met_TET*E}.
Recall from \S\ref{standard_VB_Courant_ex} and Example \ref{standard}
the Dorfman $2$-representation given by a Lagrangian splitting and the
VB-Courant algebroid structure, and recall from Example
\ref{met_TET*E} the self-dual $2$-representation defined by the same
Lagrangian splitting and the VB-algebroid structure.  A
straightforward computation, that also resembles the one in
\cite[Section 3.2]{JoMa14} for the tangent double of a Lie algebroid,
shows that the Dorfman $2$-representation and the self-dual
$2$-representation form a matched pair.  Hence, $TA\oplus_AT^*A$ is an
LA-Courant algebroid.

\subsubsection{The LA-Courant algebroid defined by a double Lie algebroid}\label{Poi_lie_2_def_by_matched_ruth}
  More generally, take a double Lie algebroid $(D,A,B,M)$ and consider
 the metric VB-algebroid
\begin{equation*}
\begin{xy}
\xymatrix{
D\oplus_B(D\duer B) \ar[d]\ar[r]& B\ar[d]\\
 A\oplus C^*\ar[r]& M}
\end{xy}
\end{equation*}
defined as in \S\ref{metric_double_VB_alg} by the VB-algebroid $(D\to
A, B\to M)$. Since $(D\to B, A\to M)$ is a VB-algebroid as well, we
get a linear Courant algebroid structure on $D\oplus_B(D\duer B)\to B$
as in \S\ref{VB_Courant_alg}. Given a linear splitting $\Sigma\colon
A\times_MB\to D$, we get a matched pair of $2$-representations as in Definition 
\ref{matched_pair_2_rep}, see \S\ref{matched_pair_2_rep_sec}.

The induced linear splitting $\tilde\Sigma\colon
B\times_M(A\oplus C^*)\to D\oplus_B(D\duer B)$ defines hence a
$2$-representation
\begin{equation}\label{double_2_rep}
(\partial_A\oplus\partial_A^*\colon C\oplus A^*\to A\oplus C^*, \nabla^A\oplus{\nabla^C}^*,\nabla^C\oplus{\nabla^A}^*,
R\oplus(-R^*)),
\end{equation}
as in \S\ref{metric_double_VB_alg}
and a Dorfman $2$-representation 
$(\partial_B\circ \pr_C\colon C\oplus A^*\to B, \Delta,\nabla,R)$
defined by
\begin{equation}\label{dorf_2_rep_1}
\begin{split}
\Delta\colon&\Gamma(A\oplus C^*)\times\Gamma(C\oplus A^*)\to\Gamma(C\oplus A^*)\\
&\Delta_{(a,\gamma)}(c,\alpha)=(\nabla_ac,\ldr{a}\alpha+\langle\nabla^*_\cdot\gamma,c\rangle),
\end{split}
\end{equation}
\begin{equation}\label{dorf_2_rep_2}
\begin{split}
  \nabla\colon&\Gamma(A\oplus C^*)\times\Gamma(B)\to\Gamma(B), \qquad \nabla_{(a,\gamma)}b=\nabla_ab
\end{split}
\end{equation}
with $A\oplus C^*$ anchored by $\rho_A$,
and 
$R\in \omega^2(A\oplus C^*, \operatorname{Hom}(B,C\oplus A^*))$,
\begin{equation}\label{dorf_2_rep_3}
R((a_1,\gamma_1), (a_2,\gamma_2))=\left(R(a_1,a_2), \langle\gamma_2,R(a_1,\cdot)\rangle
+\langle\gamma_1,R(\cdot,a_2)\rangle\right)
\end{equation}
as in Example \ref{semi-direct}, see \S\ref{VB_Courant_alg}.  A
straightforward computation shows that the matched pair conditions for
the $2$-representations describing the sides of $D$ imply that the
$2$-representation \eqref{double_2_rep} and the Dorfman $2$-representation
\eqref{dorf_2_rep_1}--\eqref{dorf_2_rep_3}
 form a matched
pair. Hence, $(D\oplus_B(D\duer B),A\oplus C^*,B,M)$ has a
natural LA-Courant algebroid structure.  In the same manner,
$(D\oplus_A(D\duer A),B\oplus C^*,A,M)$ has a natural LA-Courant
algebroid structure.  Hence, we get the following theorem.
\begin{theorem}
  Consider a matched pair of $2$-representations with the usual
  notation.  Then the split $[2]$-manifold $(A\oplus C^*)[-1]\oplus
  B^*[-2]$ endowed with the semi-direct Lie $2$-algebroid structure in
  \S\ref{semi-direct} and the Poisson bracket defined by
  \eqref{double_2_rep} and \S\ref{metric_double_VB_alg}, is a split
  Poisson Lie $2$-algebroid.

By symmetry, the split $[2]$-manifold $(B\oplus C^*)[-1]\oplus
  A^*[-2]$ also inherits a split Poisson Lie $2$-algebroid structure.
\end{theorem}

\subsection{The degenerate Courant algebroid structure on the core}
We prove in this section that the core of an LA-Courant algebroid
inherits a natural structure of degenerate Courant algebroid.  We
discuss some examples and we deduce a new way of describing the
equivalence between Courant algebroids and symplectic Lie
2-algebroids.
\begin{theorem}\label{core_courant}
  Let $(\mathbb E; B, Q;M)$ be an LA-Courant algebroid and choose a Lagrangian
  splitting. Then the core $Q^*$ inherits the structure of a
  degenerate Courant algebroid over $M$, with the anchor
  $\rho_Q\partial_Q$, the map $\mathcal D=\rho_Q^*\dr\colon
  C^\infty(M)\to\Gamma(Q^*)$, the pairing defined by $\langle \tau_1,
  \tau_2\rangle_{Q^*}=\langle \tau_1, \partial_Q\tau_2\rangle$ and the
  bracket defined by $\lb\tau_1,
  \tau_2\rb_{Q^*}=\Delta_{\partial_Q\tau_1}\tau_2-\nabla_{\partial_B\tau_2}\tau_1$
  for all $\tau_1,\tau_2\in\Gamma(Q^*)$.  This structure does not
  depend on the choice of the Lagrangian splitting, and the map
  $\partial_B\colon Q^*\to B$ preserves the brackets and the anchors.
\end{theorem}

\begin{proof}
  Theorem \ref{LA-Courant} states that the $2$-representation and the
  Dorfman $2$-representation defined by a Lagrangian splitting form a
  matched pair and that the $2$-representation is self-dual.  Hence,
  by (1) of Definition \ref{saruth}, the pairing $\langle \cdot\,,
  \cdot\rangle_{Q^*}$ is symmetric.  The map $\rho_Q^*\dr\colon
  C^\infty(M)\to\Gamma(Q^*)$ satisfies $\langle \tau, \rho_Q^*\dr
  \varphi\rangle_{Q^*}=\langle \partial_Q\tau, \rho_Q^*\dr
  \varphi\rangle=\langle \rho_Q\partial_Q\tau, \dr
  \varphi\rangle=(\rho_Q\circ\partial_Q)(\tau)(\varphi)$ for all
  $\tau\in\Gamma(Q^*)$ and $\varphi\in C^\infty(M)$.  We check
  (CA1)--(CA5) in Definition
  \ref{background_courant_notions}. Condition (CA5) is immediate by
  definition of the bracket. Condition (CA3) is exactly
  \eqref{almost_C}.  Note that (M5) and (D1) imply
\begin{align}\label{partial_A_preserves}
  \partial_B\lb \tau_1,
  \tau_2\rb_{Q^*}&=\partial_B(\Delta_{\partial_Q\tau_1}\tau_2-\nabla_{\partial_B\tau_2}\tau_1) \nonumber\\
  &=\nabla_{\partial_Q\tau_1}\partial_B\tau_2-[\partial_B\tau_2,\partial_B\tau_1]
  -\nabla_{\partial_Q\tau_1}\partial_B\tau_2=[\partial_B\tau_1,\partial_B\tau_2].
\end{align}
This and $\rho_Q\circ\partial_Q=\rho_B\circ\partial_B$ (see Remark
\ref{remark_simplifications}) imply the last assertion of the theorem.
In the same manner (M1) and
$\nabla^{Q}\circ\partial_Q=\partial_Q\circ\nabla^{Q^*}$ (by Definition
of a $2$-representation) imply the equation
\begin{align}\label{partial_Q_preserves}
\partial_Q\lb \tau_1,
\tau_2\rb_{Q^*}&=\lb\partial_Q\tau_1,\partial_Q\tau_2\rb_\Delta+\partial_B^*\langle
\tau_2, \nabla_\cdot\partial_Q\tau_1\rangle.
\end{align}
The compatibility of the bracket with the anchor (CA4) follows then
immediately from \eqref{partial_Q_preserves} with \eqref{rho_delta},
or from \eqref{partial_A_preserves} with
$\rho_Q\circ\partial_Q=\rho_B\circ\partial_B$.  Next we check (CA2)
using (M1) and $\nabla^{Q}\circ\partial_Q=\partial_Q\circ\nabla^{Q^*}$. We have
\begin{align*}
  &\rho_Q\partial_Q(\tau_1)\langle\tau_2,
  \tau_3\rangle_{Q^*}-\langle\lb \tau_1, \tau_2\rb_{Q^*},
  \tau_3\rangle_{Q^*}-\langle\tau_2, \lb \tau_1,
  \tau_3\rb_{Q^*}\rangle_{Q^*}\\
  =\,&\langle \tau_2,
  \lb \partial_Q\tau_1, \partial_Q\tau_3\rb_\Delta\rangle
  +\langle\nabla_{\partial_B\tau_2}
  \tau_1, \partial_Q\tau_3\rangle -
  \langle\tau_2, \partial_Q\Delta_{\partial_Q\tau_1}\tau_3-\partial_Q\nabla_{\partial_B\tau_3}\tau_1\rangle\\
  =\,&\langle\tau_2, -\partial_Q\Delta_{\partial_Q\tau_1}\tau_3
  +\nabla_{\partial_B\tau_3}\partial_Q\tau_1
  +\lb \partial_Q\tau_1, \partial_Q\tau_3\rb_\Delta
  +\partial_B^*\langle\nabla_{\cdot}
  \partial_Q\tau_1, \tau_3\rangle \rangle=0.
\end{align*}
Finally we check the Jacobi identity (CA1).  Using
\eqref{partial_A_preserves} and \eqref{partial_Q_preserves}, we have
for $\tau_1,\tau_2,\tau_3\in\Gamma(Q^*)$:
\begin{equation*}
\begin{split}
  &\lb \lb \tau_1,\tau_2\rb_{Q^*},\tau_3\rb_{Q^*}+\lb \tau_2,\lb
  \tau_1,\tau_3\rb_{Q^*}\rb_{Q^*}-\lb \tau_1,\lb \tau_2,\tau_3\rb_{Q^*}\rb_{Q^*}\\
  &=\lb
  \Delta_{\partial_Q\tau_1}\tau_2-\nabla_{\partial_B\tau_2}\tau_1,
  \tau_3\rb_{Q^*} +\lb \tau_2,
  \Delta_{\partial_Q\tau_1}\tau_3-\nabla_{\partial_B\tau_3}\tau_1\rb_{Q^*}\\
&  -\lb \tau_1,
  \Delta_{\partial_Q\tau_2}\tau_3-\nabla_{\partial_B\tau_3}\tau_2\rb_{Q^*}\\
  &=\Delta_{\lb\partial_Q\tau_1,\partial_Q\tau_2\rb+\partial_B^*\langle
    \tau_2,
    \nabla_\cdot\partial_Q\tau_1\rangle}\tau_3-\nabla_{\partial_B\tau_3}(
  \Delta_{\partial_Q\tau_1}\tau_2-\nabla_{\partial_B\tau_2}\tau_1)\\
  &\quad +\Delta_{\partial_Q\tau_2}
  (\Delta_{\partial_Q\tau_1}\tau_3-\nabla_{\partial_B\tau_3}\tau_1)
  -\nabla_{[\partial_B\tau_1,\partial_B\tau_3]}\tau_2\\
  &\quad-\Delta_{\partial_Q\tau_1}(\Delta_{\partial_Q\tau_2}\tau_3-\nabla_{\partial_B\tau_3}\tau_2)+\nabla_{[\partial_B\tau_2,\partial_B\tau_3]}\tau_1\\
  &=R_\nabla(\partial_B\tau_3,\partial_B\tau_2)\tau_1+\nabla_{\partial_B\tau_2}\nabla_{\partial_B\tau_3}\tau_1-R_\Delta(\partial_Q\tau_1,\partial_Q\tau_2)\tau_3+\Delta_{\partial_B^*\langle
    \tau_2, \nabla_\cdot\partial_Q\tau_1\rangle}\tau_3\\
  &\quad -\nabla_{\partial_B\tau_3}
  \Delta_{\partial_Q\tau_1}\tau_2+\Delta_{\partial_Q\tau_1}\nabla_{\partial_B\tau_3}\tau_2-\Delta_{\partial_Q\tau_2}\nabla_{\partial_B\tau_3}\tau_1-\nabla_{[\partial_B\tau_1,\partial_B\tau_3]}\tau_2.
\end{split}
\end{equation*} 
Using the equalities 
\[R_\Delta(\partial_Q\tau_1, \partial_Q\tau_2)\tau_3=R(\partial_Q\tau_1, \partial_Q\tau_2)\partial_B\tau_3
\quad \text{ by (D4)},
\]
\[
R_\nabla(\partial_B\tau_3,\partial_B\tau_2)\tau_1=R(\partial_B\tau_3,\partial_B\tau_2)\partial_Q\tau_1\quad
\text{ by the Def. of a $2$-representation},\] and \eqref{(LC10)},
this is
\begin{equation*}
\begin{split}
  &-\Delta_{\nabla_{\partial_B\tau_3}\partial_Q\tau_1}\tau_2+\nabla_{\nabla_{\partial_Q\tau_1}\partial_B\tau_3}\tau_2+\langle\nabla_{\nabla_\cdot \partial_B\tau_3}\partial_Q\tau_1,
  \tau_2\rangle \\
  &\quad
  +\nabla_{\partial_B\tau_2}\nabla_{\partial_B\tau_3}\tau_1+\Delta_{\partial_B^*\langle
    \tau_2,
    \nabla_\cdot\partial_Q\tau_1\rangle}\tau_3-\nabla_{[\partial_B\tau_1,\partial_B\tau_3]}\tau_2-\Delta_{\partial_Q\tau_2}\nabla_{\partial_B\tau_3}\tau_1.
\end{split}
\end{equation*}
By \eqref{almost_C}, we can replace 
\begin{equation*}
\begin{split}
  &-\Delta_{\nabla_{\partial_B\tau_3}\partial_Q\tau_1}\tau_2+\nabla_{\partial_B\tau_2}\nabla_{\partial_B\tau_3}\tau_1-\Delta_{\partial_Q\tau_2}\nabla_{\partial_B\tau_3}\tau_1\\
  =&-\Delta_{\partial_Q(\nabla_{\partial_B\tau_3}\tau_1)}\tau_2+\nabla_{\partial_B\tau_2}(\nabla_{\partial_B\tau_3}\tau_1)-\Delta_{\partial_Q\tau_2}(\nabla_{\partial_B\tau_3}\tau_1)
\end{split}
\end{equation*}
by 
\[-\nabla_{\partial_B(\nabla_{\partial_B\tau_3}\tau_1)}\tau_2-\rho_Q^*\dr
\langle \tau_2, \partial_Q\nabla_{\partial_B\tau_3}\tau_1\rangle
\]
and we get
\begin{equation*}
\begin{split}
  &\nabla_{\nabla_{\partial_Q\tau_1}\partial_B\tau_3
    -[\partial_B\tau_1,\partial_B\tau_3]-\partial_B(\nabla_{\partial_B\tau_3}\tau_1)}\tau_2+\langle\nabla_{\nabla_\cdot \partial_B\tau_3}\partial_Q\tau_1,
  \tau_2\rangle \\
  &\quad +\Delta_{\partial_B^*\langle \tau_2,
    \nabla_\cdot\partial_Q\tau_1\rangle}\tau_3 -\rho_Q^*\dr \langle
  \tau_2, \partial_Q\nabla_{\partial_B\tau_3}\tau_1\rangle
\end{split}
\end{equation*}
Since
$\nabla_{\partial_Q\tau_1}\partial_B\tau_3
-[\partial_B\tau_1,\partial_B\tau_3]-\partial_B(\nabla_{\partial_B\tau_3}\tau_1)=0$ by (M2),
we finally get
\begin{equation}\label{remainder}
\begin{split}
&\langle\nabla_{\nabla_\cdot \partial_B\tau_3}\partial_Q\tau_1, \tau_2\rangle
+\Delta_{\partial_B^*\langle
\tau_2,
\nabla_\cdot\partial_Q\tau_1\rangle}\tau_3
-\rho_Q^*\dr
\langle \tau_2, \partial_Q\nabla_{\partial_B\tau_3}\tau_1\rangle.
\end{split}
\end{equation}
We write $\beta:=\langle\nabla_\cdot \partial_Q\tau_1,
\tau_2\rangle\in\Gamma(B^*)$. Since
$\rho_Q\circ\partial_Q=\rho_B\circ\partial_B$ and
$\nabla^{Q}\circ\partial_Q=\partial_Q\circ\nabla^{Q^*}$, we find
$\beta=\langle\partial_Q\nabla_\cdot \tau_1,
\tau_2\rangle=\langle\nabla_\cdot \tau_1,
\partial_Q\tau_2\rangle\in\Gamma(B^*)$.  To see that
\eqref{remainder}, a section of $Q^*$, vanishes, we evaluate it on an
arbitrary $q\in\Gamma(Q)$. We use (D1) and the definition of a
$2$-representation and we get
\begin{equation*}
\begin{split}
  &\langle\nabla_{\nabla_q \partial_B\tau_3}\partial_Q\tau_1,
  \tau_2\rangle +\langle\Delta_{\partial_B^*\beta}\tau_3, q\rangle
  -\rho_Q(q) \langle
  \tau_2, \partial_Q\nabla_{\partial_B\tau_3}\tau_1\rangle\\
  =&\langle\nabla_{\partial_B\Delta_q\tau_3}\partial_Q\tau_1,
  \tau_2\rangle +\langle\Delta_{\partial_B^*\beta}\tau_3, q\rangle
  -\rho_Q(q) \langle
  \tau_2, \partial_Q\nabla_{\partial_B\tau_3}\tau_1\rangle\\
  =&\langle\beta, \partial_B\Delta_q\tau_3\rangle
  +\langle\Delta_{\partial_B^*\beta}\tau_3, q\rangle -\rho_Q(q)
  \langle\beta,\partial_B\tau_3\rangle =-\langle\lb
  q, \partial_B^*\beta\rb_\Delta,\tau_3\rangle+\langle\Delta_{\partial_B^*\beta}\tau_3,
  q\rangle.
\end{split}
\end{equation*}
Since the Dorfman connection $\Delta$ is dual to the skew-symmetric
dull bracket $\lb\cdot\,,\cdot\rb_\Delta$, this is
\[\rho_Q(\partial_B^*\beta)\langle q,\tau_3\rangle.
\]
Since $\rho_Q\circ\partial_B^*=0$ by \eqref{rho_delta}, we can
conclude.

We finally prove that the degenerate Courant algebroid structure does
not depend on the choice of the Lagrangian splitting.  Clearly the
pairing and anchor are independent of the splitting, so we only need
to check that the bracket remains the same if we choose a different
Lagrangian splitting. Assume that $\Sigma_1,\Sigma_2\colon
B\times_MQ\to\mathbb E$ are two Lagrangian splittings.  Then
$\Delta^2_{\partial_Q\tau_1}\tau_2-\nabla^2_{\partial_B\tau_2}\tau_1=
\Delta^1_{\partial_Q\tau_1}\tau_2+\phi_{12}(\partial_Q\tau_1,\partial_B\tau_2)
-\nabla^2_{\partial_B\tau_2}\tau_1-\phi_{12}(\partial_Q\tau_1,\partial_B\tau_2)
=\Delta^1_{\partial_Q\tau_1}\tau_2-\nabla^1_{\partial_B\tau_2}\tau_1$
by Proposition \ref{change_of_lift} and the considerations following
Theorem \ref{rajan}.
\end{proof}

% We believe that degenerate Courant algebroids can be interpreted as
% Dirac $[2]$-manifolds. 
We will study in a future work the N-geometric description of
degenerate Courant algebroids.

\begin{example}[Tangent Courant algebroid]
  Consider the example described in \S\ref{tangent_euclidean},
  \S\ref{adjoint} and \S\ref{tangent_Courant}.  The degenerate Courant
  algebroid structure on the core $\mathsf E$ of $T\mathsf E$ is just
  the initial Courant algebroid structure on $\mathsf E$ since
  $\Delta_{e_1}e_2=\lb e_1, e_2\rb +\nabla_{\rho(e_2)}e_1$ by
  definition and so
\[\Delta_{e_1}e_2-\nabla_{\rho(e_2)}e_1
=\lb e_1, e_2\rb.
\]
\end{example}

We have hence proved that the Courant algebroid associated to a
symplectic Lie 2-algebroid can be defined directly from any of the
splittings of the Lie 2-algebroid, and so does not need to be obtained
as a derived bracket. Recall that we have characterised split
symplectic $[2]$-manifolds in \S\ref{symplectic}.

\begin{theorem}\label{culminate}
  Let $\mathcal M$ be a symplectic Lie 2-algebroid over a base
  manifold $M$. Then the corresponding Courant algebroid is defined as
  follows. Choose any splitting $\mathcal M\simeq Q[-1]\oplus
  T^*M[-2]$ of the underlying symplectic $[2]$-manifold. Then $Q\simeq
  Q^*$ via $\partial_Q$ and $Q$ inherits a nondegenerate pairing
  given by $\langle \partial_Q\tau_1,
  \partial_Q\tau_2\rangle_{Q}=\langle\tau_1,\partial_Q\tau_2\rangle$.  The
  split Lie 2-algebroid structure on $Q\oplus T^*M$ is dual to a
  Dorfman 2-representation of $Q$ on $-l_1^*\colon Q^*\to TM$.  The
  map $l_1^*\circ\partial_Q\inv$ defines an anchor on $Q$ and the bracket
  $\lb\cdot\,,\cdot\rb_{Q}$ defined on $\Gamma(Q)$ by $\lb \partial_Q\tau_1,
  \partial_Q\tau_2\rb_{Q}=\partial_Q(\Delta_{\partial_Q\tau_1}\tau_2-\{l_1^*\tau_2,\tau_1\})$
  does not depend on the choice of the splitting. This anchor, pairing
  and bracket define a Courant algebroid structure on $Q$.% (The
  % Courant algebroid that is equivalent to the symplectic Lie
  % 2-algebroid in the usual sense.)
\end{theorem}

Note that the Courant algebroid structure is more naturally defined on
$Q^*$, but we consider the induced Courant algebroid structure on $Q$
by the vector bundle isomorphism $\Beta=\partial_Q\colon Q^*\to Q$
defined by the pairing on $Q^*$. We do this for our result to be
consistent with the constructions in \S\ref{symplectic},
\S\ref{adjoint} and \S\ref{TCourant_LACourant}.

\begin{example}[Core of the standard Courant algebroid over a Lie algebroid]
  Consider now the Example discussed in Example \ref{standard},
  \S\ref{standard_VB_Courant_ex} and \S\ref{PontLA_LACourant}; namely
  the standard LA-Courant algebroid
\begin{equation*}
\begin{xy}
\xymatrix{
TA\oplus_AT^*A \ar[d]\ar[r]& TM\oplus A^*\ar[d]\\
 A\ar[r]& M}
\end{xy}
\end{equation*}
over a Lie algebroid $A$.  The degenerate Courant algebroid structure
on the core $A\oplus T^*M$ of $TA\oplus T^*A$ is here given by
$\rho_{A\oplus T^*M}(a,\theta)=\rho_A(a)$, \[\langle (a_1,\theta_1),
(a_2, \theta_2)\rangle_{A\oplus T^*M}=\langle (a_1,\theta_1),
(\rho_A,\rho_A^*)(a_2, \theta_2)\rangle\] and the bracket defined
by \[\lb (a_1,\theta_1), (a_2, \theta_2)\rb_{A\oplus
  T^*M}=([a_1,a_2],\ldr{\rho_A(a_1)}\theta_2-\ip{\rho_A(a_2)}\dr\theta_1)\]
for all $a,a_1,a_2\in\Gamma(A)$ and
$\theta,\theta_1,\theta_2\in\Omega^1(M)$. To see this, use Lemma 5.16
in \cite{Jotz13a} or the next example; this degenerate Courant
algebroid plays a crucial role in the infinitesimal description of
Dirac groupoids \cite{Jotz14}, i.e.~in the definition of \emph{Dirac
  bialgebroids}.

\end{example}
\begin{example}[LA-Courant algebroid associated to a double Lie algebroid]
  More generally, the LA-Courant algebroids (and the
  corresponding Poisson Lie $2$-algebroids) constructed in
  \S\ref{metric_double_VB_alg}, \S\ref{semi-direct},
  \S\ref{VB_Courant_alg} and \S\ref{Poi_lie_2_def_by_matched_ruth}
and Theorem \ref{core_courant} yield the following application.

A matched pair of $2$-representations as in Definition
\ref{matched_pair_2_rep} defines two degenerate Courant algebroids.
The first one is $C\oplus A^*\to M$ with the anchor $\rho_{C\oplus
  A^*}\colon C\oplus A^*\to TM$ defined by
$\rho_A\circ\pr_A\circ(\partial_A\oplus\partial_A^*)=\rho_B\circ\partial_B\circ\pr_C$.
The pairing is defined by
\[\langle (c_1,\alpha_1), (c_2,\alpha_2)\rangle_{C\oplus A^*}=\langle
\alpha_1, \partial_Ac_2\rangle+\langle \alpha_2, \partial_Ac_1\rangle
\]
for all $\alpha_1,\alpha_2\in\Gamma(A^*)$ and $c_1,c_2\in\Gamma(C)$, 
and the bracket by
\begin{equation*}
\begin{split}
  \lb (c_1,\alpha_1), (c_2,\alpha_2)\rb_{C\oplus A^*}
  &=\Delta_{(\partial_Ac_1,\partial_A^*\alpha_1)}(c_2,\alpha_2)-\nabla_{\partial_Bc_2}(c_1,\alpha_1)\\
  &=(\nabla_{\partial_Ac_1}c_2-\nabla_{\partial_Bc_2}c_1,\ldr{\partial_Ac_1}\alpha_2+\langle\nabla^*_\cdot\partial_A^*\alpha_1,c_2\rangle-\nabla^*_{\partial_Bc_2}\alpha_1)\\
  &=(\nabla_{\partial_Ac_1}c_2-\nabla_{\partial_Bc_2}c_1,\ldr{\partial_Ac_1}\alpha_2-\ip{\partial_Ac_2}\dr_A\alpha_1).
\end{split}
\end{equation*}

The second degenerate Courant algebroid is 
$C\oplus B^*\to M$ with the anchor $\rho_{C\oplus
  B^*}\colon C\oplus B^*\to TM$ defined by
$\rho_B\circ\pr_B\circ(\partial_B\oplus\partial_B^*)=\rho_B\circ\partial_B\circ\pr_C$,
the pairing defined by
\[\langle (c_1,\beta_1), (c_2,\beta_2)\rangle_{C\oplus B^*}=\langle
\beta_1, \partial_Bc_2\rangle+\langle \beta_2, \partial_Bc_1\rangle
\]
for all $\beta_1,\beta_2\in\Gamma(B^*)$ and $c_1,c_2\in\Gamma(C)$, 
and the bracket
\begin{equation*}
\begin{split}
  \lb (c_1,\beta_1), (c_2,\beta_2)\rb_{C\oplus B^*}
  &=(\nabla_{\partial_Bc_1}c_2-\nabla_{\partial_Ac_2}c_1,\ldr{\partial_Bc_1}\beta_2-\ip{\partial_Bc_2}\dr_B\beta_1).
\end{split}
\end{equation*}
Note that in both cases, the restriction to 
$\Gamma(C)$ of the Courant bracket is the Lie algebroid bracket 
induced on $C$ by the matched pair, see \S\ref{lie_bracket_on_C}.
\end{example}

\section{VB-Dirac structures, LA-Dirac structures and pseudo Dirac
  structures}\label{sec:dirac}

In this section, we study isotropic subalgebroids of metric
VB-algebroids and Dirac structures in VB- and LA-Courant
algebroids. While we paid attention in the preceding sections to
bridge $[2]$-geometric objects to geometric structures on metric
double vector bundles, we are here more interested in classifications
of Dirac structures via the simple geometric descriptions that we
found before for VB-Courant algebroids and LA-Courant algebroids.

\subsection{VB-Dirac structures}
Let $(\mathbb E; B, Q;M)$ be a VB-Courant algebroid with core $Q^*$
and anchor $\Theta\colon\mathbb E\to TB$.  Let $D$ be a double vector
subbundle structure over $B'\subseteq B$ and $U\subseteq Q$ and with
core $K$.  Choose a linear splitting $\Sigma\colon
B\times_M Q\to \mathbb E$ that is adapted\footnote{To see that such a
  splitting exist, we work with decompositions.  Since $D$ and
  $\mathbb E$ are both double vector bundles, there exist two
  decompositions $\mathbb I_D\colon B'\times_MU\times_MK \to D$ and
  $\mathbb I\colon  B\times_M Q\times_M Q^*\to\mathbb E$. Let
  $\iota\colon D\to \mathbb E$ be the double vector bundle inclusion,
  over $\iota_U\colon U\to Q$ and $\iota_{B'}\colon B'\to B$, and with
  core $\iota_K\colon K\to Q^*$.  Then there exists
  $\phi\in\Gamma({B'}^*\otimes U^*\otimes Q^*)$ such that the map
  $\mathbb I\inv\circ\iota\circ\mathbb I_D \colon
  B'\times_MU\times_MK\to B\times_M Q\times_M Q^*$ sends
  $(b_m,u_m,k_m)$ to
  $(\iota_B(b_m),\iota_U(u_m),\iota_K(k_m)+\phi(b_m,u_m))$. Using
  local basis sections of $B$ and $Q$ adapted to $B'$ and $U$ and a
  partition of unity on $M$, extend $\phi$ to $\hat\phi\in
  \Gamma(B^*\otimes Q\otimes Q^*)$. Then define a new decomposition
  $\tilde{\mathbb I}\inv\colon \mathbb E \to B\times_M Q\times_M Q^*$ by
  $\tilde{\mathbb I}\inv(e)= \mathbb
  I\inv(e)+_B(b_m,0^Q_m,-\hat\phi(b_m,q_m))=\mathbb
  I\inv(e)+_Q(0^B_m,q_m,-\hat\phi(b_m,q_m))$ for $e\in\mathbb E$ with
  $\pi_B(e)=b_m$ and $\pi_Q(e)=q_m$. Then $(\tilde{\mathbb
    I}\circ\iota\circ\mathbb
  I_D)(b_m,u_m,k_m)=(\iota_B(b_m),\iota_U(u_m),\iota_K(k_m))$ for all
  $(b_m,u_m,k_m)\in B'\times_M U\times_MK$.  The corresponding linear
  splitting $\tilde\Sigma\colon B\times _M Q\to \mathbb E$, $\tilde
  \Sigma(b_m,q_m)=\tilde{\mathbb I}(b_m,q_m,0^{Q^*}_m)$ sends
  $(\iota_{B'}(b_m),\iota_U(u_m))$ to $\iota(\mathbb
  I_D(b_m,u_m,0^K_m)\in \iota(D)$.} to $D$, i.e. such that
$\Sigma(B'\times_MU)\subseteq D$.  Then $D$ is spanned as a vector
bundle over $B'$ by the sections $\sigma_Q(u)\an{B'}$ for all
$u\in\Gamma(U)$ and $\tau^\dagger\an{B'}$ for all $\tau\in\Gamma(K)$.

The following two propositions follow immediately from the study of
linear metrics in \S\ref{Lagr_dec} and the definition in \eqref{lambda_def}.

\begin{proposition}\label{prop1}
  In the situation described above, the double subbundle $D\subseteq
  \mathbb E$ over $B'$ is isotropic if and only if $K\subseteq
  U^\circ$ and $\Lambda$ as in \eqref{lambda_def} sends $U\otimes U$
  to ${B'}^\circ$.
\end{proposition}

\begin{proposition}\label{prop2}
  In the situation described above, $D$ is maximal isotropic if and
  only if $U=K^\circ$ and $\Lambda$  sends
  $U\otimes U$ to ${B'}^\circ$.
\end{proposition}

Now we can prove that if $D$ is maximal isotropic, then there exists a \emph{Lagrangian} 
splitting of $\mathbb E$ that is adapted to $D$. 
\begin{corollary}
  Let $(\mathbb E, B; Q, M)$ be a metric double vector bundle and
  $D\subseteq \mathbb E$ a maximal isotropic double subbundle.  Then
  there exists a Lagrangian splitting that is adapted to $D$.
\end{corollary}

\begin{proof}
  As before, let $U\subseteq Q$ and $B'\subseteq B$ be the sides of
  $D$. Then by Proposition \ref{prop2} the core of $D$ is the vector
  bundle $U^\circ\subseteq Q^*$. Choose a linear splitting $\Sigma\colon
  Q\times_M B\to \mathbb E$ that is adapted to $D$. Then $D$ is
  spanned as a vector bundle over $B'$ by the sections
  $\sigma_Q(u)\an{B'}$ for all $u\in\Gamma(U)$ and
  $\tau^\dagger\an{B'}$ for all $\tau\in\Gamma(U^\circ)$.  As in the proof
  of Theorem \ref{symmetrization}, transform $\Sigma$ into a new
  Lagrangian linear splitting $\Sigma'$. We need to show that
  $\sigma_Q'(u)\an{B'}-\sigma_Q(u)\an{B'}$ is equivalent to a section
  of ${B'}^*\otimes U^\circ$ for all $u\in\Gamma(U)$. But
  $\sigma_Q'(u)-\sigma_Q(u)=\widetilde{\frac{1}{2}\Lambda(u,\cdot)}$ by
  construction and, since $D$ is isotropic, we have
  $\Lambda(u,u')\an{B'}=0$ for all $u,u'\in\Gamma(U)$.
\end{proof}

\begin{remark}\label{invariant_part}
  Consider a Courant algebroid $\mathsf E\to M$ and its tangent double
  $T\mathsf E$.  Recall from Example \ref{metric_connections} that
  Lagrangian splittings of $T\mathsf E$ are equivalent to metric
  connections $\mx(M)\times\Gamma(\mathsf E)\to \Gamma(\mathsf E)$.
  Let $\nabla$ be such a metric connection, that is adapted to a
  maximally isotropic double subbundle $D$ over the sides $TM$ and
  $U\subseteq \mathsf E$.  Define $[\nabla]\colon
  \mx(M)\times\Gamma(U)\to\Gamma(\mathsf E/U^\perp)$ by
  $[\nabla]_Xu=\overline{\nabla_Xu}\in\Gamma(\mathsf E/U^\perp)$.  A
  second metric connection $\nabla'\colon \mx(M)\times\Gamma(\mathsf
  E)\to \Gamma(\mathsf E)$ is adapted to $D$ if and only if
  $\nabla_Xu-\nabla'_Xu\in\Gamma(U^\circ)$ for all $X\in\mx(M)$ and
  for all $u\in\Gamma(U)$. Hence, if and only if $[\nabla]=[\nabla']$.
  We call $[\nabla]$ the \textbf{invariant part of the metric
    connection adapted to $D$}.
\end{remark}

The existence of Lagrangian splittings of $\mathbb E$ adapted to
maximal isotropic double subbundles $D$ will now be used to study the
involutivity of $D$.  Note that in a very early version of this work,
we studied VB-Courant algebroids via general (not necessarily
Lagrangian) linear splittings. We found some more general objects than
Dorfman $2$-representations; basically Dorfman $2$-connections with
the additional structure object $\Lambda$ appearing in (D1) to (D5) in
Definition \ref{def_dorfman_2_conn}.  The study of the involutivity of
general (not necessarily isotropic) double subbundles $D$ of $\mathbb
E$ was then possible, and yielded very similar results.

\begin{proposition}\label{char_VB_dir}
  Let $(\mathbb E, B; Q, M)$ be a VB-Courant algebroid and $D\subseteq
  \mathbb E$ a maximal isotropic double subbundle.  Choose a
  Lagrangian splitting of $\mathbb E$ that is adapted to $D$ and
  consider the corresponding Dorfman $2$-representation, denoted as
  usual.  Then $D$ is a Dirac structure in $\mathbb E$ with support
  $B'$ if and only if
\begin{enumerate}
\item $\partial_B(U^\circ)\subseteq B'$,
\item $\nabla_ub\in\Gamma(B')$ for all $u\in\Gamma(U)$ and $b\in\Gamma(B')$,
\item $\lb u_1, u_2\rb\in\Gamma(U)$ for all $u_1,u_2\in\Gamma(U)$,
% \item $\Delta_u\tau\in\Gamma(U^\circ)$ for all $u\in\Gamma(U)$ and
%   $\tau\in\Gamma(U^\circ)$,
\item $R(u_1,u_2)$ restricts to a section of  $\Gamma(\operatorname{Hom}(B',U^\circ))$ for all
  $u_1,u_2\in\Gamma(U)$.
\end{enumerate}
\end{proposition}
A Dirac double subbundle $D$ of a VB-Courant algebroid $\mathbb E$ as
in the proposition is called a \textbf{VB-Dirac structure}.

\begin{proof}
  This is easy to prove using Lemma \ref{useful_for_dirac_w_support}
  on sections $\sigma_Q(u)$ and $\tau^\dagger$, for $u\in\Gamma(U)$
  and $\tau\in\Gamma(U^\circ)$.  Their
  anchors and Courant brackets can be described by
\begin{equation}\label{ruth_D_to_B}
\begin{split}
\Theta(\sigma_Q(u))&=\widehat{\nabla_u}\in\mx^l(B),\qquad 
\Theta(\tau^\dagger)=(\partial_B\tau)^\uparrow\in\mx^c(B),\\
\lb \sigma_Q(u_1),\sigma_Q(u_2)\rb&=\sigma_Q(\lb u_1, u_2\rb)-\widetilde{R(u_1,u_2)},\\
\lb \sigma_Q(u),\tau^\dagger\rb&=(\Delta_u\tau)^\dagger,\qquad \lb \tau_1^\dagger, \tau_2^\dagger\rb=0
\end{split}
\end{equation}
for all $u,u_1,u_2\in\Gamma(U)$ and
$\tau,\tau_1,\tau_2\in\Gamma(U^\circ)$.  The vector field
$\widehat{\nabla_u}$ is tangent to $B'$ on $B'$ if and only if for all
$\beta\in\Gamma((B')^\circ)$, $\widehat{\nabla_u}(\ell_\beta)=\ell_{\nabla_u^*\beta}$
vanishes on $B'$. That is, if and only if, for all
$\beta\in\Gamma((B')^\circ)$, $\nabla_u^*\beta$
is again a section of $(B')^\circ$. This yields (2).
The vector field $(\partial_B\tau)^\uparrow$ is tangent to $B'$
if and only if $\partial_B\tau\in\Gamma(B')$. This yields (1).
Next, $\lb \sigma_Q(u),\tau^\dagger\rb=(\Delta_u\tau)^\dagger$
is a section of $D$ over $B'$ if and only if $\Delta_u\tau\in\Gamma(U^\circ)$.
Since $\Delta_u\tau\in\Gamma(U^\circ)$ for all $u\in\Gamma(U)$ and
  $\tau\in\Gamma(U^\circ)$ if and only if $\lb u_1,
  u_2\rb\in\Gamma(U)$ for all $u_1,u_2\in\Gamma(U)$, this is
  (3). Further, $\sigma_Q(\lb u_1, u_2\rb)$ takes then values in $D$
  over $B'$, and so $\lb \sigma_Q(u_1),\sigma_Q(u_2)\rb$ takes values in $D$
  over $B'$ if and only if $R(u_1,u_2)$ restricts to a morphism $B'\to
  U^\circ$. This is (4).
\end{proof}

We get the following result for ordinary VB-Dirac structures (with support $B$) in $\mathbb E$.

\begin{corollary}\label{LA_on_U} Let $(\mathbb E, B; Q, M)$ be a VB-Courant algebroid and
  $(D,B,U,M)\subseteq \mathbb E$ a maximal isotropic double subbundle.
  Choose a Lagrangian splitting of $\mathbb E$ that is adapted to $D$
  and consider the corresponding Dorfman $2$-representation, denoted
  as usual.  If $D$ is a Dirac structure in $\mathbb E\to B$, then $U$
  inherits a Lie algebroid structure with bracket $\lb\cdot\,,\cdot\rb
  \an{\Gamma(U)\times\Gamma(U)}$ and anchor $\rho_Q\an{U}$. This Lie
  algebroid structure does not depend on the choice of Lagrangian
  splitting.
\end{corollary}

\begin{proof}
  By (3) in Proposition \ref{char_VB_dir}, $\lb\cdot\,,\cdot\rb$
  restricts to a bracket on sections of $U$.  For $u_1,u_2,u_3$,
  $\operatorname{Jac}_{\lb\cdot\,,\cdot\rb}(u_1,u_2,u_3)=\partial_B^*\omega_R(u_1,u_2,u_3)=0$
  since $\omega_R(u_1,u_2,u_3)=\langle R(u_1,u_2),
  u_3\rangle=0\in\Gamma(B^*)$ by (4) in Proposition \ref{char_VB_dir}.
  Hence, $U$ with the bracket
  $\lb\cdot\,,\cdot\rb\an{\Gamma(U)\times\Gamma(U)}$ and the anchor
  $\rho_Q\an{U}$ is a Lie algebroid.

If $\phi_{12}\in\Gamma(Q^*\wedge Q^*\otimes B^*)$ is the tensor
defined as in \S\ref{subsub:lsl}
by a change of Lagrangian splitting adapted to $D$, then, by 
Proposition \ref{change_of_lift}, 
\[\lb u, u'\rb_1=\lb u, u'\rb_2+\partial_B^*\phi_{12}(u,u')
\]
for all $u,u'$. But since both splittings $\Sigma^1,\Sigma^2\colon
B\times_M Q\to \mathbb E$ are adapted to $D$, we know that
$\sigma^1_Q(u)$ and $\sigma^2_Q(u)$ have values in $D$, and their
difference $\sigma_Q^1(u)-\sigma_Q^2(u)=\widetilde{\phi_{12}(u)}$ is a
core-linear section of $D\to B$. Hence it must takes values in
$U^\circ$, and $\phi_{12}(u,u')$ must so vanish for all
$u,u'\in\Gamma(U)$. As a consequence, 
$\lb u, u'\rb_1=\lb u, u'\rb_2$.
\end{proof}

The following two corollaries are now easy to prove.
(The first one was already found in \cite{Li-Bland12}.)
\begin{corollary}\label{Lie_1_sub}
Let $(\mathcal M,\mathcal Q)$ be a Lie $2$-algebroid, and 
$(\mathbb E\to B,Q\to M)$ the corresponding VB-Courant algebroid.
Then VB-Dirac structures in $\mathbb E$ are equivalent to 
wide Lie $1$-subalgebroids of $(\mathcal M,\mathcal Q)$.
\end{corollary}

\begin{proof}
  A wide Lie subalgebroid of $(\mathcal M, \mathcal Q)$ is a wide
  $[1]$-submanifold $U[-1]$ of $\mathcal M$ such that $\mathcal
  Q_U(\mu^\star\xi)=\mu^\star(\mathcal Q(\xi))$, $\xi\in
  C^\infty(\mathcal M)$, defines a Lie algebroid structure $\mathcal
  Q_U$ on $U$. Here, $\mu\colon U[-1]\to \mathcal M$ is the
  submanifold inclusion as in Definition \ref{def_1_subman}.

  In a splitting $Q[-1]\oplus B^*[-2]$ of $\mathcal M$, the
  homological vector field $\mathcal Q$ is given by
  \eqref{Q_in_coordinates}. Assume that the choice of local basis
  vector fields is as in the discussion after Definition
  \ref{def_1_subman}. Then we have $\mathcal
  Q_U(f)=\mu^\star(\rho_Q^*\dr f)=\rho_Q^*\dr f+ U^\circ$ for all
  $f\in C^\infty(M)$. This translates easily to
  $\rho_U=\rho_Q\an{U}$. Then we have $\mathcal
  Q_U(\tau_i+U^\circ)=\mathcal Q_U(\mu^\star\tau_i)=
  \mu^\star(\mathcal Q(\tau_k))=-\sum_{i<j}^r\langle \lb u_i,u_j\rb,
  \tau_k\rangle\bar\tau_i\bar\tau_j$ for $k=1,\ldots,r$.  This shows
  that the bracket on $U$ must be the restriction to $\Gamma(U)$ of
  the dull bracket on $\Gamma(Q)$. Finally $0=\mathcal Q_U(\mu^\star
  b_l)= \mu^\star(\mathcal
  Q(b_l))=-\sum_{i<j<k<r}\omega_R(u_i,u_j,u_k)(b_l)\bar\tau_i\bar\tau_j\bar\tau_k$
  for all $l$ shows that $\omega_R(u_1,u_2,u_3)$ must be zero for all
  $u_1,u_2,u_3\in\Gamma(U)$. This is equivalent to (3) in Proposition
  \ref{char_VB_dir} (with $B'=B$). Note that since $B'=B$, (1) and (2) in Proposition
  \ref{char_VB_dir} are trivially satisfied. Hence we can conclude.
\end{proof}

\begin{corollary}\label{VB_dir_is_VB_alg}
  A VB-Dirac structure $(D,B;U,M)$ in a VB-Courant algebroid inherits
  a linear Lie algebroid structure $(D\to B, U\to M)$.
\end{corollary}

\begin{proof}

\end{proof}

\subsection{LA-Dirac structures}\label{LA-Dirac}
Assume now that $(\mathbb E\to Q;B\to M)$ is a metric VB-algebroid,
and take a maximal isotropic double subbundle $D$ of $\mathbb E$ over
the sides $U\subseteq Q$ and $B'\subseteq B$. We will study conditions
on the self-dual $2$-representation defined by a Lagrangian splitting
and the linear Lie algebroid structure on $\mathbb E\to Q$, and on $Q$
and on $B'$, for $D$ to be an isotropic subalgebroid of $\mathbb E\to
Q$ over $U$. 

Note the similarity of the following result with Proposition \ref{char_VB_dir}.
\begin{proposition}\label{char_sub_alg}
  Let $(\mathbb E, B; Q, M)$ be a metric VB-algebroid and
  $(D,B';U,M)\subseteq \mathbb E$ a maximal isotropic double
  subbundle.  Choose a Lagrangian splitting of $\mathbb E$ that is
  adapted to $D$ and consider the corresponding self-dual
  $2$-representation, denoted as usual.  Then $D\to U$ is a
  subalgebroid of $\mathbb E\to Q$ if and only if
\begin{enumerate}
\item $\partial_Q(U^\circ)\subseteq U$,
\item $\nabla_bu\in\Gamma(U)$ for all $u\in\Gamma(U)$ and $b\in\Gamma(B')$,
\item $[b_1, b_2]\in\Gamma(B')$ for all $b_1,b_2\in\Gamma(B')$,
% \item $\Delta_u\tau\in\Gamma(U^\circ)$ for all $u\in\Gamma(U)$ and
%   $\tau\in\Gamma(U^\circ)$,
\item $R(b_1,b_2)$ restricts to a section of
  $\Gamma(\operatorname{Hom}(U,U^\circ))$ for all
  $b_1,b_2\in\Gamma(B')$.
\end{enumerate}
\end{proposition}
\begin{proof}
  This proof is very similar to the proof of Proposition
  \ref{char_VB_dir}, and left to the reader.
\end{proof}

Now let  $(\mathbb E; Q, B; M)$ be an LA-Courant algebroid.
A VB-Dirac structure $(D;U,B';M)$ in $\mathbb E$ is an \textbf{LA-Dirac structure} 
if $(D\to U, B'\to M)$ is also a subalgebroid of $(\mathbb E\to Q;B\to M)$.
We deduce from Propositions \ref{char_VB_dir} and \ref{char_sub_alg}
a characterisation of LA-Dirac structures.

\begin{proposition}\label{char_LA_D}
  Let $(\mathbb E, B; Q, M)$ be an LA-Courant algebroid and
  $(D,B';U,M)$ a maximal isotropic double
  subbundle of $\mathbb E$.  Choose a Lagrangian splitting of $\mathbb E$ that is
  adapted to $D$ and consider the corresponding matched self-dual
  $2$-representation and Dorfman $2$-representation.  Then $D\to U$ is
  an LA-Dirac structure in $\mathbb E$ if and only if
\begin{enumerate}
\item $\partial_B(U^\circ)\subseteq B'$ and $\partial_Q(U^\circ)\subseteq U$,
\item $\nabla_ub\in\Gamma(B')$ for all $u\in\Gamma(U)$ and $b\in\Gamma(B')$,
\item $\nabla_bu\in\Gamma(U)$ for all $u\in\Gamma(U)$ and $b\in\Gamma(B')$,
\item $\lb u_1, u_2\rb\in\Gamma(U)$ for all $u_1,u_2\in\Gamma(U)$,
\item $[b_1, b_2]\in\Gamma(B')$ for all $b_1,b_2\in\Gamma(B')$,
% \item $\Delta_u\tau\in\Gamma(U^\circ)$ for all $u\in\Gamma(U)$ and
%   $\tau\in\Gamma(U^\circ)$,
\item $R(u_1,u_2)$ restricts to a section of  $\Gamma(\operatorname{Hom}(B',U^\circ))$ for all
  $u_1,u_2\in\Gamma(U)$,
\item $\nabla_bu\in\Gamma(U)$ for all $u\in\Gamma(U)$ and $b\in\Gamma(B')$,
\item $[b_1, b_2]\in\Gamma(B')$ for all $b_1,b_2\in\Gamma(B')$,
% \item $\Delta_u\tau\in\Gamma(U^\circ)$ for all $u\in\Gamma(U)$ and
%   $\tau\in\Gamma(U^\circ)$,
\item $R(b_1,b_2)$ restricts to a section of
  $\Gamma(\operatorname{Hom}(U,U^\circ))$ for all
  $b_1,b_2\in\Gamma(B')$.
\end{enumerate}
\end{proposition}

Hence, we also have the following result.
\begin{corollary}
  VB-subalgebroids $(D\to U, B\to M)$ of a metric VB-algebroid
  $(\mathbb E\to Q,B\to M)$ are equivalent to wide coisotropic
  $[1]$-submanifolds of the corresponding Poisson $[2]$-manifold.

  LA-Dirac structures $(D\to U, B\to M)$ in an LA-Courant algebroid
  $(\mathbb E\to Q,B\to M)$ are equivalent to wide coisotropic Lie
  subalgebroids of the corresponding Poisson Lie 2-algebroid.
\end{corollary}

\begin{proof}
  Let $U[-1]$ be a $[1]$-submanifold of a Poisson $[2]$-manifold
  $(\mathcal M, \{\cdot\,,\cdot\})$.  Then $U[-1]$ is coisotropic if
  and only if $\mu^\star(\xi)=\mu^\star(\eta)=0$ imply
  $\mu^\star(\{\xi,\eta\})=0$ for all $\xi,\eta\in C^\infty(\mathcal
  M)$, where $\mu\colon Q[-1]\to\mathcal M$ is the inclusion as in
  Definition \ref{def_1_subman}.  In a local splitting, we find easily
  that this implies $\partial_Q(U^\circ)\subseteq U$,
  $\nabla_b^*\tau\in\Gamma(U^\circ)$ for all $b\in\Gamma(B)$ and
  $\tau\in\Gamma(U^\circ)$, and the restriction to $U$ of $R(b_1,b_2)$
  has image in $U^\circ$. By Proposition \ref{char_sub_alg}, we can
  conclude.

  The second claim follows with Corollary \ref{Lie_1_sub}.
\end{proof}

As a corollary of Theorem \ref{LA-Courant}, Proposition \ref{char_VB_dir} and
Proposition \ref{char_sub_alg}, we get the following theorem.  
\begin{theorem}
  Let $(\mathbb E, B; Q, M)$ be an LA-Courant algebroid and
  $(D,U;B,M)\subseteq \mathbb E$ a (wide) LA-Dirac structure in $\mathbb E$.

  Then $D$ is a double Lie algebroid with the VB-algebroid structure
  in Corollary \ref{VB_dir_is_VB_alg} and the VB-algebroid structure
  $(D\to U, B\to M)$.
\end{theorem}

\begin{proof}
  Let us study the two linear Lie algebroid structures on $D$. Choose
  as before a linear splitting $\Sigma\colon B\times_MQ\to \mathbb E$
  that restricts to a linear splitting $\Sigma_D\colon U\times_MB\to D$
  of $D$. The LA-Courant algebroid structure of $\mathbb E$ is then
  encoded as in Theorems \ref{main} and \ref{rajan}, respectively, by
  a Dorfman $2$-representation $(\Delta,\nabla,R)$ of $(Q,\rho_Q)$ on
  $\partial_B\colon Q^*\to B$ and by a self-dual $2$-representation
  $(\nabla,\nabla^*,R)$ of the Lie algebroid $B$ on
  $\partial_Q=\partial_Q^*\colon Q^*\to Q$. By Theorem
  \ref{LA-Courant}, the Dorfman $2$-representation and the
  $2$-representation form a matched pair as in Definition
  \ref{matched_pairs}.

  By Proposition \ref{char_VB_dir} and Corollary \ref{LA_on_U}, the
  restriction to $\Gamma(U)$ of the dull bracket on $\Gamma(Q)$ that
  is dual to $\Delta$ defines a Lie algebroid structure on $U$,
  $R\an{U\otimes U}$ can be seen as an element of
  $\Omega^2(U,\operatorname{Hom}(B,U^\circ))$ and since
  $\Delta_u\tau\in\Gamma(U^\circ)$ for all $u\in\Gamma(U)$ and
  $\tau\in\Gamma(U^\circ)$, the Dorfman connection $\Delta$ restricts
  to a map
  $\Delta^D\colon\Gamma(U)\times\Gamma(U^\circ)\to\Gamma(U^\circ)$.
  Since $\Delta^D_u(f\tau)=f\Delta^D_u\tau+\rho_Q(u)(f)\tau$ and
  $\Delta^D_{fu}\tau=f\Delta^D_u\tau+\langle u,\tau\rangle\rho_Q^*\dr
  f=f\Delta_u\tau$ for $f\in C^\infty(M)$, we find that this
  restriction is in fact an ordinary connection.  Since
  $R(u_1,u_2)^*u_3$ then vanishes for all $u_1,u_2,u_3\in\Gamma(U)$,
  it is then easy to see that the restrictions of (D1), (D4) and (D6)
  to sections of $U$ and $U^\circ$ define a ordinary
  $2$-representation.  By \eqref{ruth_D_to_B}, this $2$-representation
  $(\partial_B\colon U^\circ \to B,
  \nabla,\Delta^D,R\an{\Gamma(U)\times\Gamma(U)})$ of the Lie
  algebroid $U$ on $\partial_B\colon U^\circ \to B$ encodes the
  VB-algebroid structure that $D\to B$ inherits from the Courant
  algebroid $\mathbb E\to B$.

  In a similar manner, we find using Proposition \ref{char_sub_alg}
  that the self-dual $2$-representation $(\partial_Q\colon Q^*\to Q,
  \nabla, \nabla^*, R\in\Omega^2(B,Q^*\wedge Q^*))$ restricts to a
  $2$-representation $(\partial_B\colon U^\circ \to U, \nabla^U\colon
  \Gamma(B)\times\Gamma(U)\to\Gamma(U), \nabla^{U^\circ}\colon
  \Gamma(B)\times\Gamma(U^\circ)\to\Gamma(U^\circ),
  R\in\Omega^2(B,\operatorname{Hom}(U,U^\circ)))$ of $B$.

  A study of the restrictions to sections of $U$ and $U^\circ$ of the
  equations in Definition \ref{matched_pairs} shows then that (M1)
  restricts to (2) in Definition \ref{matched_pair_2_rep} since
  $\partial_B^*\langle \tau, \nabla^U\cdot u\rangle=0$ for all
  $u\in\Gamma(U)$ and $\tau\in\Gamma(U^\circ)$.  The equations (M2),
  and (M3) and immediately (3) and (6), respectively.  (M4) restricts
  to (5) since $\langle R(\cdot, b)u_1, u_2\rangle=0$ for all
  $u_1,u_2\in\Gamma(U)$ and $b\in\Gamma(B)$. (M5) restricts to (7)
  since the right-hand side of (M5) in (1) of Remark
  \ref{remark_simplifications} vanishes.  Finally, \eqref{almost_C}
  restricts to (1) and \eqref{(LC10)} restricts to (4) since $\langle
  \nabla_{\nabla_\cdot b}u, \tau\rangle =0$ for all $b\in \Gamma(B)$,
  $u\in\Gamma(U)$ and $\tau\in\Gamma(U^\circ)$.  Thus, the two
  $2$-representations describing the sides of $D$ given the splitting
  $\Sigma_D$ form a matched pair, which implies that $D$ is a double
  Lie algebroid (see \cite{GrJoMaMe14} or
  \S\ref{matched_pair_2_rep_sec} for a quick summary of this paper).
\end{proof}

Note finally that with a different approach as the one adopted in this
paper, we could deduce the main result in \cite{GrJoMaMe14} from our
Theorem \ref{LA-Courant}. If we had shown without the use of these
results that for each double Lie algebroid $(D, A, B,M)$ with core
$C$, the direct sum over $B$ of $D$ and $D\duer B$ defines an
LA-Courant algebroid $(D\oplus_B(D\duer B), A\oplus C^*, B,M)$ as in
\S\ref{Poi_lie_2_def_by_matched_ruth}, then we could use the last
Theorem to deduce the equations in Definition \ref{matched_pair_2_rep}
from the ones in Definition \ref{matched_pairs} and in Remark
\ref{remark_simplifications}: by construction, the double vector
subbundle $D$ of $D\oplus_B(D\duer B)$ is a VB-Dirac structure in
$D\oplus_B(D\duer B)\to B$ and a linear Lie subalgebroid in
$D\oplus_B(D\duer B)\to A\oplus C^*$.  We have chosen to use the main
theorem in \cite{GrJoMaMe14} to prove that $(D\oplus_B(D\duer B),
A\oplus C^*, B,M)$ is an LA-Courant algebroid, see
\S\ref{Poi_lie_2_def_by_matched_ruth}.

\subsection{Pseudo-Dirac structures}
We explain here the notion of pseudo-Dirac structures that was
introduced in \cite{Li-Bland12,Li-Bland14} and we compare it with our approach to
VB- and LA-Dirac structures in the tangent of a Courant algebroid. Consider a
VB-Courant algebroid
\begin{equation*}
\begin{xy}
\xymatrix{\mathbb E\ar[r]\ar[d]& Q \ar[d]\\
B\ar[r]&M
}
\end{xy}
\end{equation*}
with core $Q^*$, and a double vector subbundle
\begin{equation*}
\begin{xy}
\xymatrix{D\ar[r]\ar[d]& U\ar[d]\\
B\ar[r]&M
}
\end{xy}
\end{equation*}
in $\mathbb E$ with core $K$. 
Consider the restriction $\mathbb E\an{U}$ of $\mathbb E$ to $U$; 
i.e. $\mathbb E\an{U}=\pi_Q\inv(U)$. 
This is a double vector bundle 
\begin{equation*}
\begin{xy}
\xymatrix{\mathbb E\an{U}\ar[r]\ar[d]& U \ar[d]\\
B\ar[r]&M
}
\end{xy}
\end{equation*}
with core $Q^*$. The \textbf{total quotient} of $\mathbb E\an{U}$ by $D$ 
is the map $\mathsf q$ from 
\begin{equation*}
\begin{xy}
\xymatrix{\mathbb E\an{U}\ar[r]\ar[d]& Q \ar[d]\\
B\ar[r]&M
}
\end{xy}
\quad \text{
to
}
\quad
\begin{xy}
\xymatrix{Q^*/K\ar[r]\ar[d]& 0^M \ar[d]\\
0^M\ar[r]&M,
}
\end{xy}
\end{equation*}
defined by
\[\mathsf q(e)=\bar\tau \Leftrightarrow e-\tau^\dagger\in D.
\]
After the choice of a linear splitting of $\mathbb E$ that is adapted
to $D$, we know that each element of $\mathbb E\an{U}$ can be written
$\sigma_Q(u)(b_m)+\tau^\dagger(b_m)$ for some $u\in\Gamma(U)$,
$\tau\in\Gamma(Q^*)$ and $b_m\in B$. The image of $\sigma_Q
(u)(b_m)+\tau^\dagger(b_m)$ under $\mathsf q$ is then simply
$\bar\tau(m)$. It is easy to see that $D$ can be recovered from
$\mathsf q$.  Recall that if
$e_1=\sigma_Q(u_1)(b_m)+\tau_1^\dagger(b_m)$ and
$e_2=\sigma_Q(u_2)(b_m)+\tau_2^\dagger(b_m)\in \mathbb E$, then
\begin{equation*}
\begin{split}
  \langle e_1, e_2\rangle&=\langle \sigma_Q(u_1)(b_m)+\tau_1^\dagger(b_m), \sigma_Q(u_2)(b_m)+\tau_2^\dagger(b_m)\rangle\\
  &=\ell_{\Lambda(u_1,u_2)}(b_m)+\langle u_1(m),
  \tau_2(m)\rangle+\langle u_2(m), \tau_1(m)\rangle.
\end{split}
\end{equation*}
 In
particular,
\begin{equation*}
\begin{split}
  \langle e_1, e_2\rangle=\langle \pi_Q(e_1), \mathsf q(e_2)\rangle+\langle
  \pi_Q(e_2), \mathsf q(e_1)\rangle
\end{split}
\end{equation*}
for all $e_1, e_2\in \mathbb E\an{U}$ if and only if $\Lambda\an{U\otimes U}$ vanishes
and $K=U^\circ$, i.e. if and only if $D$ is maximal isotropic
(Proposition \ref{prop2}). 

\medskip
Now we recall Li-Bland's definition of a pseudo-Dirac structure \cite{Li-Bland14}.
\begin{definition}\label{davids_pseudo}
  Let $\mathsf E\to M$ be a Courant algebroid. A pseudo-Dirac
  structure is a pair $(U,\nabla^p)$ consisting of a subbundle
  $U\subseteq \mathsf E$ together with a map $\nabla^p\colon
  \Gamma(U)\to\Omega^1(M,U^*)$ satisfying
\begin{enumerate}
\item $\nabla^p(fu)=f\nabla^p u+\dr f\otimes\langle u,\cdot\rangle$,
\item $\dr\langle u_1,u_2\rangle=\langle\nabla^p u_1, u_2\rangle+\langle u_1, \nabla^p u_2\rangle$,
\item $\lb u_1,u_2\rb_p:=\lb u_1, u_2\rb_{\mathsf E}-\rho^*\langle
  \nabla^p u_1, u_2\rangle$ defines a bracket
  $\Gamma(U)\times\Gamma(U)\to\Gamma(U)$,
\item and 
\begin{multline}\label{alternative_curvature}
(\langle \lb u_1, u_2\rb_p,\nabla^p u_3\rangle+\ip{\rho(u_1)}\dr\langle\nabla^p u_2, u_3\rangle)+{\rm c.p.}\\
+\dr\left(\left\langle \nabla^p_{\rho(u_1)}u_2-\nabla^p_{\rho(u_2)}u_1, u_3
\right\rangle-\left\langle\lb u_1,u_2\rb_p, u_3
\right\rangle\right)=0
\end{multline}
\end{enumerate}
for all $u_1,u_2,u_3\in\Gamma(U)$ and $f\in C^\infty(M)$.
\end{definition}
Consider the tangent double $(T\mathsf E, TM, \mathsf E, M)$ where
$\mathsf E$ is a Courant algebroid over $M$. Choose a linear (wide)
Dirac structure $D$ in $T\mathsf E$, over the side $U\subseteq \mathsf
E$ and a metric connection $\nabla\colon\mx(M)\times\Gamma(\mathsf
E)\to \Gamma(\mathsf E)$ that is adapted to $D$. Li-Bland define the
pseudo-Dirac structure associated to $D$ \cite{Li-Bland14} as the map
$\nabla^p\colon\Gamma(U)\to\Gamma(\operatorname{Hom}(TM, U^*))$ that
is defined by $\nabla^pu=\mathsf q\circ Tu$ for all
$u\in\Gamma(U)$. By definition of $\sigma^\nabla_{\mathsf E}$, we have
$Tu=\sigma^\nabla_{\mathsf E}(u)+\widetilde{\nabla_\cdot u} $ and we
find that
$\nabla^pu(v_m)=\overline{\nabla_{v_m}u}=[\nabla]_{v_m}u$. The
pseudo-Dirac structure is nothing else than the invariant part of the
metric connection that is adapted to $D$ (Remark
\ref{invariant_part}).  Condition (2) in Definition
\ref{davids_pseudo} is then
\begin{equation}\label{pseudo}
\dr\langle u_1,u_2\rangle
=\langle Tu_1, Tu_2\rangle_{T\mathsf E}
=\langle u_1, \nabla^pu_2\rangle+\langle u_2 , \nabla^pu_1\rangle
\end{equation}
for all $u_1,u_2\in\Gamma(U)$ and Condition (1) is
\begin{equation}\label{pseudo2}
  \nabla^p(\varphi\cdot u)=\overline{\nabla_\cdot (\varphi\cdot u)}
  =\varphi\cdot \overline{\nabla_\cdot u}+\dr\varphi\otimes \overline{u}
=\varphi\cdot\nabla^pu+\dr\varphi\otimes \overline{u}.
\end{equation}
The bracket $\lb\cdot\,,\cdot\rb_p$ is then 
\[\lb u_1, u_2\rb_p=\lb u_1, u_2\rb_{\mathsf
  E}-\rho^*\langle\nabla^pu_1,u_2\rangle=\lb u_1, u_2\rb_{\mathsf
  E}-\rho^*\langle\nabla_\cdot u_1,u_2\rangle=\lb u_1, u_2\rb_\nabla,
\]
the bracket defined in \S\ref{adjoint}.  Finally, a straightforward
computation shows that the left-hand side of
\eqref{alternative_curvature} equals $R_\Delta^{\rm
  bas}(u_1,u_2)^*u_3\in\Gamma(B^*)$, which is zero by Proposition
\ref{char_VB_dir}.  Li-Bland proves that the bracket
$\lb\cdot\,,\cdot\rb_p$ defines a Lie algebroid structure on $U$. More
explicitly, he finds that the left-hand side $\Psi(u_1,u_2,u_3 )$ of
\eqref{alternative_curvature} defines a tensor
$\Psi\in\Omega^3(U,T^*M)$ that is related as follows to the Jacobiator of
$\lb\cdot\,,\cdot\rb_p$:
$\operatorname{Jac}_{\lb\cdot\,,\cdot\rb_p}=(\Beta\inv\circ\rho_{\mathsf
  E}^*)\Psi$. He proves so that (wide) linear Dirac structures in
$T\mathsf E$ are in bijection with pseudo-Dirac structures on $\mathsf
E$. Hence, our result in Proposition \ref{char_VB_dir} is a
generalisation of Li-Bland's result to linear Dirac structures in
general VB-Courant algebroids.

Further, our Theorem \ref{char_sub_alg} can be formulated as follows
in Li-Bland's setting. 
\begin{theorem}
  In the correspondence of linear Dirac structures with pseudo-Dirac
  connections in \cite{Li-Bland14}, LA-Dirac structures correspond
  to pseudo-Dirac connections $(U,\nabla^p)$ such that
\begin{enumerate}
\item $U\subseteq \mathsf E$ is an isotropic (or 'quadratic') subbundle,
  i.e.~$U^\perp\subseteq U$,
\item $\nabla^p$ sends $U^\perp$ to zero and so, by Condition (2) in
  Definition \ref{davids_pseudo}, has image in $U/U^\perp\subseteq
  \mathsf E/U^\perp\simeq U^*$,
\item the induced ordinary connection $\overline{\nabla^p}\colon
  \Gamma(U/U^\perp)\to\Omega^1(M,U/U^\perp)$ is flat.
\end{enumerate} 
\end{theorem}
We propose to call these pseudo-Dirac connections \textbf{quadratic
  pseudo-Dirac connections}.  Note that $\overline{\nabla^p}$ equals
$\bar\nabla\colon\mx(M)\times\Gamma(U/U^\perp)\to\Gamma(U/U^\perp)$
$\bar\nabla_X\bar u=\overline{\nabla_Xu}$, $u\in\Gamma(U)$ and
$X\in\mx(M)$, for any metric connection
$\nabla\colon\mx(M)\times\Gamma(U)\to\Gamma(U)$ such that
$[\nabla]=\nabla^p$. Such a connection must preserve $U$ by Condition
(2) in Proposition \ref{char_sub_alg}, and so also $U^\perp$ since it
is metric. The condition $R_\nabla(X_1,X_2)u\in\Gamma(U^\perp)$ for
all $X_1,X_2\in\mx(M)$ and $u\in\Gamma(U)$ in Proposition
\ref{char_sub_alg} is then equivalent to $R_{\bar\nabla}=0$.

\subsection{The Manin pair associated to an LA-Dirac structure}
Consider an LA-Courant algebroid 
\begin{equation*}
\begin{xy}
  \xymatrix{\mathbb E\ar[r]\ar[d]&Q\ar[d]\\
    B\ar[r]&M }
\end{xy}
\end{equation*}
with core $Q^*$, and an LA-Dirac structure 
\begin{equation*}
\begin{xy}
\xymatrix{D\ar[r]\ar[d]&U\ar[d]\\
B\ar[r]&M
}
\end{xy}
\end{equation*}
in $\mathbb E$ with core $U^\circ$.  Since $\partial_Q$ restricts to a
map from $U^\circ $ to $U$, we can define the vector bundle
\[\mathbb B=\frac{U\oplus
  Q^*}{\operatorname{graph}(-\partial_Q\an{U^\circ})}\to M.
\]
This vector bundle is anchored by the map
\[\rho_{\mathbb B}\colon\mathbb B\to TM, \qquad \rho_{\mathbb B}(u\oplus
\tau)=\rho_Q(u+\partial_Q\tau)=\rho_Q(u)+\rho_B(\partial_B\tau).
\]
Note that this map is well-defined because 
\[ \rho_{\mathbb
  B}(-\partial_Q\tau\oplus\tau)=\rho_Q(-\partial_Q\tau+\partial_Q\tau)=0
\]
for all $\tau\in U^\circ$.  We will show that there is a natural
symmetric non-degenerate pairing $\langle \cdot\,,\cdot\rangle_{\mathbb
  B}$ on $\mathbb B\times_M\mathbb B$ and a natural bracket
$\lb\cdot\,,\cdot\rb_{\mathbb B}$ on $\Gamma(\mathbb B)$ such that
\[ (\mathbb B\to M, \rho_{\mathbb B}, \langle
\cdot\,,\cdot\rangle_{\mathbb B}, \lb\cdot\,,\cdot\rb_{\mathbb B})
\]
is a Courant-algebroid.

We define the pairing on $\mathbb B$ by
\begin{align*}
\langle u_1\oplus\tau_1, u_2\oplus\tau_2\rangle_{\mathbb B}
=\langle u_1,\tau_1\rangle +\langle u_2,\tau_2\rangle +\langle
\tau_1, \partial_Q\tau_2\rangle.
\end{align*}
It is easy to check that this pairing is well-defined and
non-degenerate and that the
induced 
map $\mathcal D_{\mathbb B}: C^\infty(M)\to\Gamma(\mathbb B)$
given by 
\[\langle \mathcal D_{\mathbb B}f, u\oplus\tau\rangle_{\mathbb
  F}=\rho_{\mathbb B}(u\oplus\tau)(f)\]
can alternatively be defined by 
$ \mathcal D_{\mathbb B}f=0\oplus \rho_Q^*\dr f$.

Choose as before a Lagrangian splitting of $\mathbb E$ that is adapted
to $D$, and recall that the linear Courant algebroid structure and the
linear Lie algebroid structure on $\mathbb E$ are then encoded by a
Dorfman $2$-representation and a self-dual $2$-representation,
respectively, both denoted as usual.  We define the bracket on
$\Gamma(\mathbb B)$ by
\begin{equation}\label{C_bracket}
\begin{split}
  &\lb u_1\oplus\tau_1, u_2\oplus\tau_2\rb_{\mathbb B}\\
  =&(\lb u_1,
  u_2\rb_U+\nabla_{\partial_B\tau_1}u_2-\nabla_{\partial_B\tau_2}u_1 )
  \oplus (\lb \tau_1,
  \tau_2\rb_{Q^*}+\Delta_{u_1}\tau_2-\Delta_{u_2}\tau_1+\rho_Q^*\dr\langle
  \tau_1, u_2\rangle ).
\end{split}
\end{equation}
A quick computation using Remark \ref{change} and Proposition
\ref{change_of_lift} show that this bracket does not depend on the
choice of Lagrangian splitting.

\begin{theorem}\label{thm_CA_sd}
  Let $(D;U,B;M)$ be an LA-Dirac structure in a LA-Courant algebroid
  $(\mathbb E;Q,B;M)$.  Then the vector bundle \[ \mathbb
  B=\frac{U\oplus
    Q^*}{\operatorname{graph}(-\partial_Q\an{U^\circ})}\to M,\] with
  the anchor $\rho_{\mathbb B}$, the pairing $\langle
  \cdot\,,\cdot\rangle_{\mathbb B}$ and the bracket
  $\lb\cdot\,,\cdot\rb_{\mathbb B}$, is a Courant algebroid.
Further, $U$ is a Dirac structure in $\mathbb E$, via the inclusion $U\hookrightarrow\mathbb B$, 
$u\mapsto u\oplus 0$.
\end{theorem}
The proof of Theorem \ref{thm_CA_sd} can be found in Appendix
\ref{proof_of_manin_ap}.

\begin{corollary}\label{disapointing_cor}
  Let $(D;U,B;M)$ be an LA-Dirac structure in an LA-Courant algebroid
  $(\mathsf E, B; Q, M)$ (with core $Q^*$). The Manin pair $(\mathbb B, U)$ 
defined in Theorem \ref{thm_CA_sd} and the degenerate Courant algebroid $Q^*$ satisfy the following conditions:
\begin{enumerate}
%\item $\iota(U)^\circ$ is isotropic in $Q^*$, FOLLOWS FROM 3
\item There is a morphism $\psi\colon
  Q^*\to C$ of degenerate Courant algebroids and an embedding $\iota\colon U\to Q$ over the identity on $M$
\item $\iota$ is compatible with the anchors: $\rho_Q\circ\iota=\rho_C\an{U}$,
\item $\psi(Q^*)+U=C$ and
\item $\langle \psi(\tau), u\rangle_C=\langle  \iota(u), tau\rangle$ for all
  $\tau\in Q^*$ and $u\in U$.
\end{enumerate}
\end{corollary}

\begin{proof}
 Take an LA-Dirac structure $(D;U,B;M)$ in an LA-Courant
  algebroid\linebreak $(\mathbb E;Q,B;M)$.  The morphism $\psi\colon Q^*\to \mathbb E$ defined by
  $\psi(\tau)\mapsto 0\oplus \tau$ is obviously a morphism of
  degenerate Courant algebroids. Conditions (1)--(4) are then immediate.
\end{proof}

Conversely take a Manin pair $(C,U)$ over $M$ satisfying with $Q^*$ the
conditions in Corollary \ref{disapointing_cor} and identify $U$ with a
subbundle of $Q$.  If $\tau\in U^\circ\subseteq Q^*$, then
$\psi(\tau)$ satisfies
\[\langle u, \psi(\tau)\rangle_C=\langle u, \tau\rangle=0
\]
for all $u\in U$. Since $U$ is a Dirac structure, we find that $\psi$
restricts to a map $U^\circ\to U$. Conversely, we find easily that
$\psi(\tau)\in U$ if and only if $\tau\in U^\circ$.  Next choose
$\tau_1\in U^\circ $ and $\tau_2\in Q^*$.  Then since $\psi(\tau_1)\in U$,
\[\langle
\psi(\tau_1), \tau_2\rangle=\langle
\psi(\tau_1),\psi(\tau_2)\rangle_C= \langle
\tau_1,\tau_2\rangle_{Q^*}=\langle \partial_Q\tau_1, \tau_2\rangle,
\] 
which shows that $\psi\an{U^\circ}=\partial_Q\an{U^\circ}$. In
particular, $\partial_Q$ sends $U^\circ$ to $U$, and $U^\circ$ is
isotropic in $Q^*$.  Consider the vector bundle map $U\oplus Q^*\to
C$, $(u,\tau)\mapsto u+\psi(\tau)$. By assumption, this map is
surjective. Its kernel is the set of pairs $(u,\tau)$ with
$u=-\psi(\tau)$, i.e.~the graph of $-\partial_Q\an{U^\circ}\colon
U^\circ\to U$. It follows that
\begin{equation}\label{id_C}
C\simeq \frac{U\oplus Q^*}{\operatorname{graph}(-\partial_Q\an{U^\circ}\colon
U^\circ\to U)}.
\end{equation}
Hence, we can use the notation $u\oplus \tau$ for
$\overline{u+\psi(\tau)}\in C$.

In the case of an LA-Courant algebroid $(TA\oplus_AT^*A,TM\oplus
A^*,A,M)$ as in \S\ref{PontLA_LACourant}, for a Lie algebroid $A$, we
could show in \cite{Jotz14} that Manin pairs as in Corollary
\ref{disapointing_cor} are \emph{in bijection} with LA-Dirac
structures on $A$. That is, given a Manin pair $(C,U)$ with an
inclusion $U\hookrightarrow TM\oplus A^*$ and a degenerate Courant
algebroid morphism $A\oplus T^*M\to C$ satisfying (1)--(4), then via
the identification \eqref{id_C}, there exists a Lagrangian splitting
of $TA\oplus_AT^*A$ such that the Courant bracket on $C$ is given
by \eqref{C_bracket}.

In a future project we will study how this result generalises to
LA-Dirac structures in general LA-Courant algebroids, and we will
compare the data contained in the Manin pair with the infinitesimal
description of Dirac groupoids that was found by Li-Bland and Severa
in \cite{LiSe11}.

\appendix
\section{Proof of Theorem \ref{main}}\label{appendix_proof_of_main} 
Let $(\mathbb E;Q,B;M)$ be a VB-Courant algebroid and choose a
Lagrangian splitting $\Sigma\colon Q\times_MB$. We prove here that the
split linear Courant algebroid is equivalent to a Dorfman
$2$-representation.  Recall that $\mathcal S\subseteq\Gamma_B(\mathbb
E)$ is the subset $\{\tau^\dagger\mid \tau\in\Gamma(Q^*)\}\cup
\{\sigma_Q(q)\mid q\in\Gamma(Q)\}\subseteq \Gamma(\mathbb E)$.

Recall also that the tangent double $(TB\to B;TM\to M)$ has a VB-Lie
algebroid structure, which is described in \S\ref{tangent_double}.

We start with the proofs of two useful lemmas.
\begin{lemma}\label{formula_for_Dl}
For $\beta\in\Gamma(B^*)$, we have 
\[\mathcal
D(\ell_\beta)=\sigma_Q(\partial_B^*\beta)+\widetilde{\nabla_\cdot^*\beta},
\]
where $\nabla_\cdot^*\beta$ is seen as
follows as a
section of $\Gamma(\operatorname{Hom}(B,Q^*))$: 
$\left(\nabla_\cdot^*\beta\right)(b)=
\langle \nabla_\cdot^*\beta, b\rangle
\in\Gamma(Q^*)$ for all $b\in\Gamma(B)$.
\end{lemma}
\begin{proof}
  For $\beta\in\Gamma(B^*)$, the section $\dr\ell_\beta$ is a linear
  section of $T^*B\to B$. Since the anchor $\Theta$ is linear, the
  section $\mathcal D\ell_\beta=\Theta^*\dr\ell_\beta$ is linear. Since for
  any $\tau\in\Gamma(Q^*)$,
\[\langle \mathcal D(\ell_\beta),
\tau^\dagger\rangle =
\Theta(\tau^\dagger)(\ell_\beta)=q_B^*\langle \partial_B\tau, \beta\rangle,
\]
we find that $\mathcal
D(\ell_\beta)-\sigma_Q(\partial_B^*\beta)\in\Gamma(\ker\pi_Q)$.
Hence, $\mathcal D(\ell_\beta)-\sigma_Q(\partial_B^*\beta)$ is a
core-linear section of $\mathbb \E\to B$ and there exists a section
$\phi$ of $\operatorname{Hom}(B,Q^*)$ such that $\mathcal
D(\ell_\beta)-\sigma_Q(\partial_B^*\beta)=\widetilde{\phi}$.  We
have \[\ell_{\langle\phi, q\rangle}= \langle \widetilde{\phi},
\sigma_Q(q)\rangle=\langle \mathcal
D(\ell_\beta)-\sigma_Q(\partial_B^*\beta), \sigma_Q(q)\rangle
=\Theta(\sigma_Q(q))(\ell_\beta)=\ell_{\nabla_q^*\beta}
\]
and so $\phi(b)=\langle \nabla_\cdot^*\beta, b\rangle\in\Gamma(Q^*)$ for all $b\in\Gamma(B)$.
\end{proof}

\begin{lemma}
  For $q\in\Gamma(Q)$ and $\phi\in \Gamma(\operatorname{Hom}(B,Q^*))$,
  we have
\[\left\lb \sigma_Q(q), \widetilde{\phi}\right\rb= \widetilde{\lozenge_q\phi}.
\]
\end{lemma}

\begin{proof}
  Write $\phi=\sum_{i=1}^n\beta_i\otimes \tau_i$ with
  $\beta_1,\ldots,\beta_n\in\Gamma(B^*)$ and
  $\tau_1,\ldots,\tau_n\in\Gamma(Q^*)$.  Then
  $\widetilde{\phi}=\sum_{i=1}^n\ell_{\beta_i}\cdot \tau_i^\dagger$ and
  we can compute
\begin{align*}
  \left\lb \sigma_Q(q),
    \widetilde{\phi}\right\rb&=\sum_{i=1}^n\left(\ell_{\nabla^*_q\beta_i}\cdot
    \tau_i^\dagger+\ell_{\beta_i}\cdot
    (\Delta_q\tau_i)^\dagger\right) =\widetilde{\sum_{i=1}^n
    \nabla^*_q\beta_i\otimes\tau_i+\beta_i\otimes
    \Delta_q\tau_i}
\end{align*}
on the one hand, and on the other hand 
\[(\lozenge_q\phi)(b)=\Delta_q\left(\sum_{i=1}^n\langle \beta_i,
  b\rangle\tau_i\right)-\sum_{i=1}^n\langle \beta_i,
\nabla_qb\rangle\tau_i =\left( \sum_{i=1}^n
  \beta_i\otimes\Delta_q\tau_i+\nabla_q^*\beta_i\otimes\tau_i\right)
  (b)\] for any $b\in\Gamma(B)$.
\end{proof}

\medskip

Now we can express all the conditions of Lemma \ref{useful_lemma} in
terms of the data
$\partial_B,\Delta,\nabla,\lb\cdot\,,\cdot\rb_\sigma, R$ found in
\S\ref{construction_of_objects}.

\begin{proposition}\label{anchor}
The anchor satisfies $\Theta\circ\Theta^*=0$ if and only if 
\begin{enumerate}
\item $\rho_Q\circ\partial_B^*=0$ and 
\item
  $\nabla_{\partial_B^*\beta_1}^*\beta_2+\nabla_{\partial_B^*\beta_2}^*\beta_1=0$
  for all $\beta_1,\beta_2\in\Gamma(B^*)$.
\end{enumerate}
\end{proposition}

\begin{proof}
  The composition $\Theta\circ \Theta^*$ vanishes if and only if
  $(\Theta\circ \Theta^*)\dr F=0$ for all linear and pullback
  functions $F\in C^\infty(B)$. For $f\in
  C^\infty(M)$,
  $\Theta(\Theta^*\dr(q_B^*f))=((\partial_B\circ\rho_Q^*)\dr
  f)^\uparrow$. For $\beta\in\Gamma(B^*)$, we find using Lemma \ref{formula_for_Dl}
  $\Theta(\Theta^*\dr\ell_\beta)=\Theta(\mathcal
  D\ell_\beta)=\Theta(\sigma_Q(\partial_B^*\beta)+\widetilde{\nabla_\cdot^*\beta})
  =\widehat{\nabla_{\partial_B^*\beta}}+\widetilde{\partial_B\circ
    \langle\nabla^*_\cdot\beta, \cdot\rangle}$.  Here,
  $\partial_B\circ \langle\nabla^*_\cdot\beta, \cdot\rangle$ is as
  follows a morphism $B\to B$;
  $b\mapsto \partial_B(\langle\nabla^*_\cdot\beta, b\rangle)$.  On a
  linear function $\ell_{\beta'}$, $\beta'\in\Gamma(B^*)$,
  $\Theta(\Theta^*\dr\ell_\beta)(\ell_{\beta'})
  =\ell_{\nabla_{\partial_B^*\beta}^*\beta'}+\ell_{\nabla_{\partial_B^*\beta'}^*\beta}$. On
  a pullback $q_B^*f$, $f\in C^\infty(M)$, this is
  $q_B^*(\ldr{(\rho_Q\circ\partial_B^*)(\beta)}f)$. 
\end{proof}

\begin{proposition}\label{comp}
The compatibility of $\Theta$ with the Courant algebroid
bracket $\lb\cdot\,,\cdot\rb$ 
implies
\begin{enumerate}
 \item $\partial_B\circ R(q_1,q_2)=R_\nabla(q_1,q_2)$,
 \item $\rho_Q\circ
   \lb\cdot\,,\cdot\rb_\sigma=[\cdot\,,\cdot]\circ(\rho_Q, \rho_Q)$,
   that is $\Delta_q(\rho_Q^*\dr f)=\rho_Q^*\dr(\rho_Q(q)(f))$ for all
   $q\in\Gamma(Q)$ and $f\in C^\infty(M)$, and
\item $\partial_B\circ \Delta= \nabla\circ \partial_B$.
\end{enumerate}
\end{proposition}

\begin{proof}
  We have $$\Theta \left\lb \sigma_Q(q_1), \sigma_Q(q_2)\right\rb=\left[\Theta(\sigma_Q(q_1)),
    \Theta(\sigma_Q(q_2))\right] =\left[\widehat{\nabla_{q_1}},
    \widehat{\nabla_{q_2}}\right]$$ and $$\Theta\left(\sigma_Q(\lb q_1,
    q_2\rb_\sigma)-\widetilde{R(q_1,q_2)}\right)=\widehat{\nabla_{\lb q_1,
      q_2\rb_\sigma}}-\widetilde{\partial_B\circ R(q_1,q_2)}.$$ Applying both
  derivations to a pullback function $q_B^*f$ for $f\in
  C^\infty(M)$ yields
\[\left[\widehat{\nabla_{q_1}}, \widehat{\nabla_{q_2}}\right] (q_B^*f)=q_B^*([\rho_Q(q_1), \rho_Q(q_2)]f).\]
and
\[ \left(\widehat{\nabla_{\lb q_1, q_2\rb_\sigma}}-\widetilde{\partial_B\circ
  R(q_1,q_2)}\right)(q_B^*f)=q_B^*(\rho_Q\lb q_1,
q_2\rb_\sigma(f))\] Applying both vector fields to a linear function
$\ell_\beta\in C^\infty(B)$, $\beta\in\Gamma(B^*)$, we get
\[\left[\widehat{\nabla_{q_1}}, \widehat{\nabla_{q_2}}\right]
(\ell_\beta)=\ell_{\nabla_{q_1}^*\nabla_{q_2}^*\beta-\nabla_{q_2}^*\nabla_{q_1}^*\beta}
\]
and 
\[ \left(\widehat{\nabla_{\lb q_1, q_2\rb_\sigma}}-\widetilde{\partial_B\circ
  R(q_1,q_2)}\right)(\ell_\beta)=\ell_{\nabla^*_{\lb q_1,
    q_2\rb_\sigma}\beta-R(q_1,q_2)^*\partial_B^*\beta}.\] Since
$R_{\nabla^*}(q_1,q_2)=-(R_\nabla(q_1,q_2))^*$, we find that
\[\Theta \left\lb \sigma_Q(q_1), \sigma_Q( q_2)\right\rb=[\Theta(\sigma_Q(q_1)), \Theta(\sigma_Q(q_2))]\]
for all $q_1,q_2\in\Gamma(Q)$
if and only if (1) and (2) are satisfied.

In the same manner we compute for $q\in\Gamma(Q)$ and
$\tau\in\Gamma(Q^*)$:
\[  \Theta\left(\left\lb \sigma_Q(q),
      \tau^\dagger\right\rb\right)=(\partial_B\Delta_q\tau)^\uparrow\]
and 
\[\left[ \Theta(\sigma_Q(q)),
  \Theta(\tau^\dagger)\right]=\left[\widehat{\nabla_q},
  (\partial_B\tau)^\uparrow\right]
=\left(\nabla_q(\partial_B\tau)\right)^\uparrow.
\]
Hence, $\Theta\left(\left\lb \sigma_Q(q),
    \tau^\dagger\right\rb\right)=\left[ \Theta(\sigma_Q(q)),
  \Theta(\tau^\dagger)\right]$ if and only if $\partial_B\circ
\Delta= \nabla\circ \partial_B$.
\end{proof}

\begin{proposition}\label{CA3}
  The condition (3) of Lemma \ref{useful_lemma} is equivalent to
\begin{enumerate}
\item
  $R(q_1,q_2)=-R(q_2,q_1)$
  and 
\item $\lb q_1, q_2\rb_\sigma+\lb q_2, q_1\rb_\sigma=0$
\end{enumerate}
for $q_1,q_2\in\Gamma(Q)$.
\end{proposition}

\begin{proof}
Choose $q_1,q_2$ in $\Gamma(Q)$.
Then we have 
\begin{equation*}\label{equation}
  \lb \sigma_Q(q_1), \sigma_Q(q_2)\rb+\lb
  \sigma_Q(q_2),\sigma_Q(q_1)\rb=\sigma_Q(\lb q_1,q_2\rb_\sigma+\lb q_2,
  q_1\rb_\sigma)
  -\widetilde{R(q_1,q_2)}-\widetilde{R(q_2,q_1)}.
\end{equation*}
By the choice of the splitting, we have $\mathcal D\langle \sigma_Q(q_1),
\sigma_Q(q_2)\rangle=\mathcal D(0)=0$.  Hence, (3) of Lemma
\ref{useful_lemma} is true for linear sections if and only if
$R(q_1,q_2)=-R(q_2,q_1)$ and $\lb q_1, q_2\rb+\lb q_2, q_1\rb=0$.
On one linear and one core section, we have $\lb \sigma_Q(q),
\tau^\dagger\rb=(\Delta_q\tau)^\dagger$ and $\lb
\tau^\dagger, \sigma_Q(q)\rb=(-\Delta_q\tau+\rho_Q^*\dr\langle \tau,
q\rangle)^\dagger$ by definition. On core sections (3) is trivially
satisfied since both the pairing and the bracket of two core sections
vanish.
\end{proof}

\begin{proposition}\label{CA2}
The derivation formula  (2) in Lemma  \ref{useful_lemma} 
is equivalent to 
\begin{enumerate}
\item $\Delta$ is dual to $\lb\cdot\,,\cdot\rb_\sigma$, that is $\lb\cdot\,,\cdot\rb_\sigma=\lb\cdot\,,\cdot\rb_\Delta$,
\item $\lb q_1, q_2\rb_\sigma+\lb q_2, q_1\rb_\sigma=0$
  for all $q_1, q_2\in\Gamma(Q)$ and 
\item $R(q_1,q_2)^*q_3=-R(q_1,q_3)^*q_2$
for all $q_1,q_2,q_3\in\Gamma(Q)$.
\end{enumerate}
\end{proposition}

\begin{proof}
  We compute (CA2) for linear and core sections.  First of all, the
  equations \[\Theta(\tau_1^\dagger)\langle \tau_2^\dagger,
  \tau_3^\dagger\rangle =\langle \lb \tau_1^\dagger,
  \tau_2^\dagger\rb, \tau_3^\dagger\rangle +\langle
  \tau_2^\dagger, \lb
  \tau_1^\dagger,\tau_3^\dagger\rb\rangle,\]
\[\Theta(\tau_1^\dagger)\langle \tau_2^\dagger,
  \sigma_Q(q)\rangle =\langle \lb \tau_1^\dagger,
  \tau_2^\dagger\rb, \sigma_Q(q)\rangle +\langle 
  \tau_2^\dagger, \lb
  \tau_1^\dagger,\sigma_Q(q)\rb\rangle\]
and 
\[\Theta(\sigma_Q(q))\langle \tau_1^\dagger, \tau_2^\dagger\rangle
=\langle \lb \sigma_Q(q), \tau_1^\dagger\rb, \tau_2^\dagger\rangle
+\langle \tau_1^\dagger, \lb \sigma_Q(q), \tau_2^\dagger\rb\rangle\]
are trivially satisfied for all
$\tau_1,\tau_2,\tau_3\in\Gamma(Q^*)$ and $q\in\Gamma(Q)$.
Next we have for $q_1,q_2\in\Gamma(Q)$ and 
  $\tau\in\Gamma(Q^*)$:
\begin{align*}
  &\Theta(\sigma_Q(q_1))\langle \sigma_Q(q_2), \tau^\dagger\rangle -\langle \lb
  \sigma_Q(q_1), \sigma_Q(q_2)\rb, \tau^\dagger\rangle-\langle \sigma_Q(q_2), \lb
  \sigma_Q(q_1),
  \tau^\dagger\rb\rangle \\
  =\,\,&\widehat{ \nabla_{q_1}}(q_B^*\langle q_2,\tau\rangle)
  -q_B^*\langle \lb q_1, q_2\rb_\sigma, \tau\rangle-q_B^*\langle q_2,
  \Delta_{q_1}\tau\rangle\\
  =\,\,&q_B^*\Bigl(\rho_Q(q_1)\langle q_2,\tau\rangle -\langle \lb
  q_1, q_2\rb_\sigma, \tau\rangle-\langle q_2,
  \Delta_{q_1}\tau\rangle\Bigr)
\end{align*}
Hence
\[\Theta(\sigma_Q(q_1))\langle \sigma_Q(q_2), \tau^\dagger\rangle =\langle
\lb \sigma_Q(q_1), \sigma_Q(q_2)\rb, \tau^\dagger\rangle+\langle \sigma_Q(q_2), \lb
\sigma_Q(q_1), \tau^\dagger\rb\rangle
\]
for all $q_1,q_2\in\Gamma(Q)$ and $\tau\in\Gamma(Q^*)$ if and only
if $\Delta$ and $\lb\cdot\,,\cdot\rb_\sigma$ are dual to each other.

In the same manner, we compute
\begin{align*}
  &\Theta(\tau^\dagger)\langle \sigma_Q(q_1), \sigma_Q(q_2)\rangle -\langle \lb
  \tau^\dagger, \sigma_Q(q_1)\rb, \sigma_Q(q_2)\rangle-\langle \sigma_Q(q_1), \lb
  \tau^\dagger, \sigma_Q(q_2)
  \rb\rangle \\
  =\,\,&0-\langle -(\Delta_{q_1}\tau)^\dagger+(\rho_Q^*\dr\langle
  q_1, \tau\rangle)^\dagger, \sigma_Q(q_2)\rangle -\langle \sigma_Q(q_1),
  -(\Delta_{q_2}\tau)^\dagger+(\rho_Q^*\dr\langle q_2,
  \tau\rangle)^\dagger\rangle\\
  =\,\,&-q_B^*\langle \lb q_1, q_2\rb_\sigma+ \lb q_2,
  q_1\rb_\sigma,\tau\rangle.
\end{align*}
Thus, 
\[\Theta(\tau^\dagger)\langle \sigma_Q(q_1),  \sigma_Q(q_2)\rangle =\langle
\lb \tau^\dagger, \sigma_Q(q_1)\rb, \sigma_Q(q_2)\rangle+\langle
\sigma_Q(q_1), \lb \tau^\dagger, \sigma_Q(q_2) \rb\rangle \] for all
$q_1,q_2\in\Gamma(Q)$ and $\tau\in\Gamma(Q^*)$ if and only if
\[ 0=\lb q_1, q_2\rb_\sigma+\lb q_2, q_1\rb_\sigma.
\]

Finally we have $\Theta(\sigma_Q(q_1))\langle \sigma_Q(q_2), \sigma_Q(q_3)\rangle =0$ for all
$q_1,q_2,q_3\in\Gamma(Q)$, and $\langle \lb \sigma_Q(q_1), \sigma_Q(q_2)\rb,
\sigma_Q(q_3)\rangle=\ell_{-R(q_1,q_2)^*q_3}$.  This shows that
\[\Theta(\sigma_Q(q_1))\langle \sigma_Q(q_2), \sigma_Q(q_3)\rangle=\langle \lb \sigma_Q(q_1),
\sigma_Q(q_2)\rb, \sigma_Q(q_3)\rangle+\langle \sigma_Q(q_2), \lb \sigma_Q(q_1), \sigma_Q(q_3)\rb\rangle
\]
if and only if 
\[0=-R(q_1,q_2)^*q_3-R(q_1,q_3)^*q_2.
\]
\end{proof}

\begin{proposition}\label{Jacobi}
  Assume that $\Delta$ and $\lb\cdot\,,\cdot\rb_\sigma$ are dual to each
  other. The Jacobi identity in Leibniz form for sections in $\mathcal
  S$ is equivalent to
\begin{enumerate}
\item $R(q_1,q_2)\circ \partial_B=R_\Delta(q_1,q_2)$ and 
\item 
\begin{equation*}
\begin{split}
  &R(q_1, \lb q_2,q_3\rb_\Delta)-R(q_2, \lb q_1, q_3\rb_\Delta)-R(\lb q_1,q_2\rb_\Delta),q_3)\\
  &+\lozenge_{q_1}(R(q_2,q_3))-\lozenge_{q_2}(R(q_1,q_3))+\lozenge_{q_3}(R(q_1,q_2))=\,\nabla_\cdot^*\left(R(q_1,q_2)^*q_3\right)
\end{split}
\end{equation*}
\end{enumerate}
for all $q_1,q_2,q_3\in\Gamma(Q)$.
\end{proposition}
If $R$ is skew-symmetric as in (1) of Proposition \ref{CA3}, then the second
equation is the same as (D6) in Definition \ref{def_dorfman_2_conn}.

\begin{proof}
  The Jacobi identity is trivially satisfied on core sections since
  the bracket of two core sections is $0$.  Similarly, for
  $\tau_1,\tau_2\in\Gamma(Q^*)$ and $q\in\Gamma(Q)$, we find $\lb
  \sigma_Q(q),\lb \tau_1^\dagger, \tau_2^\dagger\rb\rb=0$ and $\lb \lb
  \sigma_Q(q), \tau_1^\dagger\rb, \tau_2^\dagger\rb +\lb
  \tau_1^\dagger, \lb \sigma_Q(q), \tau_2^\dagger\rb\rb=0 $.  We have
\begin{equation*}\begin{split}
 &   \left\lb \sigma_Q(q_1), \left\lb \sigma_Q(q_2),
        \tau^\dagger\right\rb\right\rb-\left\lb \sigma_Q(q_2), \left\lb
        \sigma_Q(q_1), \tau^\dagger\right\rb\right\rb\\
 &=\left\lb \sigma_Q(q_1),
      (\Delta_{q_2}\tau)^\dagger\right\rb
    -\left\lb \sigma_Q(q_2), (\Delta_{q_1}\tau)^\dagger\right\rb\\
    &=(\Delta_{q_1}\Delta_{q_2}\tau)^\dagger-
    (\Delta_{q_2}\Delta_{q_1}\tau)^\dagger,
\end{split}\end{equation*}
and 
\begin{align*}
  \left\lb \left\lb \sigma_Q(q_1), \sigma_Q(q_2)\right\rb, \tau^\dagger\right\rb
 & =\left\lb \sigma_Q(\lb q_1, q_2\rb_\Delta)-\widetilde{R(q_1,q_2)},
    \tau^\dagger\right\rb\\
&= (\Delta_{\lb q_1,
    q_2\rb_\Delta}\tau)^\dagger+(R(q_1,q_2)(\partial_B\tau))^\dagger
\end{align*}
by Lemma \ref{formulas}.
We now choose $q_1,q_2, q_3\in\Gamma(Q)$ and compute 
\begin{equation*}
\begin{split}
  &\left\lb\left\lb \sigma_Q(q_1), \sigma_Q(q_2)\right\rb, \sigma_Q(q_3)\right\rb\\
&=\left\lb
    \sigma_Q(\lb q_1, q_2\rb_\Delta)-\widetilde{R(q_1,q_2)}, \sigma_Q(q_3)
  \right\rb\\
  &=\sigma_Q(\lb \lb q_1, q_2\rb_\Delta, q_3\rb_\Delta)-\widetilde{R(\lb q_1, q_2\rb_\Delta,
  q_3)}-\mathcal D\ell_{\langle R(q_1,q_2)\cdot , q_3\rangle}+\widetilde{\lozenge_{q_3}R(q_1,q_2)}\\
  &=\sigma_Q(\lb \lb q_1, q_2\rb_\Delta, q_3\rb_\Delta)-\widetilde{R(\lb q_1, q_2\rb_\Delta,
  q_3)}\\
  &\qquad -\sigma_Q(\partial_B^*\langle R(q_1,q_2)\cdot ,
  q_3\rangle)-\widetilde{\nabla_\cdot^*\langle R(q_1,q_2)\cdot , q_3\rangle}
    +\widetilde{\lozenge_{q_3}R(q_1,q_2)}
\end{split}
\end{equation*}
and 
\begin{equation*}
\begin{split}
  \left\lb \sigma_Q(q_2),\left\lb \sigma_Q(q_1),
      \sigma_Q(q_3)\right\rb\right\rb&=\left\lb \sigma_Q(q_2),
    \sigma_Q(\lb q_1,
    q_3\rb_\Delta)-\widetilde{R(q_1,q_3)}\right\rb\\
  &=\sigma_Q(\lb q_2, \lb q_1,
  q_3\rb_\Delta\rb_\Delta)-\widetilde{R(q_2,\lb q_1,
    q_3\rb_\Delta)}-\widetilde{\lozenge_{q_2}R(q_1,q_3)}.
\end{split}
\end{equation*}
We hence find that 
\[\left\lb\left\lb \sigma_Q(q_1), \sigma_Q(q_2)\right\rb,
  \sigma_Q(q_3)\right\rb+\left\lb \sigma_Q(q_2),\left\lb \sigma_Q(q_1),
    \sigma_Q(q_3)\right\rb\right\rb=\left\lb \sigma_Q(q_1),\left\lb \sigma_Q(q_2),
    \sigma_Q(q_3)\right\rb\right\rb
\]
if and only if 
\[\left\lb\left\lb q_1, q_2\right\rb_\Delta, 
  q_3\right\rb_\Delta+\left\lb q_2,\left\lb q_1,
    q_3\right\rb_\Delta\right\rb_\Delta=\left\lb q_1,\left\lb q_2,
    q_3\right\rb_\Delta\right\rb_\Delta+\partial_B^*\langle
R(q_1,q_2)\cdot,q_3\rangle
\]
and 
\begin{equation*}
\begin{split}
  &R(\lb
  q_1,q_2\rb_\Delta,q_3)+\nabla_\cdot^*\langle
  R(q_1,q_2)\cdot,q_3\rangle-\lozenge_{q_3}R(q_1,q_2)
  +(q_2,\lb q_1,q_3\rb_\Delta)+\lozenge_{q_2}R(q_1,q_3)\\
  &=R(q_1,\lb q_2,q_3\rb_\Delta)+\lozenge_{q_1}R(q_2,q_3).
\end{split}
\end{equation*}
We conclude using \eqref{curv_dual_Jac}.
\end{proof}

A combination of Proposition \ref{anchor}, \ref{comp}, \ref{CA3},
\ref{CA2}, \ref{Jacobi} and Lemma \ref{useful_lemma} proves Theorem \ref{main}.

\section{On the LA-Courant algebroid condition}\label{appendix_LACA}
We give here Li-Bland's definition of an LA-Courant algebroid
\cite{Li-Bland12}, and we sketch the proof of Theorem \ref{LA-Courant}.

First consider a Lie algebroid $(q_A\colon A\to M, \rho_A,
[\cdot\,,\cdot])$. Then the dual $A^*$ is endowed with a linear
Poisson structure $\pi_A$, which defines a vector bundle morphism
$\pi_A^\sharp\colon T^*A^*\to TA^*$. We begin by studying this
morphism in decompositions of $T^*A^*$ and $TA^*$. Choose a connection
$\nabla\colon\mx(M)\times\Gamma(A)\to\Gamma(A)$.  
Then $\nabla$ defines a decomposition $\mathbb I_\nabla\colon A\times_MTM\times_M A\to TA$ by
$\mathbb I_\nabla(a_m,v_m,a_m')=T_ma
  v_m-\left.\frac{d}{dt}\right\an{t=0}a_m+t(\nabla_{v_m}a-a_m')$
for any $a\in\Gamma(A)$ such that $a(m)=a_m$.
In particular, the dual connection $\nabla^*\colon\mx(M)\times\Gamma(A^*)\to\Gamma(A^*)$ defines a decomposition of $TA^*$:
\begin{equation*}
\begin{xy}
\xymatrix{A^*\times_M TM\times_M A^*\ar[rr]\ar[dd]&&TM\ar[dd] &&TA^* \ar[rr]\ar[dd]&& TM\ar[dd]\\
                                    &A^*\ar[rd]&&\overset{\mathbb I_{\nabla^*}}{\longrightarrow}&&A^*\ar[dr]&\\
                  A^*\ar[rr]&& M&&A^*\ar[rr]&&M
}
\end{xy},
\end{equation*}
\[ \mathbb I_{\nabla^*} (\alpha(m), v_m, \alpha'_m)=T_m\alpha
  v_m-\left.\frac{d}{dt}\right\an{t=0}\alpha(m)+t(\nabla_{v_m}^*\alpha-\alpha'_m).
\]
Further, $\nabla$ defines a decomposition
\begin{equation*}
\begin{xy}
  \xymatrix{ A^*\times_MA\times_MT^*M\ar[rr]\ar[dd]&&A\ar[dd]&&T^*A^* \ar[rr]\ar[dd]&& A\ar[dd]\\
    &T^*M\ar[rd]&&\overset{\mathbb J_\nabla}{\longrightarrow}&&T^*M\ar[dr]&\\
    A^*\ar[rr]&& M&&A^*\ar[rr]&&M }
\end{xy}
\end{equation*}
of $T^*A^*$ by
\[ \mathbb J_\nabla(\alpha_m, a(m), \theta_m)=\dr\ell_a(\alpha_m)-(T_{\alpha_m}q_{A^*})^*(\langle
  \nabla_\cdot a, \alpha_m\rangle -\theta_m),
\]
for all $\alpha_m\in A^*$, $\theta_m\in T^*M$ and $a\in\Gamma(A)$, and 
$\nabla^*$ defines a decomposition $\mathbb J_{\nabla^*}\colon A\times_M A^*\times_M T^*M\to T^*A$
by $\mathbb J_\nabla(a_m,\alpha(m), \theta_m)=\dr\ell_\alpha(a_m)-(T_{a_m}q_{A})^*(\langle
  \nabla_\cdot^* \alpha, a_m\rangle -\theta_m)$.

  The map $\mathbb I_{\nabla^*}\inv\circ\pi_A^\sharp\circ\mathbb
  J_{\nabla}\colon A^*\times_M A\times_M TM\to A^*\times_M TM\times_M
  A^*$ is given by
\[ (\alpha_m,a_m,\theta_m)\mapsto (\alpha_m, \rho_A(a_m), -\langle
\alpha_m, \nabla_\cdot^{\rm bas}a\rangle-\rho_A^*\theta_m+\rho_A^*\langle
\nabla_\cdot a,\alpha_m\rangle)
\] (see \cite{Jotz13a}), where $\nabla^{\rm bas}$ is defined as in
Example \ref{tangent_double_al}. Recall that $T^*A^*$ is
isomorphic to $T^*A$ \cite{MaXu98}; in the decompositions $\mathbb
J_{\nabla}$ and $\mathbb J_{\nabla^*}$, of $T^*A^*$ and $T^*A$, the
isomorphism is given by
\[(\alpha_m,a_m,\theta_m)\in T^*A^*\leftrightarrow (\alpha_m,a_m,-\theta_m)\in T^*A.
\]
Note further that, via this isomorphism, $(T^*A^*)\duer A\simeq
(T^*A)\duer A=TA$ and also $TA^*\duer TM\simeq TA$.
In the decompositions, the pairing of $T^*A^*$ with its dual $TA$ is given by 
\[\langle (\alpha_m,a_m,\theta_m), (a_m,v_m,a_m')\rangle=\langle \alpha_m, a_m\rangle-\langle \theta_m, v_m\rangle
\]
and the pairing of $TA^*$ with $TA$ is given by
\[\langle (\alpha_m,v_m,\alpha'_m), (a_m,v_m,a_m')\rangle=\langle \alpha_m, a'_m\rangle+\langle \alpha'_m, a_m\rangle.
\]

In \cite{Li-Bland12} Li-Bland defines a relation
$\Pi_A\subseteq TA\times TA\simeq (T^*A^*)\duer A\,\times\,TA^*\duer TM$ 
by
\[ (V,W)\in \Pi_A \qquad \Leftrightarrow \qquad \langle V,\Phi\rangle_A=\langle W,\pi^\sharp(\Phi)\rangle_{TM}
\,\,\text{ for all }\,\, \Phi\in T^*A^*.
\]
In the decomposition $\mathbb I_{\nabla}$, the relation $\Pi_A$ is hence given by
\[ (v_m,a_m,b_m) \sim_{\Pi_A} (c_m,w_m, d_m)
\]
if and only if 
\begin{equation*}
\begin{split}
&\langle (v_m,a_m,b_m) , (\alpha_m,a_m,\theta_m)\rangle\\
&=\langle
(c_m,w_m, d_m), (\alpha_m, \rho_A(a_m), -\langle
\alpha_m, \nabla_\cdot^{\rm bas}a\rangle-\rho_A^*\theta_m+\rho_A^*\langle
\nabla_\cdot a,\alpha_m\rangle)\rangle
\end{split}
\end{equation*}
for all $(\alpha_m,a_m,\theta_m)\in A^*_m\times A_m\times T_m^*M$.  That
is, $ (v_m,a_m,b_m) \sim_{\Pi_A} (c_m,w_m, d_m)$ if and only if
$w_m=\rho_A(a_m)$, $b_m=d_m-\nabla_{c_m}^{\rm bas}a+
\nabla_{\rho_A(c_m)} a$ and $v_m=\rho_A(c_m)$.   In other words, we have
\[ (\rho_A(a_m),b_m,c_m-\nabla_{a_m}^{\rm bas}b+ \nabla_{\rho_A(a_m)}
b) \sim_{\Pi_A} (a_m,\rho_A(b_m), c_m)
\]
for all $a_m,b_m,c_m\in A_m$.
This leads to 
\begin{equation*}
\begin{split} 
&(\rho_A(a)(m),b(m),c(m)-[a,b](m)+ \nabla_{\rho_A(a)} b(m))\\
&\hspace*{3cm}\sim_{\Pi_A} (a(m),\rho_A(b)(m), c(m)+\nabla_{\rho_A(b)}a(m))
\end{split}
\end{equation*}
for all $a,b,c\in \Gamma(A)$.
But this is just 
\begin{equation*}
\begin{split} 
&T_mb\rho_A(a)(m)+_A\left.\frac{d}{dt}\right\an{t=0}b(m)+t([b,a]+c)(m)\\
&\hspace*{3cm}\sim_{\Pi_A}T_ma\rho_A(b)(m)+_A\left.\frac{d}{dt}\right\an{t=0}a(m)+tc(m)
\end{split}
\end{equation*}
for all $a,b,c\in\Gamma(A)$.

Given $a,b,c\in\Gamma(A)$, we write $(a,b,c)(m)$ for the pair 
\[ \left(
T_mb\rho_A(a)(m)+_A([b,a]+c)^\uparrow(b_m),
T_ma\rho_A(b)(m)+_Ac^\uparrow(a_m)\right)
\]
in $\Pi_A$ and we get 
\[\Pi_A=\{(a,b,c)(m)\mid a,b,c\in \Gamma(A),\, m\in M\}.
\]
Note that for $f\in C^\infty(M)$,  \[(a,b,fc)(m)=(a,b,f(m)c)(m),\] 
\[(fa,b,c)(m)=(f(m)a,b,c+\rho_A(b)(f)(m)a)(m),\]
and \[(a,fb,c)(m)=(a,f(m)b,c-\rho_A(a)(f)(m)b)(m)\]
for all $m\in M$.
As a consequence, one can easily check that given a family of sections
$\mathcal S\subseteq \Gamma(A)$ that spans point-wise $A$, we find
\begin{equation}\label{description_Pi}
  \Pi_A=\left\{\left(\left.\sum_{i}x_ia_i,\sum_jy_jb_j,\sum_kz_kc_k\right)(m)
\right\vert a_i,b_j,c_k\in\mathcal S, x_i,y_j,z_k\in\R, m\in M\right\}.
\end{equation}
Note further that $\Pi_A\subseteq TA\times TA$ can be described 
as follows.
\begin{proposition}\label{description_of_Pi_A}
A point $P\in TA\times TA$
 is an element of $\Pi_A$ if and only if 
\begin{enumerate}
\item $(q_A\circ p_A\circ\pr_1)(P)=(q_A\circ p_A\circ\pr_2)(P)$,
\end{enumerate}
i.e. $P\in TA\times_{q_A\circ p_A}TA$, 
and all the following elements of $C^\infty(TA\times_{q_A\circ p_A} TA)$ vanish on $P$:
\begin{enumerate}
\setcounter{enumi}{1}
\item
  $\ell_{\dr\ell_\alpha}\circ\pr_1-\ell_{\dr\ell_\alpha}\circ\pr_2+\dr_A\alpha\circ
  (p_A,p_A)$, for all $\alpha\in\Gamma(A^*)$,
\item $\ell_{q_A^*\dr f}\circ \pr_1-(p_A^*\ell_{\rho_A^*\dr f})\circ
  \pr_2$ and 
\item $\ell_{q_A^*\dr f}\circ \pr_2-(p_A^*\ell_{\rho_A^*\dr f})\circ
  \pr_1$  for all $f\in C^\infty(M)$.
\end{enumerate}
\end{proposition}

\begin{proof}
  Any point $P\in TA\times TA$ can be written
  $P=(T_mav_m+c^\dagger(a_m), T_nbw_n+d^\dagger(b_n))$ with $v_m,
  w_n\in TM$ and $a,b,c,d\in\Gamma(A)$.  Condition (1) is
  $m=n$. Condition (4) is equivalent to $w_m=\rho(a(m))$, and
  Condition (3) to $v_m=\rho(b(m))$.  Condition (2) is satisfied for
  $P=(T_ma\rho(b(m))+c^\dagger(a(m)), T_mb\rho(a(m))+d^\dagger(b_n))$
  if and only if $c=d+[a,b]$.
\end{proof}

\medskip

Now consider  a double vector bundle 
\begin{equation*}
\begin{xy}
  \xymatrix{\mathbb E\ar[r]\ar[d]&B\ar[d]\\
    Q\ar[r]&M }
\end{xy}
\end{equation*} endowed with a VB-Lie algebroid structure $(\mathbf
b, [\cdot\,,\cdot])$ on $\mathbb E\to Q$ and a linear 
metric $\langle\cdot\,,\cdot\rangle$ on $\mathbb E\to B$ 
(hence, $\mathbb E$ has core $Q^*$). Let $\rho_B\colon B\to TM$ be the anchor 
of the induced Lie algebroid structure on $B$.

The relation $\Pi_{\mathbb E}$ defined as above by the Lie algebroid
structure on $\mathbb E$ over $Q$ is then a relation $\Pi_{\mathbb E}$ of
triple vector bundles \cite{Li-Bland12}
\begin{equation*}
\begin{xy}
  \xymatrix{T\mathbb E\ar[rr]\ar[dd]\ar[rd]&&TQ\ar[dd]\ar[rd]&\\
&\mathbb E\ar[dd]\ar[rr]&&Q \ar[dd]\\
TA\ar[rd]\ar[rr]&&TM \ar[rd]&\\
 &   A\ar[rr]&&M }
\end{xy}
\qquad 
\overset{\Pi_{\mathbb E}}{\longrightarrow}
\qquad 
\begin{xy}
  \xymatrix{T\mathbb E\ar[rr]\ar[dd]\ar[rd]&&\mathbb E\ar[dd]\ar[rd]&\\
&TQ\ar[dd]\ar[rr]&&Q \ar[dd]\\
TA\ar[rd]\ar[rr]&&A \ar[rd]&\\
 &   TM\ar[rr]&&M }
\end{xy}
\end{equation*}
Li-Bland's definition \cite{Li-Bland12} is the following.
\begin{definition}
Let $(\mathbb E;Q,B;M)$ be a double vector bundle with
a VB-Courant algebroid structure (over $B$) and a VB-algebroid structure
$ (\mathbb E\to Q,B\to M)$.
Then $(\mathbb E, B,Q,M)$ is an LA-Courant
algebroid if $\Pi_{\mathbb E}$ is a Dirac structure with support in
$\overline{T\mathbb E}\times T\mathbb E$.
\end{definition}

We choose a Lagrangian splitting $\Sigma\colon B\times_M Q\to \mathbb
E$. We denote as usual by $(\partial_Q\colon Q^*\to Q, \nabla^Q,
\nabla^{Q^*}, R\in \Omega^2(B,\operatorname{Hom}(Q,Q^*)))$ the induced
$2$-representation\footnote{As always, we omit the upper indices of
  the two linear connections when which one is used is clear from
  its index and term.} of the Lie algebroid $B$, and by
$(\partial_B\colon Q^*\to B,
\nabla,\Delta,R\in\Omega^2(Q,\Hom(B,Q^*)))$ the induced Dorfman
$2$-representation of the anchored vector bundle $(Q,\rho_Q)$.  We
begin with studying the vector bundle structure of $\Pi_{\mathbb E}$.
\begin{proposition}
The image of  $\Pi_{\mathbb E}$ under  $(T\pi_B\times T\pi_B)\colon T\mathbb E\times T\mathbb
E \to TB\times TB$ is $\Pi_B$.
\end{proposition}

\begin{proof}
  By \eqref{description_Pi}, it is sufficient to show that the set of
  points \[ (\sigma_B(b_1)+\tau_1^\dagger, \sigma_B(b_2)+\tau_2^\dagger,
  \sigma_B(b_3)+\tau_3^\dagger)(q_m)
\]
for $b_1,b_2,b_3\in\Gamma(B)$, $\tau_1,\tau_2,\tau_3\in\Gamma(Q^*)$
and $q_m\in Q$ projects under $(T\pi_B,T\pi_B)$ to $\Pi_B$.  It is
easy to see that $(T\pi_B,T\pi_B)((\sigma_B(b_1)+\tau_1^\dagger, \sigma_B(b_2)+\tau_2^\dagger,
  \sigma_B(b_3)+\tau_3^\dagger)(q_m))
=(b_1,b_2,b_3)(m)$.
\end{proof}

It seems at this point useful to enumerate the sections of $\mathbb
E\to B$, of $\mathbb E\to Q$, and of $T\mathbb E\to TB$ that we are
working with here.  Recall that after the choice of a Lagrangian
splitting $\Sigma\colon Q\times_M B\to \mathbb E$, the vector bundle
$\mathbb E\to Q$ is spanned by sections $\sigma_B(b)$ for all
$b\in\Gamma(B)$ and $\tau^\dagger$ for all $\tau\in\Gamma(Q^*)$. 
We have used those sections to define $\Pi_{\mathbb E}$, as the relation 
that the Lie algebroid $\mathbb E\to Q$ induces.

In the same manner
(and with a slight abuse of notation), the vector bundle
$\mathbb E\to B$ is spanned by sections $\sigma_Q(q)$ for all
$q\in\Gamma(Q)$ and $\tau^\dagger$ for all $\tau\in\Gamma(Q^*)$.

As a consequence, the vector bundle $T\mathbb E\to TB$ is spanned
by the sections 
\[T\sigma_Q(q), \qquad T\tau^\dagger, \qquad (\sigma_Q(q))^\times \quad\text{ and }\quad (\tau^\dagger)^\times
\]
for all $q\in\Gamma(Q)$ and $\tau\in\Gamma(Q^*)$. (Recall that for a
vector bundle $E\to M$, and a section $e\in\Gamma_M(E)$, the section
$e^\times$ is defined in \eqref{cross_sect_def}.)  Since $\mathbb E\to
B$ has a Courant algebroid structure, its tangent double has a natural
Courant algebroid structure, see the proof of Theorem
\ref{tangent_courant_double}.

\begin{proposition}\label{crucial}
For $b_1,b_2,b_3\in\Gamma(B)$, $\tau_1,\tau_2,\tau_3\in\Gamma(Q^*)$ and
$q\in\Gamma(Q)$, 
the pair 
\begin{align*}
  &T_{q(m)} \left(\sigma_B(b_1)+\tau_1^\dagger\right)\mathsf
  b\left(\left(\sigma_B(b_2)+\tau_2^\dagger\right)(q(m))\right)\\
  &\qquad+\left(\left[\sigma_B(b_1)+\tau_1^\dagger,\sigma_B(b_2)+\tau_2^\dagger\right]+\sigma_B(b_3)+\tau_3^\dagger\right)^\times(\sigma_B(b_1)+\tau_1^\dagger)(q(m))\\
  \sim_{\Pi_{\mathbb E}}&
  T_{q(m)}\left(\sigma_B(b_2)+\tau_2^\dagger\right)\mathsf
  b\left(\left(\sigma_B(b_1)+\tau_1^\dagger\right)(q(m))\right)\\
  &\qquad\qquad \qquad
  +\left(\sigma_B(b_3)+\tau_3^\dagger\right)^\times\left(\left(\sigma_B(b_2)+\tau_2^\dagger\right)(q(m))\right)
\end{align*}
is equal to
\begin{align*}
  &T_{b_1(m)}\left(\sigma_Q(q)+\tau_1^\dagger\right)\left(T_mb_1\rho_B(b_2)(m)+[b_1,b_2]^\uparrow(b_1(m))+b_3^\uparrow(b_1(m))\right)\\
  +\Bigl(-\sigma_Q&(\nabla_{b_2}q)-R(b_1,b_2)q^\dagger+\tau_3^\dagger+\nabla_{b_1}\tau_2^\dagger-\nabla_{b_2}\tau_1^\dagger+\sigma_Q(\partial_Q\tau_2)\Bigr)^\times\left(\sigma_Q(q)+\tau_1^\dagger\right)(b_1(m))\\
  \sim_{\Pi_{\mathbb E}}&
  T_{b_2(m)}\left(\sigma_Q(q)+\tau_2^\dagger\right)(T_mb_2\rho(b_1)(m)+b_3^\uparrow(b_2(m)))\\
  &\qquad
  +\left(-\sigma_Q(\nabla_{b_1}q)+\tau_3^\dagger+\sigma_Q(\partial_Q\tau_1)\right)^\times(\sigma_Q(q)+\tau_2^\dagger)(b_2(m)).
\end{align*}
\end{proposition}

\begin{proof}
  This proof is tedious, but straightforward. To show that two vectors
  on $\mathbb E$ are equal, we evaluate them on linear and pullback
  functions in $C^\infty(\mathbb E)$.  We identify $\mathbb E$ with
  $\mathbb E\duer B$ using the linear metric. We consider the four
  types of functions on $\mathbb E$: $\ell_{\sigma_Q(q)}$,
  $\ell_{\tau^\dagger}$, $\pi_B^*\ell_\beta$ and $\pi_B^*q_B^*f$ for
  all $q\in\Gamma(Q)$, $\tau\in\Gamma(Q^*)$, $\beta\in\Gamma(B^*)$ and
  $f\in C^\infty(M)$.  Recall that the horizontal lift
  $\sigma^\star_{Q^{**}}\colon \Gamma(Q^{**})\to\Gamma_Q^l(\mathbb
  E\duer Q)$ can be defined by $\langle\sigma_{Q^{**}}(q),\sigma_B(b)\rangle=0$ and
  $\langle \sigma_{Q^{**}}(q),\tau^\dagger\rangle= q_Q^*\langle q,
  \tau\rangle$ for all $q\in\Gamma(Q)$, $\tau\in\Gamma(Q^*)$ and
  $b\in\Gamma(B)$, where we identify $Q^{**}$ with $Q$ via the
  canonical isomorphism. The sections of $\mathbb E\duer Q$ over $Q$
  define linear functions on $\mathbb E$ and we have
  $\ell_{\sigma^\star_{Q^{**}}(q)}=\ell_{\sigma_Q(q)}$,
  $\ell_{\beta^\dagger}=\pi_B^*\ell_{\beta}$,
  $\pi_Q^*\ell_{\tau}=\ell_{\tau^\dagger}$ and
  $\pi_Q^*q_Q^*f=\pi_B^*q_B^*f$ for all $q\in\Gamma(Q)$,
  $\tau\in\Gamma(Q^*)$, $\beta\in\Gamma(B^*)$ and $f\in
  C^\infty(M)$. We use these equalities of functions to prove the
  equalities of vectors in Proposition \ref{crucial}. The details are
  left to the reader.
\end{proof}

The last proposition was quite technical, but it yields useful and
relatively simple sections of $\Pi_{\mathbb E}\to \Pi_B$:
%  is spanned pointwise by the
% following sections of $T\mathbb E\times T\mathbb E\an{\Pi_B}\to
% \Pi_B$:
\begin{equation}\label{type1}
  \left((\tau^\dagger)^\times\circ\pr_1,
    (\tau^\dagger)^\times\circ\pr_2\right),
\end{equation}
\begin{equation}\label{type2}
\left( T\tau^\dagger\circ\pr_1
    -(\nabla_{p_B\circ\pr_2}\tau^\dagger)^\times\circ\pr_1,
    \sigma_Q(\partial_Q\tau)^\times\circ\pr_2\right),
\end{equation}
\begin{equation}\label{type3}
\left( (\nabla_{p_B\circ
      \pr_1}\tau^\dagger+\sigma_Q(\partial_Q\tau))^\times\circ\pr_1,
    T\tau^\dagger\circ\pr_2\right)\end{equation}
 and 
\begin{equation}\label{type4}
\begin{split}
&\left(T\sigma_Q(q)\circ\pr_1 -(\sigma_Q(\nabla_{p_B\circ
        \pr_2}q)+\widetilde{R(p_B\circ \pr_1, p_B\circ \pr_2)q})^\times\circ\pr_1,\right.\\
&   \qquad \qquad \qquad\qquad \qquad \qquad \qquad \qquad \left. T\sigma_Q(q)\circ\pr_2-(\sigma_Q(\nabla_{p_B\circ
        \pr_1}q))^\times\circ\pr_2\right)
\end{split}
\end{equation} for all $q\in\Gamma(Q)$ and $\tau\in\Gamma(Q^*)$.
% This shows that
% the set $\Pi_{\mathbb E}$ is a subbundle of $T\mathbb E\times T\mathbb
% E$ over $\Pi_B\subseteq TB\times TB$.
Now we study the compatibility of $\Pi_{\mathbb E}$ with the metric on $\mathbb E\to B$.
\begin{theorem}\label{theorem_isotropy}
The vector bundle $\Pi_{\mathbb E}\to \Pi_B$ is maximal isotropic in
$\overline{T\mathbb E}\times T\mathbb E\an{\Pi_B}\to \Pi_B$ 
if and only if 
\begin{enumerate}
\item $\partial_Q=\partial_Q^*$,
\item
  $(\nabla^Q)^*_b\tau=\nabla^{Q^*}_b\tau$
for all $b\in\Gamma(B)$ and $\tau\in\Gamma(Q^*)$, and 
\item $\langle R(b_1,b_2)q_1, q_2\rangle+\langle R(b_1,b_2)q_2, q_1\rangle
=0$
for all $b_1,b_2\in\Gamma(B)$ and $q_1,q_2\in\Gamma(Q)$.
\end{enumerate}
\end{theorem}
In other words, $\Pi_{\mathbb E}\to \Pi_B$ is maximal isotropic in
$\overline{T\mathbb E}\times T\mathbb E\an{\Pi_B}\to \Pi_B$ if and
only if the $2$-representation $(\partial_Q\colon Q^*\to Q, \nabla^Q,
\nabla^{Q^*},R)$ is self-adjoint. This yields the following result:
\begin{corollary}
  The VB-algebroid structure on $\mathbb E\to Q$ is compatible with
  the linear metric if and only if $\Pi_{\mathbb E}\to \Pi_B$ is
  maximal isotropic in $\overline{T\mathbb E}\times T\mathbb
  E\an{\Pi_B}\to \Pi_B$.
\end{corollary}

\begin{proof}[Proof of Theorem \ref{theorem_isotropy}]
The vector bundle $T\mathbb E\to TB$ has rank twice the rank of 
$\mathbb E\to B$. Hence, we have 
$\operatorname{rank}(T\mathbb E\times T\mathbb E\to TB\times
TB)=2\cdot2\cdot(\operatorname{rank}(Q)+\operatorname{rank}(Q^*))=8\operatorname{rank}(Q)$. It is easy to see that 
$\Pi_{\mathbb E}\to\Pi_B$ has rank 
$\operatorname{rank}(Q)+3\operatorname{rank}(Q^*)=4\operatorname{rank}(Q)$. Hence,
it is sufficient to find when $\Pi_{\mathbb E}\to \Pi_B$
is isotropic in $\overline{T\mathbb E}\times T\mathbb E\an{\Pi_B}\to \Pi_B$.

\medskip

The pairing 
\begin{align*}
  \Bigl\langle &T\left(\sigma_Q(q_1)+\chi_1^\dagger\right)+\left(-\sigma_Q(\nabla_{b_2}q_1)-R(b_1,b_2)q_1^\dagger+\sigma_1^\dagger+\nabla_{b_1}\tau_1^\dagger-\nabla_{b_2}\chi_1^\dagger+\sigma_Q(\partial_Q\tau_1)\right)^\times,\\
  & T\left(\sigma_Q(q_2)+\chi_2^\dagger\right)+\left(-\sigma_Q(\nabla_{b_2}q_2)-R(b_1,b_2)q_2^\dagger+\sigma_2^\dagger+\nabla_{b_1}\tau_2^\dagger-\nabla_{b_2}\chi_2^\dagger+\sigma_Q(\partial_Q\tau_2)\right)^\times\Bigr\rangle\\
  &\qquad\qquad \qquad\qquad \qquad\qquad \qquad\qquad  \qquad\left(Tb_1+([b_1,b_2]+b_3)^\times\right)(\rho_B(b_2))\\
  -\Bigl\langle &T\left(\sigma_Q(q_1)+\tau_1^\dagger\right)
  +\left(-\sigma_Q(\nabla_{b_1}q_1)+\sigma_1^\dagger+\sigma_Q(\partial_Q\chi_1)\right)^\times,\\
&\qquad
  T\left(\sigma_Q(q_2)+\tau_2^\dagger\right)
  +\left(-\sigma_Q(\nabla_{b_1}q_2)+\sigma_2^\dagger+\sigma_Q(\partial_Q\chi_2)\right)^\times
  \Bigr\rangle (Tb_2+b_3^\times)(\rho_B(b_1))
\end{align*}
equals 
\begin{align}
  &\rho_B(b_2)\langle
  \chi_1,q_2\rangle+\rho_B(b_2)\langle\chi_2,q_1\rangle\label{pairing_long_condition}-\rho_B(b_1)\langle
  \tau_1,q_2\rangle-\rho_B(b_1)\langle\tau_2,q_1\rangle\\
  -&\langle \nabla_{b_2}q_1, \chi_2\rangle-\langle R(b_1,b_2)q_1, q_2\rangle
  +\langle\nabla_{b_1}\tau_1,q_2\rangle-\langle\nabla_{b_2}\chi_1,
  q_2\rangle\nonumber\\
  +& \langle \partial_Q\tau_1, \chi_2\rangle-\langle \nabla_{b_2}q_2,
  \chi_1\rangle-\langle R(b_1,b_2)q_2, q_1\rangle
  +\langle\nabla_{b_1}\tau_2,q_1\rangle-\langle\nabla_{b_2}\chi_2, q_1\rangle+
  \langle \partial_Q\tau_2, \chi_1\rangle\nonumber\\
  +&\langle \nabla_{b_1}q_1, \tau_2\rangle
  -\langle \partial_Q\chi_1,\tau_2\rangle+\langle \nabla_{b_1}q_2,
  \tau_1\rangle -\langle \partial_Q\chi_2,\tau_1\rangle.\nonumber
\end{align}
This vanishes for all $b_1,b_2\in\Gamma(B)$, $q_1,q_2\in\Gamma(Q)$,
$\chi_1,\chi_2,\tau_1,\tau_2\in\Gamma(Q^*)$
if and only if
\begin{itemize}
\item (setting $b_1=b_2=0$, $q_1=q_2=0$, $\chi_2=\tau_1=0$)
\[\langle \partial_Q\tau_2,\chi_1\rangle-\langle \partial_Q\chi_1,\tau_2\rangle=0\]
for all $\chi_1,\tau_2\in\Gamma(Q^*)$, which is equivalent to
$\partial_Q=\partial_Q^*$.
\item (setting $b_2=0$, $q_2=0$, $\chi_1=\chi_2=\tau_1=0$)
  \[-\rho_B(b_1)\langle\tau_2,q_1\rangle+\langle\nabla_{b_1}\tau_2,q_1\rangle+\langle
  \nabla_{b_1}q_1, \tau_2\rangle =0
\]
for all $\tau_2\in\Gamma(Q^*)$, $q_1\in\Gamma(Q)$ and $a\in\Gamma(B)$.
This is equivalent to $\nabla_{b_1}^*\tau=\nabla_{b_1}\tau$ for all
$\tau\in\Gamma(Q^*)$ and $b_1\in\Gamma(B)$.
\item Using the two equations found above, 
several terms in \eqref{pairing_long_condition}
cancel. This yields
$-\langle R(b_1,b_2)q_1, q_2\rangle-\langle R(b_1,b_2)q_2, q_1\rangle=0$
for all $b_1,b_2\in\Gamma(B)$ and $q_1,q_2\in\Gamma(Q)$.
\end{itemize}
\end{proof}
% We now assume that $\mathbb E$ has in addition a linear Courant
% algebroid structure over $B$. Then, \cite{Li-Bland12} defines the
% compatibility by definition, $\mathbb E$ is an LB-Courant algebroid if
% and only if $\Pi_{\mathbb E}\to \Pi_B$ is a Dirac structure with
% support in $T\mathbb E\times \overline{T\mathbb E}\to TB\times TB$ .

Hence we have found that the second condition in Lemma \ref{useful_for_dirac_w_support} is satisfied if and only if
$(\mathbb E\to Q; B\to M)$ is a metric VB-algebroid.  We study
the remaining two conditions in Lemma \ref{useful_for_dirac_w_support} on the sections \eqref{type1},
\eqref{type2}, \eqref{type3}, \eqref{type4}.

\begin{theorem}
The anchor $\rho_{\overline{T\mathbb E}\times T\mathbb E}$ sends $\Pi_{\mathbb
  E}$ to $T\Pi_B$ if and only if 
\begin{enumerate}
\item $\rho_B\circ\partial_B=\rho_Q\circ\partial_Q$,
\item
  $\partial_B(\nabla_b\tau)=[b,\partial_B\tau]+\nabla_{\partial_Q\tau}b$,
\item $[\rho_B(b), \rho_Q(q)]=\rho_Q(\nabla_bq)-\rho_B(\nabla_qb)$ and
\item
  $\partial_BR(b_1,b_2)q=-\nabla_q[b_1,b_2]+[\nabla_qb_1,b_2]+[b_1,\nabla_qb_2]+\nabla_{\nabla_{b_2}q}b_1-\nabla_{\nabla_{b_1}q}b_2$
\end{enumerate}
for all $b,b_1,b_2\in\Gamma(B)$, $q\in\Gamma(Q)$ and $\tau\in\Gamma(Q^*)$.
\end{theorem}

\begin{proof}
First note that, by construction, the map
\[T(q_B\circ p_B\circ\pr_1)-T(q_B\circ p_B\circ\pr_2)\colon T(TB\times
TB)\to TM
\]
sends the images under $\rho_{\overline{T\mathbb E}\times T\mathbb E}$
of our four types of sections to zero.  That is, we have
$\rho_{\overline{T\mathbb E}\times T\mathbb E}(\Pi_{\mathbb
  E})\subseteq T(TB\times_{q_B\circ p_B}TB)$ by construction.  Hence,
the anchor $\rho_{\overline{T\mathbb E}\times T\mathbb E}$ sends
$\Pi_{\mathbb E}$ to $T\Pi_B$ if and only if $\rho_{\overline{T\mathbb
    E}\times T\mathbb E}$ applied to our four special types of
sections annihilate all the $\R$-valued functions in Proposition
\ref{description_of_Pi_A}.

The anchor of $\left((\tau^\dagger)^\times\circ\pr_1,
  (\tau^\dagger)^\times\circ\pr_2\right)$ is
\[\left((\Theta(\tau^\dagger))^\uparrow\circ\pr_1,
  (\Theta(\tau^\dagger))^\uparrow\circ\pr_2\right) =\left(
  (\partial_B\tau^\dagger)^\uparrow\circ\pr_1,
  (\partial_B\tau^\dagger)^\uparrow\circ\pr_2\right)\] for
$\tau\in\Gamma(Q^*)$. This vanishes on all the functions defining
$\Pi_B$.  The anchor
of \[\left(T\tau^\dagger-(\nabla_{b_2}\tau^\dagger)^\times,
  \sigma_Q(\partial_Q\tau)^\times\right)((b_1,b_2,b_3)(m))\]
is \[\left(\widehat{[\partial_B\tau^\uparrow,\cdot]}-(\partial_B\nabla_{b_2}\tau^\uparrow)^\uparrow,
  \widehat{\nabla_{\partial_Q\tau}\cdot}^\uparrow\right)((b_1,b_2,b_3)(m)).\]
This derivation, which we call here $X$ sends $\ell_{q_B^*\dr
  f}\circ\pr_1-(p_B^*\ell_{\rho_B^*\dr f})\circ \pr_2$ to $0$. We have
further
\begin{align*}
  X(\ell_{q_B^*\dr f}\circ\pr_2-(p_B^*\ell_{\rho_B^*\dr
    f})\circ \pr_1)&=
  \widehat{\nabla_{\partial_Q\tau}\cdot}(q_B^*
  f)(b_2(m))-(\partial_B\tau)^\uparrow(\ell_{\rho_B^*\dr f})(b_1(m))\\
  &=(\rho_Q\partial_Q\tau)(f)(m)-(\rho_B\partial_B\tau)(f)(m).%\\
%&=p_B^*q_B^*\langle \dr \varphi, \rho_Q\partial_Q\chi-\rho_B\partial_B\chi\rangle
\end{align*} This vanishes for all $f\in C^\infty(M)$ and all
$m\in M$ if and only if $\rho_B\partial_B(\tau)=\rho_Q\partial_Q(\tau)$. Finally,
a computation yields
\begin{align*}
&X(\ell_{\dr\ell_\beta}\circ\pr_1-\ell_{\dr\ell_\beta}\circ\pr_2+\dr_B\beta\circ(p_B,p_B))\\
=\,&\langle \beta,
-\partial_B\nabla_{b_2}\tau+\nabla_{\partial_Q\tau}b_2+[b_2,\partial_B\tau]\rangle
+(\rho_B\partial_B\tau-\rho_Q\partial_Q\tau)\langle \beta, b_2\rangle(m).
\end{align*}
We find hence that $X$ vanishes on all the functions defining $\Pi_B$ as in
Proposition \ref{description_of_Pi_A} if and only if
$[b_2,\partial_B\tau]=\partial_B(\nabla_{b_2}\tau)-\nabla_{\partial_Q\tau}b_2$
and
$\rho_B\partial_B(\tau)=\rho_Q\partial_Q(\tau)$.

In a similar manner, the anchors of the elements of the type \[\left((\nabla_a\tau^\dagger+
  \sigma_Q(\partial_Q\tau))^\times
  ,T\tau^\dagger\right)((a,b,c)(m))\]
% is \[\left((\partial_B(\nabla_a\tau)^\dagger+
%   \widehat{\nabla_{\partial_Q\tau}\cdot})^\dagger,
%   \widehat{[\partial_B\tau^\dagger,\cdot]} \right)((a,b,c)(m)).\] This
% sends \[\ell_{q_B^*\dr \varphi}\circ\pr_1-(p_B^*\ell_{\rho_B^*\dr
%   \varphi})\circ \pr_2 \quad \text{ to }\quad p_B^*q_B^*\langle \dr
% \varphi, \rho_Q\partial_Q\tau-\rho_B\partial_B\tau\rangle,\]
% \[\ell_{q_B^*\dr
%   \varphi}\circ\pr_2-(p_B^*\ell_{\rho_B^*\dr \varphi})\circ \pr_1
% \quad \text{ to }\quad 0\] and
% $\ell_{\dr\ell_\eta}\circ\pr_1-\ell_{\dr\ell_\eta}\circ\pr_2+\dr_B\eta\circ(p_B,p_B)$
% to \[\langle
% \eta, \partial_B\nabla_a\tau-[a,\partial_B\tau]-\nabla_{\partial_Q\tau}a\rangle(m)+
% (\rho_Q\partial_Q(\tau(m))-\rho_B\partial_B(\tau(m)))\langle\eta,
% a\rangle.\] Hence, we find here again that the anchors of the vectors
% of type (3)
annihilate all the functions in Proposition \ref{description_of_Pi_A}
if and only if $\rho_Q\partial_Q=\rho_B\partial_B$ and
$\partial_B\nabla_b\tau=[b,\partial_B\tau]+\nabla_{\partial_Q\tau}b$
for all $b\in\Gamma(B)$ and $\tau\in\Gamma(Q^*)$.

The anchor of \begin{align*} \left( T\sigma_Q(q)-\left(\sigma_Q(\nabla_{b_2}q)+R(b_1,b_2)q^\dagger\right)^\times,
    T\tilde q-\sigma_Q(\nabla_{b_1}q)^\times\right)((b_1,b_2,b_3)(m))
\end{align*}
is 
\begin{align*}
  \left(
    \widehat{\left[\widehat{\nabla_q\cdot},\cdot\right]}
-\left(\widehat{\nabla_{\nabla_{b_2}q}\cdot}+\partial_BR(b_1,b_2)q^\uparrow\right)^\uparrow,
    \widehat{\left[\widehat{\nabla_q\cdot},\cdot\right]}
-\widehat{\nabla_{\nabla_{b_1}q}\cdot}^\uparrow\right)((b_1,b_2,b_3)(m)).
\end{align*}
We call this vector $Y$. Applying $Y$ to $\ell_{q_B^*\dr
  f}\circ\pr_1-(p_B^*\ell_{\rho_B^*\dr f})\circ\pr_2$
yields
\begin{align*}
&\ell_{\ldr{\widehat{\nabla_q\cdot}}q_B^*\dr f}\left( (Tb_1+([b_1,b_2]+c)^\times)\rho_B(b_2)(m)
\right)\\
&-\langle q_B^*\dr f, \widehat{\nabla_{\nabla_{b_2}q}\cdot}+\partial_BR(b_1,b_2)q^\dagger\rangle
(b_1(m))
-\widehat{\nabla_q\cdot}(\ell_{\rho_B^*\dr f})(b_2(m))\\
=\,&\ell_{q_B^*\dr(\rho_Q(q)(f))}\left( (Tb_1+([b_1,b_2]+c)^\times)\rho_B(b_2)(m)
\right)-\rho_Q(\nabla_{b_2}q)(f)(m)-\langle\nabla_q^*(\rho_B^*\dr f),
b_m\rangle\\
=\,&\rho_B(b_2(m))\rho_Q(q)(f)-\rho_Q(\nabla_{b_2}q)(f)(m)-\rho_Q(q(m))\langle
\rho_B^*\dr f, b_2\rangle+\langle\rho_B^*\dr f, \nabla_qb_2\rangle\\
=\,&\langle \dr f, [\rho_B(b_2), \rho_Q(q)]-\rho_Q(\nabla_{b_2}q)+\rho_B(\nabla_qb_2)\rangle(m)
\end{align*}
and applying $Y$ to $\ell_{q_B^*\dr f}\circ\pr_2-(p_B^*\ell_{\rho_B^*\dr f})\circ\pr_1$ yields
% \begin{align*}
% &\ell_{q_B^*\dr(\rho_Q(q)(\varphi))}(Tb+c^\dagger)\rho(a)(m)-\langle q_B^*\dr
% \varphi, \widehat{\nabla_{\nabla_aq}\cdot}\rangle(b_m)
% -\widehat{\nabla_q\cdot}(\ell_{\rho_B^*\dr \varphi})(a_m)\\
% =\,&\rho_B(a(m))\rho_Q(q)(\varphi)-\rho_Q(\nabla_aq(m))(
% \varphi)-\langle\nabla_q^*\rho_B^*\dr \varphi, a(m)\rangle\\
\[\langle \dr f, [\rho_B(b_1), \rho_Q(q)]-\rho_Q(\nabla_{b_1}q)+\rho_B(\nabla_qb_1)\rangle(m).\]
%\end{align*} 
Now we apply $Y$ to
$\ell_{\dr\ell_\beta}\circ\pr_1-\ell_{\dr\ell_\beta}\circ\pr_2+\dr_B\beta\circ(p_B,p_B)$.
Let $\Phi_t$ be the flow of $\widehat{\nabla_q\cdot}\in\mx(B)$, and
$\phi_t$ the flow of $\rho_Q(q)\in\mx(M)$. Since for each $t$,
$\Phi_t$ is a vector bundle morphism $B\to B$ over $\phi_t$, we can
define for each section $b\in\Gamma(B)$ a new section
$b^t\in\Gamma(B)$ by $b^t(m)=\Phi_t(b(\phi_{-t}(m)))$ for all $m\in
M$.  We find that
$Y(\ell_{\dr\ell_\beta}\circ\pr_1-\ell_{\dr\ell_\beta}\circ\pr_2+\dr_B\beta\circ(p_B,p_B))$ equals 
\begin{equation*}
\begin{split}
&\ell_{\ldr{\widehat{\nabla_q\cdot}}\dr\ell_\beta}\left((Tb_1+([b_1,b_2]+b_3)^\times)\rho_B(b_2)(m)
\right)-\langle \dr\ell_\beta,
\widehat{\nabla_{\nabla_{b_2}q}\cdot}+\partial_BR(b_1,b_2)q^\uparrow\rangle(b_1(m))\\
&-\ell_{\ldr{\widehat{\nabla_q\cdot}}\dr\ell_\beta}\left((Tb_2+b_3^\times)\rho_B(b_1)(m)
\right)+\langle \dr\ell_\beta,
\widehat{\nabla_{\nabla_{b_1}q}\cdot}\rangle(b_2(m))\\
&+\left.\frac{d}{dt}\right\an{t=0}\dr_B\beta(b_1^t(m),
b_2^t(m)).
\end{split}
\end{equation*}
The first term is 
\[\ell_{\dr\ell_{\nabla_q^*\beta}}\left((Tb_1+([b_1,b_2]+b_3)^\times)\rho_B(b_2)(m)
\right)=\rho_B(b_2(m))\langle\nabla_q^*\beta, b_1\rangle+\langle \nabla_q^*\beta,
[b_1,b_2]+b_3\rangle.\]
The second term is
\[-(\ell_{\nabla_{\nabla_{b_2}q}^*\beta}+q_B^*\langle\beta,\partial_BR(b_1,b_2)q\rangle)(b_1(m))=-\langle
\nabla_{\nabla_{b_2}q}^*\beta,
b_1(m)\rangle-\langle\beta,\partial_BR(b_1,b_2)q\rangle(m).
\]
The third and fourth terms are 
\[-\ell_{\dr\ell_{\nabla_q^*\beta}}\left((Tb_2+b_3^\times)\rho_B(b_1)(m)
\right)=-\rho_B(b_1(m))\langle\nabla_q^*\beta, b_2\rangle-\langle \nabla_q^*\beta,
b_3\rangle(m)
\]
and 
\[\ell_{\nabla_{\nabla_{b_1}q}^*\beta}(b_2(m))=\langle \nabla_{\nabla_{b_1}q}^*\beta, b_2\rangle(m)
\]
and the fifth term is
\begin{align*}
&+\left.\frac{d}{dt}\right\an{t=0}\left(\rho_B(b_1^t)\langle\beta,b_2^t\rangle
-\rho_B(b_2^t)\langle \beta, b_1^t\rangle-\langle\beta, [b_1^t,b_2^t]\rangle
\right)\left(\phi_t(m)\right)\\
=\,
&-\rho_B(\nabla_qb_1(m))\langle \beta,
b_2\rangle-\rho_B(b_1(m))\langle\beta,\nabla_qb_2\rangle\\
&+\rho_B(\nabla_qb_2(m))\langle \beta,
b_1\rangle+\rho_B(b_2(m))\langle\beta,\nabla_qb_1\rangle+\langle\beta(m),
[\nabla_qb_1, b_2]+[b_1,\nabla_qb_2]\rangle\\
&+\rho_Q(q(m)) \left(\rho_B(b_1)\langle\beta,b_2\rangle
-\rho_B(b_2)\langle \beta, b_1\rangle-\langle\beta, [b_1,b_2]\rangle
\right).
\end{align*}
% =\,&\left\langle \eta,
%   -\nabla_q[a,b]+\nabla_{\nabla_bq}a-\partial_BR(a,b)q-\nabla_{\nabla_aq}b+[\nabla_qa,
%   b]+[a,\nabla_qb]\right\rangle(m)\\
% &+\rho_B(b)\rho_Q(q)\langle \eta,
% a\rangle-\rho_B(b)\langle\eta,\nabla_qa\rangle+\rho_Q(q)\langle\eta,
% [a,b]\rangle\\
% &-\rho_B(\nabla_bq)\langle\eta,a\rangle-\rho_B(a)\rho_Q(q)\langle\eta,b\rangle+\rho_B(a)\langle\eta,\nabla_qb\rangle+\rho_Q(\nabla_aq)\langle
% \eta,b\rangle \\
% &-\rho_B(\nabla_qa(m))\langle \eta,
% b\rangle-\rho_B(a(m))\langle\eta,\nabla_qb\rangle\\
% &+\rho_B(\nabla_qb(m))\langle \eta,
% a\rangle+\rho_B(b(m))\langle\eta,\nabla_qa\rangle\\
% &+\rho_Q(q(m)) \left(\rho_B(a)\langle\eta,b\rangle
% -\rho_B(b)\langle \eta, a\rangle-\langle\eta, [a,b]\rangle
% \right)\\
Hence, we get
\begin{align*}
&\left\langle \beta,
  -\nabla_q[b_1,b_2]+\nabla_{\nabla_{b_2}q}b_1-\partial_BR(b_1,b_2)q-\nabla_{\nabla_{b_1}q}b_2+[\nabla_qb_1,
  b_2]+[b_1,\nabla_qb_2]\right\rangle(m)\\
&+\left([\rho_B(b_2), \rho_Q(q)](m) -\rho_B(\nabla_{b_2}q(m))+\rho_B(\nabla_qb_2(m)) 
\right)\langle \beta, b_1\rangle\\
&+\left(-[\rho_B(b_1), \rho_Q(q)](m)+\rho_Q(\nabla_{b_1}q(m))-\rho_B(\nabla_qb_1(m))\right)\langle\beta,b_2\rangle
\end{align*}
Hence, we find that the anchors of sections of type \eqref{type4} vanish on all the functions defining $\Pi_B$
as in Proposition \ref{description_of_Pi_A} if and only if 
\[[\rho_B(b), \rho_Q(q)]=\rho_Q(\nabla_{b}q)-\rho_B(\nabla_q{b})
\]
and 
\[\partial_BR(b_1,b_2)q=-\nabla_q[b_1,b_2]+\nabla_{\nabla_{b_2}q}b_1-\nabla_{\nabla_{b_1}q}b_2+[\nabla_qb_1,
  b_2]+[b_1,\nabla_qb_2]
\]
for all $b,b_1,b_2\in\Gamma(B)$ and $q\in\Gamma(Q)$.
\end{proof}

\begin{proposition}\label{theorem_integrability}
  Assume that $\partial_Q=\partial_Q^*$.  The Courant brackets of
  extensions of any two of the special sections of $\Pi_{\mathbb
    E}\to\Pi_B$ restrict to a section of $\Pi_{\mathbb E}\to\Pi_B$ if
  and only if
\begin{enumerate}
\item
  $(\Delta_{\partial_Q\tau_1}\tau_2-\nabla_{\partial_B\tau_2}\tau_1)+(\Delta_{\partial_Q\tau_2}\tau_1-\nabla_{\partial_B\tau_1}\tau_2)=\rho_Q^*\dr\langle
  \tau_1, \partial_Q\tau_2\rangle$,
\item $\partial_Q(\Delta_q\tau)=\nabla_{\partial_B\tau}q+\lb
  q, \partial_Q\tau\rb+\partial_B^*\langle \tau, \nabla_\cdot q\rangle$,
\item
  $\partial_BR(b_1,b_2)q=-\nabla_q[b_1,b_2]+[\nabla_qb_1,b_2]+[b_1,\nabla_qb_2]+\nabla_{\nabla_{b_2}q}b_1-\nabla_{\nabla_{b_1}q}b_2$,
\item
  $R(q,\partial_Q\tau)b-R(b,\partial_B\tau)q=\Delta_q\nabla_b\tau-\nabla_b\Delta_q\tau+\Delta_{\nabla_bq}\tau-\nabla_{\nabla_qb}\tau-\langle\nabla_{\nabla_\cdot b}q, \tau\rangle$,
\item$\partial_QR(q_1,q_2)b=-\nabla_b\lb q_1,q_2\rb+\lb q_1,
  \nabla_{b}q_2\rb+\lb \nabla_bq_1,
  q_2\rb+\nabla_{\nabla_{q_2}b}q_1-\nabla_{\nabla_{q_1}b}q_2+\partial_B^*\langle
  R(\cdot, b)q_1, q_2\rangle$.
\item
\begin{equation*}
\begin{split}
  &\nabla_{b_2}R(q_1,q_2)b_1-\nabla_{b_1}R(q_1,q_2)b_2+R(q_1,q_2)[b_1,b_2]\\
  &+R(\nabla_{b_1}q_1,q_2)b_2+R(q_1,\nabla_{b_1}q_2)b_2-R(\nabla_{b_2}q_1,q_2)b_1-R(q_1,\nabla_{b_2}q_2)b_1\\
  &+\Delta_{q_1}R(b_1,b_2)q_2-\Delta_{q_2}R(b_1,b_2)q_1-R(b_1,b_2)\lb q_1,q_2\rb\\
  &-R(\nabla_{q_1}b_1,b_2)q_2-R(b_1,\nabla_{q_1}b_2)q_2+R(\nabla_{q_2}b_1,b_2)q_1-R(b_1,\nabla_{q_2}b_2)q_1\\
  =&\,\langle R(b_1,\nabla_\cdot b_2)q_1,q_2\rangle +\langle R(\nabla_\cdot
  b_1, b_2)q_1,q_2\rangle -\rho_Q^*\dr\langle R(b_1,b_2)q_1, q_2\rangle.
\end{split}
\end{equation*}
\end{enumerate}
for all $b,b_1,b_2\in\Gamma(B)$, $q,q_1,q_2\in\Gamma(Q)$, $\tau,\tau_1,\tau_2\in\Gamma(Q^*)$.
\end{proposition}

In order to prove this, we need to extend the four types of
sections \eqref{type1},\eqref{type2},\eqref{type3},\eqref{type4} to
sections of $T\mathbb E\times T\mathbb E\to TB\times TB$.

\begin{lemma}\label{extensions}
\begin{enumerate}
\item The sections of type \eqref{type1} are already restrictions to $\Pi_B$ 
of sections $\left((\tau^\dagger)^\times,
    (\tau^\dagger)^\times\right)\in\Gamma_{TB\times TB}(T\mathbb E\times T\mathbb E)$, $\tau\in\Gamma(Q^*)$.
\item For $\chi\in\Gamma(Q^*)$, the vector bundle morphism $\nabla_\cdot\chi\colon B\to Q^*$ 
can be written
\[\nabla_\cdot\chi=\sum_{i=1}^n\ell_{\xi_i}\cdot\chi_i
\]
with some $\xi_1,\ldots,\xi_n\in\Gamma(B^*)$ and $\chi_1,\ldots,\chi_n\in\Gamma(Q^*)$.
The section 
\[\left( T\chi^\dagger, 0\right)
    -\sum_{i=1}^n(p_B\circ\pr_2)^*\ell_{\xi_i}\cdot\left((\chi_i^\dagger)^\times,0\right)
+\left(0,\sigma_Q(\partial_Q\chi)^\times\right)\] of $T\mathbb E\times T\mathbb E\to TB\times TB$ restricts 
to $\left( T\chi^\dagger
    -(\nabla_{p_B\circ\pr_2}\chi^\dagger)^\times,
    \sigma_Q(\partial_Q\chi)^\times\right)$ on $\Pi_B$.
\item In the same manner, for $ \tau\in\Gamma(Q^*)$, the morphism $\nabla_\cdot\tau\colon B\to Q^*$ can be written
\[\nabla_\cdot\tau=\sum_{i=1}^k\ell_{\eta_i}\cdot\tau_i
\]
with some $\eta_1,\ldots,\eta_k\in\Gamma(B^*)$ and $\tau_1,\ldots,\tau_n\in\Gamma(Q^*)$.
The section \[\left( \sigma_Q(\partial_Q\tau)^\times, 0\right)
    +\sum_{i=1}^k(p_B\circ\pr_1)^*\ell_{\eta_i}\cdot\left((\tau_i^\dagger)^\times,0\right)
+\left(0,T\tau^\dagger\right)\] of $T\mathbb E\times T\mathbb E\to TB\times TB$ restricts 
to 
$\left( (\nabla_{p_B\circ
      \pr_1}\tau^\dagger+\sigma_Q(\partial_Q\tau))^\times,
    T\tau^\dagger\right)$ on $\Pi_B$.
\item Finally, for $q\in\Gamma(Q)$, the morphism
$\nabla_\cdot q\colon B\to Q$ can be written 
\[\nabla_\cdot q=\sum_{i=1}^l\ell_{\gamma_i}\cdot q_i
\]
with some $\gamma_1,\ldots,\gamma_l\in\Gamma(B^*)$ and $q_1,\ldots,q_n\in\Gamma(Q)$,
and the tensor $R(\cdot,\cdot)q\colon B\times B\to Q^*$ can be written 
\[R(b_1,b_2)q=\sum_{i=1}^m\sum_{j=1}^p\ell_{\alpha_i}(b_1)\ell_{\beta_j}(b_2)\tau_{ij}
\]
with some  $\alpha_1,\ldots,\alpha_m,\beta_1,\ldots,\beta_p\in\Gamma(B^*)$ and $\tau_{ij}\in\Gamma(Q^*)$, $i=1,\ldots,m$, $j=1,\ldots,p$.
The section
\begin{equation*}
\begin{split}
  & \left(T\sigma_Q(q) , T\sigma_Q(q)\right)-\sum_{i=1}^l(p_B\circ\pr_2)^*\ell_{\gamma_i}\cdot\left((\sigma_Q(q_i))^\times,0\right)-\sum_{i=1}^l(p_B\circ\pr_1)^*\ell_{\gamma_i}\cdot\left(0,(\sigma_Q(q_i))^\times\right)\\
  &-\sum_{i=1}^m\sum_{j=1}^p(p_B\circ\pr_1)^*\ell_{\alpha_i}(p_B\circ\pr_2)^*\ell_{\beta_j}\cdot\left((\tau_{ij}^\dagger)^\times,0\right)
\end{split}
\end{equation*}
restricts to 
\[\left(T\sigma_Q(q) -(\sigma_Q(\nabla_{p_B\circ
        \pr_2}q)+R(p_B\circ \pr_1, p_B\circ \pr_2)q^\dagger)^\times,
    T\sigma_Q(q)-(\sigma_Q(\nabla_{p_B\circ
        \pr_1}q))^\times\right)\]
on $\Pi_B$.
\end{enumerate}
\end{lemma}

\begin{proof}[Proof of Theorem \ref{theorem_integrability}]
We compute successively the Courant brackets of the
extensions of our four special types of sections.
The Courant bracket of the extension $S$ of
\begin{align*}
  &\Biggl(T\left(\sigma_Q(q_1)+\chi_1^\dagger\right)+\left(-\sigma_Q(\nabla_{p_B\circ\pr_2}q_1)-R(p_B\circ\pr_1,p_B\circ\pr_2)q_1^\dagger+\sigma_1^\dagger\right)^\times\\
&\hspace*{6cm}+\left(\nabla_{p_B\circ\pr_1}\tau_1^\dagger-\nabla_{p_B\circ\pr_2}\chi_1^\dagger+\sigma_Q(\partial_Q\tau_1)\right)^\times,\\
  &\hspace*{3cm}T\left(\sigma_Q( q_1)+\tau_1^\dagger\right)
  +\left(-\sigma_Q(\nabla_{p_B\circ\pr_1}q_1)+\sigma_1^\dagger+\sigma_Q(\partial_Q\chi_1)\right)^\times\Biggr)
\end{align*}
with $\left( (\sigma_2^\dagger)^\times,
  (\sigma_2^\dagger)^\times\right)$ equals
$ \left( (\Delta_{q_1}\sigma_2^\dagger)^\times,
  (\Delta_{q_1}\sigma_2^\dagger)^\times\right)$,
which restricts to a section of type \eqref{type1} of $\Pi_{\mathbb
  E}$ on $\Pi_B$.  Since
%\begin{align*}
  $\left\langle S, \left( (\sigma_2^\dagger)^\times,
    (\sigma_2^\dagger)^\times\right)\right\rangle=p_B^*\left\langle
    \sigma_Q(q_1)+\chi_1^\dagger, \sigma_2^\dagger\right\rangle
  -p_B^*\left\langle\sigma_Q(q_1)+\tau_1^\dagger, \sigma_2^\dagger
  \right\rangle= p_B^*q_B^*(\langle q_1,\sigma_2\rangle-\langle
  q_1,\sigma_2\rangle)=0$,
%\end{align*}
the Courant bracket of $\left( (\sigma_2^\dagger)^\times,
  (\sigma_2^\dagger)^\times\right)$ with $S$ % the extension of 
% \begin{align*}
%   &\Biggl(T\left(\tilde
%     q_1+\chi_1^\dagger\right)+\left(-\widetilde{\nabla_{p_B\circ\pr_2}q_1}-R(p_B\circ\pr_1,p_B\circ\pr_2)q_1^\dagger+\sigma_1^\dagger+\nabla_{p_B\circ\pr_1}\tau_1^\dagger-\nabla_{p_B\circ\pr_2}\chi_1^\dagger+\widetilde{\partial_Q\tau_1}\right)^\dagger,\\
%   &\hspace*{6cm}T\left(\tilde q_1+\tau_1^\dagger\right)
%   +\left(-\widetilde{\nabla_{p_B\circ\pr_1}q_1}+\sigma_1^\dagger+\widetilde{\partial_Q\chi_1}\right)^\dagger\Biggr)
% \end{align*}
equals
$ -\left( (\Delta_{q_1}\sigma_2^\dagger)^\times,
  (\Delta_{q_1}\sigma_2^\dagger)^\times\right)$.
Hence the bracket of sections of type \eqref{type1} 
with the four types of sections have values in $\Pi_{\mathbb E}$ on $\Pi_B$.
\medskip

Choose $\chi_1,\chi_2,\tau_1,\tau_2\in\Gamma(Q^*)$.  We write
\[\left( T\chi_1^\dagger, 0\right)
    -\sum_{i=1}^n(p_B\circ\pr_2)^*\ell_{\xi^1_i}\cdot\left(({\chi^1_i}^\dagger)^\times,0\right)
+\left(0,\sigma_Q(\partial_Q\chi_1)^\times\right)\] for the extension of 
$ \left( T\chi_1^\dagger -(\nabla_{p_B\circ\pr_2}\chi_1^\dagger)^\times,
    \sigma_Q(\partial_Q\chi_1)^\times\right)$ as in Lemma \ref{extensions},
\[\left( T\chi_2^\dagger, 0\right)
    -\sum_{j=1}^m(p_B\circ\pr_2)^*\ell_{\xi^2_j}\cdot\left(({\chi^2_j}^\dagger)^\times,0\right)
+\left(0,\sigma_Q(\partial_Q\chi_2)^\times\right)\] for the extension of 
$ \left( T\chi_2^\dagger -(\nabla_{p_B\circ\pr_2}\chi_2^\dagger)^\times,
    \sigma_Q(\partial_Q\chi_2)^\times\right)$,
\[\left( \sigma_Q(\partial_Q\tau_1)^\times, 0\right)
    +\sum_{i=1}^k(p_B\circ\pr_1)^*\ell_{\eta^1_i}\cdot\left(({\tau^1_i}^\dagger)^\times,0\right)
+\left(0,T\tau_1^\dagger\right)\]for the extension of 
$\left( (\nabla_{p_B\circ
      \pr_1}\tau_1^\dagger+\sigma_Q(\partial_Q\tau_1))^\times,
    T\tau_1^\dagger\right)$  and 
\[\left( \sigma_Q(\partial_Q\tau_2)^\times, 0\right)
    +\sum_{j=1}^l(p_B\circ\pr_1)^*\ell_{\eta^2_i}\cdot\left(({\tau^2_i}^\dagger)^\times,0\right)
+\left(0,T\tau_2^\dagger\right)\]for the extension of 
$\left( (\nabla_{p_B\circ
      \pr_1}\tau_2^\dagger+\sigma_Q(\partial_Q\tau_2))^\times,
    T\tau_2^\dagger\right)$.
We find 
\begin{equation*}
\begin{split}
&\left\lb\left( T\chi_1^\dagger, 0\right)
    -\sum_{i=1}^n(p_B\circ\pr_2)^*\ell_{\xi^1_i}\cdot\left(({\chi^1_i}^\dagger)^\times,0\right)
+\left(0,\sigma_Q(\partial_Q\chi_1)^\times\right), \right.\\
&\qquad \qquad\left. \left( T\chi_2^\dagger, 0\right)
    -\sum_{j=1}^m(p_B\circ\pr_2)^*\ell_{\xi^2_j}\cdot\left(({\chi^2_j}^\dagger)^\times,0\right)
+\left(0,\sigma_Q(\partial_Q\chi_2)^\times\right)\right\rb=0,
\end{split}
\end{equation*}
\begin{align*}
\begin{split}
&\left\lb
\left( \sigma_Q(\partial_Q\tau_2)^\times, 0\right)
    +\sum_{j=1}^l(p_B\circ\pr_1)^*\ell_{\eta^2_i}\cdot\left(({\tau^2_i}^\dagger)^\times,0\right)
+\left(0,T\tau_2^\dagger\right),\right.\\
&\qquad \qquad\qquad \qquad\left.\left( \sigma_Q(\partial_Q\tau_1)^\times, 0\right)
    +\sum_{i=1}^k(p_B\circ\pr_1)^*\ell_{\eta^1_i}\cdot\left(({\tau^1_i}^\dagger)^\times,0\right)
+\left(0,T\tau_1^\dagger\right)\right\rb=0
\end{split}
\end{align*}
and
\begin{align*}
&\left\lb 
\left( T\chi^\dagger, 0\right)
    -\sum_{i=1}^n(p_B\circ\pr_2)^*\ell_{\xi_i}\cdot\left((\chi_i^\dagger)^\times,0\right)
+\left(0,\sigma_Q(\partial_Q\chi)^\times\right),\right.\\
&\qquad \qquad\qquad \qquad\left.\left( \sigma_Q(\partial_Q\tau)^\times, 0\right)
    +\sum_{i=1}^k(p_B\circ\pr_1)^*\ell_{\eta_i}\cdot\left((\tau_i^\dagger)^\times,0\right)
+\left(0,T\tau^\dagger\right)
\right\rb\\
=&\left((-\Delta_{\partial_Q\tau}\chi^\dagger+\mathcal D_{\mathbb
    E}q_B^*\langle
  \chi,\partial_Q\tau\rangle)^\times),0\right)+\sum_{i=1}^k(q_B\circ
p_B\circ
\pr_1)^*\langle \partial_B\chi,\eta_i\rangle\left((\tau_i^\dagger)^\times,0\right)\\
&+\sum_{i=1}^n(q_B\circ p_B\circ\pr_2)^*\langle\xi_i,
\chi\rangle\cdot\left((\chi_i^\dagger)^\times,0\right)
+\left(0,(\Delta_{\partial_Q\chi}\tau^\dagger)^\times\right)\\
=&\left( (-\Delta_{\partial_Q\tau}\chi+\rho_Q^*\dr\langle
  \chi,\partial_Q\tau\rangle+\nabla_{\partial_B\chi}\tau+\nabla_{\partial_B\tau}\chi)^\dagger)^\times,
  (\Delta_{\partial_Q\chi}\tau^\dagger)^\times\right).
\end{align*}
The restriction of this to $\Pi_B$ is a section (of type \eqref{type1}) of $\Pi_{\mathbb E}$
if and only if 
\[-\Delta_{\partial_Q\tau}\chi+\rho_Q^*\dr\langle
  \chi,\partial_Q\tau\rangle+\nabla_{\partial_B\chi}\tau+\nabla_{\partial_B\tau}\chi=\Delta_{\partial_Q\chi}\tau.
\]
Since using $\partial_Q=\partial_Q^*$
\begin{align*}
&\left\langle 
\left( T\chi^\dagger, 0\right)
    -\sum_{i=1}^n(p_B\circ\pr_2)^*\ell_{\xi_i}\cdot\left((\chi_i^\dagger)^\times,0\right)
+\left(0,\sigma_Q(\partial_Q\chi)^\times\right),\right.\\
&\qquad \qquad\qquad \qquad\left.\left( \sigma_Q(\partial_Q\tau)^\times, 0\right)
    +\sum_{i=1}^k(p_B\circ\pr_1)^*\ell_{\eta_i}\cdot\left((\tau_i^\dagger)^\times,0\right)
+\left(0,T\tau^\dagger\right)
\right\rangle\\
=&p_B^*q_B^*\left(\langle \partial_Q\tau,
      \chi\rangle- \langle \tau, \partial_Q\chi\rangle\right)=0,
\end{align*}
the bracket
\begin{align*}
&\left\lb \left( \sigma_Q(\partial_Q\tau)^\times, 0\right)
    +\sum_{i=1}^k(p_B\circ\pr_1)^*\ell_{\eta_i}\cdot\left((\tau_i^\dagger)^\times,0\right)
+\left(0,T\tau^\dagger\right)
,\right.\\
&\qquad \qquad\qquad \qquad\left.
 \left( T\chi^\dagger, 0\right)   -\sum_{i=1}^n(p_B\circ\pr_2)^*\ell_{\xi_i}\cdot\left((\chi_i^\dagger)^\times,0\right)
+\left(0,\sigma_Q(\partial_Q\chi)^\times\right)
\right\rb
\end{align*}
then also restricts to a section of $\Pi_{\mathbb E}$ on $\Pi_B$.
\medskip

% Next, note that each $q\in\Gamma(Q)$ defines a derivation 
% $D_q$ of $\Hom(B,Q^*)$ with symbol $\rho_Q(q)$ by 
% \[(D_q\Phi)(a)=\Delta_q(\Phi(a))-\Phi(\nabla_qa)\]
% for $\Phi\in\Gamma(\Hom(B,Q^*))$ and $a\in\Gamma(B)$.

We compute 
\begin{equation*}
\begin{split}
&\left\lb
\left(T\sigma_Q(q) , T\sigma_Q(q)\right)-\sum_{i=1}^l(p_B\circ\pr_2)^*\ell_{\gamma_i}\cdot\left((\sigma_Q(q_i))^\times,0\right)\right.\\
&\qquad\qquad-\sum_{i=1}^l(p_B\circ\pr_1)^*\ell_{\gamma_i}\cdot\left(0,(\sigma_Q(q_i))^\times\right)\\
&\qquad\qquad-\sum_{i=1}^m\sum_{j=1}^p(p_B\circ\pr_1)^*\ell_{\alpha_i}(p_B\circ\pr_2)^*\ell_{\beta_j}\cdot\left((\sigma_{ij}^\dagger)^\times,0\right),\\
&\qquad\qquad\qquad\qquad\left.\left( \sigma_Q(\partial_Q\tau)^\times, 0\right)
    +\sum_{i=1}^k(p_B\circ\pr_1)^*\ell_{\eta_i}\cdot\left((\tau_i^\dagger)^\times,0\right)
+\left(0,T\tau^\dagger\right)
\right\rb\\
=&\left( (\sigma_Q(\lb
    q, \partial_Q\tau\rb)-R(q,\partial_Q\tau)^\dagger)^\times,0\right)\\
& +\sum_{i=1}^k\left((p_B\circ\pr_1)^*\ell_{\nabla_q^*\eta_i}\cdot\left((\tau_i^\dagger)^\times,0\right)
+(p_B\circ\pr_1)^*\ell_{\eta_i}\cdot\left((\Delta_q\tau_i^\dagger)^\times,0\right)\right)\\
&+\left(0, T\Delta_q\tau^\dagger
\right)+\sum_{i=1}^l(q_B\circ
p_B\circ\pr_2)^*\langle\gamma_i, \partial_B\tau\rangle\cdot\left((\sigma_Q(q_i))^\times,0\right)\\
&+\sum_{i=1}^l(p_B\circ\pr_1)^*\ell_{\gamma_i}\cdot\left(0,((-\Delta_{q_i}\tau+\rho_Q^*\dr\langle
  q_i,\tau\rangle)^\dagger)^\times\right)\\
&-\mathcal D_{\overline{T\mathbb E}\times T\mathbb
  E}\sum_{i=1}^l(p_B\circ\pr_1)^*\ell_{\gamma_i}(q_B\circ
p_B\circ\pr_2)^*\langle q_i,\tau\rangle
\\
&+\sum_{i=1}^m\sum_{j=1}^p(p_B\circ\pr_1)^*\ell_{\alpha_i}(q_B\circ p_B\circ\pr_2)^*\langle\partial_B\tau,\beta_j\rangle\cdot\left((\sigma_{ij}^\dagger)^\times,0\right)\\
\end{split}
\end{equation*}
We have 
\begin{equation*}
\begin{split}
&\mathcal D_{\overline{T\mathbb E}\times T\mathbb
  E}\sum_{i=1}^l(p_B\circ\pr_1)^*\ell_{\gamma_i}(q_B\circ
p_B\circ\pr_2)^*\langle q_i,\tau\rangle\\
&=\sum_{i=1}^l(p_B\circ\pr_1)^*\ell_{\gamma_i}(0,\mathcal D_{T\mathbb E}p_B^*q_B^*\langle q_i,\tau\rangle)
-\sum_{i=1}^l(q_B\circ
p_B\circ\pr_2)^*\langle q_i,\tau\rangle(\mathcal D_{T\mathbb
  E}p_B^*\ell_{\gamma_i},0)\\
&=\sum_{i=1}^l(p_B\circ\pr_1)^*\ell_{\gamma_i}(0,(\Theta^*\dr q_B^*\langle q_i,\tau\rangle)^\times)
-\sum_{i=1}^l(q_B\circ
p_B\circ\pr_2)^*\langle q_i,\tau\rangle\left((\Theta^*\dr\ell_{\gamma_i})^\times,0\right)\\
&=\sum_{i=1}^l(p_B\circ\pr_1)^*\ell_{\gamma_i}(0,((\rho_Q^*\dr \langle q_i,\tau\rangle)^\dagger)^\times)\\
&\quad-\sum_{i=1}^l(q_B\circ
p_B\circ\pr_2)^*\langle q_i,\tau\rangle\left((\sigma_Q(\partial^*\gamma_i)+\widetilde{\nabla_\cdot^*\gamma_i}
)^\times,0\right).
\end{split}
\end{equation*}
Hence, the bracket equals
\begin{multline}
  \left( (\sigma_Q(\lb
    q, \partial_Q\tau\rb)-R(q,\partial_Q\tau)^\dagger)^\times,0\right)
  +\sum_{i=1}^k(p_B\circ\pr_1)^*\ell_{\nabla_q^*\eta_i}\cdot\left((\tau_i^\dagger)^\times,0\right)\\
    +\sum_{i=1}^k(p_B\circ\pr_1)^*\ell_{\eta_i}\cdot\left((\Delta_q\tau_i^\dagger)^\times,0\right)
  +\left(0, T\Delta_q\tau^\dagger \right)\\+\sum_{i=1}^l(q_B\circ
  p_B\circ\pr_2)^*\langle\gamma_i, \partial_B\tau\rangle\cdot\left((\sigma_Q(q_i))^\times,0\right)+\sum_{i=1}^l(p_B\circ\pr_1)^*\ell_{\gamma_i}\cdot\left(0,(-\Delta_{q_i}\tau^\dagger)^\times\right)\\
  +\sum_{i=1}^l(q_B\circ p_B\circ\pr_2)^*\langle
  q_i,\tau\rangle\left((\sigma_Q(\partial^*\gamma_i)+\widetilde{\nabla_\cdot^*\gamma_i}
    )^\times,0\right)\\
  +\sum_{i=1}^m\sum_{j=1}^p(p_B\circ\pr_1)^*\ell_{\alpha_i}(q_B\circ
  p_B\circ\pr_2)^*\langle\partial_B\tau,\beta_j\rangle\cdot\left((\sigma_{ij}^\dagger)^\times,0\right).
\end{multline}
On $TB\times_{q_B\circ p_B} TB$ we have 
\[\sum_{i=1}^l(q_B\circ
p_B\circ\pr_2)^*\langle\gamma_i, \partial_B\tau\rangle\cdot\left((\sigma_Q(q_i))^\times,0\right)=\left(\sigma_Q(\nabla_{\partial_B\tau}q)^\times,0\right)
\]
and 
\[\sum_{i=1}^l(q_B\circ
p_B\circ\pr_2)^*\langle
q_i,\tau\rangle\left(\left(\sigma_Q(\partial^*\gamma_i)\right)^\times,0\right)
=\left(\left(\sigma_Q(\partial_B^*\langle\nabla_\cdot q, \tau
      \rangle)\right)^\times, 0)\right).
\]
(Note that since $\langle\nabla_\cdot q, \tau
      \rangle$ is a section of $B^*$, $\partial_B^*\langle\nabla_\cdot q, \tau
      \rangle$ is a section of $Q$.)
Since 
\begin{equation*}
\begin{split}
\sum_{i=1}^k\langle\nabla_q^*\eta_i, b_1\rangle\cdot \tau_i
+\langle\eta_i,
b_1\rangle\cdot\Delta_q\tau_i
&=\sum_{i=1}^k-\langle\eta_i, \nabla_qb_1\rangle\cdot \tau_i
+\Delta_q\left(\langle\eta_i,
b_1\rangle\cdot\tau_i\right)\\
&=-\nabla_{\nabla_qb_1}\tau+\Delta_q\nabla_{b_1}\tau,
\end{split}
\end{equation*}
we find
\begin{equation*}
\begin{split}
&\sum_{i=1}^k\left((p_B\circ\pr_1)^*\ell_{\nabla_q^*\eta_i}\cdot\left((\tau_i^\dagger)^\times,0\right)
+(p_B\circ\pr_1)^*\ell_{\eta_i}\cdot\left((\Delta_q\tau_i^\dagger)^\times,0\right)\right)(b_1,b_2,b_3)(m)\\
&=\left( \left((-\nabla_{\nabla_qb_1}\tau+\Delta_q\nabla_{b_1}\tau)^\dagger\right)^\times,0\right)(b_1,b_2,b_3)(m)
\end{split}
\end{equation*}
 at $(b_1,b_2,b_3)(m)=((Tb_1+([b_1,b_2]+b_3)^\times)(\rho_B(b_2)(m)),
(Tb_2+b_3^\times)(\rho_B(b_1)(m)))\in\Pi_B$.
Note also that
\begin{equation*}
\begin{split}
\sum_{i=1}^l-\langle \gamma_i,b_1 \rangle\cdot\Delta_{q_i}\tau
=-\Delta_{\nabla_{b_1}q}\tau+\sum_{i=1}^l\langle q_i,\tau\rangle\rho_Q^*\dr\langle \gamma_i,b_1 \rangle
\end{split}
\end{equation*}
and 
\begin{equation*}
\begin{split}
\sum_{i=1}^l\langle
q_i,\tau\rangle\cdot\langle\nabla_\cdot^*\gamma_i,b_1\rangle
=-\langle\nabla_{\nabla_\cdot b_1}q,\tau\rangle+\sum_{i=1}^l\langle q_i,\tau\rangle\rho_Q^*\dr\langle \gamma_i,b_1 \rangle.
\end{split}
\end{equation*}
The bracket is hence at $(b_1,b_2,b_3)(m)\in\Pi_B$:
\begin{multline*}
  \Biggl( \sigma_Q(\lb
    q,\partial_Q\tau\rb)^\times+\sigma_Q(\nabla_{\partial_B\tau}q)^\times+\sigma_Q(\partial_B^*\langle\tau,
      \nabla_\cdot
      q\rangle)^\times
 +\left(-R(q,\partial_Q\tau)b_1^\dagger+R(b_1,\partial_B\tau)q^\dagger+\Delta_q\nabla_{b_1}\tau^\dagger\right)^\times\\
-\biggl(\nabla_{\nabla_qb_1}\tau^\dagger
    +\langle\tau,
    \nabla_{\nabla_\cdot
      b_1}q\rangle^\dagger-\sum_{i=1}^l\langle q_i,\tau\rangle\rho_Q^*\dr\langle \gamma_i,b_1 \rangle^\dagger\biggr)^\times, \\
 T\Delta_q\tau^\dagger+\biggl(-\Delta_{\nabla_{b_1}q}\tau^\dagger+\sum_{i=1}^l\langle q_i,\tau\rangle\rho_Q^*\dr\langle \gamma_i,b_1 \rangle^\dagger\biggr)^\times \Biggr)
 \bigl( (b_1,b_2,b_3)(m)\bigr).
\end{multline*}
This is an element of $\Pi_{\mathbb E}$ (of type
\eqref{type1}+\eqref{type3}) if and only if 
\[\partial_Q\Delta_q\tau=\lb q,\partial_Q\tau\rb+\nabla_{\partial_B\tau}q+\partial_B^*\langle\tau,
    \nabla_\cdot
    q\rangle
\]
and 
\[-R(q,\partial_Q\tau)b_1+R(b_1,\partial_B\tau)q+\Delta_q\nabla_{b_1}\tau
-\nabla_{\nabla_qb_1}\tau-\langle\tau, \nabla_{\nabla_\cdot b_1}q\rangle
=-\Delta_{\nabla_{b_1}q}\tau+\nabla_{b_1}\Delta_q\tau.
\]

\bigskip

Choose finally $q_1, q_2\in\Gamma(Q)$.  
We write 
\begin{multline}\label{q_1}
\left(T\sigma_Q(q_1), T\sigma_Q(q_1)\right) -\sum_i(p_B\circ \pr_2)^*\ell_{\xi_i}\left({\sigma_Q(
    q_i)}^\times,0\right)\\
-\sum_k\sum_l(p_B\circ
  \pr_1)^*\ell_{\alpha_k}(p_B\circ
  \pr_2)^*\ell_{\beta_l}\cdot\left({\tau_{kl}^\dagger}^\times,0\right)-\sum_i(p_B\circ \pr_1)^*\ell_{\xi_i}\left(0,{\sigma_Q(q_i)}^\times\right)
\end{multline}
for the extension of
\[\left(T\sigma_Q(q_1) -(\sigma_Q(\nabla_{p_B\circ
      \pr_2}q_1)+R(p_B\circ \pr_1, p_B\circ
  \pr_2)q_1^\dagger)^\times, T\sigma_Q(q_1)-(\sigma_Q(\nabla_{p_B\circ \pr_1}q_1))^\times\right)
\]
and 
\begin{multline}\label{q_2}
\left(T\sigma_Q(q_2), T\sigma_Q(q_2)\right)-\sum_j(p_B\circ \pr_2)^*\ell_{\eta_j}\left(\sigma_Q(r_j)^\times,0\right)\\
-\sum_s\sum_t(p_B\circ
  \pr_1)^*\ell_{\gamma_s}(p_B\circ
  \pr_2)^*\ell_{\delta_t}\left({\chi_{st}^\dagger}^\times,0\right)-\sum_j(p_B\circ
  \pr_1)^*\ell_{\eta_j}\left(0, \sigma_Q(r_j)^\times\right)
\end{multline}
for the extension of
\[\left(T\sigma_Q(q_2) -(\sigma_Q(\nabla_{p_B\circ
      \pr_2}q_2)+R(p_B\circ \pr_1, p_B\circ
  \pr_2)q_2^\dagger)^\times, T\sigma_Q(q_2)-\sigma_Q(\nabla_{p_B\circ \pr_1}q_2)^\times\right)
\]
The bracket of these two extensions equals
\begin{multline*}
\Bigl\lb \left(T\sigma_Q(q_1), T\sigma_Q(q_1)\right), \left(T\sigma_Q(q_2),
    T\sigma_Q(q_2)\right)\Bigr\rb\\
- \Bigl\lb \left(T\sigma_Q(q_1), T\sigma_Q(q_1)\right), \sum_j(p_B\circ \pr_2)^*\ell_{\eta_j}\left(\sigma_Q(r_j)^\times,0\right)
  \Bigr\rb\\
-\Bigl\lb \left(T\sigma_Q(q_1), T\sigma_Q(q_1)\right),\sum_s\sum_t(p_B\circ
  \pr_1)^*\ell_{\gamma_s}(p_B\circ
  \pr_2)^*\ell_{\delta_t}\left({\chi_{st}^\dagger}^\times,0\right)
\Bigr\rb\\
 -\Bigl\lb \left(T\sigma_Q(q_1), T\sigma_Q(q_1)\right), \sum_j(p_B\circ \pr_1)^*\ell_{\eta_j} \left(0,\sigma_Q(r_j)^\times\right)
    \Bigr\rb\\
-\Bigl\lb  \sum_i(p_B\circ \pr_2)^*\ell_{\xi_i}\left(\sigma_Q(q_i)^\times,0\right), \left( T\sigma_Q(q_2), T\sigma_Q(q_2)\right)
\Bigr\rb\\
-\Bigl\lb\sum_k\sum_l(p_B\circ
  \pr_1)^*\ell_{\alpha_k}(p_B\circ
  \pr_2)^*\ell_{\beta_l}\cdot\left({\tau_{kl}^\dagger}^\times,0\right),
  \left( T\sigma_Q(q_2), T\sigma_Q(q_2)\right)
\Bigr\rb\\
-\Big\lb\sum_i(p_B\circ
  \pr_1)^*\ell_{\xi_i}\left(0,\sigma_Q(q_i)^\times\right),
\left( T\sigma_Q(q_2), T\sigma_Q(q_2)\right)
\Bigr\rb.
\end{multline*}
(The remaining terms all vanish.) This is
\begin{multline*}
  \left(T(\sigma_Q(\lb q_1,q_2\rb)-\widetilde{R(q_1,q_2)}),
    T(\sigma_Q(\lb q_1,q_2\rb)-\widetilde{R(q_1,q_2)})\right)\\
 -\sum_j(p_B\circ \pr_2)^*\ell_{\eta_j}\left(\left(\sigma_Q(\lb q_1,
      r_j\rb)-\widetilde{R(q_1,r_j)}\right)^\times,0\right)\\
  -\sum_j(p_B\circ
  \pr_2)^*\ell_{\nabla_{q_1}^*\eta_j}\left(\sigma_Q(r_j)^\times,0\right)
  \\
  -\sum_s\sum_t(p_B\circ
  \pr_1)^*\ell_{\nabla_{q_1}^*\gamma_s}(p_B\circ
  \pr_2)^*\ell_{\delta_t}\left({\chi_{st}^\dagger}^\times,0\right)
  -\sum_s\sum_t(p_B\circ \pr_1)^*\ell_{\gamma_s}(p_B\circ
  \pr_2)^*\ell_{\nabla_{q_1}^*\delta_t}\left({\chi_{st}^\dagger}^\times,0\right)\\
  -\sum_s\sum_t(p_B\circ \pr_1)^*\ell_{\gamma_s}(p_B\circ
  \pr_2)^*\ell_{\delta_t}\left({\Delta_{q_1}\chi_{st}^\dagger}^\times,0\right)\\
  -\sum_j(p_B\circ \pr_1)^*\ell_{\nabla_{q_1}^*\eta_j}
  \left(0,\sigma_Q(r_j)^\times\right) -\sum_j(p_B\circ
  \pr_1)^*\ell_{\eta_j} \left(0,\left(\sigma_Q(\lb
      q_1,r_j\rb)-\widetilde{R(q_1,r_j)}\right)^\times\right)\\
  +\sum_i(p_B\circ
  \pr_2)^*\ell_{\nabla_{q_2}^*\xi_i}\left(\sigma_Q(q_i)^\times,0\right)+\sum_i(p_B\circ
  \pr_2)^*\ell_{\xi_i}\left(\left(\sigma_Q(\lb
      q_2,q_i\rb)-\widetilde{R(q_2,q_i)}\right)^\times,0\right)\\
  -\mathcal D_{\overline{T\mathbb E}\times T\mathbb E}\left\langle
    \sum_i(p_B\circ \pr_2)^*\ell_{\xi_i}\left(\sigma_Q(q_i)^\times,0\right), \left( T\sigma_Q(q_2), T\sigma_Q(q_2)\right)\right\rangle\\
  +\sum_k\sum_l(p_B\circ \pr_1)^*\ell_{\alpha_k}(p_B\circ
  \pr_2)^*\ell_{\nabla_{q_2}^*\beta_l}\cdot\left({\sigma_{kl}^\dagger}^\times,0\right)
  +\sum_k\sum_l(p_B\circ
  \pr_1)^*\ell_{\nabla_{q_2}^*\alpha_k}(p_B\circ
  \pr_2)^*\ell_{\beta_l}\cdot\left({\sigma_{kl}^\dagger}^\times,0\right)\\
  +\sum_k\sum_l(p_B\circ \pr_1)^*\ell_{\alpha_k}(p_B\circ
  \pr_2)^*\ell_{\beta_l}\cdot\left({\Delta_{q_2}\sigma_{kl}^\dagger}^\times,0\right)\\
  -\mathcal D_{\overline{T\mathbb E}\times T\mathbb
    E}\left\langle\sum_k\sum_l(p_B\circ
    \pr_1)^*\ell_{\alpha_k}(p_B\circ
    \pr_2)^*\ell_{\beta_l}\cdot\left({\sigma_{kl}^\dagger}^\times,0\right),
    \left( T\sigma_Q(q_2), T\sigma_Q(q_2)\right)
  \right\rangle\\
  +\sum_i(p_B\circ \pr_2)^*\ell_{\nabla_{q_2}^*\xi_i}\left(0,\sigma_Q(q_i)^\times\right)+\sum_i(p_B\circ
  \pr_2)^*\ell_{\xi_i}\left(0,\left(\sigma_Q(\lb
        q_2,q_i\rb)-\widetilde{R(q_2,q_i)}\right)^\times\right)\\
  -\mathcal D_{\overline{T\mathbb E}\times T\mathbb
    E}\left\langle\sum_i(p_B\circ \pr_1)^*\ell_{\xi_i}\left(0,\sigma_Q(
        q_i)^\times\right), \left( T\sigma_Q(q_2), T\sigma_Q(q_2)\right)
  \right\rangle.
\end{multline*}

We evaluate this at $(b_1,b_2,b_3)(m)=((Tb_1+([b_1,b_2]+b_3)^\times)(\rho_B(b_2)(m)),
(Tb_2+b_3^\times)(\rho_B(b_1)(m)))$.
We have 
\begin{equation*}
\begin{split}
-\sum_j\langle b_2, \eta_j\rangle\lb q_1,
    r_j\rb
-\sum_j\langle b_2, \nabla_{q_1}^*\eta_j\rangle r_j
&=\sum_j(-\lb q_1,
    \langle b_2, \eta_j\rangle\cdot r_j\rb+(\rho_Q(q_1)\langle b_2,
    \eta_j\rangle
-\langle b_2, \nabla_{q_1}^*\eta_j\rangle)\cdot r_j)\\
&=-\lb q_1, \nabla_{b_2} q_2\rb +\nabla_{\nabla_{q_1}b_2}q_2
\end{split}
\end{equation*}
and in the same manner
\begin{equation*}
\begin{split}
&\sum_s\sum_t\left(\langle \nabla_{q_1}^*\gamma_s, b_1\rangle\langle\delta_t, b_2\rangle\chi_{st}
+\langle b_1, \gamma_s\rangle\langle b_2,\nabla_{q_1}^*\delta_t\rangle\chi_{st}
+\langle b_1,\gamma_s\rangle\langle b_2,
\delta_t\rangle\Delta_{q_1}\chi_{st}\right)\\
=&\sum_s\sum_t\left(-\langle \gamma_s, \nabla_{q_1} b_1\rangle\langle\delta_t, b_2\rangle\chi_{st}
-\langle b_1, \gamma_s\rangle\langle \delta_t, \nabla_{q_1}b_2\rangle\chi_{st}
+\rho_Q(q_1)( \langle \gamma_s, b_1\rangle\cdot\langle\delta_t, b_2\rangle)+\langle b_1,\gamma_s\rangle\langle b_2,
\delta_t\rangle\Delta_{q_1}\chi_{st}\right)\\
=&-R(\nabla_{q_1} b_1, b_2)q_2-R(b_1,\nabla_{q_1}b_2)q_2+\Delta_{q_1}(R(b_1,b_2)q_2).
\end{split}
\end{equation*}

We also have $\mathcal D_{\overline{T\mathbb E}\times T\mathbb
  E}\left\langle \sum_i(p_B\circ
  \pr_2)^*\ell_{\xi_i}\left(\sigma_Q(q_i)^\times,0\right), \left(
    T\sigma_Q(q_2), T\sigma_Q(q_2)\right)\right\rangle=0$ and
\begin{equation*}
\begin{split}
&\mathcal D_{\overline{T\mathbb E}\times T\mathbb E}\left\langle\sum_k\sum_l(p_B\circ
  \pr_1)^*\ell_{\alpha_k}(p_B\circ
  \pr_2)^*\ell_{\beta_l}\cdot\left({\tau_{kl}^\dagger}^\times,0\right),
  \left( T\sigma_Q(q_2), T\sigma_Q(q_2)\right)
\right\rangle\\
&=-\mathcal D_{\overline{T\mathbb E}\times T\mathbb E}\left(\sum_k\sum_l(p_B\circ
  \pr_1)^*\ell_{\alpha_k}\cdot(p_B\circ
  \pr_2)^*\ell_{\beta_l}\cdot(q_B\circ p_B\circ
  \pr_1)^*\langle\tau_{kl}, q_2\rangle
\right)\\
&=\sum_k\sum_l(p_B\circ
  \pr_1)^*\ell_{\alpha_k}\cdot(p_B\circ
  \pr_2)^*\ell_{\beta_l}\cdot\left(\left(\rho_Q^*\dr\langle\tau_{kl},
      q_2\rangle^\dagger\right)^\times,0\right)\\
&+\sum_k\sum_l(p_B\circ
  \pr_2)^*\ell_{\beta_l}\cdot(q_B\circ p_B\circ
  \pr_1)^*\langle\tau_{kl},
  q_2\rangle\cdot\left(\left(\sigma_Q(\partial_B^*\alpha_k)+\widetilde{\nabla_\cdot^*\alpha_k}\right)^\times,0\right)\\
&-\sum_k\sum_l(p_B\circ
  \pr_1)^*\ell_{\alpha_k}\cdot(q_B\circ p_B\circ
  \pr_1)^*\langle\tau_{kl}, q_2\rangle\cdot\left(0,\left(\sigma_Q(\partial_B^*\beta_l)+\nabla_\cdot^*\beta_l^\dagger\right)^\times\right)
\end{split}
\end{equation*}
which is 
\begin{multline*}
\sum_k\sum_l\langle \alpha_k,b_1\rangle\cdot\langle\beta_l,b_2\rangle\cdot\left(\left(\rho_Q^*\dr\langle\tau_{kl},
      q_2\rangle^\dagger\right)^\times,0\right)
+\left(\left(\sigma_Q(\partial_B^*\langle R(\cdot,b_2)q_1,q_2\rangle)-\langle R(\nabla_\cdot b_1,b_2)q_1,q_2\rangle^\dagger\right)^\times,0\right)\\
+\sum_k\sum_l\langle\beta_l,b_2\rangle\cdot\langle\tau_{kl},
  q_2\rangle\cdot\left(\left(\rho_Q^*\dr\langle\alpha_k,b_1\rangle^\dagger\right)^\times,0\right)
-\left(0, \left(\sigma_Q(\partial_B^*\langle R(b_1, \cdot)q_1,q_2\rangle)-\langle R( b_1,\nabla_\cdot b_2)q_1,q_2\rangle^\dagger\right)^\times\right)\\
-\sum_k\sum_l\langle\alpha_k,b_1\rangle\cdot\langle\tau_{kl},
  q_2\rangle\cdot\left(0, \left(\rho_Q^*\dr\langle\beta_l,b_2\rangle^\dagger\right)^\times\right)
\end{multline*}
at $(b_1,b_2,b_3)(m)$.
We evaluate the bracket at $(b_1,b_2,b_3)(m)$ and reorganize the terms to get
\begin{multline*}
  \left(T\sigma_Q(\lb q_1,q_2\rb)-TR(q_1,q_2){b_1}^\dagger-\left(R(q_1,q_2)([{b_1},{b_2}]+c)^\dagger\right)^\times, T\sigma_Q(\lb q_1,q_2\rb)-TR(q_1,q_2){b_2}^\dagger-\left(R(q_1,q_2)c^\dagger\right)^\times\right)\\
  +\Bigl(- \sigma_Q(\lb
    q_1,\nabla_{b_2}q_2\rb)^\times+\sigma_Q(\nabla_{\nabla_{q_1}{b_2}}q_2)^\times+\sigma_Q(\lb
    q_2,\nabla_{b_2}q_1\rb)^\times-\sigma_Q(\nabla_{\nabla_{q_2}{b_2}}q_1)^\times
  -\sigma_Q(\partial_B^*\langle
    R(\cdot,{b_2})q_1,q_2\rangle)^\times,\\
  \qquad- \sigma_Q(\lb
    q_1,\nabla_{b_1}q_2\rb)^\times+\sigma_Q(\nabla_{\nabla_{q_1}{b_1}}q_2)^\times+\sigma_Q(\partial_B^*\langle
    R({b_1}, \cdot)q_1,q_2\rangle)^\times +\sigma_Q(\lb
    q_2,\nabla_{b_1}q_1\rb)^\times-\sigma_Q(\nabla_{\nabla_{q_2}{b_1}}q_1)^\times\Bigr)\\
   +\left((R(q_1,\nabla_{b_2}q_2){b_1}^\dagger)^\times+\left(R(\nabla_{q_1} {b_1}, {b_2})q_2^\dagger+R({b_1},\nabla_{q_1}{b_2})q_2^\dagger-\Delta_{q_1}(R({b_1},{b_2})q_2)^\dagger\right)^\times-(R(q_2,\nabla_{b_2}q_1){b_1}^\dagger)^\times, 0\right)\\
  +\left(0, (R(q_1,\nabla_{b_1}q_2){b_2}^\dagger)^\times\right)+\left(\left(-R(\nabla_{q_2} {b_1}, {b_2})q_1^\dagger-R({b_1},\nabla_{q_2}{b_2})q_1^\dagger+\Delta_{q_2}(R({b_1},{b_2})q_1)^\dagger\right)^\times, 0\right)\\
  -\sum_k\sum_l\langle
  \alpha_k,{b_1}\rangle\cdot\langle\beta_l,{b_2}\rangle\cdot\left(\left(\rho_Q^*\dr\langle\tau_{kl},
      q_2\rangle^\dagger\right)^\times,0\right)-\sum_k\sum_l\langle\beta_l,{b_2}\rangle\cdot\langle\tau_{kl},
  q_2\rangle\cdot\left(\left(\rho_Q^*\dr\langle\alpha_k,{b_1}\rangle^\dagger\right)^\times,0\right)\\
+\left(\left(\langle R(\nabla_\cdot {b_1},{b_2})q_1,q_2\rangle^\dagger\right)^\times,0
\right)
-\left(0, \left(\langle R( {b_1},\nabla_\cdot {b_2})q_1,q_2\rangle^\dagger\right)^\times\right)
-\left(0, \left(R(q_2,\nabla_{b_1}q_1){b_2}^\dagger\right)^\times\right)\\
+\sum_k\sum_l\langle\alpha_k,{b_1}\rangle\cdot\langle\tau_{kl},
q_2\rangle\cdot\left(0,
  \left(\rho_Q^*\dr\langle\beta_l,{b_2}\rangle^\dagger\right)^\times\right)\\
\end{multline*}

This is an element of $\Pi_{\mathbb E}$ if and only if 
\begin{equation*}
\begin{split}
&    -\nabla_{b_1}\lb q_1, q_2\rb-\partial_QR(q_1,q_2){b_1}\\
=&- \lb
    q_1,\nabla_{b_1}q_2\rb+\nabla_{\nabla_{q_1}{b_1}}q_2+\partial_B^*\langle
    R({b_1}, \cdot)q_1,q_2\rangle+\lb
    q_2,\nabla_{b_1}q_1\rb-\nabla_{\nabla_{q_2}{b_1}}q_1\\
    =&\,- \lb
    q_1,\nabla_{b_1}q_2\rb+\nabla_{\nabla_{q_1}{b_1}}q_2+\partial_B^*\langle
    R({b_1}, \cdot)q_1,q_2\rangle-\lb \nabla_{b_1}q_1, q_2\rb-\nabla_{\nabla_{q_2}{b_1}}q_1
\end{split}
\end{equation*}
and 
\begin{multline*}
   -\nabla_{b_1}R(q_1,q_2){b_2}+\nabla_{b_2}R(q_1,q_2){b_1}-R({b_1},{b_2})\lb q_1, q_2\rb-R(q_1,q_2)b_3\\
-\langle R( {b_1},\nabla_\cdot {b_2})q_1,q_2\rangle
-R(q_2,\nabla_{b_1}q_1){b_2}\\
+\sum_k\sum_l\langle\alpha_k,{b_1}\rangle\cdot\langle\tau_{kl},
q_2\rangle\cdot\rho_Q^*\dr\langle\beta_l,{b_2}\rangle+R(q_1,\nabla_{b_1}q_2){b_2}\\
=-R(q_1,q_2)([{b_1},{b_2}]+b_3)
  +R(q_1,\nabla_{b_2}q_2){b_1}+R(\nabla_{q_1} {b_1}, {b_2})q_2+R(b_1,\nabla_{q_1}{b_2})q_2\\
-\Delta_{q_1}(R({b_1},{b_2})q_2)-R(q_2,\nabla_{b_2}q_1){b_1}
\\
 -R(\nabla_{q_2} {b_1}, {b_2})q_1-R({b_1},\nabla_{q_2}{b_2})q_1+\Delta_{q_2}(R({b_1},{b_2})q_1)\\
  -\sum_k\sum_l\langle
  \alpha_k,{b_1}\rangle\cdot\langle\beta_l,{b_2}\rangle\cdot\rho_Q^*\dr\langle\tau_{kl},
      q_2\rangle\\
-\sum_k\sum_l\langle\beta_l,{b_2}\rangle\cdot\langle\tau_{kl},
  q_2\rangle\cdot\rho_Q^*\dr\langle\alpha_k,{b_1}\rangle+\langle
  R(\nabla_\cdot {b_1},{b_2})q_1,q_2\rangle.
\end{multline*}
The last equation can be rewritten
\begin{equation*}
\begin{split}
  &  -\nabla_{b_1}R(q_1,q_2){b_2}+\nabla_{b_2}R(q_1,q_2){b_1}-R({b_1},{b_2})\lb q_1, q_2\rb-\langle R( {b_1},\nabla_\cdot {b_2})q_1,q_2\rangle\\
  &-R(q_2,\nabla_{b_1}q_1){b_2}+R(q_1,\nabla_{b_1}q_2){b_2}\\
  =&-R(q_1,q_2)[{b_1},{b_2}]
  +R(q_1,\nabla_{b_2}q_2){b_1}+R(\nabla_{q_1} {b_1},
  {b_2})q_2+R({b_1},\nabla_{q_1}{b_2})q_2\\
&-\Delta_{q_1}(R({b_1},{b_2})q_2)-R(q_2,\nabla_{b_2}q_1){b_1} -R(\nabla_{q_2} {b_1},
  {b_2})q_1-R({b_1},\nabla_{q_2}{b_2})q_1\\
&+\Delta_{q_2}(R({b_1},{b_2})q_1)  -\rho_Q^*\dr\langle R({b_1},{b_2})q_1,q_2\rangle+\langle
  R(\nabla_\cdot {b_1},{b_2})q_1,q_2\rangle.
\end{split}
\end{equation*}
With (DS6), this is 
\begin{equation*}
\begin{split}
&R(q_1,q_2)[{b_1},{b_2}]-R({b_1},{b_2})\lb q_1, q_2\rb\\
 &  -\nabla_{b_1}R(q_1,q_2){b_2}+R(\nabla_{b_1}q_1, q_2){b_2}+R(q_1,\nabla_{b_1}q_2){b_2}\\
&+\nabla_{b_2}R(q_1,q_2){b_1}-R(\nabla_{b_2}q_1, q_2){b_1}-R(q_1,\nabla_{b_2}q_2){b_1}\\
&+\Delta_{q_1}(R({b_1},{b_2})q_2)-R(\nabla_{q_1} {b_1},
  {b_2})q_2-R({b_1},\nabla_{q_1}{b_2})q_2\\
&-\Delta_{q_2}(R({b_1},{b_2})q_1)+R(\nabla_{q_2} {b_1}, {b_2})q_1+R({b_1},\nabla_{q_2}{b_2})q_1\\
&=-\rho_Q^*\dr\langle R({b_1},{b_2})q_1,q_2\rangle+\langle R(\nabla_\cdot {b_1},{b_2})q_1,q_2\rangle+
\langle R( {b_1},\nabla_\cdot {b_2})q_1,q_2\rangle.
\end{split}
\end{equation*}
\end{proof}

\section{Proof of Theorem \ref{thm_CA_sd}}\label{proof_of_manin_ap}
Note that in the following computations, we will make extensive use of
the identity $\partial_Q=\partial_Q^*$ without always mentioning it.
We begin by proving the two following Lemmas.
\begin{lemma}\label{lemma_looks_like_basic}
  Consider an LA-Courant algebroid $(\mathbb E, Q, B,M)$.  The bracket
  $\lb\cdot\,,\cdot\rb_{Q^*}$ on sections of the core $Q^*$ satisfies
  the following equation:
\begin{equation}\label{bracket_on_Q*_basic}
\begin{split}
  R(\partial_B\tau_1,\partial_B\tau_2)q
  =&-\Delta_q\lb\tau_1,\tau_2\rb_{Q^*}+\lb \Delta_q\tau_1,\tau_2\rb_{Q^*}+\lb\tau_1,\Delta_q\tau_2\rb_{Q^*}\\
  &+\Delta_{\nabla_{\partial_B\tau_2}q}\tau_1-\Delta_{\nabla_{\partial_B\tau_1}q}\tau_2
-\rho_Q^*\dr\langle\tau_1,\nabla_{\partial_B\tau_2}q\rangle
\end{split}
\end{equation}
for all $q\in\Gamma(Q)$ and $\tau_1,\tau_2\in\Gamma(Q^*)$.
\end{lemma}

\begin{proof}
The proof is just a computation using (M1) and \eqref{(LC10)}.
We have 
\begin{equation*}
\begin{split}
  &\Delta_q\lb\tau_1,\tau_2\rb_{Q^*}-\lb
  \Delta_q\tau_1,\tau_2\rb_{Q^*}-\lb\tau_1,\Delta_q\tau_2\rb_{Q^*}
+\Delta_{\nabla_{\partial_B\tau_1}q}\tau_2\\
&-\Delta_{\nabla_{\partial_B\tau_2}q}\tau_1
+\rho_Q^*\dr\langle\tau_1,\nabla_{\partial_B\tau_2}q\rangle\\
  =&\Delta_q\Delta_{\partial_Q\tau_1}\tau_2-\Delta_q\nabla_{\partial_B\tau_2}\tau_1
-\Delta_{\partial_Q(\Delta_q\tau_1)}\tau_2+\nabla_{\partial_B\tau_2}\Delta_q\tau_1-\Delta_{\partial_Q\tau_1}\Delta_q\tau_2\\
  & +\nabla_{\partial_B(\Delta_q\tau_2)}\tau_1
  +\Delta_{\nabla_{\partial_B\tau_1}q}\tau_2-\Delta_{\nabla_{\partial_B\tau_2}q}\tau_1
+\rho_Q^*\dr\langle\tau_1,\nabla_{\partial_B\tau_2}q\rangle
\end{split}
\end{equation*}
Replacing $\Delta_q\Delta_{\partial_Q\tau_1}\tau_2-\Delta_{\partial_Q\tau_1}\Delta_q\tau_2$ 
by $R_\Delta(q,\partial_Q\tau_1)\tau_2+\Delta_{\lb q,\partial_Q\tau_1\rb}\tau_2$
and reordering the terms yields
\begin{equation*}
\begin{split}
  &R_\Delta(q,\partial_Q\tau_1)\tau_2+\Delta_{\lb
    q,\partial_Q\tau_1\rb-\partial_Q(\Delta_q\tau_1)+\nabla_{\partial_B\tau_1}q}\tau_2
  -\Delta_q\nabla_{\partial_B\tau_2}\tau_1+\nabla_{\partial_B\tau_2}\Delta_q\tau_1\\
  & +\nabla_{\partial_B(\Delta_q\tau_2)}\tau_1-\Delta_{\nabla_{\partial_B\tau_2}q}\tau_1
+\rho_Q^*\dr\langle\tau_1,\nabla_{\partial_B\tau_2}q\rangle.
\end{split}
\end{equation*}
Since
$R_\Delta(q,\partial_Q\tau_1)\tau_2=R(q,\partial_Q\tau_1)\partial_B\tau_2$
by (D4), we can now use \eqref{(LC10)} and
$\nabla_q\circ\partial_B=\partial_B\circ\Delta_q$ to replace \[
R_\Delta(q,\partial_Q\tau_1)\tau_2-\Delta_q\nabla_{\partial_B\tau_2}\tau_1
+\nabla_{\partial_B\tau_2}\Delta_q\tau_1-\Delta_{\nabla_{\partial_B\tau_2}q}\tau_1+\nabla_{\nabla_q\partial_B\tau_2}\tau_1\]
by
\[-\langle\nabla_{\nabla_\cdot\partial_B\tau_2 }q, \tau_1\rangle
+R(\partial_B\tau_2,\partial_B\tau_1)q.\]
We use (M1) to replace $\Delta_{\lb
  q, \partial_Q\tau_1\rb-\partial_Q(\Delta_q\tau_1)+\nabla_{\partial_B\tau_1}q}\tau_2$
by $-\Delta_{\partial_B^*\langle \tau_1, \nabla_\cdot
  q\rangle}\tau_2$.  These two steps yield that the right hand side
of our equation is
\begin{equation*}
\begin{split}
  &-\langle\nabla_{\nabla_\cdot\partial_B\tau_2 }q, \tau_1\rangle
+R(\partial_B\tau_2,\partial_B\tau_1)q-\Delta_{\partial_B^*\langle \tau_1, \nabla_\cdot
    q\rangle}\tau_2
  +\rho_Q^*\dr\langle\tau_1,\nabla_{\partial_B\tau_2}q\rangle.
\end{split}
\end{equation*}
To conclude, let us show  that 
\[-\langle\nabla_{\nabla_\cdot\partial_B\tau_2 }q, \tau_1\rangle
-\Delta_{\partial_B^*\langle
  \tau_1, \nabla_\cdot q\rangle}\tau_2
+\rho_Q^*\dr\langle\tau_1,\nabla_{\partial_B\tau_2}q\rangle\in\Gamma(Q^*)
\]
vanishes. On $q'\in\Gamma(Q)$, this is 
\begin{equation*}
\begin{split}
  &-\langle \nabla_{\nabla_{q'}(\partial_B\tau_2)}q,
  \tau_1\rangle+\langle \lb \partial_B^*\langle\nabla_\cdot q,
  \tau_1\rangle, q'\rb,\tau_2\rangle
  +\rho_Q(q')\langle\tau_1, \nabla_{\partial_B\tau_2}q\rangle\\
  =&-\langle \nabla_{\nabla_{q'}(\partial_B\tau_2)}q,
  \tau_1\rangle
  +\langle\Delta_{q'}\tau_2, \partial_B^*(\langle\nabla_\cdot q, \tau_1\rangle)\rangle\\
  =&-\langle \nabla_{\nabla_{q'}(\partial_B\tau_2)}q,
  \tau_1\rangle +\langle\nabla_{\partial_B(\Delta_{q'}\tau_2)} q,
  \tau_1\rangle=0.
\end{split}
\end{equation*}
We have used \eqref{rho_delta}$\rho_Q\circ\partial_B^*=0$ and the
duality of the Dorfman connection with the dull bracket in the first
line, as well as for the first equality. To conclude, we have used
$\partial_B\circ\Delta_{q'}=\nabla_{q'}\circ\partial_B$ by (D1).
\end{proof}

\begin{lemma}\label{bracket_and_pullback}
  The bracket on $Q^*$ satisfies
\begin{equation}\label{bracket_pullback}
\lb \rho_Q^*\dr f,\tau\rb_{Q^*}=0
\end{equation}
for all $f\in C^\infty(M)$ and $\tau\in\Gamma(Q^*)$.
\end{lemma}

\begin{proof}
   We have $\Delta_{q}(\rho_Q^*\dr f)=\rho_Q^*\dr
(\rho_Q(q)f)$ by the definition of a Dorfman connection (Definition \ref{the_def}).

By \eqref{almost_C}, we have $\lb \rho_Q^*\dr
f,\tau\rb_{Q^*}=\nabla_{\partial_B\rho_Q^*\dr
  f}\tau-\Delta_{\partial_Q\tau}(\rho_Q^*\dr
f)+\rho_Q^*\dr((\rho_Q\partial_Q\tau)f)$.  Since
$\partial_B\rho_Q^*=0$ by \eqref{rho_delta}, and
$\Delta_{\partial_Q\tau}(\rho_Q^*\dr
f)=\rho_Q^*\dr(\rho_Q(\partial_Q\tau)(f))$, the bracket $\lb
\rho_Q^*\dr f,\tau\rb_{Q^*}$ is $0$.
\end{proof}

\medskip

Now we check that the bracket $\lb\cdot\,,\cdot\rb_{\mathbb B}$ in
Theorem \ref{thm_CA_sd} is well-defined. We have for all
$\upsilon\in\Gamma(U^\circ)$, $\tau\in\Gamma(Q^*)$ and $u\in\Gamma(U)$:
\begin{equation*}
\begin{split}
  \lb u\oplus \tau, (-\partial_Q\upsilon)\oplus\upsilon\rb&=
  \left(-\lb u,\partial_Q\upsilon\rb_U+\nabla_{\partial_B\tau}(-\partial_Q\upsilon)-\nabla_{\partial_B\upsilon}u\right)\\
  &\quad \oplus
  \left(\lb\tau,\upsilon\rb_{Q^*}+\Delta_u\upsilon-\Delta_{-\partial_Q\upsilon}\tau+\rho_Q^*\dr\langle\tau,-\partial_Q\upsilon\rangle\right).
\end{split}
\end{equation*}
By (M1), the properties of $2$-representations and \eqref{almost_C}, this is 
\begin{equation*}
\begin{split}
  &\left(-\partial_Q(\Delta_u\upsilon)+\partial_B^*\langle\upsilon,\nabla_\cdot
    u\rangle
    +\cancel{\nabla_{\partial_B\upsilon}u}-\partial_Q\nabla_{\partial_B\tau}\upsilon-\cancel{\nabla_{\partial_B\upsilon}u}\right)\\
  &\quad \oplus
  \left(-\cancel{\Delta_{\partial_Q\upsilon}\tau}+\nabla_{\partial_B\tau}\upsilon
    +\cancel{\rho_Q^*\dr\langle\tau,\partial_Q\upsilon\rangle}
    +\Delta_u\upsilon+\cancel{\Delta_{\partial_Q\upsilon}\tau}-\cancel{\rho_Q^*\dr\langle\tau,\partial_Q\upsilon\rangle}\right)\\
  =&\left(-\partial_Q(\Delta_u\upsilon)+\partial_B^*\langle\upsilon,\nabla_\cdot
    u\rangle-\partial_Q\nabla_{\partial_B\tau}\upsilon\right) \oplus
  \left(\nabla_{\partial_B\tau}\upsilon+\Delta_u\upsilon\right).
\end{split}
\end{equation*}
Since $\upsilon\in\Gamma(U^\circ)$ and $\nabla_b$ preserves $\Gamma(U)$
for all $b\in\Gamma(B)$, the section $\langle\upsilon,\nabla_\cdot
u\rangle$ of $B^*$ vanishes and we get
\begin{equation*}
\begin{split}
\lb u\oplus \tau, (-\partial_Q\upsilon)\oplus\upsilon\rb
=&\left(-\partial_Q(\Delta_u\upsilon+\nabla_{\partial_B\tau}\upsilon)\right) \oplus \left(\nabla_{\partial_B\tau}\upsilon+\Delta_u\upsilon\right).
\end{split}
\end{equation*}
Because $\Delta_u$ preserves as well $\Gamma(U^\circ)$, the sum
$\nabla_{\partial_B\tau}\upsilon+\Delta_u\upsilon$ is a section of
$U^\circ$, and so $\lb u\oplus \tau, (-\partial_Q\upsilon)\oplus\upsilon\rb$
is zero in $\mathbb B$.  \medskip 

We now check the Courant algebroid
axioms (CA1), (CA2) and (CA3).  The last one, (CA3), is immediate:
\begin{equation*}
\begin{split}
  &\lb u_1\oplus\tau_1, u_2\oplus\tau_2\rb_{\mathbb{B}}+\lb
  u_2\oplus\tau_2, u_1\oplus\tau_1\rb_{\mathbb{B}}\\
  &=0\oplus
  \left(\rho_Q^*\dr\langle\tau_1,\partial_Q\tau_2\rangle+\rho_Q^*\dr\langle\tau_1,u_2\rangle
    +\rho_Q^*\dr\langle\tau_2,u_1\rangle\right)
  =0\oplus\rho_Q^*\dr\langle u_1\oplus\tau_1, u_2\oplus\tau_2\rangle_{\mathbb B}\\
  &=\mathcal D_{\mathbb B}\langle u_1\oplus\tau_1,
  u_2\oplus\tau_2\rangle_{\mathbb B}.
\end{split}
\end{equation*}

Next we prove (CA2). We have, using \eqref{almost_C} to replace $\lb \tau_1,
\tau_2\rb_{Q^*}$ by
$-\Delta_{\partial_Q\tau_2}\tau_1+\nabla_{\partial_B\tau_1}\tau_2+\rho_Q^*\dr\langle\tau_1,\partial_Q\tau_2\rangle$:
\begin{equation*}\label{metric}
\begin{split}
  &\langle \lb u_1\oplus\tau_1, u_2\oplus\tau_2\rb, u_3\oplus\tau_3\rangle_{\mathbb B}\\
  &=\langle \lb u_1,
  u_2\rb_U+\nabla_{\partial_B\tau_1}u_2-\nabla_{\partial_B\tau_2}u_1, \tau_3\rangle\\
  &+\langle
  -\Delta_{\partial_Q\tau_2}\tau_1+\nabla_{\partial_B\tau_1}\tau_2+\rho_Q^*\dr\langle\tau_1,\partial_Q\tau_2\rangle+\Delta_{u_1}\tau_2-\Delta_{u_2}\tau_1+\rho_Q^*\dr\langle
  \tau_1, u_2\rangle, u_3+\partial_Q\tau_3\rangle\\
  &=\rho_Q(u_1)\langle u_2,\tau_3\rangle-\langle u_2, \Delta_{u_1}\tau_3\rangle+\langle\nabla_{\partial_B\tau_1}u_2-\nabla_{\partial_B\tau_2}u_1,\tau_3\rangle\\
  &-\langle\Delta_{\partial_Q\tau_2}\tau_1,u_3\rangle+\langle\nabla_{\partial_B\tau_1}\tau_2,
  u_3\rangle+\rho_Q(u_3)\langle\tau_1,\partial_Q\tau_2\rangle
  +\langle\Delta_{u_1}\tau_2, u_3\rangle-\langle\Delta_{u_2}\tau_1, u_3\rangle\\
  &+\rho_Q(u_3)\langle\tau_1,u_2\rangle-\langle\Delta_{\partial_Q\tau_2}\tau_1,\partial_Q\tau_3\rangle+\langle\nabla_{\partial_B\tau_1}\tau_2,\partial_Q\tau_3\rangle+\rho_Q(\partial_Q\tau_3)\langle\tau_1, \partial_Q\tau_2\rangle \\
  &+\langle\Delta_{u_1}\tau_2,\partial_Q\tau_3\rangle-\langle\Delta_{u_2}\tau_1,\partial_Q\tau_3\rangle+\rho_Q(\partial_Q\tau_3)\langle\tau_1,u_2\rangle.
\end{split}
\end{equation*}
We sum $\langle \lb u_1\oplus\tau_1, u_2\oplus\tau_2\rb,
u_3\oplus\tau_3\rangle_{\mathbb B}$ with $\langle u_2\oplus\tau_2,
\lb u_1\oplus\tau_1, u_3\oplus\tau_3\rb \rangle_{\mathbb B}$, and 
replace only in the first summand the term
$\langle\Delta_{u_1}\tau_2,\partial_Q\tau_3\rangle$ by
$\rho_Q(u_1)\langle\tau_2,\partial_Q\tau_3\rangle-\langle\tau_2,\lb
u_1,\partial_Q\tau_3\rb\rangle$.  This yields
\begin{equation*}
\begin{split}
  &\langle \lb u_1\oplus\tau_1, u_2\oplus\tau_2\rb, u_3\oplus\tau_3\rangle_{\mathbb B}+\langle u_2\oplus\tau_2, \lb u_1\oplus\tau_1, u_3\oplus\tau_3\rb \rangle_{\mathbb B}\\
  &=\rho_Q(u_1)\langle u_2,\tau_3\rangle-\cancel{\langle u_2, \Delta_{u_1}\tau_3\rangle}+\langle\nabla_{\partial_B\tau_1}u_2-\nabla_{\partial_B\tau_2}u_1,\tau_3\rangle\\
  &-\langle\Delta_{\partial_Q\tau_2}\tau_1,u_3\rangle+\langle\nabla_{\partial_B\tau_1}\tau_2,
  u_3\rangle+\rho_Q(u_3)\langle\tau_1,\partial_Q\tau_2\rangle
  +\cancel{\langle\Delta_{u_1}\tau_2, u_3\rangle}-\langle\Delta_{u_2}\tau_1, u_3\rangle\\
  &+\rho_Q(u_3)\langle\tau_1,u_2\rangle-\langle\Delta_{\partial_Q\tau_2}\tau_1,\partial_Q\tau_3\rangle+\langle\nabla_{\partial_B\tau_1}\tau_2,\partial_Q\tau_3\rangle+\rho_Q(\partial_Q\tau_3)\langle\tau_1, \partial_Q\tau_2\rangle\\
  & +\rho_Q(u_1)\langle\tau_2,\partial_Q\tau_3\rangle-\langle\tau_2,\lb u_1,\partial_Q\tau_3\rb\rangle-\langle\Delta_{u_2}\tau_1,\partial_Q\tau_3\rangle+\rho_Q(\partial_Q\tau_3)\langle\tau_1,u_2\rangle\\
  &+\rho_Q(u_1)\langle u_3,\tau_2\rangle-\cancel{\langle u_3, \Delta_{u_1}\tau_2\rangle}+\langle\nabla_{\partial_B\tau_1}u_3-\nabla_{\partial_B\tau_3}u_1,\tau_2\rangle\\
  &-\langle\Delta_{\partial_Q\tau_3}\tau_1,u_2\rangle+\langle\nabla_{\partial_B\tau_1}\tau_3,
  u_2\rangle+\rho_Q(u_2)\langle\tau_1,\partial_Q\tau_3\rangle
  +\cancel{\langle\Delta_{u_1}\tau_3, u_2\rangle}-\langle\Delta_{u_3}\tau_1, u_2\rangle\\
  &+\rho_Q(u_2)\langle\tau_1,u_3\rangle-\langle\Delta_{\partial_Q\tau_3}\tau_1,\partial_Q\tau_2\rangle+\langle\nabla_{\partial_B\tau_1}\tau_3,\partial_Q\tau_2\rangle+\rho_Q(\partial_Q\tau_2)\langle\tau_1, \partial_Q\tau_3\rangle\\
  &+\langle\Delta_{u_1}\tau_3,\partial_Q\tau_2\rangle-\langle\Delta_{u_3}\tau_1,\partial_Q\tau_2\rangle+\rho_Q(\partial_Q\tau_2)\langle\tau_1,u_3\rangle.
\end{split}
\end{equation*}
% Four terms cancel each other pairwise: $\langle\Delta_{u_1}\tau_2,
% u_3\rangle$ with $-\langle u_3, \Delta_{u_1}\tau_2\rangle$, and
% $-\langle u_2, \Delta_{u_1}\tau_3\rangle$ with
% $\langle\Delta_{u_1}\tau_3, u_2\rangle$.  
We reorder the remaining
terms and replace eight times sums like
$\rho_Q(\partial_Q\tau_2)\langle\tau_1,u_3\rangle-\langle\Delta_{\partial_Q\tau_2}\tau_1,
u_3\rangle$ by $\langle\lb \partial_Q\tau_2, u_3\rb,
\tau_1\rangle$, and, using ${\nabla^Q}^*=\nabla^{Q^*}$, three times sums like
$\langle\nabla_{\partial_B\tau_1}u_2,\tau_3\rangle+\langle
u_2,\nabla_{\partial_B\tau_1}\tau_3\rangle$ by
$\rho_B(\partial_B\tau_1)\langle
u_2,\tau_3\rangle$.  This leads to
\begin{equation*}
\begin{split}
  &\langle \lb u_1\oplus\tau_1, u_2\oplus\tau_2\rb, u_3\oplus\tau_3\rangle_{\mathbb B}+\langle u_2\oplus\tau_2, \lb u_1\oplus\tau_1, u_3\oplus\tau_3\rb \rangle_{\mathbb B}\\
  &=\rho_Q(u_1)\langle u_2\oplus\tau_2,u_3\oplus\tau_3\rangle_{\mathbb B}\\
&+\cancel{\langle \lb \partial_Q\tau_2, \partial_Q\tau_3\rb, \tau_1\rangle}
+\cancel{\langle \lb \partial_Q\tau_2, u_3\rb, \tau_1\rangle}
+\cancel{\langle \lb u_3, \partial_Q\tau_2\rb, \tau_1\rangle}
+\cancel{\langle \lb u_2, u_3\rb, \tau_1\rangle}\\
&+\cancel{\langle\lb u_3,u_2\rb, \tau_1\rangle}
+\cancel{\langle\lb \partial_Q\tau_3, u_2\rb, \tau_1\rangle}
+\cancel{\langle\lb u_2, \partial_Q\tau_3\rb,\tau_1\rangle}
+\cancel{\langle\lb\partial_Q\tau_3, \partial_Q\tau_2\rb,\tau_1\rangle}\\
&+\rho_B(\partial_B\tau_1)\langle u_2,\tau_3\rangle
+\rho_B(\partial_B\tau_1)\langle\tau_3,\partial_Q\tau_2\rangle
+\rho_B(\partial_B\tau_1)\langle u_3,\tau_2\rangle\\
&-\langle\nabla_{\partial_B\tau_2}u_1,\tau_3\rangle
  -\langle\tau_2,\lb u_1,\partial_Q\tau_3\rb\rangle-\langle\nabla_{\partial_B\tau_3}u_1,\tau_2\rangle
  +\langle\Delta_{u_1}\tau_3,\partial_Q\tau_2\rangle.
\end{split}
\end{equation*}
Because $\lb\cdot\,,\cdot\rb$ is skew-symmetric on $\Gamma(U)$,
$\langle \lb u_2, u_3\rb, \tau_1\rangle$ and $\langle\lb u_3,u_2\rb,
\tau_1\rangle$ cancel each other, etc. 
The four last terms cancel each other by (M1) and $\partial_Q=\partial_Q^*$.  This
yields
\begin{equation*}
\begin{split}
  &\langle \lb u_1\oplus\tau_1, u_2\oplus\tau_2\rb, u_3\oplus\tau_3\rangle_{\mathbb B}+\langle u_2\oplus\tau_2, \lb u_1\oplus\tau_1, u_3\oplus\tau_3\rb \rangle_{\mathbb B}\\
  &=(\rho_Q(u_1)+\rho_B\partial_B\tau_1)\langle
  u_2\oplus\tau_2,u_3\oplus\tau_3\rangle_{\mathbb B}=\rho_{\mathbb
    B}(u_1\oplus\tau_1)\langle
  u_2\oplus\tau_2,u_3\oplus\tau_3\rangle_{\mathbb B}.
\end{split}
\end{equation*}

\medskip Finally we check the Jacobi identity in Leibniz form
(CA1). We will check that
\[\operatorname{Jac}_{\lb\cdot\,,\cdot\rb}(u_1\oplus\tau_1,
u_2\oplus\tau_2, u_3\oplus\tau_3)=(-\partial_Q\upsilon)\oplus \upsilon
\]
with \[\upsilon=(R(\partial_B\tau_1,\partial_B\tau_2)u_3-R(u_1,u_2)\partial_B\tau_3)+\text{cyclic
  permutations}.
\] 
Since by Proposition \ref{char_sub_alg} $R(u_1,u_2)$ has image in
$U^\circ$ for all $u_1,u_2\in\Gamma(U)$ and $R(b_1,b_2)$ restricts to
a morphism $U\to U^\circ$ for all $b_1,b_2\in\Gamma(B)$ by Proposition
\ref{char_VB_dir}, $\upsilon$ is a section of $U^\circ$ and so
$\operatorname{Jac}_{\lb\cdot\,,\cdot\rb}(u_1\oplus\tau_1,
u_2\oplus\tau_2, u_3\oplus\tau_3)$ will be zero in $\mathbb B$.

We compute 
\begin{equation*}
\begin{split}
  &\lb\lb u_1\oplus\tau_1, u_2\oplus\tau_2\rb, u_3\oplus\tau_3\rb\\
  &=\bigl\lb(\lb u_1,
  u_2\rb_U+\nabla_{\partial_B\tau_1}u_2-\nabla_{\partial_B\tau_2}u_1
  ) \oplus (\lb \tau_1,
  \tau_2\rb_{Q^*}+\Delta_{u_1}\tau_2-\Delta_{u_2}\tau_1+\rho_Q^*\dr\langle
  \tau_1, u_2\rangle
  ), u_3\oplus\tau_3\bigr\rb\\
  &=\Bigl( \lb \lb u_1,
  u_2\rb_U, u_3\rb_U+\lb \nabla_{\partial_B\tau_1}u_2-\nabla_{\partial_B\tau_2}u_1, u_3\rb_U\\
  &\qquad+ \nabla_{\partial_B(\lb
    \tau_1,\tau_2\rb_{Q^*}+\Delta_{u_1}\tau_2-\Delta_{u_2}\tau_1+\rho_Q^*\dr\langle\tau_1,
    u_2\rangle )}u_3 -\nabla_{\partial_B\tau_3}(\lb u_1,
  u_2\rb_U+\nabla_{\partial_B\tau_1}u_2-\nabla_{\partial_B\tau_2}u_1)
  \Bigr)\\
  &\oplus\Bigl(\lb\lb \tau_1,
  \tau_2\rb_{Q^*}, \tau_3\rb_{Q^*}+\lb\Delta_{u_1}\tau_2-\Delta_{u_2}\tau_1+\rho_Q^*\dr\langle\tau_1, u_2\rangle, \tau_3 \rb_{Q^*}\\
  &\qquad+\Delta_{\lb u_1,
    u_2\rb_U+\nabla_{\partial_B\tau_1}u_2-\nabla_{\partial_B\tau_2}u_1}\tau_3-\Delta_{u_3}(\lb
  \tau_1,
  \tau_2\rb_{Q^*}+\Delta_{u_1}\tau_2-\Delta_{u_2}\tau_1+\rho_Q^*\dr\langle
  \tau_1, u_2\rangle)\\
  &\qquad+\rho_Q^*\dr\left\langle\lb \tau_1,
  \tau_2\rb_{Q^*}+\Delta_{u_1}\tau_2-\Delta_{u_2}\tau_1+\rho_Q^*\dr\langle
  \tau_1, u_2\rangle, u_3 \right\rangle \Bigr).
\end{split}
\end{equation*}
Using $\partial_B\lb
    \tau_1,\tau_2\rb_{Q^*}=[\partial_B\tau_1,\partial_B\tau_2]$,
   $\Delta_q(\rho_Q^*\dr f)=\rho_Q^*\dr(\rho_Q(q)(f))$, \eqref{bracket_pullback}
and \eqref{rho_delta}, this is
\begin{equation*}
\begin{split}
  &\Bigl( \lb \lb u_1,
  u_2\rb_U, u_3\rb_U+\lb \nabla_{\partial_B\tau_1}u_2-\nabla_{\partial_B\tau_2}u_1, u_3\rb_U\\
  &\qquad+ \nabla_{[\partial_B
    \tau_1,\partial_B\tau_2]+\partial_B(\Delta_{u_1}\tau_2-\Delta_{u_2}\tau_1)}u_3 -\nabla_{\partial_B\tau_3}(\lb u_1,
  u_2\rb_U+\nabla_{\partial_B\tau_1}u_2-\nabla_{\partial_B\tau_2}u_1)
  \Bigr)\\
  &\oplus\Bigl(\lb\lb \tau_1,
  \tau_2\rb_{Q^*}, \tau_3\rb_{Q^*}+\lb\Delta_{u_1}\tau_2-\Delta_{u_2}\tau_1, \tau_3 \rb_{Q^*}\\
  &\qquad+\Delta_{\lb u_1,
    u_2\rb_U+\nabla_{\partial_B\tau_1}u_2-\nabla_{\partial_B\tau_2}u_1}\tau_3-\Delta_{u_3}(\lb
  \tau_1,
  \tau_2\rb_{Q^*}+\Delta_{u_1}\tau_2-\Delta_{u_2}\tau_1)\\
  &\qquad+\rho_Q^*\dr\left\langle\lb \tau_1,
  \tau_2\rb_{Q^*}+\Delta_{u_1}\tau_2-\Delta_{u_2}\tau_1, u_3 \right\rangle \Bigr).
\end{split}
\end{equation*}
In the same manner, we compute 
\begin{equation*}
\begin{split}
  &\lb u_1\oplus\tau_1,\lb u_2\oplus\tau_2,
  u_3\oplus\tau_3\rb\rb\\
  &=\bigl\lb u_1\oplus\tau_1, (\lb u_2,
  u_3\rb_U+\nabla_{\partial_B\tau_2}u_3-\nabla_{\partial_B\tau_3}u_2
  ) \oplus (\lb \tau_2,
  \tau_3\rb_{Q^*}+\Delta_{u_2}\tau_3-\Delta_{u_3}\tau_2+\rho_Q^*\dr\langle
  \tau_2, u_3\rangle
  )\bigr\rb\\
  &=\Bigl(\lb u_1, \lb u_2, u_3\rb_U\rb_U+\lb u_1,
  \nabla_{\partial_B\tau_2}u_3-\nabla_{\partial_B\tau_3}u_2\rb_U\\
&\qquad+\nabla_{\partial_B\tau_1}(\lb u_2,
  u_3\rb_U+\nabla_{\partial_B\tau_2}u_3-\nabla_{\partial_B\tau_3}u_2)-\nabla_{[ \partial_B\tau_2,
  \partial_B\tau_3]+\partial_B(\Delta_{u_2}\tau_3-\Delta_{u_3}\tau_2)
 }u_1\Bigr)\\
&\oplus\Bigl(\lb \tau_1, \lb \tau_2,
  \tau_3\rb_{Q^*}\rb_{Q^*}+\lb\tau_1, \Delta_{u_2}\tau_3-\Delta_{u_3}\tau_2\rb_{Q^*}+\rho_Q^*\dr(\rho_Q(\partial_Q\tau_1)\langle
  \tau_2, u_3\rangle)\\
&\qquad +\Delta_{u_1}(\lb \tau_2,
  \tau_3\rb_{Q^*}+\Delta_{u_2}\tau_3-\Delta_{u_3}\tau_2)+\rho_Q^*\dr(\rho_Q(u_1)\langle
  \tau_2, u_3\rangle)-\Delta_{ \lb u_2,
  u_3\rb_U+\nabla_{\partial_B\tau_2}u_3-\nabla_{\partial_B\tau_3}u_2
  }\tau_1\\
&\qquad +\rho_Q^*\dr \langle\tau_1, \lb u_2,
  u_3\rb_U+\nabla_{\partial_B\tau_2}u_3-\nabla_{\partial_B\tau_3}u_2
\rangle 
\Bigr)
\end{split}
\end{equation*}
Since $\operatorname{Jac}_{\lb\cdot\,,\cdot\rb_U}(u_1,u_2,u_3)=0$ and
$\partial_B\circ\Delta_q=\nabla_q\circ\partial_B$ for all $q\in\Gamma(Q)$, the
$U$-term of
$\operatorname{Jac}_{\lb\cdot\,,\cdot\rb}(u_1\oplus\tau_1,u_2\oplus\tau_2,u_3\oplus\tau_3)$
equals
\begin{equation*}
\begin{split}
&\lb \nabla_{\partial_B\tau_1}u_2-\nabla_{\partial_B\tau_2}u_1, u_3\rb_U\\
  &+ \nabla_{[\partial_B
    \tau_1,\partial_B\tau_2]+\nabla_{u_1}\partial_B\tau_2-\nabla_{u_2}\partial_B\tau_1}u_3 -\nabla_{\partial_B\tau_3}(\lb u_1,
  u_2\rb_U+\nabla_{\partial_B\tau_1}u_2-\nabla_{\partial_B\tau_2}u_1)\\
&+\lb u_2,
  \nabla_{\partial_B\tau_1}u_3-\nabla_{\partial_B\tau_3}u_1\rb_U\\
&+\nabla_{\partial_B\tau_2}(\lb u_1,
  u_3\rb_U+\nabla_{\partial_B\tau_1}u_3-\nabla_{\partial_B\tau_3}u_1)-\nabla_{[ \partial_B\tau_1,
  \partial_B\tau_3]+\nabla_{u_1}\partial_B\tau_3-\nabla_{u_3}\partial_B\tau_1
 }u_2\\
& -\lb u_1,
  \nabla_{\partial_B\tau_2}u_3-\nabla_{\partial_B\tau_3}u_2\rb_U\\
&-\nabla_{\partial_B\tau_1}(\lb u_2,
  u_3\rb_U+\nabla_{\partial_B\tau_2}u_3-\nabla_{\partial_B\tau_3}u_2)+\nabla_{[ \partial_B\tau_2,
  \partial_B\tau_3]+\nabla_{u_2}\partial_B\tau_3-\nabla_{u_3}\partial_B\tau_2
 }u_1.
\end{split}
\end{equation*}
Note that since for any $b_1,b_2\in\Gamma(B)$,  $R(b_1,b_2)$ restricts to a section
of $\operatorname{Hom}(U,U^\circ)$ (see \S\ref{LA-Dirac}), the
last summand on the right hand side of (M4) vanishes on sections of $U$.
By sorting out the terms and using (M4) on sections of $U$, we get
\begin{equation*}
\begin{split}
  &-R_\nabla(\partial_B\tau_3, \partial_B\tau_1)u_2-R_\nabla(\partial_B\tau_1,\partial_B\tau_2)u_3
  -R_\nabla(\partial_B\tau_2,\partial_B\tau_3)u_1\\
&+\partial_QR(u_1,u_2)\partial_B\tau_3+\partial_QR(u_2,u_3)\partial_B\tau_1+\partial_QR(u_3,u_1)\partial_B\tau_2.
\end{split}
\end{equation*}
Since $R_\nabla=\partial_Q\circ R$, this is $-\partial_Q\upsilon$. 
We conclude by computing the $Q^*$-part of
$\operatorname{Jac}_{\lb\cdot\,,\cdot\rb}(u_1\oplus\tau_1,u_2\oplus\tau_2,u_3\oplus\tau_3)$.
Again, because
$\operatorname{Jac}_{\lb\cdot\,,\cdot\rb_{Q^*}}(\tau_1,\tau_2,\tau_3)=0$,
we get
\begin{equation*}
\begin{split}
  &\lb\Delta_{u_1}\tau_2-\Delta_{u_2}\tau_1, \tau_3
  \rb_{Q^*}+\Delta_{\lb u_1,
    u_2\rb_U+\nabla_{\partial_B\tau_1}u_2-\nabla_{\partial_B\tau_2}u_1}\tau_3-\Delta_{u_3}(\lb
  \tau_1,
  \tau_2\rb_{Q^*}+\Delta_{u_1}\tau_2-\Delta_{u_2}\tau_1)\\
  &+\rho_Q^*\dr\left\langle\lb \tau_1,
    \tau_2\rb_{Q^*}+\cancel{\Delta_{u_1}\tau_2}-\cancel{\Delta_{u_2}\tau_1}, u_3
  \right\rangle+\lb\tau_2,
  \Delta_{u_1}\tau_3-\Delta_{u_3}\tau_1\rb_{Q^*}+\rho_Q^*\dr(\rho_Q(\partial_Q\tau_2)\langle
  \tau_1, u_3\rangle)\\
  &+\Delta_{u_2}(\lb \tau_1,
  \tau_3\rb_{Q^*}+\Delta_{u_1}\tau_3-\Delta_{u_3}\tau_1)+\cancel{\rho_Q^*\dr(\rho_Q(u_2)\langle
  \tau_1, u_3\rangle)}-\Delta_{ \lb u_1,
    u_3\rb_U+\nabla_{\partial_B\tau_1}u_3-\nabla_{\partial_B\tau_3}u_1
  }\tau_2\\
  &+\rho_Q^*\dr \langle\tau_2, \cancel{\lb u_1,
  u_3\rb_U}+\nabla_{\partial_B\tau_1}u_3-\nabla_{\partial_B\tau_3}u_1
  \rangle -\lb\tau_1,
  \Delta_{u_2}\tau_3-\Delta_{u_3}\tau_2\rb_{Q^*}-\rho_Q^*\dr(\rho_Q(\partial_Q\tau_1)\langle
  \tau_2, u_3\rangle)\\
  &-\Delta_{u_1}(\lb \tau_2,
  \tau_3\rb_{Q^*}+\Delta_{u_2}\tau_3-\Delta_{u_3}\tau_2)-\cancel{\rho_Q^*\dr(\rho_Q(u_1)\langle
  \tau_2, u_3\rangle)}+\Delta_{ \lb u_2,
    u_3\rb_U+\nabla_{\partial_B\tau_2}u_3-\nabla_{\partial_B\tau_3}u_2
  }\tau_1\\
  &-\rho_Q^*\dr \langle\tau_1, \cancel{\lb u_2,
  u_3\rb_U}+\nabla_{\partial_B\tau_2}u_3-\nabla_{\partial_B\tau_3}u_2
  \rangle
 \end{split}
\end{equation*}
The six cancelling terms cancel by the duality of the dull bracket and
the Dorfman connection.  Reordering the terms, we get using Lemma
\ref{lemma_looks_like_basic}:
\begin{equation*}
\begin{split}
&-R_\Delta(u_3,
u_1)\tau_2-R_\Delta(u_2,u_3)\tau_1-R_\Delta(u_1,u_2)\tau_3\\
&+R(\partial_B\tau_2,\partial_B\tau_3)u_1+R(\partial_B\tau_1,\partial_B\tau_2)u_3
-R(\partial_B\tau_1,\partial_B\tau_3)u_2\\
&+\cancel{\rho_Q^*\dr\langle\tau_2,\nabla_{\partial_B\tau_3}u_1\rangle}
+\cancel{\rho_Q^*\dr\langle\tau_1,\nabla_{\partial_B\tau_2}u_3\rangle}-\cancel{\rho_Q^*\dr\langle\tau_1,\nabla_{\partial_B\tau_3}u_2\rangle}\\
&-\rho_Q^*\dr\langle
\tau_2,\partial_Q\Delta_{u_3}\tau_1\rangle
+\rho_Q^*\dr\left\langle\lb \tau_1,
  \tau_2\rb_{Q^*}, u_3\right\rangle+\rho_Q^*\dr(\rho_Q(\partial_Q\tau_2)\langle
  \tau_1, u_3\rangle)\\
&+\rho_Q^*\dr \langle\tau_2, \nabla_{\partial_B\tau_1}u_3-\cancel{\nabla_{\partial_B\tau_3}u_1}
\rangle -\rho_Q^*\dr(\rho_Q(\partial_Q\tau_1)\langle
  \tau_2, u_3\rangle)-\rho_Q^*\dr \langle\tau_1, \cancel{\nabla_{\partial_B\tau_2}u_3}-\cancel{\nabla_{\partial_B\tau_3}u_2}
\rangle.
 \end{split}
\end{equation*}
For the third use of Lemma \ref{lemma_looks_like_basic}, we had to
replace $-\lb \tau_2,\Delta_{u_3}\tau_1\rb_{Q^*}$ by $\lb
\Delta_{u_3}\tau_1,\tau_2\rb_{Q^*}-\rho_Q^*\dr\langle
\tau_2, \partial_Q\Delta_{u_3}\tau_1\rangle$. This is why we get the
first term on the fourth line. Six terms cancel immediately
pairwise. Using $R_\Delta=R\circ\partial_B$, we get
\[\upsilon+\rho_Q^*\dr f\]
with $f\in C^\infty(M)$ defined by  
\begin{equation*}
\begin{split}
  f=&-\langle
  \tau_2,\partial_Q\Delta_{u_3}\tau_1\rangle+\left\langle\Delta_{\partial_Q\tau_1}\tau_2-\nabla_{\partial_B\tau_2}\tau_1,
    u_3\right\rangle+\rho_Q(\partial_Q\tau_2)\langle
  \tau_1, u_3\rangle\\
  &+\langle\tau_2, \nabla_{\partial_B\tau_1}u_3 \rangle
  -\rho_Q(\partial_Q\tau_1)\langle
  \tau_2, u_3\rangle\\
  =&\langle \tau_2,
  -\partial_Q\Delta_{u_3}\tau_1+\lb\partial_Q\tau_1,u_3\rb+\nabla_{\partial_B\tau_1}u_3\rangle+\rho_Q(\partial_Q\tau_2)\langle
  \tau_1, u_3\rangle-\langle\nabla_{\partial_B\tau_2}\tau_1,
  u_3\rangle.
 \end{split}
\end{equation*}
By (M1), the first pairing equals 
\[-\langle\tau_1,\nabla_{\partial_B\tau_2}u_3\rangle.
\]
Hence, we find using ${\nabla^{Q}}^*=\nabla^{Q^*}$ and
$\rho_B\circ\partial_B=\rho_Q\circ\partial_Q$ that $f=0$, and so
\[\operatorname{Jac}_{\lb\cdot\,,\cdot\rb}(u_1\oplus\tau_1,u_2\oplus\tau_2,u_3\oplus\tau_3)=(-\partial_Q\upsilon)\oplus\upsilon.
\]

\def\cprime{$'$} \def\polhk#1{\setbox0=\hbox{#1}{\ooalign{\hidewidth
  \lower1.5ex\hbox{`}\hidewidth\crcr\unhbox0}}} \def\cprime{$'$}
\providecommand{\bysame}{\leavevmode\hbox to3em{\hrulefill}\thinspace}
\providecommand{\MR}{\relax\ifhmode\unskip\space\fi MR }
% \MRhref is called by the amsart/book/proc definition of \MR.
\providecommand{\MRhref}[2]{%
  \href{http://www.ams.org/mathscinet-getitem?mr=#1}{#2}
}
\providecommand{\href}[2]{#2}

 \end{document}